\documentclass[AFST]{cedram}
\usepackage{stackrel}
\usepackage{mathrsfs}
\usepackage{eucal}
\usepackage{framed}
\usepackage{pstricks}
\newcommand{\Hilbert}{\mathcal{H}}
\newcommand{\Nat}{\mathbb{N}}
\newcommand{\Int}{\mathbb{Z}}
\newcommand{\Real}{\mathbb{R}}
\newcommand{\Com}{\mathbb{C}}
\newcommand{\Sphere}{\mathbb{S}}
\newcommand{\sii}{L^2}
\newcommand{\dom}{\mathop{\mathrm{dom}}\nolimits}
\newcommand{\ran}{\mathop{\mathrm{ran}}\nolimits}
\newcommand{\supp}{\mathop{\mathrm{supp}}\nolimits}
\newcommand{\der}{\mathrm{d}}
\newcommand{\dist}{\mathop{\mathrm{dist}}\nolimits}
\newcommand{\obal}{\mathop{\mathrm{span}}\nolimits}
\newcommand{\const}{\mathrm{const}}    
\newcommand{\eps}{\varepsilon}
\newcommand{\interior}{\mathrm{int}}
\newcommand{\esssup}{\mathop{\mathrm{ess\;\!sup}}}
\newcommand{\diag}{\mathop{\mathrm{diag}}\nolimits}
\newcommand{\divergence}{\mathop{\mathrm{div}}\nolimits}
\newcommand{\cf}{\emph{cf}}
\newcommand{\ie}{\emph{i.e.}}
\newcommand{\eg}{\emph{e.g.}}
\newcommand{\etc}{\emph{etc}}

\newcommand{\vertiii}[1]{{\left\vert\kern-0.25ex\left\vert\kern-0.25ex\left\vert #1 
    \right\vert\kern-0.25ex\right\vert\kern-0.25ex\right\vert}}
\newenvironment{psmallmatrix}
  {\left(\begin{smallmatrix}}
  {\end{smallmatrix}\right)}
\begingroup
    \makeatletter
    \@for\theoremstyle:=definition,remark,plain\do{%
        \expandafter\g@addto@macro\csname th@\theoremstyle\endcsname{%
            \addtolength\thm@preskip\parskip
             }%
        }
\endgroup
\theoremstyle{remark}
\newtheorem{Problem}{Open Problem} 
\newenvironment{OProblem}{\begin{framed}\begin{Problem}}{\end{Problem}\end{framed}}
\title[Spectral geometry]{Spectral geometry: old questions and new answers}
\author{\firstname{David} \lastname{Krej\v{c}i\v{r}\'{\i}k}}
\address{Department of Mathematics, 
Faculty of Nuclear Sciences and Physical Engineering, 
Czech Technical University in Prague, 
Trojanova 13, 12000 Prague 2, Czech Republic}
\email{david.krejcirik@fjfi.cvut.cz}

\begin{document}

\begin{abstract}
These are extended notes based on a series of lectures given by the author at the \emph{Institut de Math\'ematiques de Toulouse} in June 2025. The goal is to present classical problems as well as the most recent developments in spectral geometry of the Laplace operator in Euclidean domains, subject to Dirichlet, Neumann and Robin boundary conditions. Among the topics covered, there are embedded eigenvalues in unbounded domains, spectral isoperimetric inequalities, nodal-line and hot-spots conjectures, inverse problems of hearing the shape of a drum and geometrically induced eigenvalues and Hardy-type inequalities in curved tubes. 
\end{abstract}

\maketitle



\section{Introduction}
%
Spectral geometry is concerned with the interaction between the spectrum of linear differential operators and the geometry of underlying Riemannian manifolds. 
How does the geometry influence the spectrum?
Which geometric shapes are spectrally optimal?
What can be deduced about the geometry from the spectral data?
These are just examples of a variety of questions covered by the field,
and partially by these notes.

Among the many books on the subject, 
let us point out the classical masterpiece~\cite{Polya-Szego} 
by P{\'o}lya and Szeg{\H{o}}, 
Henrot's compact book~\cite{Henrot} about shape optimisation
as well as his editorial follow-up~\cite{Henrot2}
and the most recent textbook~\cite{Levitin-Mangoubi-Polterovich}
by Levitin, Mangoubi and Polterovich.
For Riemannian manifolds, 
the classical books  
are due to Chavel~\cite{Chavel-evs} 
and Schoen and Yau~\cite{Schoen-Yau}.

In these notes, we exclusively focus on 
the Laplace operator in Euclidean domains, 
subject to Dirichlet, Neumann and Robin boundary conditions. 
Our goal is to acquaint the reader with the realm of spectral geometry 
by a selection of some classical mysteries 
and recent progresses in the field,
biasedly chosen by the taste of the author 
and his own contributions. 
A collection of persisting and new open problems are identified in the text.

\paragraph{Acknowledgement}
%
I would like to thank the organisers of the summer school  
\emph{Control, Inverse Problems and Spectral Theory} 
(Toulouse, June 2025) for their kind invitation, their hospitality and the opportunity to give these lectures;
I am particularly indebted to Julien Royer. 
I would also like to thank Monika Winklmeier
for inviting me to the summer school 
\emph{Aspects of spectral theory for linear operators}
(Bogot\'a, June 2025),
where these lecture notes started to crystallise.
Finally, I am grateful to Rami Band and Pedro Freitas 
for useful remarks on a previous version of these notes.


\addtolength{\parskip}{-0.15cm}
\setcounter{tocdepth}{2}
\tableofcontents
\addtolength{\parskip}{0.15cm}


\subsection{The problem and motivations}
We are concerned with the boundary value problem
\begin{equation}\label{problem}
\left\{
\begin{aligned}
  -\Delta \psi &= \lambda \;\! \psi
  && \mbox{in} \quad \Omega \subset \Real^d 
  \,, \\
  \frac{\partial \psi}{\partial n} + \alpha \, \psi &= 0
  && \mbox{on} \quad \partial\Omega 
  \,, 
\end{aligned}  
\right.
\end{equation}
where $d \geq 1$ is the dimension, $\Omega$~is an open set, 
$n$ the outward unit normal vector field of~$\partial\Omega$ 
and
\begin{center}
\begin{tabular}{ccccccccc}
  $\alpha$ & $\in$ & $(-\infty,0)$ 
  & $\cup$ & $\{0\}$ 
  & $\cup$ & $(0,\infty)$ 
  & $\cup$ & $\{\infty\}$\,.
  \\
  &&
  \mbox{\small negative Robin}
  && \mbox{\small Neumann}
  &&
  \mbox{\small positive Robin}
  && \mbox{\small Dirichlet}
\end{tabular}  
\end{center}  
The Dirichlet boundary condition is understood as $\psi=0$
(formally obtained after dividing 
the Robin boundary condition of~\eqref{problem} 
by~$\alpha$ and sending the parameter to infinity).
The function $\psi:\Omega\to\Com$ plays the role of an ``eigenvector''
and the number $\lambda \in \Com$ plays the role of an ``eigenvalue''.

The differential (Helmholtz) equation in~\eqref{problem} 
is a stationary counterpart 
of fundamental evolution equations in physics
modelling various phenomena in nature: 

\begin{itemize} 
\item
The wave equation ($\partial_t^2 u -\Delta u=0$)
is a classical model for a vibrating string,
membrane or elastic solid, but it also models propagation
of electromagnetic waves, moreover it arises 
in relativistic quantum mechanics and cosmology.
\item
The heat equation ($\partial_t u -\Delta u=0$), 
also known as the diffusion equation,
describes in typical applications the evolution in time 
of the density of some quantity such as the heat,
chemical concentration, \etc. It also represents 
the simplest version of the Fokker--Planck equation
describing the stochastic motion of a Brownian particle
(see Figure~\ref{Fig.Brown}).
\item
Finally, the Schr\"odinger equation ($i\partial_t u = -\Delta u$)
is the fundamental equation
of quantum theory, which is probably the best physical theory 
mankind has ever had (at least from the point of view of 
the technological impact and the number of experiments 
confirming it).  
\end{itemize}
\begin{figure}[h!]
\begin{center}
\includegraphics[width=0.4  \textwidth]{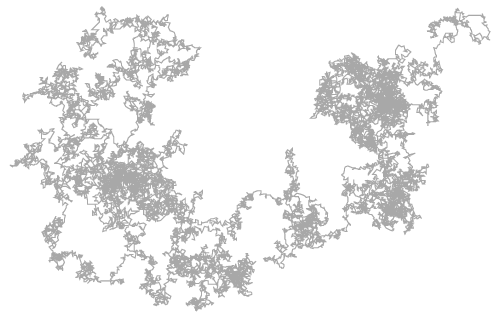}
\end{center}
\caption{The Brownian motion in~$\Real^2$.}
\label{Fig.Brown}
\end{figure}

The solutions~$\lambda$ and~$\psi$ of~\eqref{problem}
usually have direct physical interpretations,
depending on the given situation.
For instance, the numbers~$\lambda$  have the meaning of
\begin{itemize} 
\item
squares of resonant frequences for vibrating systems,
\item
decay rates for dissipative systems,
\item
bound-state energies for quantum systems.
\end{itemize}

The Dirichlet boundary conditions are used to describe 
vibrations of an elastic membrane whose boundary is fixed,
heat flow in a medium whose boundary is kept at zero temperature 
(a cooling mug),
killing boundary conditions for the Brownian motion,
the motion of a quantum particle
which is confined to a region by the barrier associated 
with a large chemical potential (nanostructures), \etc. 

The Neumann boundary conditions are used to describe
the vibration of a membrane at those parts of the boundary
which are free to move,
the flow of a fluid through a channel or past an obstacle,
the flow of heat in a medium with an insulated boundary (a vacuum flask),
reflecting boundary conditions for the Brownian motion, \etc. 

In electromagnetism and quantum mechanics, 
the Robin boundary conditions are used as an approximation 
for materials with thin layers (\eg~stealth aircrafts)
and strongly localised potentials, respectively.
More visually, the positive Robin boundary conditions 
approximate the vibration of a membrane attached 
to the drum shell by spring of positive friction.
The negative Robin boundary conditions appear
in acoustics in connection with
propagation of sonic waves through elastic cylinders. 

In summary,
the ``spectral problem''~\eqref{problem}
represents a unifying mathematical framework
for various (possibly very different) physical systems.

The objective of these notes is to study 
the interplay between the geometry of~$\Omega$ 
and the spectrum of~\eqref{problem}. 
Since~$\Omega$ is allowed to be unbounded 
and/or possibly irregular, 
first we have to provide an adequate 
interpretation of the boundary value problem~\eqref{problem}.

\subsection{The operator-theoretic framework}\label{Sec.Sobolev}
To interpret~\eqref{problem} as a spectral problem 
for a self-adjoint operator in a Hilbert space,
we define 
$$
  \eqref{problem} 
  \quad :\Longleftrightarrow \quad
  \lambda \in \sigma(-\Delta_\alpha^\Omega)
  \,.
$$
Here 
$
  -\Delta_\alpha^\Omega : 
  \dom(-\Delta_\alpha^\Omega) \subset \sii(\Omega) \to \sii(\Omega)
$
is an operator, called the Robin Laplacian,
acting in the Lebesgue space $\sii(\Omega)$ 
of square-integrable functions $\psi:\Omega\to\Com$.

The Robin Laplacian $-\Delta_\alpha^\Omega$ is introduced 
as the operator associated with a sesquilinear form 
$
  \delta_\alpha^\Omega: 
  \dom(\delta_\alpha^\Omega) 
  \times \dom(\delta_\alpha^\Omega) \subset
  \sii(\Omega) \times \sii(\Omega) \to \Com
$ 
via the representation relationship
\cite[Sec.~VI.2.1]{Kato}
\begin{equation}\label{representation}
  \forall \phi \in \dom(\delta_\alpha^\Omega) \,, \
  \psi \in \dom(-\Delta_\alpha^\Omega) \,, \qquad
  (\phi,-\Delta_\alpha^\Omega\psi) 
  = \delta_\alpha^\Omega(\phi,\psi) 
  \,.
\end{equation}
Here $(\cdot,\cdot)$ denotes the inner product of $\sii(\Omega)$.
Depending on the boundary conditions, we set 
\begin{equation}\label{depending}
\begin{aligned}
  \delta_D^\Omega(\phi,\psi) 
  &:= \int_\Omega \overline{\nabla\phi}\cdot\nabla\psi
  \,, \quad
  &\dom(\delta_D^\Omega) &:= W_0^{1,2}(\Omega)
  \,,
  && (\alpha=\infty)
  \\
  \delta_\alpha^\Omega(\phi,\psi) 
  &:= \int_\Omega \overline{\nabla\phi}\cdot\nabla\psi
  + \alpha \int_{\partial\Omega} \overline{\phi}\,\psi
  \,, \quad
  &\dom(\delta_\alpha^\Omega) &:= W^{1,2}(\Omega)
  \,.
  && (\alpha \in \Real)
\end{aligned}    
\end{equation}

The Sobolev space $W^{1,2}(\Omega)$ is composed of functions
$\psi \in \sii(\Omega)$ such that also $\nabla\psi \in \sii(\Omega)$. 
The other Sobolev space $W_0^{1,2}(\Omega)$ is defined as the closure 
of $C_0^\infty(\Omega)$ with respect to the norm 
$(\|\cdot\|^2 + \|\nabla\cdot\|^2)^{1/2}$,
where $\|\cdot\|$ denotes the norm of $\sii(\Omega)$.
It is customary to think of $W_0^{1,2}(\Omega)$ 
as the subspace of $W^{1,2}(\Omega)$ composed of functions 
vanishing on the boundary~$\partial\Omega$ in a very weak sense
(\cf~\cite[Thm.~VI.3.5]{Edmunds-Evans}).

Note that the actions of $\delta_D^\Omega$ and $\delta_N^\Omega$ are the same,
but the latter has a larger domain,
$\dom(\delta_D^\Omega) \subset \dom(\delta_N^\Omega)$.
Here (and in the sequel), we (usually) write~$D$ and~$N$
instead of $\alpha=0$ and $\alpha=\infty$ for the Dirichlet
and Neumann boundary conditions, respectively.

In this way,  
the Dirichlet and Neumann Laplacians 
are defined for arbitrary open sets
(since it is the case of the  Sobolev spaces 
$W_0^{1,2}(\Omega)$ and $W^{1,2}(\Omega)$).
For the Robin Laplacian with $\alpha \in \Real \setminus \{0\}$,
we implicitly require that the boundary term 
is a relatively small perturbation:
\begin{equation}\label{Ass.Robin}
  \forall \delta > 0 \,, \ \exists C_\delta \geq 0 \,, \ 
  \forall \psi \in W^{1,2}(\Omega) \,, \qquad
  \int_{\partial\Omega} |\psi|^2
  \leq \delta \int_\Omega |\nabla\psi|^2 
  + C_\delta \int_\Omega |\psi|^2 
  \,,
\end{equation}
where~$\psi$ on~$\partial\Omega$ is interpreted 
as the boundary trace of $\psi \in W^{1,2}(\Omega)$.

Inequality~\eqref{Ass.Robin} is particularly satisfied for 
Lipschitz sets~$\Omega$ with compact boundary~$\partial\Omega$ 
due to the well-known boundary-trace embedding 
$W^{1,2}(\Omega) \hookrightarrow \sii(\partial\Omega)$ 
(see, \eg, \cite[Thm.~3.37]{McLean} or~\cite[Thm.~1.5.1.3]{Grisvard})
together with the Ehrling inequality \cite[Lem.~1.5.3]{Schwarz}.
It also holds for smooth sets admitting a uniform tubular neighbourhood,
which can be verified with the help of 
parallel (also called Fermi) coordinates~\cite{KRT}. 

The operator $-\Delta_\alpha^\Omega$ is fully characterised 
by~\eqref{representation}. 
It is easy to see that 
\begin{equation}\label{Laplace}
\begin{aligned} 
  -\Delta_\alpha^\Omega \psi &= -\Delta\psi \,,
  \\
  \dom(-\Delta_\alpha^\Omega) &=
  \left\{
  \psi \in  \dom(\delta_\alpha^\Omega) : \ \Delta\psi \in \sii(\Omega)
  \ \land \ 
  \left.
  \frac{\partial \psi}{\partial n} + \alpha \, \psi 
  \,\right|_{\partial\Omega} 
  = 0
  \right\} 
  ,  
\end{aligned}
\end{equation}
where $\Delta\psi$ is the distributional Laplacian of~$\psi$
and the ``boundary condition'' is interpreted in a weak sense:
\begin{equation}\label{bc}
  \left.
  \frac{\partial \psi}{\partial n} + \alpha \, \psi 
  \,\right|_{\partial\Omega}
  = 0
  \quad :\Longleftrightarrow \quad
  \forall \phi \in \dom(\delta_\alpha^\Omega) \,, \quad
  \delta_\alpha^\Omega(\phi,\psi) = (\phi,-\Delta\psi)
  \,.
\end{equation}
Note that this is automatically satisfied in the Dirichlet case,
because of the definition of $W_0^{1,2}(\Omega)$ 
and the distributional Laplacian,
so  
$$  
  \dom(-\Delta_D^\Omega) = \{\psi \in W_0^{1,2}(\Omega):
  \ \Delta\psi \in \sii(\Omega)\}
  \,.
$$ 

If~$\Omega$ is ``nice'' (\eg, smooth and bounded), 
then 
\begin{equation}\label{nice}
  \dom(-\Delta_\alpha^\Omega) =
  \left\{
  \psi \in W^{2,2}(\Omega) : \ 
  \frac{\partial \psi}{\partial n} + \alpha \, \psi
  = 0
  \quad\mbox{on}\quad \partial\Omega
  \right\} 
  ,  
\end{equation}
where the boundary condition holds in the sense of traces
(and should be interpreted as $\psi=0$ on $\partial\Omega$
if $\alpha=\infty$). 
In general, however, it is not true that 
$\dom(-\Delta_\alpha^\Omega) \subset W^{2,2}(\Omega)$.   

\subsection{Point and continuous spectra}\label{Sec.spectrum}
Let~$H$ be any self-adjoint operator 
in a Hilbert space~$\Hilbert$,
which is always assumed to be complex and separable throughout these notes.
Recall that the self-adjointness means $H=H^*$.
Here the adjoint~$H^*$ satisfies the usual relationship
$(\phi,H\psi) = (H^*\phi,\psi)$
for every $\psi \in \dom H$ and 
$$
  \phi \in \dom H^*
  := \{\phi \in \Hilbert: \ \exists \eta \in \Hilbert, \
  \forall \psi \in \dom H, \ (\phi,H\psi) = (\eta,\psi)\}  
  \,;
$$
then the action $H^*\phi:=\eta$ is well defined 
for any~$H$ with $\dom H$ dense in~$\Hilbert$.
The spectrum of~$H$ is the disjoint union 
\begin{equation}\label{spectrum}
  \sigma(H) := \sigma_\mathrm{p}(H) \ \dot\cup \ \sigma_\mathrm{c}(H)
  \,,
\end{equation}
where the \emph{point} and \emph{continuous} spectra
are respectively defined by
$$
\begin{aligned}
  \sigma_\mathrm{p}(H) &:= \big\{
  \lambda \in \Com : \, \exists 
  \stackrel[\psi\not=0 \qquad\qquad]{}{\psi \in \dom H}, 
  \
  H\psi=\lambda\psi
  \big\} 
  ,
  \\
  \sigma_\mathrm{c}(H) &:= \Big\{
  \lambda \in \Com\setminus\sigma_\mathrm{p}(H) : \, 
  \exists 
  \stackrel[\|\psi_n\|=1 \qquad\qquad]{}{\{\psi_n\}_{n\in\Nat} \subset \dom H}, 
  \
  \|H\psi_n-\lambda\psi_n\| \xrightarrow[n\to\infty]{} 0
  \Big\} .
\end{aligned}
$$
Here~$\|\cdot\|$ denotes the norm of~$\Hilbert$.
Any point $\lambda \in \sigma_\mathrm{p}(H)$ 
(respectively, $\lambda \in \sigma_\mathrm{c}(H)$)
is called an \emph{eigenvalue} 
(respectively, \emph{approximate eigenvalue}) of~$H$. 
The corresponding vector~$\psi$
(respectively, the sequence $\{\psi_n\}_{n\in\Nat}$)
is called an \emph{eigenvector} 
(respectively, \emph{approximate eigenvector}) of~$H$. 

We stress that our definition~\eqref{spectrum}
does coincide with the usual definition of the spectrum 
of a general self-adjoint operator~$H$ in a Hilbert space~$\Hilbert$.
For an operator which is merely assumed to be closed,
it is true that there is additionally the so-called \emph{residual spectrum},
which is formed by those complex numbers 
$\lambda \not\in \sigma_\mathrm{p}(H)$ 
for which the closure of $\ran(H-\lambda I)$
does not coincide with~$\Hilbert$
(\ie~the inverse operator $(H-\lambda I)^{-1}$
is not densely defined).
However, this pathological part of the spectrum 
is always empty in the important case
of self-adjoint operators.

The Robin Laplacian $-\Delta_\alpha^\Omega$ is self-adjoint by definition.
It follows that $\sigma(-\Delta_\alpha^\Omega) \subset \Real$
(see the proof of Proposition~\ref{Prop.positivity}
for the idea of a proof). 

\subsection{The geometric layout}
We use the following classification 
of Euclidean open sets due to Glazman~\cite{Glazman} 
(see also \cite[Sec.~X.6.1]{Edmunds-Evans}):
$\Omega \subset \Real^d$ is
\begin{itemize}
\item
\emph{quasi-conical}
if it contains arbitrarily large balls;
\item
\emph{quasi-cylindrical}
if it is not quasi-conical
but it contains infinitely many (pairwise) disjoint
identical (\ie~congruent) balls;
\item
\emph{quasi-bounded}
if it is neither quasi-conical nor quasi-cylindrical.
\end{itemize}

Bounded sets 
(\ie~those contained in a ball)
represent a subset of quasi-bounded sets,
but the latter class is much larger
(it includes unbounded sets which are ``narrow at infinity'').
The whole Euclidean space~$\Real^d$ or its conical sector
are examples of quasi-conical sets.
The infinite sequence of disjoint 
identical (respectively, expanding) balls
is an example of a quasi-cylindrical (respectively, quasi-conical) set. 
Finally, an infinite (solid) cylinder $\Real \times \omega$,
where~$\omega$ is a $(d-1)$-dimensional open set,
is a quasi-cylindrical set.
See Figure~\ref{Fig.Glazman} for typical examples in~$\Real^2$.

\begin{figure}[h!]
\begin{center}
\begin{tabular}{ccc}
\includegraphics[width=0.3\textwidth]{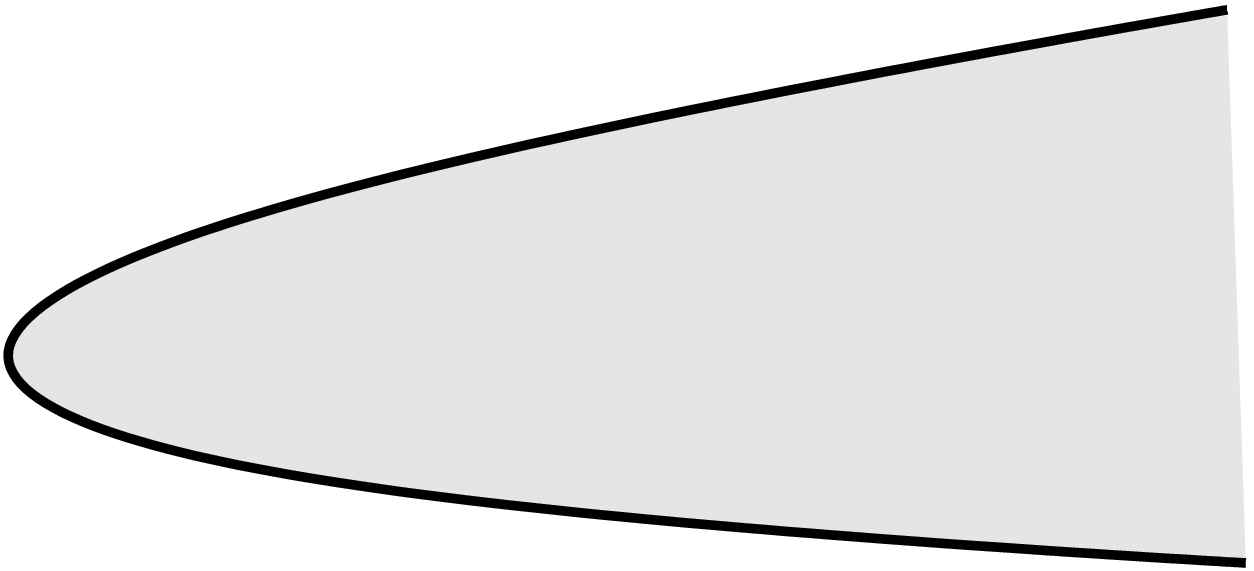}
&\includegraphics[width=0.3\textwidth]{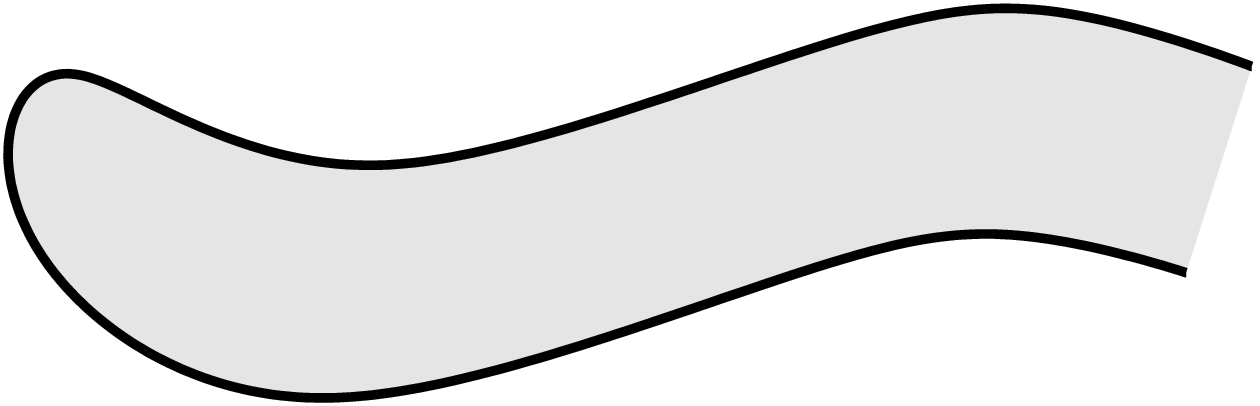}
&\includegraphics[width=0.3\textwidth]{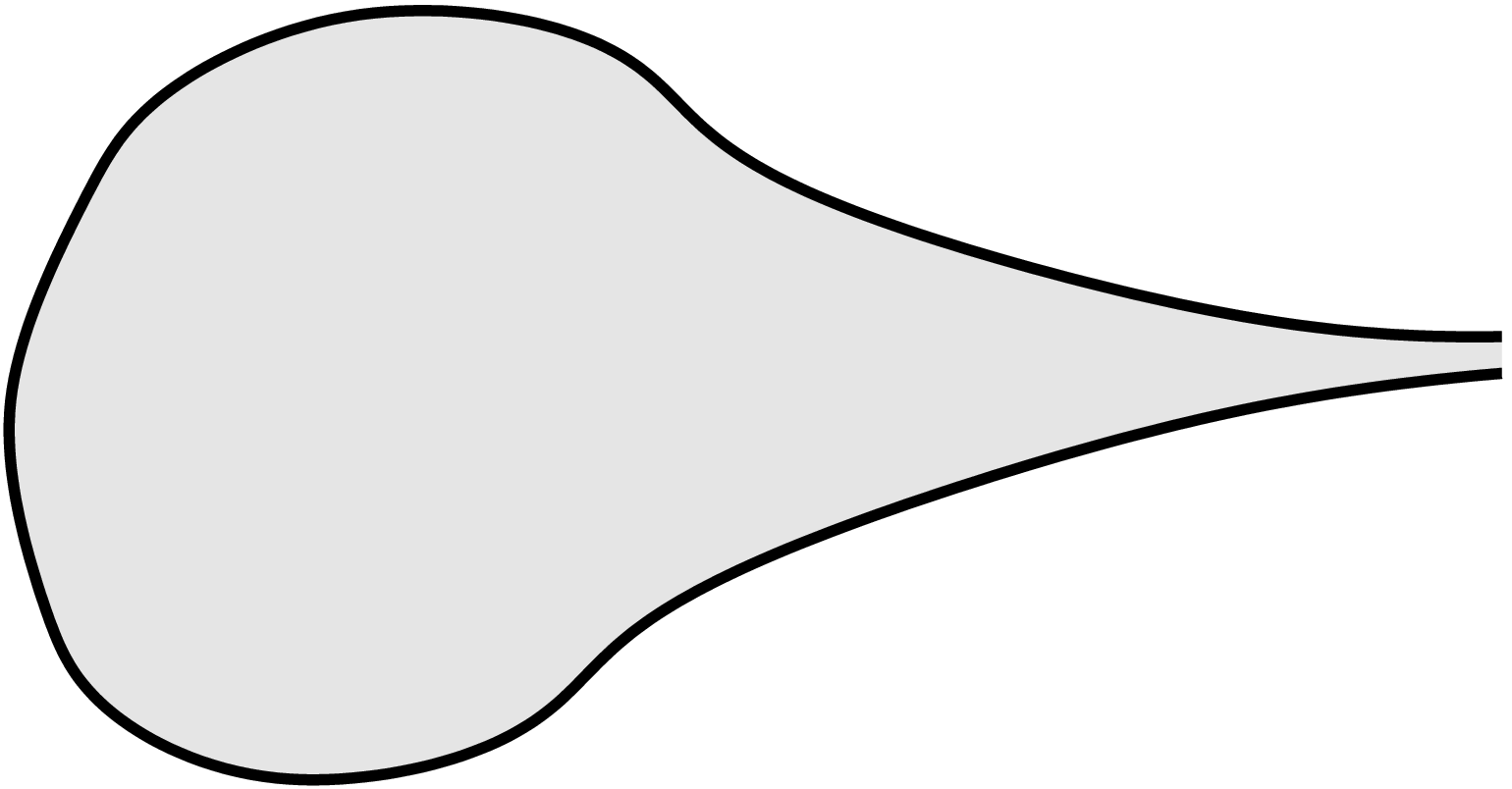}
\\
quasi-conical
& quasi-cylindrical
& quasi-bounded
\end{tabular}
\end{center}
\caption{Examples of planar domains as regards the Glazman classification.}
\label{Fig.Glazman}
\end{figure}

In physics, the particular subclasses of open sets 
are particularly related to the following problems:

\begin{center}
\begin{tabular}{|c|c|c|}
\hline
quasi-conical 
& quasi-cylindrical 
& quasi-bounded
\\ \hline\hline
\begin{tabular}{c}
scatterers
\end{tabular}
&
\begin{tabular}{c}
waveguides 
\end{tabular}
&
resonators
\\ \hline
\end{tabular}  
\end{center}  

In the following sections, 
we are interested in spectral properties of
the Robin Laplacian as regards the above classification.

\section{Quasi-conical domains or Scatterers}
%
Let~$\Omega$ be an arbitrary quasi-conical open set.
By definition, there exist sequences
of centres $\{x_j\}_{j\in\Nat} \subset \Omega$ 
and radii $\{R_j\}_{j\in\Nat} \subset (0,\infty)$
such that 
\begin{equation}\label{centres}
  B_{R_j}(x_j) \subset \Omega
  \qquad \mbox{and} \qquad
  R_j \xrightarrow[j\to\infty]{} \infty
  \,.
\end{equation}
Here $B_R(x_0) := \{x\in\Real^d: |x-x_0|<R\}$
denotes the open ball of radius $R>0$ centred at $x_0 \in \Real^d$.
Occasionally, we abbreviate $B_R := B_R(0)$.
Notice that~$\Omega$ is necessarily unbounded.

\subsection{Construction of approximate eigenfunctions}
The differential equation of~\eqref{problem} admits 
a classical solution (plane waves)
\begin{equation}\label{plane.waves}
  w_k(x) := e^{i k \cdot x}
  \qquad\mbox{with any}\quad
  k\in\Real^d
  \quad\mbox{such that}\quad
  |k|^2 = \lambda \in [0,\infty)
  \,.
\end{equation}
However, $w_k \not\in \sii(\Omega)$, so~$w_k$ is not an eigenfunction.
Anyway, the observation enables us to construct 
approximate eigenfunctions and prove the following result.
\begin{thm}\label{Thm.conical}
If~$\Omega$ is a quasi-conical open set 
and $\alpha \in \Real \cup \{\infty\}$, 
then
$$
  \sigma(-\Delta_\alpha^\Omega)
  \supset [0,\infty)
  \,.
$$
\end{thm}
\begin{proof}
Let~$\varphi$ be a function from $C_0^\infty(\Real^d)$,
normalised to~$1$ in $\sii(\Real^d)$, \ie~$\|\varphi\|_{\sii(\Real^d)}=1$.
For any $n\in\Nat^* := \Nat \setminus \{0\}$
(following the French nature of the journal, 
natural numbers contain zero in our convention) 
and $\{a_n\}_{n\in\Nat^*} \subset \Real^d$, 
we set
$$
  \varphi_n(x) := N_n \, \varphi\!\left(\frac{x-a_n}{n}\right)
  \qquad \mbox{with} \qquad
  N_n := n^{-d/2}
  \,.
$$
The prefactor~$N_n$ is chosen in such a way that
also each~$\varphi_n$ is normalised to~$1$ in $\sii(\Real^d)$.
Indeed, by an obvious change of variables, we have
\begin{equation}\label{normalisation}
\begin{aligned}
  \|\varphi_n\|_{\sii(\Real^d)}^2
  &= |N_n|^2 \int_{\Real^d} 
  \left|\varphi\!\left(\frac{x-a_n}{n}\right)\right|^2 
  \, \der x
  \\
  &= \int_{\Real^d} |\varphi(y)|^2 \, \der y
  = \|\varphi\|_{\sii(\Real^d)}^2 
  = 1
  \,.
\end{aligned}  
\end{equation}
With respect to the support of~$\varphi$,
the support of~$\varphi_n$ is translated by the vector~$a_n$
and scaled by~$n$
(see Figure~\ref{Fig.rozpliz}):
\begin{equation}\label{support}
  \supp \varphi_n = a_n + n \, \supp \varphi
  \,.
\end{equation}
By the property of the set~$\Omega$ being quasi-conical,
for each $n \in \Nat^*$ there exists $j_n \in \Nat^*$   
such that 
$$
  \supp \varphi_n \subset B_{R_{j_n}}(x_{j_n}) \subset \Omega
  \qquad \mbox{with} \qquad
  a_n := x_{j_n}
  \,,
$$
where the radii~$R_j$ and centres~$x_j$ satisfy~\eqref{centres}.
Hence, 
$
  \varphi_n 
  \in C_0^\infty(\Omega)
  \in \dom(-\Delta_\alpha^\Omega)
$.
For any $n\in\Nat^*$, we define
\begin{equation}\label{singular}
  \psi_n(x) := \varphi_n(x) \, e^{i k \cdot x}
  , 
\end{equation}
which also belongs to $\dom(-\Delta_\alpha^\Omega)$.
By~\eqref{normalisation}, $\|\psi_n\|=1$ for every $n \in \Nat^*$.

\begin{figure}[h!] 
\begin{center}
\includegraphics[width=0.7\textwidth]{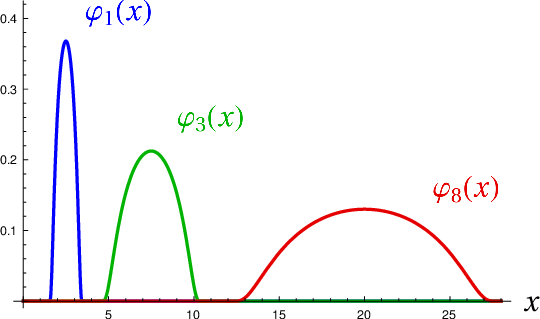} 
\caption{The sequence of functions $\varphi_n$ for $d=1$
smearing out and travelling to $+\infty$ as $n \to \infty$.}\label{Fig.rozpliz}
\end{center}
\end{figure}

In order to ensure that $\{\psi_n\}_{n \in \Nat^*}$
is the approximate eigenfunction corresponding 
to the approximate eigenvalue~$|k|^2$,
it remains to verify that 
$-\Delta_\alpha^\Omega \psi_n - |k|^2 \psi_n \to 0$
in $\sii(\Omega)$ as $n \to \infty$.
Since $\psi_n \in C_0^\infty(\Omega)$, the action of $-\Delta_\alpha^\Omega$
is that of the classical Laplacian.
We compute
$$
\begin{aligned}
  \nabla\psi_n(x) 
  &= \left[ \nabla\varphi_n(x) + i k \, \varphi_n(x) \right] e^{ik\cdot x} \,,
  \\
  \Delta\psi_n(x) &= \nabla \cdot \nabla \psi_n(x)
  = \left[ \Delta\varphi_n(x) + 2ik\cdot\nabla\varphi_n(x)
  -|k|^2 \, \varphi_n(x) \right] e^{ik\cdot x} \,.
\end{aligned}
$$
Consequently,
$$
  -\Delta_\alpha^\Omega \psi_n(x) - |k|^2 \psi_n(x) 
  = \left[-\Delta\varphi_n(x) - 2ik\cdot\nabla\varphi_n(x)\right] e^{ik\cdot x}
  \,,
$$
and therefore 
$$  
  \|-\Delta_\alpha^\Omega \psi_n - k^2 \psi_n\|
  \leq \|\Delta\varphi_n\| + 2\,|k| \, \|\nabla\varphi_n\|
  \,.
$$
The right-hand side vanishes as $n \to \infty$, indeed:
\begin{equation*}
\begin{aligned}
  \|\nabla\varphi_n\|^2
  &= |N_n|^2 \int_{\Real^d} 
  \left|\frac{1}{n}\nabla\varphi\!\left(\frac{x-a_n}{n}\right)\right|^2 
  \, \der x
  \\
  &= \frac{1}{n^2} \int_{\Real^d} |\nabla\varphi(y)|^2 \, \der y
  = \frac{1}{n^2} \|\nabla\varphi\|^2 
  \,,
  \\
  \|\Delta\varphi_n\|^2
  &= |N_n|^2 \int_{\Real^d} 
  \left|\frac{1}{n^2}\Delta\varphi\!\left(\frac{x-a_n}{n}\right)\right|^2 
  \, \der x
  \\
  &= \frac{1}{n^4} \int_{\Real^d} |\Delta\varphi(y)|^2 \, \der y
  = \frac{1}{n^4} \|\Delta\varphi\|^2 
  \,.
\end{aligned}
\end{equation*}
This concludes the proof of the theorem.
\end{proof}

\subsection{Is there a negative spectrum ?}
First of all, this is never the case if $\alpha \geq 0$
(including $\alpha=\infty$),
irrespectively of the geometry.

\begin{prop}\label{Prop.positivity}
Let $\Omega$ be any open set 
and $\alpha \in [0,\infty) \cup \{\infty\}$. 
Then
$$
  \sigma(-\Delta_\alpha^\Omega) \subset [0,\infty) \,.
$$
\end{prop}
\begin{proof}
Let $\lambda \in \sigma(-\Delta_\alpha^\Omega)$. 
By our definition of the spectrum~\eqref{spectrum},
there exists a sequence $\{\psi_n\}_{n \in \Nat} \subset \dom(-\Delta_\alpha^{\Omega})$ 
such that $\|\psi_n\|=1$ for every $n \in \Nat$
and $-\Delta_\alpha^{\Omega}\psi_n-\lambda\psi_n \to 0$ as $n \to \infty$.  
We have
$$
\begin{aligned}
  \lambda 
  &= \lambda \, \|\psi_n\|^2 \\
  &= \liminf_{n\to\infty} \lambda \, \|\psi_n\|^2 \\
  &= \liminf_{n\to\infty} (\psi_n,\lambda\psi_n) \\
  &= \liminf_{n\to\infty} (\psi_n,-\Delta_\alpha^\Omega\psi_n) \\
  &= \liminf_{n\to\infty} \delta_\alpha^\Omega(\psi_n,\psi_n) \\
  &\geq \liminf_{n\to\infty} \|\nabla\psi_n\|^2 \\
  &\geq 0
  \,,
\end{aligned}
$$
where the last but one inequality employs $\alpha \geq 0$.
\end{proof}
\begin{coro}\label{Corol.positivity}
Let~$\Omega$ be any quasi-conical open set 
and $\alpha \in [0,\infty) \cup \{\infty\}$.
Then
$$
  \sigma(-\Delta_\alpha^\Omega)
  = [0,\infty)
  \,.
$$
\end{coro}

If~$\alpha$ is negative, however, there might be some negative spectrum.

\begin{exam}[Half-axis]\label{Ex.half}
Let us look for negative eigenvalues of the Robin Laplacian
$-\Delta_\alpha^{(0,\infty)}$
with $\alpha < 0$.
The eigenvalue problem~\eqref{problem} for $\Omega:=(0,\infty)$ reads 
\begin{equation*} 
\left\{
\begin{aligned}
  -\psi'' &= \lambda \;\! \psi
  && \mbox{in} \quad (0,\infty)
  \,, \\
  -\psi' + \alpha \, \psi &= 0
  && \mbox{at} \quad 0
  \,.
\end{aligned}  
\right.
\end{equation*}
Assuming $\lambda = -k^2$ with $k > 0$,
the differential equation admits the general solution
$
  \psi(x) = A e^{kx} + B e^{-kx}
$
with $A,B \in \Com$. 
The requirement $\psi \in \sii((0,\infty))$ dictates $A=0$.
The boundary condition at~$0$ requires $k = -\alpha$,
which leads to the normalised eigenfunction
$$
  \psi(x) = \sqrt{-2\alpha} \, e^{\alpha x}
$$
corresponding to the negative eigenvalue 
$-\alpha^2 \in \sigma_\mathrm{p}(-\Delta_\alpha^{(0,\infty)})$.
Looking similarly for non-negative eigenvalues,
it is easily seen that 
$
  \sigma_\mathrm{p}(-\Delta_\alpha^{(0,\infty)})
  \cap [0,\infty) = \varnothing
$
for every $\alpha \in \Real \cup \{\infty\}$.
In summary,
$$
\begin{aligned}
  \sigma_\mathrm{p}(-\Delta_\alpha^{(0,\infty)}) 
  &=
  \begin{cases}
    \{-\alpha^2\} 
    & \mbox{if} \quad \alpha \in (-\infty, 0) \,, 
    \\
    \varnothing 
    & \mbox{if} \quad \alpha \in [0,+\infty) \cup \{\infty\} \,,
  \end{cases}
  \\
  \sigma_\mathrm{c}(-\Delta_\alpha^{(0,\infty)}) 
  &= [0,+\infty) \,.
\end{aligned}  
$$
The $\alpha$-dependence of the spectrum 
can be seen in Figure~\ref{Fig.half}.
\end{exam}
\begin{figure}[h!] 
\begin{center}
\includegraphics[width=0.7\textwidth]{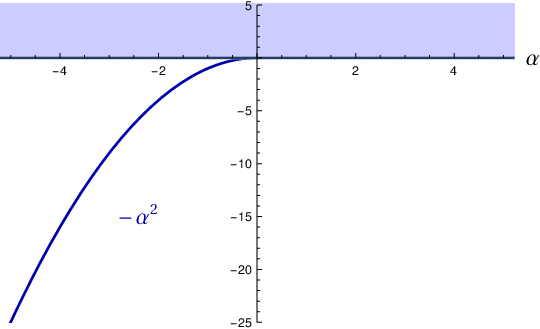} 
\caption{The spectrum of $-\Delta_\alpha^{(0,\infty)}$ 
as a function of~$\alpha$.}\label{Fig.half}
\end{center}
\end{figure}

\subsection{Embedded eigenvalues}
Example~\ref{Ex.half} shows that the half-axis admits no
eigenvalues inside the interval $[0,\infty)$.
The latter always belongs to the spectrum of $-\Delta_\alpha^\Omega$
for any quasi-conical set due to Theorem~\ref{Thm.conical}. 
\begin{center}
\fbox{Is there a quasi-conical set 
with eigenvalues inside $[0,\infty)$ ?}
\end{center}
Since these eigenvalues are necessarily not isolated points in the spectrum,
they are called \emph{embedded} eigenvalues
(in the ``continuum'' $[0,\infty)$).
It turns out that the existence/ab\-sence of embedded eigenvalues
is a highly non-trivial question in spectral theory.

\begin{exam}[Disconnected sets]
Let $\omega \subset \Real^{d-1}$ 
be any non-empty smooth bounded open set.
Then~\eqref{Ass.Robin} with~$\omega$ instead of~$\Omega$ holds true 
and the embedding $W^{1,2}(\omega) \hookrightarrow \sii(\omega)$
is compact (see \cite[Thm.~V.4.17]{Edmunds-Evans}
or the proof of Theorem~\ref{Thm.bounded.Neumann} below).
It follows that~$-\Delta_\alpha^\Omega$ is an operator 
with compact resolvent, so it possesses an infinite sequence 
of eigenvalues accumulating at~$+\infty$.  
If $\Omega := \omega \cup (\Real^d\setminus\overline{\omega})$,
all the eigenvalues of $-\Delta_\alpha^{\omega}$ 
are eigenvalues of $-\Delta_\alpha^\Omega$.
In particular, non-negative eigenvalues of $-\Delta_\alpha^{\omega}$ 
are embedded eigenvalues of $-\Delta_\alpha^\Omega$. 
(There are infinitely many of them.)
\end{exam}

In view of this example, let us henceforth restrict
to~$\Omega$ being a \emph{domain},
\ie, an open connected set.
There is a subclass of quasi-conical domains 
for which the question of embedded eigenvalues
is relatively easy to study.
Let~$\Omega$ be an \emph{exterior domain},
\ie, a domain $\Omega\subset\Real^d$ satisfying that
there exists a positive number~$R$ such that 
$\Omega \setminus B_R(0) = \Real^d \setminus B_R(0)$,
where $B_R(0)$ is the open ball of radius~$R$
centred at the origin of~$\Real^d$.
For such domains, 
already in 1943 Rellich~\cite{Rellich_1943} proved that there are 
no \emph{positive} eigenvalues,
irrespectively of the boundary conditions imposed on~$\partial\Omega$. 

\begin{theo}[Rellich 1943~\cite{Rellich_1943}]\label{Thm.Rellich}
Let $\Omega$ be an exterior domain
and $\alpha \in \Real \cup \{\infty\}$. 
Then
$$
  \sigma_\mathrm{p}(-\Delta_\alpha^\Omega) 
  \cap (0,\infty)
  = \varnothing \,.
$$
\end{theo}
\begin{proof}
The argument is elementary if $d=1$, 
so let us assume $d \geq 2$.
We are inspired by the proof of~\cite[Thm.~XIII.56]{RS4}.
Let us consider a solution $\psi \in \sii(\Omega)$ 
of $-\Delta\psi=\lambda\psi$ in $\Omega$ with $\lambda>0$.
By definition, there exists an open ball $B_R(0) \supset \Omega^c$ 
such that $-\Delta\psi=\lambda\psi$
in $\Real^d \setminus B_R(0)$.
In spherical coordinates
$(r,\theta) \in (R,\infty)\times \Sphere^{d-1}$, 
the action of the Laplacian reads 
$$
  -\Delta 
  = - r^{-(d-1)} \partial_r r^{d-1} \partial_r
  - r^{-2} \Delta_\theta
$$
and $\sigma(-\Delta^{\Sphere^{d-1}}) = \{l(l+d-2)\}_{l\in\Nat}$.
Expanding~$\psi$ 
in the eigenfunctions of $-\Delta^{\Sphere^{d-1}}$,
it follows that there must necessarily exist $l\in\Nat$ 
with a non-trivial solution 
$\varphi \in \sii((R,\infty),r^{d-1}\, \der r)$ to
\begin{equation}\label{Rellich}
  - r^{-(d-1)} [ r^{d-1} \varphi'(r) ]' 
  + l(l+d-2) r^{-2} = \lambda \varphi(r)
\end{equation}
for $r \geq R$.
There exist two independent solutions
$$
\begin{aligned}
  \varphi_1(r) &:=
  r^{-(d-2)/2} \, J_{(2l+d-2)/2}\big(\sqrt{\lambda} \, r \big)
  \,,
  \\
  \varphi_2(r) &:=
  r^{-(d-2)/2} \, Y_{(2l+d-2)/2}\big(\sqrt{\lambda} \, r \big)
  \,,
\end{aligned}   
$$
where $J_\nu$ and $Y_\nu$ are Bessel functions of order~$\nu$,
see \cite[Sec.~9]{Abramowitz-Stegun}.
Using the asymptotics \cite[Sec.~9.2]{Abramowitz-Stegun}
$$
\begin{aligned}
  \varphi_1(r) &= 
  \sqrt{\frac{2}{\pi\lambda}} \,
  \frac{\cos\big(\sqrt{\lambda} \, r-\frac{\pi}{4}(2l+d-1)\big)}
  {r^{(d-1)/2}} 
  + O(r^{-{d/2}})
  \,, 
  \\
  \varphi_2(r) &=
  \sqrt{\frac{2}{\pi\lambda}} \,
  \frac{\sin\big(\sqrt{\lambda} \, r-\frac{\pi}{4}(2l+d-1)\big)}
  {r^{(d-1)/2}} 
  + O(r^{-{d/2}})
  \,,
\end{aligned}  
$$
as $r \to \infty$, 
we see that none of the solutions belongs to 
$\sii((R,\infty),r^{d-1}\, \der r)$.
Hence $\psi = 0$ in $\Real^d \setminus B_R(0)$.
By the Harnack inequality 
(see, \eg, \cite[Thm.~8.20]{Gilbarg-Trudinger}),
necessarily $\psi = 0$ in $\Omega$
(alternatively, the same conclusion can be achieved 
by the classical strong maximum principle 
\cite[Sec.~3.2]{Gilbarg-Trudinger}).
\end{proof}

Other quasi-conical domains were considered by Jones in 1953 \cite{Jones_1953}.
The idea of Theorem~\ref{Thm.Rellich} was further developed
by Kato in 1959 \cite{Kato_1959} in more general situations
(including Schr\"o\-ding\-er operators).
More recently, D'Ancona and Racke in 2012
\cite{D'Ancona-Racke_2012} 
excluded embedded eigenvalues of the Dirichlet Laplacian
by imposing a repulsive-type condition on the geometry of the boundary
of tubular-type quasi-conical sets.
Finally, Bonnet-Ben Dhia, Fliss, Hazard and Tonnoir 
in 2016 \cite{Bonnet-Fliss-Hazard-Tonnoir_2016}
excluded the existence of non-zero eigenvalues 
in non-convex conical sectors.

In Theorem~\ref{Thm.Rellich}, 
it is essential that the zero eigenvalue
is excluded from the statement.
\begin{exam}[Exterior of a ball]
Let $d \geq 2$.
To construct an example of an exterior domain 
possessing zero as an embedded eigenvalue, 
note that for $\lambda=0$ 
the equation~\eqref{Rellich} admits two different 
independent solutions 
$$
  \varphi_1(r) := r^{l}
  \qquad\mbox{and}\qquad
  \varphi_2(r) := 
  \begin{cases}
  \log(r) & \mbox{if} \quad l=0 \ \land \ d=2 \,,
  \\
  r^{-(l+d-2)} & \mbox{otherwise} \,.
  \end{cases}
$$
In particular, now
$\varphi_2 \in \sii((R,\infty),r^{d-1}\, \der r)$
whenever $d \geq 5$ 
(or~$l$ is sufficiently large irrespectively 
of the dimension $d \geq 2$).
At the same time, if $l=0$ and $d \geq 5$,
it is straightforward to verify that 
$-\varphi_2'(r) + \alpha\varphi_2(r) = 0$ 
if, and only if, $d-2+\alpha r = 0$.
It follows that~$\psi$ defined by $\psi(x):=\varphi_2(|x|)$ 
belongs to $W^{1,2}(\Omega)$ with 
$\Omega := \Real^d \setminus \overline{B_R(0)}$
and solves the Robin problem 
$$
\left\{
\begin{aligned}
  -\Delta\psi &= 0 
  && \mbox{in} \quad \Omega  
  \,,
  \\
  \frac{\partial \psi}{\partial n}
  + \alpha \psi &= 0
  && \mbox{on} \quad \partial\Omega 
  \,,
\end{aligned}  
\right.
\qquad \mbox{with} \qquad
  \alpha := - \frac{d-2}{R} 
  \,.
$$
Consequently, $0 \in \sigma_\mathrm{p}(-\Delta_\alpha^\Omega)$.
Since $[0,\infty) \subset \sigma(-\Delta_\alpha^\Omega)$
by Theorem~\ref{Thm.conical}, the exterior of 
large-dimensional balls is an example of domains
with zero being an embedded eigenvalue.~%
\end{exam}

The previous example entails the Robin Laplacian
with a negative boundary parameter. 
For non-negative parameters, in fact,
Theorem~\ref{Thm.Rellich} can be improved
by excluding the zero eigenvalue as well. 

\begin{prop}\label{Thm.Rellich.bis}
Let $\Omega \subset \Real^d$ be an exterior domain
and $\alpha \in [0,\infty) \cup \{\infty\}$. 
Then
$$
  \sigma_\mathrm{p}(-\Delta_\alpha^\Omega) 
  \cap [0,\infty)
  = \varnothing \,.
$$
\end{prop}
\begin{proof}
By contradiction, assume that 
$0 \in \sigma_\mathrm{p}(-\Delta_\alpha^\Omega)$.
Let~$\psi$ be the corresponding eigenfunction.
Then 
$$
  0 = \delta_\alpha^\Omega[\psi] 
  = \int_\Omega |\nabla\psi|^2 
  + \alpha \int_{\partial\Omega} |\psi|^2
  \geq \int_\Omega |\nabla\psi|^2
  \,.
$$
Consequently, $\nabla\psi = 0$ in~$\Omega$,
which implies that~$\psi$ is constant 
in the exterior domain~$\Omega$,
which is impossible unless $\psi = 0$ 
identically in~$\Omega$.
\end{proof}

Leaving the realm of exterior domains,
the challenge persists whether 
there exist examples of quasi-conical domains~$\Omega$ 
for which the Robin Laplacian $-\Delta_\alpha^\Omega$
with non-negative~$\alpha$ possesses embedded eigenvalues.
This long-standing question was given a positive answer only recently
for the Dirichlet Laplacian by the present lecturer~\cite{KL4}.

\begin{figure}[h!] 
\begin{center}
\includegraphics[width=0.98\textwidth]{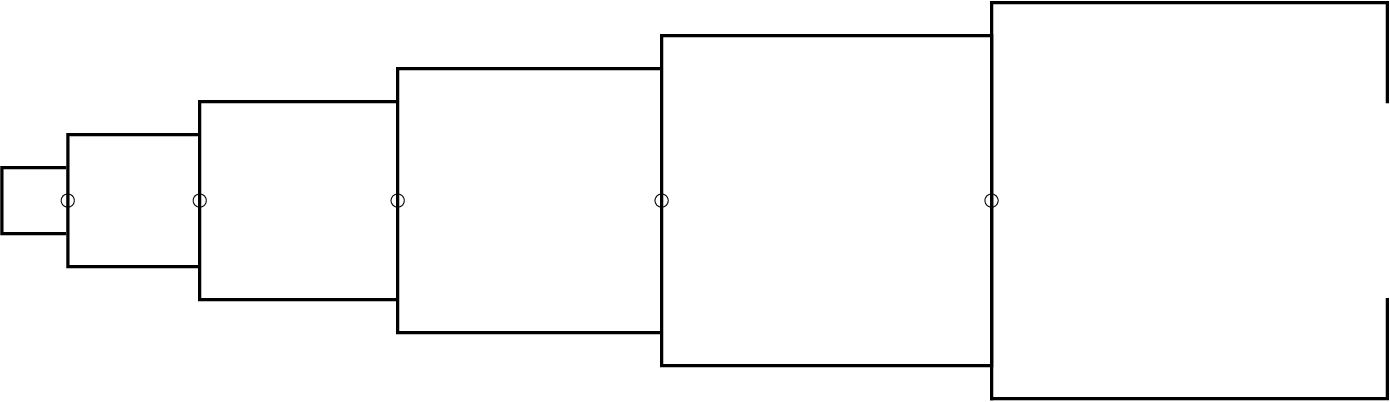} 
\caption{A quasi-conical domain with embedded eigenvalues
due to~\cite{KL4}.}\label{Fig.embedded}
\end{center}
\end{figure}
\begin{theo}[Krej\v{c}i\v{r}\'ik \& Lotoreichik 2024~\cite{KL4}]\label{Thm.embedded}
Given any positive number~$\lambda$,
there exists a quasi-conical domain~$\Omega$ such that 
$$
  \lambda \in \sigma_\mathrm{p}(-\Delta_D^\Omega) 
  \,.
$$
\end{theo}
\begin{proof}
The idea of the construction of the example of~\cite{KL4} 
is as follows,
see Figure~\ref{Fig.embedded}.

We begin by considering a union~$\Omega'$ 
of disjoint adjacent enlarging cubes,
for which 
$
  \sigma(-\Delta_D^{\Omega'})
  = \overline{\sigma_\mathrm{p}(-\Delta_D^{\Omega'})} 
  = [0,\infty)
$. 
Here the equality 
$\sigma(-\Delta_D^{\Omega'}) = [0,\infty)$
follows by the fact that~$\Omega'$ is quasi-conical
(see Corollary~\ref{Corol.positivity}),
while 
$
  \sigma(-\Delta_D^{\Omega'})
  = \overline{\sigma_\mathrm{p}(-\Delta_D^{\Omega'})} 
$
is a consequence of the property that 
the spectrum of the Dirichlet Laplacian in cubes 
(or, more generally, bounded sets)
is exhausted by eigenvalues only
(see Theorem~\ref{Thm.bounded} below).

Taking the lowest eigenvalue in the smallest cube 
and drilling a small hole into the neighbouring larger cube,
the eigenvalue is perturbed into an eigenvalue 
of the two connected cubes, while its eigenfunction 
is mostly located in the smallest cube.
Continuing this procedure of drilling smaller and smaller holes 
between adjacent cubes iteratively till infinity,
it is possible to ensure that the eigenvalue remains
an eigenvalue in the obtained connected set~$\Omega$.

The extension to any positive number~$\lambda$ is achieved by scaling.  
\end{proof}

As a matter of fact, in~\cite{KL4}, we also argue that 
the domain~$\Omega$ can be constructed in such a way that 
the absolutely continuous spectrum of the Dirichlet Laplacian is empty:
$$
  \sigma_\mathrm{ac}(-\Delta_D^\Omega) = \varnothing
  \,.
$$
We leave as an open problem whether the singularly continuous 
spectrum is empty as well. 

\begin{OProblem}
In the example~$\Omega$ of~\cite{KL4} (see Figure~\ref{Fig.embedded}),
is it possible to select the sizes of the windows so small that
$$
  \sigma_\mathrm{sc}(-\Delta_D^\Omega) = \varnothing
  \,?
$$
\end{OProblem}

In the case of affirmative answer,
we would get the Dirichlet Laplacian on 
a connected quasi-conical open set having only purely point
spectrum, which densely fills the non-negative semi-axis.

\section{The Hardy inequality}
%
According to Corollary~\ref{Corol.positivity},
the spectrum of the Dirichlet Laplacian $-\Delta_D^\Omega$
equals $[0,\infty)$ for any quasi-conical set~$\Omega$, 
irrespectively of the geometry of~$\Omega$.
This apparently boring nature of quasi-conical geometries 
is just illusive, as we have seen in the preceding subsection
devoted to embedded eigenvalues.
Now we focus on other fine spectral properties,
which are non-trivial even in the case 
the whole Euclidean space $\Omega=\Real^d$
(for which the role of geometry is played by the dimension~$d$).

\subsection{Subcriticality of high dimensions} 
The following theorem is one of the most important results 
established in this course.
(By recalling the definition of the Sobolev space $W_0^{1,2}(\Real^d)$
given in Section~\ref{Sec.Sobolev}, the set of test functions 
in~\eqref{Hardy1} can be replaced by $C_0^\infty(\Real^d)$.)
\begin{theo}[Hardy inequality]\label{Thm.Hardy1}
Let $d \geq 3$. Then
\begin{equation}\label{Hardy1}
  \forall \psi \in W_0^{1,2}(\Real^d)
  \,, \qquad
  \int_{\Real^d} |\nabla\psi(x)|^2 \, \der x
  \geq \frac{(d-2)^2}{4} \int_{\Real^d} \frac{|\psi(x)|^2}{|x|^2} \, \der x
  \,.
\end{equation}
\end{theo}
\begin{proof}
For any $c\in\Real$, we have
\begin{align*}
  0 &\leq
  \int_{\Real^d} 
  \left|\nabla\psi(x) - c \, \frac{x}{|x|^2} \, \psi(x)\right|^2 \der x
  \\
  &= \int_{\Real^d} |\nabla\psi(x)|^2 \, \der x
  + c^2 \int_{\Real^d} \frac{|\psi(x)|^2}{|x|^2} \, \der x
  - c \int_{\Real^d} \frac{x}{|x|^2} \cdot \nabla|\psi|^2(x) \, \der x
  \\
  &= \int_{\Real^d} |\nabla\psi(x)|^2 \, \der x
  + c^2 \int_{\Real^d} \frac{|\psi(x)|^2}{|x|^2} \, \der x
  + c \int_{\Real^d} \divergence\!\left(\frac{x}{|x|^2}\right) 
  \, |\psi(x)|^2 \, \der x
  \\
  &= \int_{\Real^d} |\nabla\psi(x)|^2 \, \der x
  + [c^2 + c \, (d-2)] \int_{\Real^d} \frac{|\psi(x)|^2}{|x|^2} \, \der x
  \,,
\end{align*}
where the second equality employs an integration by parts
(or, more precisely, the divergence theorem).
Consequently,
\begin{equation}\label{parabola}
  \int_{\Real^d} |\nabla\psi(x)|^2 \, \der x
  \geq -[c^2 + c (d-2)] 
  \int_{\Real^d} \frac{|\psi(x)|^2}{|x|^2} \, \der x
\end{equation}
for every $c\in\Real$.
Optimising with respect to~$c$ 
(the parabola achieves its (positive) maximum for $c=-(d-2)/2$,
see Figure~\ref{Fig.parabola}),
we arrive at the desired inequality with the right constant. 
\end{proof}
\begin{figure}[h!] 
\begin{center}
\includegraphics[width=0.6\textwidth]{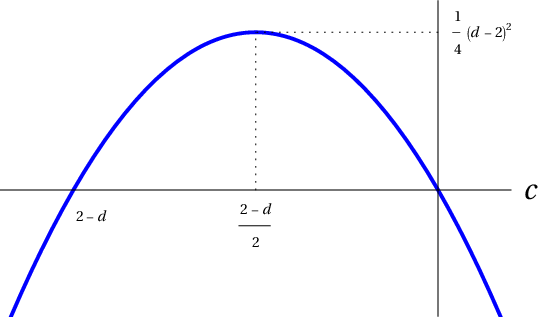} 
\caption{The maximum of the parabola corresponds
to the best constant in~\eqref{parabola}.}\label{Fig.parabola}
\end{center}
\end{figure}
\begin{rema}[Points are negligible except for $d=1$]\label{Rem.Hardy}
Where did we use the requirement $d \geq 3$ in the proof?
The inequality~\eqref{Hardy1} is trivial if $d=2$
(interpreting the right-hand side of~\eqref{Hardy1} as being zero),
so we should comment on the case $d=1$.
The point is that the vector field $x \mapsto x/|x|^2$ is too singular
in dimension one, in order to justify the usage of the divergence theorem.
More specifically, one customarily justifies the manipulations above 
by using functions $\psi \in C_0^\infty(\Real^d \setminus \{0\})$
instead of $\psi \in W_0^{1,2}(\Real^d)$ 
and employs the density result 
of $C_0^\infty(\Real^d \setminus \{0\})$ in $W_0^{1,2}(\Real^d)$.
However, this density holds if, and only if, $d \geq 2$
(see, \eg, \cite[Corol.~VIII.6.4]{Edmunds-Evans}). 

In other words, $W_0^{1,2}(\Real^d) = W_0^{1,2}(\Real^d \setminus \{0\})$
if, and only if, $d \geq 2$. 
That is, points are ``negligible'' for the Sobolev space $W_0^{1,2}$.
More generally, sets of \emph{capacity} zero 
are negligible for the Sobolev space $W_0^{1,2}$.
This statement is an analogue of the fact that
sets of measure zero are negligible for the Lebesgue space $\sii$.
Physically, points cannot carry charge if $d \geq 2$.
\end{rema}

The Hardy inequality~\eqref{Hardy1} is related to spectral properties
of the Dirichlet Laplacian in the following way.
The left-hand side of~\eqref{Hardy1} is just the quadratic form 
$\delta_D^\Omega[\psi] := \delta_D^\Omega(\psi,\psi)$
of the Dirichlet Laplacian in~$\Real^d$
with $\psi \in \dom(\delta_D^{\Real^d})$.
The right-hand side of~\eqref{Hardy1} is the quadratic form
of the operator of multiplication by the function
\begin{equation}\label{rho}
  \rho(x) := \frac{(d-2)^2}{4} \frac{1}{|x|^2}
  \,.	
\end{equation}
Hence, we can write
\begin{equation}\label{Hardy.bis}
  -\Delta_D^{\Real^d} \geq \rho
\end{equation}
in the sense of quadratic forms in $\sii(\Real^d)$.
We also have $-\Delta_\alpha^{\Real^d} \geq \rho$
for every $\alpha\in\Real$,
just because $W_0^{1,2}(\Real^d)=W^{1,2}(\Real^d)$
(there is no boundary for~$\Real^d$).

By Corollary~\ref{Corol.positivity}, 
the spectrum of $-\Delta_D^{\Real^d}$ starts by zero,
so it is impossible that~\eqref{Hardy.bis} holds with~$\rho$ being
replaced by a positive constant.
Anyway, if $d \geq 3$,
inequality~\eqref{Hardy.bis} with a positive function~$\rho$
vanishing at infinity is admissible.
In summary, although the spectrum of the Dirichlet Laplacian 
$-\Delta_D^{\Real^d}$ starts by zero, there is a ``sort of repulsivity''
at the zero energy if $d \geq 3$.
We say that $-\Delta_D^{\Real^d}$ is \emph{subcritical}
(or \emph{satisfies a Hardy inequality}), 
meaning that there exists a positive function~$\rho$ 
satisfying~\eqref{Hardy.bis}.  

\subsection{General notion of subcriticality} 
The notion of subcriticality can be generalised to abstract operators,
at least if the Hilbert space is a function space.

First of all, let us recall that there is a natural order relation
between operators by means of the corresponding quadratic forms.
The following definition takes domains into account,
which is necessary for unbounded operators. 
This notion was already used in~\eqref{Hardy.bis}.
\begin{defi}[Operator inequality]\label{Def.operator.ineq}
Let $H_-$, $H_+$ be two self-ad\-joint operators in~$\Hilbert$
that are bounded from below,
and let $h_-$, $h_+$ be the associated sesquilinear forms.
\begin{center}
$H_- \leq H_+$ 
\quad $:\Longleftrightarrow$ \quad
$
\begin{aligned}
(i) \ & \ \dom h_-  \supset \dom h_+ 
  \,,
  \\
(ii) \ & \ \forall \psi \in \dom h_+, \quad
  h_-[\psi] \leq h_+[\psi]
  \,.
\end{aligned}
$
\end{center}
We say that the inequality $H_- \leq H_+$
holds \emph{in the sense of quadratic forms}.
\end{defi}

Let~$H$ be any non-negative self-adjoint operator in $\sii(\Omega)$   
and let~$h$ be its associated sesquilinear form. 
Given any positive function $\rho \in L_\mathrm{loc}^1(\Omega)$,
we denote by~$M_\rho$ the operator of multiplication in $\sii(\Omega)$  
generated by~$\rho$.
With an abuse of notation, we often write~$\rho$ instead of~$M_\rho$.
The operator~$M_\rho$ is associated with the quadratic form
$$
\begin{aligned}
  m_\rho[\psi] &:= \int_\Omega \rho(x) \, |\psi(x)|^2 \, \der x
  \,, \\
  \dom(m_\rho) &:= 
  \left\{\psi \in \sii(\Omega): 
  \int_\Omega \rho(x) \, |\psi(x)|^2 \, \der x < \infty \right\}
  .
\end{aligned}  
$$
We say that~$H$ is \emph{subcritical} if $H \geq \rho$.
The inequality $H \geq \rho$ is called 
the \emph{(generalised) Hardy inequality},
so we alternatively say that~$H$ 
\emph{satisfies a (generalised) Hardy inequality}.

If the spectrum of~$H$ starts by zero 
but there is no positive $\rho\in L_\mathrm{loc}^1(\Omega)$ 
such that $H \geq \rho$ (\ie~$H$ is not subcritical),
we say that~$H$ is \emph{critical}.
If the spectrum of~$H$ starts below zero,
we say that~$H$ is \emph{supercritical}. 

Hence, if $d \geq 3$,
the Dirichlet Laplacian $-\Delta_D^{\Real^d}$ is subcritical
and satisfies the Hardy inequality~\eqref{Hardy1}.
The Hardy inequality finds applications
in many areas of mathematics and physics.
Here we just mention its role in the quantum stability of matter.

\subsection{Stability of matter}
There is a strong experimental evidence that our world
is composed of atoms and that an atom looks like a microscopic
planetary system 
(\cf~Ruth\-er\-ford's gold-foil experiment with $\alpha$ particles),
see Figure~\ref{Fig.atom}.
There is a heavy, positively charged nucleus,
made of protons and neutrons, which is surrounded by light,
negatively charged electrons.
Although the proton is much (about 1800 times) heavier than the electron,
the gravitational force is negligible on the microscopic level
and it is rather the electrostatic, Coulomb force that bound the electrons
to orbit around the nucleus.

\begin{figure}[h!] 
\begin{center}
\includegraphics[width=0.4\textwidth]{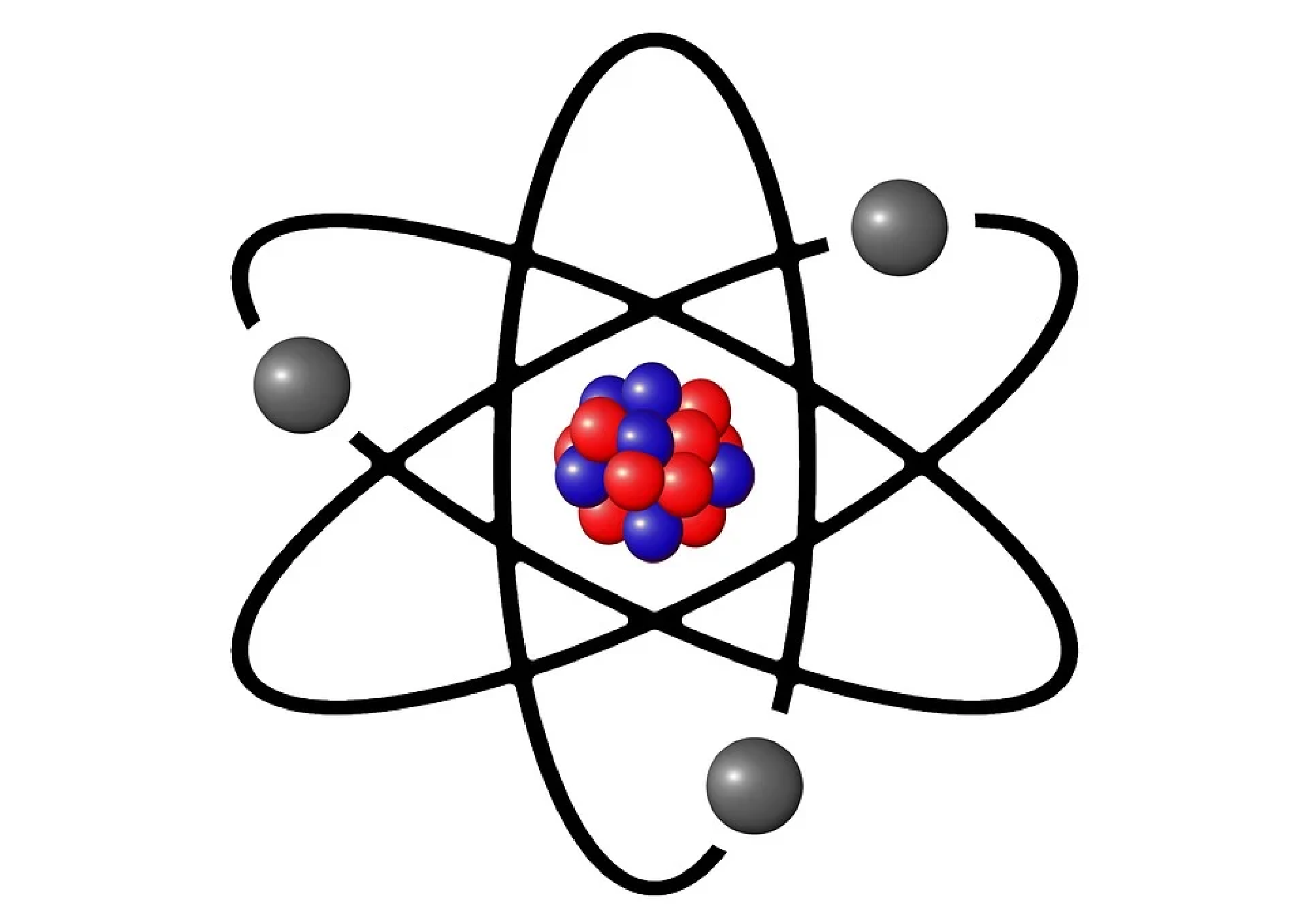} 
\caption{Rutherford's planetary model of the atom.}\label{Fig.atom}
\end{center}
\end{figure}

Now, the following classical paradox arises:
According to the laws of classical electrodynamics,
an accelerated charged particle emits electromagnetic radiation
and loses in this way its total energy.
Consequently, the electron particle would move on a spiral trajectory
and finally \textbf{collapse} on the nucleus.
The atoms should not be stable.
(For instance, the lifetime of a hydrogen atom calculated according
to the classical electrodynamics is less than $1$~nanosecond!)
 
Let us look at the simplest chemical element - hydrogen -
and argue that it cannot be classically stable. 
In classical physics,
the hydrogen atom is described by the Hamilton function
\begin{equation}\label{Hydrogen.clas}
  H(x,p) := \frac{|p|^2}{2m} - \frac{e^2}{|x|}
\end{equation}
in the phase space $\Real^3\times\Real^3 \ni (x,p)$.
Here~$x$ and~$p$ is the position and momentum, respectively, 
$m$~is the reduced mass of the electron-proton couple  
(\ie~$m^{-1}=m_e^{-1}+m_p^{-1}$)
and $e \approx 1.6 \times 10^{-19} \, \mathrm{C}$ is the elementary charge.
The first term represents the kinetic energy of the system,
while the second term is the Coulomb electrostatic potential. 
The instability of the atom can be then mathematically understood
through the unboundedness of the total energy from below, \ie,
\begin{equation}\label{collapse}
  \inf_{(x,p) \in \Real^3\times\Real^3} H(x,p)  = - \infty 
  \,,
\end{equation}
which is exactly caused by making the distance~$|x|$ between 
the electron and the nucleus infinitesimal.

At the same time, the measured spectra of the radiation
absorbed or emitted by an atom consists of discrete frequencies.
This suggests that only a discrete set of electron orbits is allowed.
Contrary to the laws of classical physics,
according to which the energy of a planet
varies continuously with the dimension of the orbit,
which can be arbitrary.

There are other important experimental facts which
cannot be explained on the level of classical physics,
like the corpuscular behaviour of light (photoelectric effect),
the particle-wave duality of matter (Bragg's experiment),
the black-body radiation, \etc.

These strong disagreements between experimental data
and foundations of classical mechanics
lead to a crisis of physics in the beginning of the last century.
Quantum mechanics was invented on the basis of very practical
physical reasons to explain the paradoxes.

In quantum mechanics, the momentum~$p$ is represented
by the differential operator 
\begin{equation}\label{momentum}
  p := -i \hbar \nabla
  \,, \qquad
  \dom p := W^{1,2}(\Real^3;\Com^3)
  \,,
\end{equation}
in the auxiliary Hilbert space $\sii(\Real^3;\Com^3)$,
where~$\hbar \approx 10^{-34} \mathrm{J\,s}$ is the reduced Planck constant.
The position~$x$ is just an operator of multiplication.
The square $|p|^2 := p^* p = -\hbar^2\Delta$ is therefore a multiple of
the (Dirichlet) Laplacian in the scalar Hilbert space $\sii(\Real^3)$.
The hydrogen atom is consequently described by the Hamilton operator
$$
  H := -\frac{\hbar^2}{2m} \Delta - \frac{e^2}{|x|}
$$
acting in the Hilbert space $\sii(\Real^3)$.

A quantum-mechanical analogue of the lowest energy 
of the classical system~\eqref{collapse} is the variational quantity
$$
  E_1 := \inf_{\stackrel[\|\psi\|=1]{}{\psi \in \dom(H)}}(\psi,H\psi) 
  \,.	
$$
We claim that $E_1 > -\infty$, 
which implies the \textbf{stability}
of the hydrogen atom in the quantum setting.
Indeed, for every $\psi \in \dom(H)$, one has
\begin{equation}\label{stability}
\begin{aligned}
  (\psi,H\psi) 
  &= \frac{\hbar^2}{2m} \int_{\Real^3} |\nabla\psi(x)|^2 \, \der x 
  - e^2 \int_{\Real^3} \frac{|\psi(x)|^2}{|x|} \, \der x  
  \\
  &\geq \frac{\hbar^2}{8m} \int_{\Real^3} \frac{|\psi(x)|^2}{|x|^2} \, \der x 
  - e^2 \int_{\Real^3} \frac{|\psi(x)|^2}{|x|} \, \der x 
  \\
  &\geq \min_{r>0} 
  \left( \frac{\hbar^2}{8m r^2} - \frac{e^2}{r} \right)
  = -2\frac{me^4}{\hbar^2}
  \,,
\end{aligned}
\end{equation}
where the first estimate is the Hardy inequality~\eqref{Hardy1} for $d=3$,
the second inequality employs the normalisation $\|\psi\|=1$
and the last equality is elementary (see Figure~\ref{Fig.stability}).
We therefore get the bound
$$
  E_1 \geq -2\frac{me^4}{\hbar^2}
  > -\infty
  \,.
$$

It is remarkable that this estimate is not so far 
from the actual value
$$
  E_1 = -\frac{1}{2}\,\frac{me^4}{\hbar^2}
  \,,
$$
which can be obtained by solving the spectral problem 
for the hydrogen atom explicitly in terms of special functions
(see, \eg, \cite[Sec.~4.2]{Griffiths}).

\begin{figure}[h!] 
\begin{center}
\includegraphics[width=0.55\textwidth]{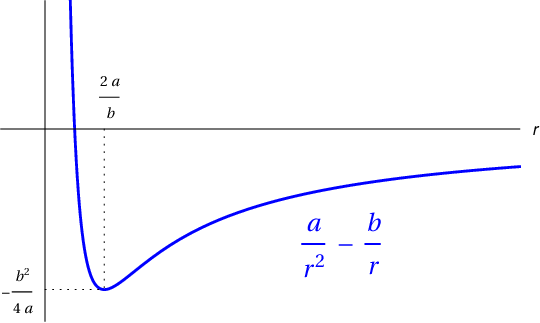} 
\caption{The function appearing in the estimate~\eqref{stability}
with positive constants $a,b$.}\label{Fig.stability}
\end{center}
\end{figure}
\begin{rema}[Uncertainty principle]
Probably the deepest reason behind the stability of atoms
in quantum mechanics is the non-commutative feature of the theory.
It is reflected in the \emph{Heisenberg uncertainty relations}
implying an inevitable limitations for the preparation of states
with sharper and sharper values of both position and momentum.
From this point of view, the Hardy inequality of Theorem~\ref{Thm.Hardy1}
can be interpreted as a sort of the \emph{uncertainty principle}.
Indeed, the boundedness from below of
the hydrogen Hamiltonian~$H$ is its consequence
and $E_1 > -\infty$ is equivalent to
$$
  \forall \psi \in W_0^{1,2}(\Real^d)
  \,, \qquad
  \frac{\hbar^2}{2m}  \int_{\Real^3} |\nabla\psi(x)|^2 \, \der x
  - e^2 \int_{\Real^3} \frac{|\psi(x)|^2}{|x|} \, \der x
  > -\infty
  \,.
$$
The classical counterpart of the energy form is unbounded from below
because of the singularity of the potential energy
at the nucleus position $x=0$.
However, a quantum electron is not allowed to reach the nucleus,
because a strict localisation close to the nucleus
would make the kinetic energy very large.
\end{rema}

\subsection{Criticality of low dimensions}
It turns out that the Dirichlet Laplacian $-\Delta_D^{\Real^d}$
is critical in dimensions $d=1,2$.
In other words, there is no Hardy inequality,
that is, no inequality of the type~\eqref{Hardy.bis}
with a positive function~$\rho$ is admissible.  
\begin{theo}\label{Thm.virtual}
Let $d = 1,2$. 
For any positive function $\rho \in L_\mathrm{loc}^1(\Real^d)$,
one has
\begin{equation}\label{virtual}
  \inf_{\stackrel[\psi\not=0]{}{\psi \in C_0^\infty(\Real^d)}}
  \left(
  \int_{\Real^d} |\nabla\psi(x)|^2 \, \der x
  - \int_{\Real^d} \rho(x) \, |\psi(x)|^2 \, \der x
  \right)
  < 0
  \,.
\end{equation}
\end{theo}

Before proving the theorem, 
let us first comment on why the result~\eqref{virtual}
contradicts the validity of the Hardy inequality~\eqref{Hardy.bis}. 
The latter precisely means that if $\psi \in W_0^{1,2}(\Real^d)$,
then $\rho^{1/2}\psi \in \sii(\Real^d)$ 
and the quadratic form 
\begin{equation}\label{formQ}
  Q[\psi] :=
  \int_{\Real^d} |\nabla\psi(x)|^2 \, \der x
  - \int_{\Real^d} \rho(x) \, |\psi(x)|^2 \, \der x
\end{equation}
is non-negative. Since $C_0^\infty(\Real^d) \subset W_0^{1,2}(\Real^d)$,
it follows that $Q[\psi] \geq 0$ for every $\psi \in C_0^\infty(\Real^d)$.
But this is an obvious contradiction with~\eqref{virtual}.
So, indeed, no Hardy inequality~\eqref{Hardy.bis} is available 
in dimensions $d=1,2$.

\begin{proof}
Clearly, to establish~\eqref{virtual},
it is enough to find a (so-called ``trial'') 
function $\psi \in C_0^\infty(\Real^d)$
such that $Q[\psi] < 0$.
Forgetting for a moment that~$1$ 
(\ie~the constant function everywhere equal to one)
is not admissible and using the pointwise identity $\nabla 1 = 0$,
we formally have 
\begin{equation}\label{test.formal}
  Q[1] = -\int_{\Real^d} \rho(x) \, \der x < 0 \,.
  \qquad \mbox{(formally!)}
\end{equation}
Hence, the idea is to use a trial function which approximates~$1$,
but it is still an admissible element of $C_0^\infty(\Real^d)$.
We thus look for a sequence
$\{\psi_n\}_{n=1}^\infty \subset C_0^\infty(\Real^d)$
such that
\begin{enumerate}
\item[(i)]
$
  \forall x \in \Real^d, \quad
  \psi_n(x) \xrightarrow[n\to\infty]{} 1
$,
\item[(ii)]
$
  \|\nabla\psi_n\| \xrightarrow[n\to\infty]{} 0
$.
\end{enumerate}
Such a sequence exists only in dimensions $d=1,2$.

\fbox{$d=1$.} \
If $d=1$,
we pick a function $\varphi \in C_0^\infty(\Real)$ such that 
$$
  0 \leq \varphi \leq 1
  \,, \qquad
  \varphi=1 \quad\mbox{on}\quad [-1,1]
  \,, \qquad
  \varphi=0 \quad\mbox{outside}\quad [-2,2]
  \,.
$$
For every $n\in\Nat^*$, we then define 
(\cf~Figure~\ref{Fig.1D})
$$
  \psi_n(x):=\varphi\left(\frac{x}{n}\right)
  \,.
$$
Notice that $\psi_n=1$ on $[-n,n]$ and $\psi_n=0$ outside $[-2n,2n]$,
so it is certainly an admissible approximation of the constant function~$1$;
in fact $\psi_n \to 1$ pointwise as $n \to \infty$.
By an obvious change of variables, we have
$$
  \int_{\Real} |\psi_n'(x)|^2 \, \der x
  = \int_{\Real} \left|\frac{1}{n}\,\varphi'\left(\frac{x}{n}\right)\right|^2 \, \der x
  = \frac{1}{n} \int_{\Real} |\varphi'(x)|^2 \, \der x
  \xrightarrow[n \to \infty]{} 0 \,,
$$
so the first term on the right-hand side of~\eqref{formQ}
vanishes as $n \to \infty$.
For the second term, we have
$$
  \int_{\Real^d} \rho(x) \, |\psi_n(x)|^2 \, \der x
  \xrightarrow[n \to \infty]{}
  \int_{\Real^d} \rho(x) \, \der x
$$
by the monotone convergence theorem (the limit can be infinite).
In summary,
$$
  Q[\psi_n]  \xrightarrow[n \to \infty]{}
  -\int_{\Real^d} \rho(x) \, \der x
  \,,
$$
so the formal result~\eqref{test.formal} is obtained in a limit sense.
Since the right hand-side is negative (possibly $-\infty$),
there obviously exists $n \in \Nat^*$ such that $Q[\psi_n] < 0$. 
This concludes the proof in the one-dimensional case.

\fbox{$d=2$.} \
If $d=2$, we have to use a more refined approximation of~$1$. 
We start by picking a function $\eta \in C^\infty([0,1])$ such that
$$
  0 \leq \eta \leq 1
  \,, \qquad
  \eta=0 \quad\mbox{on}\quad [0,\mbox{$\frac{1}{4}$}]
  \,, \qquad
  \eta=1 \quad\mbox{on}\quad [\mbox{$\frac{3}{4}$},1]
  \,.
$$
For every $n\in\Nat$ with $n \geq 2$, we then define 
(\cf~Figure~\ref{Fig.1D})
$$
  \psi_n(x) :=
  \begin{cases}
    1
    & \mbox{if} \quad |x| \leq n \,,
    \\
    \displaystyle
    \eta\left(\frac{\log n^2 - \log |x|}{\log n^2 - \log n}\right)
    & \mbox{if} \quad  n < |x| < n^2 ,
    \\
    0
    & \mbox{if} \quad |x| \geq n^2 \,.
  \end{cases}
$$
Again $\psi_n \in C_0^\infty(\Real^2)$ for every $n \geq 2$
and $\psi_n \to 1$ pointwise as $n \to \infty$.
Passing to polar coordinates
and making an obvious change of variables, 
we have
$$
\begin{aligned}
  \int_{\Real^2} |\nabla\psi_n(x)|^2 \, \der x
  &= 2\pi \int_{n}^{n^2} 
  \left| \frac{-1}{r\,(\log n^2 - \log n)}
  \ \eta'\left(
  \frac{\log n^2 -\log r}{\log n^2 - \log n}
  \right) \right|^2
  \, r \, \der r
  \\
  &= \frac{2\pi}{\log n^2 - \log n} \int_{0}^{1} 
  \left|\eta'(s) \right|^2
  \, \der s
  \xrightarrow[n \to \infty]{} 0 \,,
\end{aligned}
$$
so the first term on the right-hand side of~\eqref{formQ}
again vanishes as $n \to \infty$.
The rest of the proof is the same as in the one-dimensional case.
\end{proof}
\begin{figure}[h!]
\begin{center}
\begin{tabular}{cc}
\includegraphics[width=0.45\textwidth]{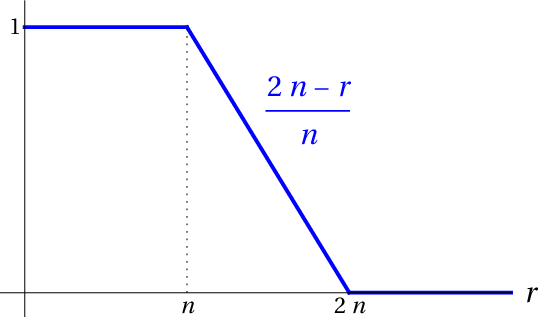}
&\includegraphics[width=0.45\textwidth]{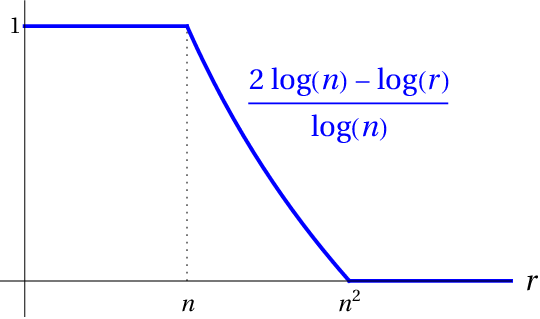}
\end{tabular}
\end{center}
\caption{The radial profile of the functions $\psi_n$ approximating~$1$ 
pointwise as $n \to \infty$ are smoothed version of 
the profiles on the left and right for $d=1$ and $d=2$, respectively.}
\label{Fig.1D}
\end{figure}

Let us summarise the dimensional features of the Euclidean space~$\Real^d$.
Due to Corollary~\ref{Corol.positivity}, 
the spectrum of the Dirichlet Laplacian in~$\Real^d$ 
is the same for all dimensions $d \geq 1$, 
namely it is equal to the interval $[0,\infty)$.
However, there is a fundamental difference at the zero energy.
There is a ``sort of repulsivity'' (respectively, ``sort of attractivity'')
at the zero energy if $d \geq 3$ (respectively, $d=1,2$).
We have quantified this by the respective existence or non-existence 
of Hardy inequalities.
More specifically, 
Theorems~\ref{Thm.Hardy1} and~\ref{Thm.virtual} can be 
summarise into the following theorem.
\begin{theo}[$-\Delta_D^{\Real^d}$\label{Thm.critical}
satisfies a Hardy inequality
if, and only if, $d \geq 3$.]\label{Thm.d-Hardy}
$$
-\Delta_D^{\Real^d}
\quad\mbox{is}\quad
\begin{cases}
  \mbox{subcritical} & \mbox{if} \quad d \geq 3 \,,
  \\
  \mbox{critical} & \mbox{if} \quad d = 1,2 \,,
\end{cases}
$$
\end{theo}

This observation has far reaching consequences in many areas
of physics and mathematics. For instance, in stochastic analysis,
it is related to the very different behaviour of the Brownian motion 
in~$\Real^d$ depending on whether $d=1,2$ or $d \geq 3$. 
Namely, the Brownian motion is \emph{recurrent} 
on the real line and in the plane 
(meaning that the Brownian particle visits every region 
infinitely many times), 
while it is \emph{transient} in $\Real^d$ with $d \geq 3$
(meaning that it escapes from any bounded region
after some time forever).  

In this course, we have interpreted Theorem~\ref{Thm.d-Hardy}
through the stability of matter: $\Real^3$~is the lowest dimensional
Euclidean space for which the atoms and molecules 
are quantum-mechanically stable.

\subsection{The half-line and the optimality}\label{Sec.alternative.Hardy}
The Hardy inequality (Theorem~\ref{Thm.Hardy1}) is so important
that we dedicate this section to provide more insights into it.
First of all, we give an alternative proof of it
based on the following one-dimensional Hardy inequality.

\begin{theo}[One-dimensional Hardy inequality]\label{Thm.Hardy.1D}
One has
\begin{equation}\label{Hardy.1D}
  \forall \psi \in W_0^{1,2}((0,\infty))
  \,, \qquad
  \int_0^\infty |\psi'(x)|^2 \, \der x
  \geq \frac{1}{4} \int_0^\infty \frac{|\psi(x)|^2}{|x|^2} \, \der x
  \,.
\end{equation}
\end{theo}

Obviously, the same inequality holds for 
$\psi \in W_0^{1,2}(\Real \setminus \{0\})$
with the integrals over~$\Real$.
Note carefully that, in any case,  
the function~$\psi$ is required to vanish at the origin.
Without this requirement there is no one-dimensional Hardy inequality
due to Theorem~\ref{Thm.virtual}.
In other words, while $-\Delta_D^\Real$ is critical,
both $-\Delta_D^{\Real\setminus\{0\}}$ 
and $-\Delta_D^{(0,\infty)}$ are subcritical.\
Indeed, \eqref{Hardy.1D}~is equivalent to the inequality
$$
  -\Delta_D^{(0,\infty)} \geq \rho
$$
in the sense of quadratic forms in $\sii((0,\infty))$,
where~$\rho$ is given by~\eqref{rho} with $d=1$.

\begin{proof}[Alternative proof of Theorem~\ref{Thm.Hardy1}]
By density, 
it is enough to prove the Hardy inequality~\eqref{Hardy1} 
of Theorem~\ref{Thm.Hardy1}
for $\psi \in C_0^\infty(\Real^d\setminus\{0\})$,
\cf~Remark~\ref{Rem.Hardy}.
Passing to spherical coordinates
$(r,\theta) \in (0,\infty) \times \Sphere^{d-1}$
and neglecting the angular-derivative term,
we get the bound
\begin{align*}
  \int_{\Real^d} |\nabla\psi(x)|^2 \, \der x
  &= \int_{(0,\infty) \times \Sphere^{d-1}} \left(
  |\partial_r\tilde\psi(r,\theta)|^2
  + \frac{|\nabla_{\!\theta}\tilde\psi(r,\theta)|^2}{r^2}
  \right) r^{d-1} \, \der r \, \der \theta
  \\
  &\geq \int_{(0,\infty) \times \Sphere^{d-1}}
  |\partial_r\tilde\psi(r,\theta)|^2
  \ r^{d-1} \, \der r \, \der\theta
  =: t[\tilde\psi]
  \,,
\end{align*}
where~$\tilde\psi$ is the function~$\psi$
expressed in the spherical coordinates,
$\der\theta$ is the surface element of 
the $(d-1)$-dimensional sphere $\Sphere^{d-1}$
and $\nabla_{\!\theta}$ denotes the spherical gradient.
Making the change of test function
\begin{equation}\label{change.trial}
  \phi(r,\theta) := \sqrt{r^{d-1}} \, \tilde\psi(r,\theta)
\end{equation}
and integrating by parts, we arrive at the identity
\begin{equation*}
  t[\tilde\psi]
  = \int_{(0,\infty) \times \Sphere^{d-1}} \left\{
  |\partial_r\phi(r,\theta)|^2
  + \frac{(d-1)(d-3)}{4} \,
  \frac{|\phi(r,\theta)|^2}{r^2}
  \right\}
  \der r \, \der\theta
  \,.
\end{equation*}
For every $\theta \in \Sphere^{d-1}$, the function
$r \mapsto \phi(r,\theta)$ belongs to $C_0^\infty((0,\infty))$.
Consequently, applying the one-dimensional Hardy inequality~\eqref{Hardy.1D}
of Theorem~\ref{Thm.Hardy.1D}
with help of Fubini's theorem, we finally get
$$
  t[\tilde\psi]
  \geq \left[\frac{1}{4}+
  \frac{(d-1)(d-3)}{4} 
  \right]
  \int_{(0,\infty) \times \Sphere^{d-1}}
  \frac{|\phi(r,\theta)|^2}{r^2}
  \ \der r \, \der \theta
  \,.
$$
This estimate coincides with the desired inequality~\eqref{Hardy1}
after coming back to Cartesian coordinates.

Again we see that the proof does not give
any non-trivial inequality in low dimensions $d=1,2$.
In $d=2$, our proof holds but the outcome is trivial
(the right-hand side of~\eqref{Hardy1} vanishes).
In $d=1$, the technical reason for the breakdown of the proof 
is mentioned in Remark~\ref{Rem.Hardy}:
we are not allowed to restrict 
to functions supported outside the origin.
We can still take $\psi \in C_0^\infty(\Real)$,
which is a dense subspace of $W^{1,2}(\Real)$,
but then the test function~$\phi$ would not vanish at the origin,
so Theorem~\ref{Thm.Hardy.1D} does not apply.
Indeed, in one dimension, the ``spherical'' coordinates are trivial,
there is no Jacobian, so in fact $\phi=\tilde{\psi}$.
\end{proof}

The one-dimensional Hardy inequality of Theorem~\ref{Thm.Hardy.1D}
can be established by the idea of the original proof of Theorem~\ref{Thm.Hardy1}.
An alternative proof goes as follows.
\begin{proof}[Proof of Theorem~\ref{Thm.Hardy.1D}]
For every $\psi \in C_0^\infty((0,\infty))$,
\begin{align*}
  \int_0^\infty \frac{|\psi(x)|^2}{x^2} \, \der x
  &= - \int_0^\infty \frac{\der}{\der x}\!\left(\frac{1}{x}\right) 
  |\psi(x)|^2\, \der x
  \\
  &= \int_0^\infty \frac{1}{x} \,
  2 \ \Re \!\left\{ \overline{\psi(x)} \, \psi'(x) \right\} \, \der x
  \\
  &\leq 2 \, \sqrt{\int_0^\infty \frac{|\psi(x)|^2}{x^2}\,\der x} \,
  \sqrt{\int_0^\infty |\psi'(x)|^2\,\der x}
  \,,
\end{align*}
where the second equality follows by an integration by parts
and the inequality is due to the Schwarz inequality.
The result is a square-root version of the desired inequality.
\end{proof}

Both Theorems~\ref{Thm.Hardy1} and~\ref{Thm.Hardy.1D} 
are optimal in various aspects \cite{Devyver-Fraas-Pinchover_2014}.
Here we discuss the optimality of Theorem~\ref{Thm.Hardy.1D} only.
The analogous claims for the multidimensional Theorem~\ref{Thm.Hardy1}
can be obtained by means of spherical coordinates,
using the optimising functions independent of the spherical variables.

\subsubsection*{Non-attainability}
The Hardy inequality~\eqref{Hardy.1D} 
is never achieved (by a non-trivial function),
meaning that there is no (non-zero) function $\psi \in W_0^{1,2}((0,\infty))$
for which there is equality in~\eqref{Hardy.1D}.
Indeed, for any $\psi \in C_0^\infty((0,\infty))$,
it is easy to check that
\begin{equation}\label{Hardy.identity}
  a[\psi] :=
  \int_0^\infty
  \left(
  |\psi'(x)|^2 - \frac{1}{4} \frac{|\psi(x)|^2}{x^2}
  \right)
  \der x
  = \int_0^\infty \left|
  \frac{\der}{\der x}\left(\frac{\psi(x)}{\sqrt{x}}\right)
  \right|^2 x \, \der x
  \geq 0
\end{equation}
and by density the identity extends to all $\psi \in W_0^{1,2}((0,\infty))$.
(This is yet another proof of Theorem~\ref{Thm.Hardy.1D}.)
Now, assume that there exists $\psi \in W_0^{1,2}((0,\infty))$ such that
the Hardy inequality turns into equality. Then $a[\psi] = 0$.
It follows from identity~\eqref{Hardy.identity} that
$\psi(x) = C \sqrt{x}$ for a.e.\ $x \in (0,\infty)$
with some constant $C \in \Com$.
But this is an admissible function from $W_0^{1,2}((0,\infty))$
only if $C=0$.

\subsubsection*{Asymptotic attainability}
The Hardy inequality~\eqref{Hardy.1D} is achieved asymptotically, 
meaning that 
$$
  \inf_{\stackrel[\psi\not=0]{}{\psi \in W_0^{1,2}((0,\infty))}} 
  \frac{\displaystyle a[\psi]}
  {\displaystyle \|\psi\|^2}
  = 0
  \,.
$$
Motivated by the result of the previous remark,
we construct an optimising sequence
by regularising the square-root function $x \mapsto \sqrt{x}$.

Let $\xi \in C^\infty([0,1])$ be such that $\xi=0$
in a right neighbourhood of~$0$ and $\xi=1$ 
in a left neighbourhood of~$1$.
For every $n \in \Nat$ with $n \geq 2$, 
we define (see Figure~\ref{Fig.optimal})
$$
  \xi_n(x) :=
  \begin{cases}
    0 & \mbox{if} \quad x \in [0,1/n^2) \,,
    \\
    \xi\left(\frac{\displaystyle\log(n^2 x)}{\displaystyle\log n}\right) 
    & \mbox{if} \quad x \in [1/n^2,1/n) \,,
    \\
    1 & \mbox{if} \quad x \in [1/n,n) \,,
    \\
    \xi\left(\frac{\displaystyle\log(n^2/x)}{\displaystyle\log n}\right) 
    & \mbox{if} \quad x \in [n,n^2) \,,
    \\
    0 & \mbox{if} \quad x \in [n^2,\infty) \,.
  \end{cases}
$$
Note that $\xi_n \in C_0^\infty((0,\infty))$
and $\xi_n(x) \to 1$ as $n \to \infty$ for every $x \in (0,\infty)$.
We set $\psi_n(x) := \xi_n(x) \sqrt{x}$ for every $x>0$.
Then $\|\psi_n\| \to \infty$ as $n \to \infty$, while
$$
\begin{aligned}
  a[\psi_n] 
  &= \int_0^\infty |\xi_n'(x)|^2 \, x \, \der x
  \\
  &= \frac{1}{\log^2 n} \int_{1/n^2}^{1/n} 
  \left|
  \xi'\left(\frac{\displaystyle\log(n^2 x)}{\displaystyle\log n}\right)
  \right|^2 
  \, \frac{1}{x} \, \der x
  \\
  & \qquad
  + \frac{1}{\log^2 n} \int_{n}^{n^2}
  \left|
  \xi'\left(\frac{\displaystyle\log(n^2/x)}{\displaystyle\log n}\right)
  \right|^2 
  \, \frac{1}{x} \, \der x
  \\
  &\leq  
  \frac{\|\xi'\|_\infty^2}{\log^2 n} \int_{1/n^2}^{1/n} 
  \frac{1}{x} \, \der x
  + \frac{\|\xi'\|_\infty^2}{\log^2 n} \int_{n}^{n^2}
  \frac{1}{x} \, \der x
  \\
  &=
  2 \, \frac{\|\xi'\|_\infty^2}{\log n}   
  \xrightarrow[n\to\infty]{}
  0 \,,
\end{aligned}  
$$
where we have denoted $\|\xi'\|_\infty := \max|\xi'|$.

\subsubsection*{Optimality of the constant}
The constant $\frac{1}{4}$ in~\eqref{Hardy.1D}
cannot be improved, meaning that 
$$
  \inf_{\stackrel[\psi\not=0]{}{\psi \in W_0^{1,2}((0,\infty))}} 
  \frac{\displaystyle \int_0^\infty |\psi'(x)|^2 \, \der x}
  {\displaystyle \int_0^\infty \frac{|\psi(x)|^2}{x^2} \, \der x}
  = \frac{1}{4}
  \,.
$$
Using the same sequence $\{\psi_n\}_{n=2}^\infty$ as above, 
one has 
$$
  \int_0^\infty \frac{|\psi_n(x)|^2}{x^2} \,\der x
  \geq \int_{1/n}^n \frac{1}{x} \,\der x
  = 2 \log n 
  \xrightarrow[n \to \infty]{}
  \infty 
  \,.
$$
Consequently,
$$
  \frac{\displaystyle \int_0^\infty |\psi_n'(x)|^2 \, \der x}
  {\displaystyle \int_0^\infty \frac{|\psi_n(x)|^2}{x^2} \,\der x}
  = \frac{1}{4} +
  \frac{a[\psi_n]}
  {\displaystyle \int_0^\infty \frac{|\psi_n(x)|^2}{x^2} \,\der x} 
  \xrightarrow[n \to \infty]{}
  \frac{1}{4}
  \,.
$$

\subsubsection*{Optimality of the weight}
The weight on the right-hand side of~\eqref{Hardy.1D}
cannot be improved, meaning that 
$$
  \inf_{\stackrel[\psi\not=0]{}{\psi \in W_0^{1,2}((0,\infty))}} 
  \frac{a[\psi] + v[\psi]}{\|\psi\|^2}
  < 0
  \,, \qquad \mbox{where} \qquad
  v[\psi] := \int_0^\infty V(x) \, |\psi(x)|^2 \, \der x
  \,,
$$
for any non-positive non-trivial function 
$V \in L_\mathrm{loc}^1((0,\infty))$.
In other words, 
the shifted operator $-\Delta_D^{(0,\infty)} - \rho$ is critical.
Obviously, this result is stronger than the optimality of the constant above.
(In the terminology of~\cite{Devyver-Fraas-Pinchover_2014},
combining this result with the non-attainability above,
the shifted operator $-\Delta_D^{(0,\infty)} - \rho$
is actually \emph{null-critical}.)

It is enough to show that there exists 
a trial function $\psi \in W_0^{1,2}((0,\infty))$
such that $a[\psi] + v[\psi] < 0$.
Using still the same sequence $\{\psi_n\}_{n=2}^\infty$ as above, 
one has
$$
  \lim_{n \to \infty} \left( a[\psi_n] + v[\psi_n] \right)
  = \lim_{n \to \infty} v[\psi_n]
  = \int_0^\infty V(x) \, x \, \der x
  < 0
  \,,
$$
where the last equality follows by the monotone convergence theorem
(the final integral can be $-\infty$).
Consequently, there exists $n_0 \geq 2$ such that 
$a[\psi_n] + v[\psi_n] < 0$ for all $n \geq n_0$.

\subsubsection*{Optimality of the weight at infinity}
The weight on the right-hand side of~\eqref{Hardy.1D} 
is optimal also in the sense that it has the optimal decay at infinity,
meaning that 
$$
  \inf_{\stackrel[\psi\not=0]{}{\psi \in W_0^{1,2}((0,\infty)\setminus K)}} 
  \frac{\displaystyle \int_0^\infty |\psi'(x)|^2 \, \der x}
  {\displaystyle \int_0^\infty \frac{|\psi(x)|^2}{x^2} \, \der x}
  = \frac{1}{4}
$$
for any compact set $K \in (0,\infty)$.
To prove it, we modify the optimising sequence $\{\psi_n\}_{n=2}^\infty$
from above as follows.
For every $n \in \Nat$ with $n \geq 2$, 
we now define (see Figure~\ref{Fig.optimal})
$$
  \tilde\xi_n(x) :=
  \begin{cases}
    0 & \mbox{if} \quad x \in [0,n) \,,
    \\
    \xi\left(\frac{\displaystyle\log(x/n)}{\displaystyle\log n}\right) 
    & \mbox{if} \quad x \in [n,n^2) \,,
    \\
    1
    & \mbox{if} \quad x \in [n^2,2 n^2) \,,
    \\
    \xi\left(\frac{\displaystyle\log(2n^4/x)}{\displaystyle\log n^2}\right) 
    & \mbox{if} \quad x \in [2 n^2,2n^4) \,,
    \\
    0 & \mbox{if} \quad x \in [2 n^4,\infty) \,.
  \end{cases}
$$

Note that $\tilde{\xi}_n \in C_0^\infty((0,\infty))$
and $\inf\supp \tilde{\xi}_n \geq n \to \infty$ as $n \to \infty$.
In particular, 
given any compact set $K \in (0,\infty)$,
$K \cap \supp \tilde{\xi}_n = \varnothing$ 
for all~$n$ sufficiently large.
We set $\tilde{\psi}_n(x) := \tilde{\xi}_n(x) \sqrt{x}$ for every $x>0$.
Then
$$
\begin{aligned}
  a[\tilde{\psi}_n] 
  &= \int_0^\infty |\tilde{\xi}_n'(x)|^2 \, x \, \der x
  \\
  &= \frac{1}{\log^2 n} \int_{n}^{n^2} 
  \left|
  \xi'\left(\frac{\displaystyle\log(x/n)}{\displaystyle\log n}\right)
  \right|^2 
  \, \frac{1}{x} \, \der x
  \\
  & \qquad
  + \frac{1}{\log^2 n^2} \int_{2n^2}^{2n^4}
  \left|
  \xi'\left(\frac{\displaystyle\log(2n^4/x)}{\displaystyle\log n^2}\right)
  \right|^2 
  \, \frac{1}{x} \, \der x
  \\
  &\leq  
  \frac{\|\xi'\|_\infty^2}{\log^2 n} \int_{n}^{n^2} 
  \frac{1}{x} \, \der x
  + \frac{\|\xi'\|_\infty^2}{\log^2 n^2} \int_{2n^2}^{2n^4}
  \frac{1}{x} \, \der x
  \\
  &=
  \frac{3}{2} \, \frac{\|\xi'\|_\infty^2}{\log n}   
 \,,
\end{aligned}  
$$
while
$$
  \int_0^\infty \frac{|\tilde{\psi}_n(x)|^2}{x^2} \, \der x
  = \int_n^{2 n^4} \frac{|\tilde{\xi}_n(x)|^2}{x} \, \der x
  \geq \int_{n^2}^{2 n^2} \frac{1}{x} \, \der x
  = \log 2
  \,.
$$
Consequently,
$$
  \frac{\displaystyle \int_0^\infty |\tilde\psi_n'(x)|^2 \, \der x}
  {\displaystyle \int_0^\infty \frac{|\tilde\psi_n(x)|^2}{x^2} \,\der x}
  = \frac{1}{4} +
  \frac{a[\tilde\psi_n]}
  {\displaystyle \int_0^\infty \frac{|\tilde\psi_n(x)|^2}{x^2} \,\der x} 
  \xrightarrow[n \to \infty]{}
  \frac{1}{4}
  \,.
$$

\begin{figure}[h!]
\begin{center}
\begin{tabular}{cc}
\includegraphics[width=0.45\textwidth]{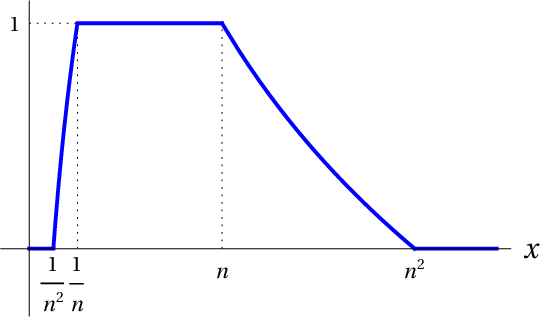}
& \includegraphics[width=0.45\textwidth]{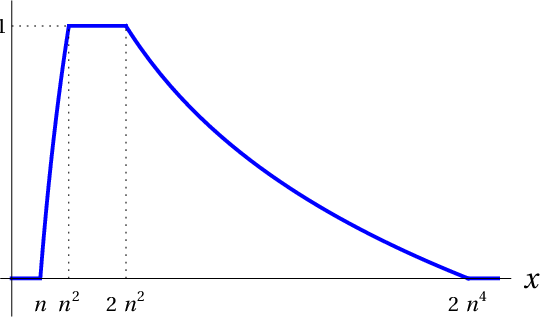}
\end{tabular}
\end{center}
\caption{The regularising functions $\xi_n$ and  $\tilde\xi_n$
are smoothed versions of the profiles 
on the left and right, respectively.}
\label{Fig.optimal}
\end{figure}

\subsection{A logarithmic Hardy inequality and a generic subcriticality}
Let us conclude this section by the following 
one-dimensional Hardy-type inequality.

\begin{lemm}\label{Lem.Hardy.2D}
For any positive number $x_0$,
\begin{equation}\label{Hardy.2D}
  \forall \psi \in W_0^{1,2}((x_0,\infty))
  \,, \qquad
  \int_{x_0}^\infty
  |\psi'(x)|^2 \, x \, \der x
   \geq \frac{1}{4}
  \int_{x_0}^\infty \frac{|\psi(x)|^2}{x^2\log^2(x/x_0)} \, x \, \der x
  \,.
\end{equation}
\end{lemm}
\begin{proof}
It is enough to prove the inequality for~$\psi$ from $C_0^\infty((x_0,\infty))$,
a dense subspace of $W_0^{1,2}((x_0,\infty))$. 
For any real constant $c$, 
we employ the usual integration-by-parts trick:
\begin{align*}
\lefteqn{
  \int_{x_0}^\infty
  \left|
  \psi'(x)- \frac{c}{x \log(x/x_0)} \psi(x)
  \right|^2 x \, \der x
  }
  \\
  &= \int_{x_0}^\infty |\psi'(x)|^2 \, x \, \der x
  + c^2 \int_{r_0}^\infty \frac{|\psi(x)|^2}{x^2 \log^2(x/x_0)} \, x \, \der x
  - c \int_{x_0}^\infty \frac{(|\psi(x)|^2)'}{\log(x/x_0)} \, \der x
  \\
  &= \int_{x_0}^\infty |\psi'(x)|^2 \, x \, \der x
  + (c^2-c) \int_{x_0}^\infty 
  \frac{|\psi(x)|^2}{x^2 \log^2(x/x_0)} \, x \, \der x
  \,.
\end{align*}
Choosing $c:=1/2$, we get~\eqref{Hardy.2D}.
\end{proof}

We note that the left-hand side of~\eqref{Hardy.2D}
is the radial component of the quadratic form 
of the two-dimensional Laplacian in the exterior
of the disk $B_{x_0}(0)$.
It follows that the Dirichlet Laplacian in this exterior is subcritical,
a property which cannot be deduced from the classical 
Hardy inequalities of Theorems~\ref{Thm.Hardy1} and~\ref{Thm.Hardy.1D}. 
What is more, we have the following robust result.

\begin{theo}\label{Thm.generic}
Let $\Omega \subset \Real^d$ be an arbitrary open set
satisfying 
$$
  \Real^d \setminus \overline{\Omega}
  \not= \varnothing
  \,.
$$
Then $-\Delta_D^\Omega$ is subcritical. 
\end{theo}
\begin{proof}
The proof is very similar to the alternative proof
of the classical Hardy inequality 
presented in Section~\ref{Sec.alternative.Hardy}.
Without loss of generality, we may assume that
$0 \in \Omega^\mathrm{ext} := \Real^d \setminus \overline{\Omega}$.
By hypothesis, there exists $\eps>0$ such that
the ball $B_\eps(0)$ is contained in~$\Omega^\mathrm{ext}$
(because it is a non-empty open set).
Let $\psi \in C_0^\infty(\Omega)$
and extend it by zero to the whole~$\Real^d$.

\fbox{$d \not= 2$}
We can proceed exactly as in the alternative 
proof of Theorem~\ref{Thm.Hardy1}:
passing to spherical coordinates,
neglecting the angular-derivative term
and using the one-dimensional Hardy inequality (Theorem~\ref{Thm.Hardy.1D})
after the change of trial function~\eqref{change.trial},
we arrive at the inequality ($\Real^d$ can be replaced by~$\Omega$)
\begin{equation}\label{Hardy.generic}
  \int_{\Real^d} |\nabla\psi(x)|^2 \, \der x
  \geq \frac{(d-2)^2}{4} \int_{\Real^d} \frac{|\psi(x)|^2}{|x|^2} \, \der x
  \,.
\end{equation}
It looks like the classical Hardy inequality of Theorem~\ref{Thm.Hardy1},
but the difference is that the present inequality holds
in \emph{all} dimensions (including $d=1,2$).
This is due to the fact that the function
$$
  r \mapsto \sqrt{r^{d-1}} \, \tilde\psi(r,\theta)
  \,,
$$
where~$\tilde\psi$ is the function~$\psi$
expressed in the spherical coordinates,
belongs (for every $\theta \in \Sphere^{d-1}$) to $W_0^{1,2}((0,\infty))$
even if $d=1,2$, just because it is identically zero
in a neighbourhood of $r=0$ by the hypothesis.
By density, \eqref{Hardy.generic}~extends to all $\psi \in W_0^{1,2}(\Omega)$
and we may write
$$
  -\Delta_D^\Omega \geq \frac{(d-2)^2}{4} \frac{1}{|x|^2}
  \,,
$$
in the sense of quadratic forms in~$\sii(\Omega)$.
The right-hand side is a positive function whenever $d \not= 2$.

\fbox{$d = 2$}
If $d=2$, inequality~\eqref{Hardy.generic} still holds, but it is trivial.
In the two-dimensional situation, we slightly modify the proof above.
Passing to polar coordinates
$(r,\theta) \in (0,\infty) \times \Sphere^{1}$
and neglecting the angular-derivative term as above,
but using Lemma~\ref{Lem.Hardy.2D} (instead of Theorem~\ref{Thm.Hardy1}),
we get the bound
\begin{align*}
  \int_{\Omega} |\nabla\psi(x)|^2 \, \der x
  &= \int_{(\eps,\infty) \times \Sphere^{1}} \left(
  |\partial_r\tilde\psi(r,\theta)|^2
  + \frac{|\nabla_{\!\theta}\tilde\psi(r,\theta)|^2}{r^2}
  \right) r \, \der r \, \der \theta
  \\
  &\geq \int_{(\eps,\infty) \times \Sphere^{1}}
  |\partial_r\tilde\psi(r,\theta)|^2
  \ r \, \der r \, \der\theta
  \\
  &\geq \frac{1}{4} \int_{(\eps,\infty) \times \Sphere^{1}}
  \frac{|\tilde\psi(r,\theta)|^2}{r^2 \log^2(r/\eps)}
  \ r \, \der r \, \der\theta
  \\
  &= \frac{1}{4} \int_{\Omega}
  \frac{|\psi(x)|^2}{|x|^2 \log^2(|x|/\eps)}
  \ \der x
  \,.
\end{align*}
By density, it extends to all $\psi \in W_0^{1,2}(\Omega)$
and we may write
$$
  -\Delta_D^\Omega \geq \frac{1}{4} \frac{1}{|x|^2\log^2(|x|/\eps)}
  \,,
$$
in the sense of quadratic forms in~$\sii(\Omega)$.
\end{proof}
\begin{rema}\label{Rem.generic}
Let the boundary~$\partial\Omega$ be sufficiently regular,
say continuous.
Then the boundary~$\partial\Omega$ (if non-empty) is $(d-1)$-dimensional
and the domain~$\Omega$ cannot lie on both sides of any 
part of its boundary (\cf~\cite[Sec.~3.21]{Adams2}). 
Then the exterior~$\Omega^\mathrm{ext}$
is always non-empty whenever $\Omega \not= \Real^d$.
\end{rema}
%

\section{Quasi-bounded domains or Resonators}
%
Recall that quasi-bounded sets are those  
which contains neither a sequence of balls of diverging radius
nor a sequence of balls of the same radius.
Bounded sets (\ie\ those which are contained in a ball)
are a special subclass of quasi-bounded sets,
but the latter class is much wider. 
In addition to bounded sets, 
it contains unbounded geometries which are ``narrow at infinity''
(recall Figure~\ref{Fig.Glazman}),
or more precisely: 
\begin{equation}\label{q-bounded.equivalent}
  \mbox{unbounded $\Omega$ is quasi-bounded}
  \qquad\Longleftrightarrow\qquad
  \lim_{\stackrel[x\in\Omega]{}{|x|\to\infty}}
  \dist(x,\partial\Omega) = 0
  \,.
\end{equation}
\begin{exam}[Spiny urchin]\label{Ex.urchin}
A highly irregular quasi-bounded domain 
(originally due to Clark~\cite{Clark_1967})
is obtained by removing an infinite sequence of half-lines (``spines'')
from the plane (see Figure~\ref{Fig.urchin}):
$$
  \Omega := \Real^2 \setminus \bigcup_{m=1}^\infty S_m
  \,,
$$
where 
$$
  S_m := \Big\{
  (r\cos\theta,r\sin\theta) : \
  r \geq m
  \ \land \
  \theta = n\pi/2^m
  \mbox{ for } n=1,2,\dots,2^{m+1}
  \Big\} .
$$
Note that the exterior of~$\Omega$ is empty.
\end{exam}

Recall that the spectrum of the Dirichlet Laplacian $-\Delta_D^\Omega$
is non-negative for any set $\Omega \subset \Real^d$
(\cf~Proposition~\ref{Prop.positivity}).
For quasi-conical sets, we have seen that the whole interval $[0,\infty)$
constitutes the spectrum (\cf~Corollary~\ref{Corol.positivity}).
The quasi-bounded sets~$\Omega$ are the other extreme case:
the spectrum of $-\Delta_D^\Omega$ is typically composed 
of isolated points only (at least under some weak regularity assumptions,
including the spiny urchin, see Proposition~\ref{Prop.urchin}).

\begin{figure}[h!]
\begin{center}
\begin{tikzpicture}[scale=3]
\draw (0.1,0) -- (1,0);
\draw (0,0.1) -- (0,1);
\draw (-0.1,0) -- (-1,0);
\draw (0,-0.1) -- (0,-1);
\draw (0.141421,0.141421) -- (0.707107,0.707107);
\draw (-0.141421,0.141421) -- (-0.707107,0.707107);
\draw (-0.141421,-0.141421) -- (-0.707107,-0.707107);
\draw (0.141421,-0.141421) -- (0.707107,-0.707107);
\draw (0.277164,0.114805) -- (0.92388,0.382683);
\draw (0.114805,0.277164) -- (0.382683,0.92388);
\draw (-0.114805,0.277164) -- (-0.382683,0.92388);
\draw (-0.277164,0.114805) -- (-0.92388,0.382683);
\draw (0.277164,-0.114805) -- (0.92388,-0.382683);
\draw (0.114805,-0.277164) -- (0.382683,-0.92388);
\draw (-0.114805,-0.277164) -- (-0.382683,-0.92388);
\draw (-0.277164,-0.114805) -- (-0.92388,-0.382683);
\draw (0.392314,0.0780361) -- (0.980785,0.19509);
\draw (0.332588,0.222228) -- (0.83147,0.55557);
\draw (0.0780361,0.392314) -- (0.19509,0.980785);
\draw (0.222228,0.332588) -- (0.55557,0.83147);
\draw (-0.392314,0.0780361) -- (-0.980785,0.19509);
\draw (-0.332588,0.222228) -- (-0.83147,0.55557);
\draw (-0.0780361,0.392314) -- (-0.19509,0.980785);
\draw (-0.222228,0.332588) -- (-0.55557,0.83147);
\draw (0.392314,-0.0780361) -- (0.980785,-0.19509);
\draw (0.332588,-0.222228) -- (0.83147,-0.55557);
\draw (0.0780361,-0.392314) -- (0.19509,-0.980785);
\draw (0.222228,-0.332588) -- (0.55557,-0.83147);
\draw (-0.392314,-0.0780361) -- (-0.980785,-0.19509);
\draw (-0.332588,-0.222228) -- (-0.83147,-0.55557);
\draw (-0.0780361,-0.392314) -- (-0.19509,-0.980785);
\draw (-0.222228,-0.332588) -- (-0.55557,-0.83147);
\draw (0.497592,0.0490086) -- (0.995185,0.0980171);
\draw (0.47847, 0.145142) -- (0.95694, 0.290285);
\draw (0.440961, 0.235698) -- (0.881921, 0.471397);
\draw (0.386505, 0.317197) -- (0.77301, 0.634393);
\draw (0.317197, 0.386505) -- (0.634393, 0.77301);
\draw (0.235698, 0.440961) -- (0.471397, 0.881921);
\draw (0.145142, 0.47847) -- (0.290285, 0.95694);
\draw (0.0490086, 0.497592) -- (0.0980171, 0.995185);
\draw (-0.497592,0.0490086) -- (-0.995185,0.0980171);
\draw (-0.47847, 0.145142) -- (-0.95694, 0.290285);
\draw (-0.440961, 0.235698) -- (-0.881921, 0.471397);
\draw (-0.386505, 0.317197) -- (-0.77301, 0.634393);
\draw (-0.317197, 0.386505) -- (-0.634393, 0.77301);
\draw (-0.235698, 0.440961) -- (-0.471397, 0.881921);
\draw (-0.145142, 0.47847) -- (-0.290285, 0.95694);
\draw (-0.0490086, 0.497592) -- (-0.0980171, 0.995185);
\draw (0.497592,-0.0490086) -- (0.995185,-0.0980171);
\draw (0.47847, -0.145142) -- (0.95694, -0.290285);
\draw (0.440961, -0.235698) -- (0.881921, -0.471397);
\draw (0.386505, -0.317197) -- (0.77301, -0.634393);
\draw (0.317197, -0.386505) -- (0.634393, -0.77301);
\draw (0.235698, -0.440961) -- (0.471397, -0.881921);
\draw (0.145142, -0.47847) -- (0.290285, -0.95694);
\draw (0.0490086, -0.497592) -- (0.0980171, -0.995185);
\draw (-0.497592,-0.0490086) -- (-0.995185,-0.0980171);
\draw (-0.47847, -0.145142) -- (-0.95694, -0.290285);
\draw (-0.440961, -0.235698) -- (-0.881921, -0.471397);
\draw (-0.386505, -0.317197) -- (-0.77301, -0.634393);
\draw (-0.317197, -0.386505) -- (-0.634393, -0.77301);
\draw (-0.235698, -0.440961) -- (-0.471397, -0.881921);
\draw (-0.145142, -0.47847) -- (-0.290285, -0.95694);
\draw (-0.0490086, -0.497592) -- (-0.0980171, -0.995185);
%
\draw (0.599277, 0.0294406) -- (0.998795, 0.0490677);
\draw (0.593506, 0.0880383) -- (0.989177, 0.14673);
\draw (0.582019, 0.145788) -- (0.970031, 0.24298);
\draw (0.564926, 0.202134) -- (0.941544, 0.33689);
\draw (0.542394, 0.256533) -- (0.903989, 0.427555);
\draw (0.514637, 0.308462) -- (0.857729, 0.514103);
\draw (0.481925, 0.35742) -- (0.803208, 0.595699);
\draw (0.444571, 0.402935) -- (0.740951, 0.671559);
\draw (0.402935, 0.444571) -- (0.671559, 0.740951);
\draw (0.35742, 0.481925) -- (0.595699, 0.803208);
\draw (0.308462, 0.514637) -- (0.514103, 0.857729);
\draw (0.256533, 0.542394) -- (0.427555, 0.903989);
\draw (0.202134, 0.564926) -- (0.33689, 0.941544);
\draw (0.174171, 0.574164) -- (0.290285, 0.95694);
\draw (0.145788, 0.582019) -- (0.24298, 0.970031);
\draw (0.0880383, 0.593506) -- (0.14673, 0.989177);
\draw (0.0294406, 0.599277) -- (0.0490677, 0.998795);
\draw (-0.599277, 0.0294406) -- (-0.998795, 0.0490677);
\draw (-0.593506, 0.0880383) -- (-0.989177, 0.14673);
\draw (-0.582019, 0.145788) -- (-0.970031, 0.24298);
\draw (-0.564926, 0.202134) -- (-0.941544, 0.33689);
\draw (-0.542394, 0.256533) -- (-0.903989, 0.427555);
\draw (-0.514637, 0.308462) -- (-0.857729, 0.514103);
\draw (-0.481925, 0.35742) -- (-0.803208, 0.595699);
\draw (-0.444571, 0.402935) -- (-0.740951, 0.671559);
\draw (-0.402935, 0.444571) -- (-0.671559, 0.740951);
\draw (-0.35742, 0.481925) -- (-0.595699, 0.803208);
\draw (-0.308462, 0.514637) -- (-0.514103, 0.857729);
\draw (-0.256533, 0.542394) -- (-0.427555, 0.903989);
\draw (-0.202134, 0.564926) -- (-0.33689, 0.941544);
\draw (-0.174171, 0.574164) -- (-0.290285, 0.95694);
\draw (-0.145788, 0.582019) -- (-0.24298, 0.970031);
\draw (-0.0880383, 0.593506) -- (-0.14673, 0.989177);
\draw (-0.0294406, 0.599277) -- (-0.0490677, 0.998795);
\draw (0.599277, -0.0294406) -- (0.998795, -0.0490677);
\draw (0.593506, -0.0880383) -- (0.989177, -0.14673);
\draw (0.582019, -0.145788) -- (0.970031, -0.24298);
\draw (0.564926, -0.202134) -- (0.941544, -0.33689);
\draw (0.542394, -0.256533) -- (0.903989, -0.427555);
\draw (0.514637, -0.308462) -- (0.857729, -0.514103);
\draw (0.481925, -0.35742) -- (0.803208, -0.595699);
\draw (0.444571, -0.402935) -- (0.740951, -0.671559);
\draw (0.402935, -0.444571) -- (0.671559, -0.740951);
\draw (0.35742, -0.481925) -- (0.595699, -0.803208);
\draw (0.308462, -0.514637) -- (0.514103, -0.857729);
\draw (0.256533, -0.542394) -- (0.427555, -0.903989);
\draw (0.202134, -0.564926) -- (0.33689, -0.941544);
\draw (0.174171, -0.574164) -- (0.290285, -0.95694);
\draw (0.145788, -0.582019) -- (0.24298, -0.970031);
\draw (0.0880383, -0.593506) -- (0.14673, -0.989177);
\draw (0.0294406, -0.599277) -- (0.0490677, -0.998795);
\draw (-0.599277, -0.0294406) -- (-0.998795, -0.0490677);
\draw (-0.593506, -0.0880383) -- (-0.989177, -0.14673);
\draw (-0.582019, -0.145788) -- (-0.970031, -0.24298);
\draw (-0.564926, -0.202134) -- (-0.941544, -0.33689);
\draw (-0.542394, -0.256533) -- (-0.903989, -0.427555);
\draw (-0.514637, -0.308462) -- (-0.857729, -0.514103);
\draw (-0.481925, -0.35742) -- (-0.803208, -0.595699);
\draw (-0.444571, -0.402935) -- (-0.740951, -0.671559);
\draw (-0.402935, -0.444571) -- (-0.671559, -0.740951);
\draw (-0.35742, -0.481925) -- (-0.595699, -0.803208);
\draw (-0.308462, -0.514637) -- (-0.514103, -0.857729);
\draw (-0.256533, -0.542394) -- (-0.427555, -0.903989);
\draw (-0.202134, -0.564926) -- (-0.33689, -0.941544);
\draw (-0.174171, -0.574164) -- (-0.290285, -0.95694);
\draw (-0.145788, -0.582019) -- (-0.24298, -0.970031);
\draw (-0.0880383, -0.593506) -- (-0.14673, -0.989177);
\draw (-0.0294406, -0.599277) -- (-0.0490677, -0.998795);
\end{tikzpicture}
\end{center}
\caption{\emph{Spiny urchin} as an example of 
a quasi-bounded domain.}
\label{Fig.urchin}
\end{figure}
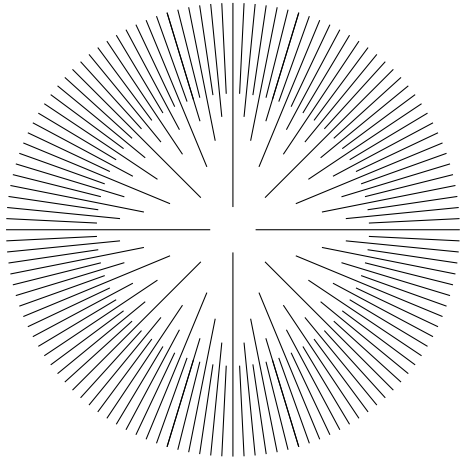

Because our life time is finite (unfortunately for these lectures,
but fortunately for other respects of our life), 
in this section we shall almost exclusively
consider quasi-bounded sets which are \textbf{bounded}.
Then we have a classical interpretation of the spectrum
of the Dirichlet Laplacian in a bounded domain~$\Omega$:
it is composed of squares of resonant frequences of 
an elastic membrane of shape~$\Omega$ with fixed edges. 
Musically talented readers will support our expectation
that there is just a countable set of such frequences.
Let us confirm this intuition by a mathematical analysis.

\subsection{Discrete and essential spectra}
First of all, let us make precise the distinction between spectra
composed of non-degenerate intervals and isolated points.

Let~$H$ be a self-adjoint operator in a Hilbert space~$\Hilbert$.
In Section~\ref{Sec.spectrum}, we decomposed the spectrum of~$H$
into the disjoint union of the point and continuous spectra.
To better describe the situations where  
the continuous spectrum is composed of a unique point 
(\eg, a compact operator in an infinite-dimensional space)
or the point spectrum is a non-degenerate interval 
(\eg, the Dirichlet Laplacian in the disconnected version 
of Figure~\ref{Fig.embedded}, 
see the proof of Theorem~\ref{Thm.embedded}),
an alternative decomposition is 
\begin{equation}\label{spectrum.ess}
  \sigma(H) = \sigma_\mathrm{disc}(H) \ \dot\cup \ \sigma_\mathrm{ess}(H)
  \,, 
\end{equation}
where the \emph{discrete} and \emph{essential} spectra
are respectively defined by
$$
\begin{aligned}
  \sigma_\mathrm{disc}(H) &:= 
  \big\{
  \lambda \in \sigma_\mathrm{p}(H) : \  
  \lambda \mbox{ is isolated} \quad \land \quad 
  \dim\ker(H-\lambda I) < \infty  
  \big\} 
  \,,
  \\
  \sigma_\mathrm{ess}(H) &:= 
  \sigma(H) \setminus \sigma_\mathrm{disc}(H)
  \,.
\end{aligned}
$$

In other words, the discrete spectrum catches the properties
of the spectrum in finite-dimensional vector spaces.
All the ugly ``rarities'' due to the infinite dimension
are then included in the essential spectrum.
More specifically, the essential spectrum contains 
either accumulation points of~$\sigma(H)$
or isolated eigenvalues of infinite multiplicity.

If the essential (respectively, discrete) spectrum is empty,
we say that the spectrum is \emph{purely discrete} 
(respectively, \emph{purely essential}).

Due to Corollary~\ref{Corol.positivity}, 
the spectrum of the Dirichlet Laplacian
in quasi-conical domains is purely essential.
Our goal is to show that the situation in bounded domains is quite opposite,
namely the spectrum is purely discrete.

\subsection{The minimax principle}
A powerfull tool to study the discrete and essential spectra
of self-adjoint operators~$H$ is the \emph{minimax principle}.
For operators acting in finite-di\-men\-sion\-al vector spaces,
it says that the spectrum of~$H$ can fully be characterised variationally. 
\begin{theo}[Finite dimensions]\label{minimax.finite}
Let~$H$ be a self-adjoint operator in a Hilbert space~$\Hilbert$ 
with $N := \dim\Hilbert < \infty$.
Let us arrange its eigenvalues into a non-decreasing sequence
$
  \sigma(H) 
  = \{\lambda_n\}_{n=1}^N=\{\lambda_1 \leq \lambda_2 \leq \dots \leq \lambda_N\}
$,
where each eigenvalue is repeated according to its multiplicity.
Then, for every $n \in \{1,\dots,N\}$,
\begin{equation}\label{infsup}
  \lambda_n
  = \min_{\stackrel[\dim\mathscr{L}_n=n]{}{\mathscr{L}_n \subset \Hilbert}} \
  \max_{\stackrel[\psi \not= 0]{}{\psi \in \mathscr{L}_n}}
  \frac{(\psi,H\psi)}{\ \|\psi\|^2}
  \,,
\end{equation}
where~$\mathscr{L}_n$ is any $n$-dimensional subspace of~$\Hilbert$.
\end{theo}
\begin{proof}
Let us denote the right-hand side of~\eqref{infsup} by $\lambda_n'$.
Our aim is to show that $\lambda_n'=\lambda_n$ for every $n \in \{1,\dots,N\}$.
We follow the proof of \cite[Thm.~4.5.1]{Davies}.

\fbox{$\lambda_n \geq \lambda_n'$}
Let $\{\psi_n\}_{n=1}^N$ denote the eigenvectors of~$H$
corresponding to $\{\lambda_n\}_{n=1}^N$.
By the spectral theorem, they can be normalised in such a way 
that $\{\psi_n\}_{n=1}^N$ is an orthonormal basis of
the space~$\Hilbert$. 
For every $\psi \in \mathcal{M}_n := \obal\{\psi_1,\dots,\psi_n\}$, one has
$$
  (\psi,H\psi) 
  = \sum_{k=1}^n \lambda_k \, |(\psi_k,\psi)|^2 
  \leq \lambda_n \sum_{k=1}^n |(\psi_k,\psi)|^2 
  = \lambda_n \|\psi\|^2
  \,.
$$
Consequently, choosing 
$\mathscr{L}_n := \mathcal{M}_n$ in~\eqref{infsup},
one gets $\lambda_n' \leq \lambda_n$ for every $n \in \{1,\dots,N\}$.

\fbox{$\lambda_n \leq \lambda_n'$}
If $n=1$, the formula~\eqref{infsup} reduces to
$$
  \lambda_1' = 
  \min_{\stackrel[\psi \not= 0]{}{\psi \in \dom H }}
  \frac{(\psi,H\psi)}{\ \|\psi\|^2}
  \,.
$$
Using that $\{\psi_n\}_{n=1}^N$ is an orthonormal basis,
one has, for every $\psi \in \dom H $, 
$$
  (\psi,H\psi) = \sum_{k=1}^N \lambda_k \, |(\psi_k,\psi)|^2 
  \geq \lambda_1 \sum_{k=1}^N |(\psi_k,\psi)|^2 
  = \lambda_1 \|\psi\|^2
  \,.
$$
Consequently, $\lambda_1' \geq \lambda_1$.

If $n \in \{2,\dots,N\}$, we introduce the operator
$$
  P := \sum_{k=1}^{n-1} \psi_k (\psi_k,\cdot)
  \,, \qquad
  \dom P := \Hilbert
  \,.
$$ 
It is an orthogonal projection on~$\Hilbert$ 
(\ie, $P^2=P$ and $P^*=P$) with range $\mathcal{M}_{n-1}$.
Clearly, $\dim\ran P = n-1$. 
Let~$\mathscr{L}_n$ be any $n$-dimensional subspace of $\dom H$. 
Since $\dim\ran P|_{\mathscr{L}_n} < \dim\mathscr{L}_n$,
there must exist a non-zero vector $\phi \in \mathscr{L}_n$ such that $P\phi=0$. 
We then have $(\psi_k,\phi)=0$ for all $k \leq n-1$.
It follows that
$$
  (\phi,H\phi) = \sum_{k=n}^N \lambda_k \, |(\psi_k,\phi)|^2 
  \geq \lambda_n \sum_{k=n}^N |(\psi_k,\phi)|^2
  = \lambda_n \|\phi\|^2
  \,.
$$
We conclude that
$$
  \sup_{\stackrel[\psi \not= 0]{}{\psi \in \mathscr{L}_n}}
  \frac{(\psi,H\psi)}{\ \|\psi\|^2}
  \geq  
  \frac{(\phi,H\phi)}{\ \|\phi\|^2}
  \geq \lambda_n
  \,.
$$
Consequently, $\lambda_n' \geq \lambda_n$ for every $n \in \{1,\dots,N\}$.
\end{proof}

In infinite-dimensional vector spaces~$\Hilbert$, 
it is still possible to characterise variationally
discrete eigenvalues of~$H$ lying 
\emph{below the essential spectrum}.
It is necessary to restrict to operators~$H$ 
which are \emph{bounded from below},
\ie, there exists $c \in \Real$ 
such that $(\psi,H\psi) \geq c \|\psi\|^2$
for every $\psi \in \dom H$. 
Then there exists a form~$h$ in~$\Hilbert$ 
which is uniquely associated to~$H$ 
via the representation relationship
$h(\phi,\psi) = (\phi,H\psi)$
for every $\phi \in \dom h$ and $\psi \in \dom H$.
The following theorem is a standard spectral result, 
which is therefore admitted here
(see, \eg, \cite[Sec.~4.5]{Davies} for a proof).

\begin{theo}[Infinite dimensions]\label{minimax}
Let~$H$ be a self-adjoint operator 
in a Hilbert space~$\Hilbert$ with $\dim\Hilbert=\infty$, 
which is bounded from below.
Let $\{\lambda_n\}_{n=1}^\infty$  
be a non-decreasing sequence of numbers defined by
\begin{equation}\label{infsup.ess}
  \lambda_n
  := 
  \inf_{\stackrel[\dim\mathscr{L}_n=n]{}{\mathscr{L}_n \subset \dom H}} \
  \sup_{\stackrel[\psi \not= 0]{}{\psi \in \mathscr{L}_n}}
  \frac{(\psi,H\psi)}{\ \|\psi\|^2}
  =  
  \inf_{\stackrel[\dim\mathscr{L}_n=n]{}{\mathscr{L}_n \subset \dom h}} \
  \sup_{\stackrel[\psi \not= 0]{}{\psi \in \mathscr{L}_n}}
  \frac{h[\psi]}{\ \|\psi\|^2}
  \,,
\end{equation}
where~$\mathscr{L}_n$ is any $n$-dimensional subspace
of the corresponding domain
(the validity of the second equality in~\eqref{infsup.ess}
is part of the statement of the theorem).
Then
\begin{enumerate}
\item
$
  \displaystyle
  \lambda_\infty
  := \lim_{n\to\infty} \lambda_n
  = \inf\sigma_\mathrm{ess}(H)
  \,,
$
\smallskip \\
with the convention that $\sigma_\mathrm{ess}(H)=\varnothing$
if $\lambda_\infty = +\infty$;
\item
$
  \{\lambda_n\}_{n=1}^\infty \cap (-\infty,\lambda_\infty)
  = \sigma_\mathrm{disc}(H) \cap (-\infty,\lambda_\infty)
  \,,
$
\medskip \\
with each $\lambda_n \in (-\infty,\lambda_\infty)$ being
an eigenvalue of~$H$ repeated a number of times equal to its multiplicity.
\end{enumerate}
\end{theo}

As a consequence of Theorem~\ref{minimax},
we see that if the essential spectrum of~$H$ is empty,
then~$H$ possesses an infinite sequence $\{\lambda_n\}_{n \in \Nat^*}$ 
of discrete eigenvalues tending to infinity. 
By the spectral theorem, the corresponding eigenvectors 
$\{\psi_n\}_{n \in \Nat^*}$
form a \emph{complete orthonormal set} in 
(or an \emph{orthonormal basis} of) $\Hilbert$.
Here the orthonormality has the same meaning as in finite-dimensional spaces,
while the completeness means that if $(\psi_n,\psi)=0$
for every $n \in \Nat^*$ with an arbitrary $\psi \in \Hilbert$, 
then necessarily $\psi=0$.
Consequently, one has the decomposition
\begin{equation}\label{D.ONB}
  \forall \psi \in \Hilbert
  \,, \qquad
  \psi = \sum_{n=1}^\infty c_n \, \psi_n
  \qquad\mbox{with}\qquad
  c_n := (\psi_n,\psi)
  \,.
\end{equation}
Vice versa, if~$H$ possesses an infinite sequence of discrete eigenvalues
tending to infinity and the corresponding eigenvectors 
form a complete orthonormal set, then 
$\sigma_\mathrm{ess}(H) = \varnothing$. 

Regardless of whether the essential spectrum of~$H$ is empty or not,
the bottom of the spectrum of~$H$ always coincides with~$\lambda_1$:
\begin{equation}\label{Rayleigh}
      \inf\sigma(H) = \lambda_1 
  = \inf_{\stackrel[\psi \not= 0]{}{\psi \in \dom H}}
  \frac{(\psi,H\psi)}{\ \|\psi\|^2}
  =  
  \inf_{\stackrel[\psi \not= 0]{}{\psi \in \dom h}}
  \frac{h[\psi]}{\ \|\psi\|^2}
  .
\end{equation}

The minimax principle can be also used to
compare spectra of different operators.
Recall Definition~\ref{Def.operator.ineq} introducing 
the order relation between (possibly unbounded) operators 
and let us write $\lambda_n(H)$ if we want to point out
the dependence of the numbers~\eqref{infsup.ess} on the operator~$H$.
\begin{coro}\label{Corol.minimax}
If $H_-$, $H_+$ are two self-adjoint operators in~$\Hilbert$
that are bounded from below. Then
$$
  H_- \leq H_+
  \qquad \Longrightarrow \qquad
  \forall n \in \Nat^* \,, \quad
  \lambda_n(H_-) \leq \lambda_n(H_+)
  \,.
$$
\end{coro}

Let us conclude this overview of the minimax principle
by the following useful observation,
which is rarely made explicit in the literature.
We stress that eigenvalues at the bottom
of the essential spectrum are included 
in this proposition, too.
\begin{prop}\label{Prop.achieved}
If~\eqref{infsup.ess} is achieved, 
\ie~there exists a non-trivial $\psi \in \dom h$ such that
\begin{equation}\label{achieved}
  \lambda_k = \frac{h[\psi]}{\|\psi\|^2}
\end{equation}
and~$\psi$ is orthogonal to all the eigenvectors 
corresponding to 
$\lambda_1,\dots,\lambda_{k-1} < \lambda_\infty$ 
(\ie~no orthogonality condition if $k=1$ 
or if there are no eigenvalues
below the essential spectrum,
while $\lambda_k = \lambda_\infty$ is permitted),
then~$\lambda_k$ is an eigenvalue of~$H$
and~$\psi$ is a corresponding eigenvector. 
\end{prop}
\begin{proof}
In the first part of the proof, 
we are inspired by \cite[proof of Lem.~XI~1.1]{Edmunds-Evans}.
Let $\psi_1,\dots,\psi_{k-1}$ denote the orthonormal eigenvectors 
of~$H$ corresponding to $\lambda_1,\dots,\lambda_{k-1}$.
Denoting $\mathcal{G} := \obal\{\psi_1,\dots,\psi_{k-1}\}$,
we have the orthogonal sum decomposition 
$\Hilbert = \mathcal{G} \oplus \mathcal{G}^\bot$.
Let~$H_1$ and~$H_2$ be the restrictions of~$H$
to~$\mathcal{G}$ and~$\mathcal{G}^\bot$, respectively.
Clearly, $H_1$ and~$H_2$ are self-adjoint operators in 
$\mathcal{G}$ and~$\mathcal{G}^\bot$, respectively,
and $H = H_1 \oplus H_2$.
It is easy to verify that 
$
  \sigma(H_2) = \sigma(H) \setminus
  \{\lambda_1,\dots,\lambda_{k-1}\}  
$.
By~\eqref{Rayleigh}, it thus follows that 
$$
  \lambda_k = 
  \inf_{\stackrel[\psi \not= 0]{}{\psi \in \mathcal{G}^\bot \cap \dom h}}
  \frac{h[\psi]}{\|\psi\|^2}  
  \,.
$$
Now, let us assume that this infimum is achieved.
That is, there exists a non-trivial vector  
$\psi \in \mathcal{G}^\bot \cap \dom h$ 
such that~\eqref{achieved} holds. 
In particular, $\psi$~is the critical point of the functional 
$$
  J[\psi] := \frac{h[\psi]}{\|\psi\|^2}
  \,.
$$
Consequently, the first variation
$$
\begin{aligned}
  \lim_{\eps \to 0} 
  \frac{J[\psi + \eps \phi] - J[\psi]}{\eps}
  &= \lim_{\eps \to 0} \frac{1}{\eps}
  \left(
  \frac{h[\psi] + 2\eps\Re h(\phi,\psi) + \eps^2 h[\phi]}
  {\|\psi\|^2 + 2\eps\Re(\phi,\psi) + \eps^2 \|\phi\|^2}
  - \frac{h[\psi]}{\|\psi\|^2}
  \right)
  \\
  &= 2 \Re h(\phi,\psi) - 2 \lambda_k \Re(\phi,\psi)
  \\
  &= 2 \Re \left[
  h(\phi,\psi) - \lambda_k (\phi,\psi)
  \right]
\end{aligned}  
$$
must vanish, 
where $\phi \in \mathcal{G}^\bot \cap\dom h$ is arbitrary.
Using the arbitrariness of~$\phi$,
we conclude that 
$$
  \forall \phi \in \mathcal{G}^\bot \cap\dom h
  \,, \qquad
  h(\phi,\psi) = \lambda_k (\phi,\psi) 
  \,.
$$
It follows that $\psi \in \dom H$ and $H\psi = \lambda_k\psi$.
\end{proof}

\subsection{Strings}\label{Sec.strings}
By a string we mean a one-dimensional bounded interval
$I_a := (0,a)$ of length $a>0$, see Figure~\ref{Fig.string}.

\begin{figure}[h!]
\begin{center}
\includegraphics[width=0.3\textwidth]{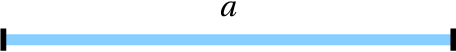}
\end{center}
\caption{Interval~$I_{a}$ of length $a$.}
\label{Fig.string}
\end{figure}

\subsubsection*{Dirichlet boundary conditions}
The point spectrum of the Dirichlet Laplacian $-\Delta_\alpha^{I_a}$
is determined by non-trivial solutions 
$\psi \in W^{2,2}(I_a)$
of the boundary-value problem
\begin{equation}\label{D.interval}
\left\{ 
\begin{aligned}
  -\psi'' &= \lambda \psi 
  && \mbox{in} \quad (0,a)
  \,,
  \\
  \psi &= 0 
  && \mbox{at} \quad 0, a
  \,.
\end{aligned}  
\right.
\end{equation}
Since $-\Delta_D^{I_a}$ is a non-negative operator
(\cf~Proposition~\ref{Prop.positivity}),
we know that $\lambda \geq 0$.  
Then the general solution of 
the differential equation of~\eqref{D.interval} reads
(the special case $\lambda=0$, when the fundamental solutions 
are given by linear functions, is covered by this formula
after taking the limit $\lambda \to 0$)
\begin{equation}\label{D.fundamental}
  \psi(x) = A \, \frac{\sin(\sqrt{\lambda}x)}{{\sqrt{\lambda}}} 
  + B \cos(\sqrt{\lambda}x)
  \,,
\end{equation}
where $A,B \in \Com$ are constants to be determined 
by the boundary conditions of~\eqref{D.interval}. 
Requiring $\psi(0)=0$ implies $B=0$.
The other boundary condition $\psi(a)=0$ yields
\begin{equation}\label{D.trans}
  \frac{\sin(\sqrt{\lambda}a)}{\sqrt{\lambda}} = 0
  \,,
\end{equation}
\ie\ $\sqrt{\lambda} a \in \pi\Nat^*$ 
(the case of $\sqrt{\lambda} = 0$ 
leads to the impossibility $a=0$,
so it must be excluded).
In summary,
\begin{equation}\label{D.spec}
  \sigma_\mathrm{p}\big(-\Delta_D^{I_a}\big)
  = \left\{ 
  \lambda_n^D(I_a) := \left(\frac{n\pi}{a}\right)^2  
  \right\}_{n = 1}^\infty
  .
\end{equation}
The normalised eigenfunctions corresponding to 
the eigenvalues $\lambda_n^D(I_a)$ are given by
\begin{equation}\label{D.efs}
  \psi_n^D(x)
  := 
  \sqrt{\frac{2}{a}} \, \sin\left(\frac{n\pi}{a} x\right)
  .  
\end{equation}

It is a standard result of Fourier analysis 
(see, \eg, \cite[Sec.~7]{Kaplan})
that $\{\psi_n^D\}_{n \in \Nat^*}$ is 
a complete orthonormal set in $\sii(I_a)$.
The minimax principle (Theorem~\ref{minimax}) implies
$\sigma_\mathrm{ess}(-\Delta_D^{I_a}) = \varnothing$.
Hence the spectrum of $-\Delta_D^{I_a}$ is purely discrete
and it is exhausted by the eigenvalues~\eqref{D.spec}.

Interpreting the eigenvalues as squares of resonant frequencies
of a vibrating string with fixed ends, we get the intuitive result
that enlarging the string leads to lower tones.
At the same time, the result tells us that 
enlarging a box to which a quantum particle is constrained
diminishes its bound-state energies.

\subsubsection*{Robin boundary conditions}
The point spectrum of the Robin Laplacian $-\Delta_\alpha^{I_a}$
with $\alpha \in \Real$
is determined by non-trivial solutions 
$\psi \in W^{2,2}(I_a)$
of the boundary-value problem
\begin{equation}\label{interval}
\left\{ 
\begin{aligned}
  -\psi'' &= \lambda \psi 
  && \mbox{in} \quad (0,a)
  \,,
  \\
  -\psi' +\alpha \psi &= 0 
  && \mbox{at} \quad 0
  \,,
  \\
  \psi' +\alpha \psi &= 0 
  && \mbox{at} \quad a
  \,.
\end{aligned}  
\right.
\end{equation}
Notice the switch of sign at the boundary conditions,
which is due to the fact that the Robin problem in~\eqref{problem}
is defined through the \emph{outward} pointing normal. 
The general solution of the differential equation of~\eqref{interval} 
again reads~\eqref{D.fundamental}
(for the case $\lambda<0$ we take $\sqrt{\lambda} := i \sqrt{|\lambda|}$,
which leads to hyperbolic functions),
where $A,B \in \Com$ are constants to be determined 
by the boundary conditions of~\eqref{interval} now. 
The latter can conveniently be written in the matrix form
$
  M_\lambda 
  \begin{psmallmatrix}
    A \\ B
  \end{psmallmatrix} 
  = 
  \begin{psmallmatrix}
    0 \\ 0
  \end{psmallmatrix} 
$
with
\begin{equation}\label{det}
  M_\lambda :=
  \begin{pmatrix}
    -1 & \alpha \\
    \cos(\sqrt{\lambda}a) 
    + \alpha \, 
    \frac{\displaystyle\sin(\sqrt{\lambda}a)}{\displaystyle\sqrt{\lambda}}
    & -\sqrt{\lambda} \sin(\sqrt{\lambda}a) 
    + \alpha \, \cos(\sqrt{\lambda}a)
  \end{pmatrix} 
  .
\end{equation}
Since we are looking for solutions with $A,B$ 
not being simultaneously equal to zero,
the eigenvalues are determined by the implicit equation
\begin{equation}\label{trans}
  0 \stackrel[]{\downarrow}{=} 
  \det(M_\lambda) 
  = - 2 \, \alpha \, \cos(\sqrt{\lambda}\,a) 
  + (\lambda-\alpha^2) \, \frac{\sin(\sqrt{\lambda}\,a)}{\sqrt{\lambda}} 
\end{equation}
(here the arrow points to the requirement,
the other equality is the computation of the determinant).
Note that~\eqref{trans} formally coincides with~\eqref{D.trans}
after dividing by~$\alpha^2$ and sending~$\alpha$ to infinity.  
Let us denote the solutions of~\eqref{trans} by $\lambda_n^\alpha(I_a)$,
where $\{\lambda_n^\alpha(I_a)\}_{n=1}^\infty$ is an increasing sequence. 
The corresponding eigenfunctions will be denoted by~$\psi_n^\alpha$.
As usual, we normalise them in such a way that 
$\psi_n^\alpha$ has norm~$1$ in $\sii(I_a)$ for all $n \in \Nat^*$,
moreover we choose~$\psi_1^\alpha$ positive.

If $\alpha=0$, we \emph{a priori} know that $\lambda \geq 0$
due to Proposition~\ref{Prop.positivity}.
The implicit equation~\eqref{trans} reduces to 
$\sqrt{\lambda} \sin(\sqrt{\lambda}\,a) = 0$, 
\ie\ $\sqrt{\lambda}\,a \in \pi\Nat$ (now the case $\sqrt{\lambda} = 0$
is admissible, leading to a constant eigenfunction). 
Consequently,
\begin{equation}\label{N.spec.point} 
  \sigma_\mathrm{p}\big(-\Delta_N^{I_a}\big)
  = \left\{ 
  \lambda_n^N(I_a) = \left(\frac{(n-1)\pi}{a}\right)^2  
  \right\}_{n = 1}^\infty
  . 
\end{equation}
The normalised eigenfunctions corresponding to 
the eigenvalues $\lambda_n^N(I_a)$ are given by
\begin{equation}\label{N.efs}
  \psi_n^N(x)
  = 
  \begin{cases}
  \displaystyle
  \sqrt{\frac{1}{a}}
  & \mbox{if} \quad n = 1 \,,
  \medskip \\
  \displaystyle
  \sqrt{\frac{2}{a}} \, \cos\left(\frac{(n-1)\pi}{a} x\right)
  & \mbox{if} \quad n \geq 2 \,. 
  \end{cases} 
\end{equation}
Again, $\{\psi_n^N\}_{n=1}^\infty$ is 
a complete orthonormal set in~$\sii(I_a)$,
so the spectrum of $-\Delta_N^{I_a}$ is purely discrete
and it is exhausted by the eigenvalues~\eqref{N.spec.point}.

The Neumann spectrum~\eqref{N.spec.point} 
looks like the Dirichlet spectrum~\eqref{N.spec.point},
but the former additionally contains the zero more.
As in the Dirichlet case, 
the result~\eqref{N.spec.point} confirms the intuition that
enlarging the length of a vibrating string with free ends
leads to lower tones.
It also explains why the \emph{piccolo} produces higher tones 
than the \emph{flute}: both can be modelled by a tube with open ends
but the piccolo is half of the length of the flute's. 

If $\alpha > 0$, we again have $\lambda \geq 0$
due to Proposition~\ref{Prop.positivity}.
In this case, however, $\lambda = 0$ is not an admissible 
solution of~\eqref{trans}, so the spectrum of $-\Delta_\alpha^{I_a}$
is positive.

If $\alpha < 0$, the situation changes dramatically. 
First of all, there is always a negative eigenvalue. 
If $\alpha > -2/a$, there is just one negative eigenvalue
and the remaining eigenvalues are positive.
If $\alpha = -2/a$, there is one negative eigenvalue,
one zero eigenvalue and the remaining eigenvalues are positive.
If $\alpha < -2/a$, there are two negative eigenvalues
and the remaining eigenvalues are positive.
See Figure~\ref{Fig.Robin} for a visualisation.

For every $\alpha \in \Real$. the spectrum of $-\Delta_\alpha^{I_a}$
is purely discrete. We have already seen it for the Neumann case $\alpha=0$.
It follows that $\dom(\delta_N^{I_a}) = W^{1,2}(I_a)$ 
is compactly embedded in $\sii(I_a)$.
The other values of~$\alpha$ can be covered 
by noticing that $\dom(\delta_\alpha^{I_a}) = \dom(\delta_0^{I_a})$ 
by definition of the Robin Laplacian. 

\subsubsection*{Combined boundary conditions}
It is also possible to consider 
the one-dimensional operator $-\Delta_{DN}^{I_a}$
that acts as the Laplacian in the interval~$I_a$,
subject to a Dirichlet (respectively, Neumann) boundary condition
at~$0$ (respectively, $a$). 
Proceeding as above, we obtain that the spectrum is purely discrete
and equal to the set
\begin{equation}\label{DN.spec.bis}
  \sigma_\mathrm{p}\big(-\Delta_{DN}^{I_a}\big)
  = \left\{ 
  \left(\frac{(2n-1)\pi}{2a}\right)^2  
  \right\}_{n = 1}^\infty
  \,.
\end{equation}
The corresponding eigenfunctions are given by
\begin{equation}\label{DN.efs}
  \psi_n^{DN}(x) := 
  \sqrt{\frac{2}{a}} \, \sin\left(\frac{(2n-1)\pi}{2a} \, x\right)
\end{equation}
and they form a complete orthonormal set in $\sii(I_a)$.

The operator $-\Delta_{DN}^{I_a}$ is a classical model for
resonant vibrations of a string with one end fixed and the other free.
It also models standing waves in a \emph{clarinet},
\ie\ a tube with one open end and one closed end (at the reed).
On the other hand, $-\Delta_{N}^{I_a}$ models the situation of a \emph{flute},
\ie\ a tube with both ends open.
Considering the hypothetical situation of a clarinet and a flute
of the same length, 
we see by comparing~\eqref{DN.spec.bis} with~\eqref{N.spec.point} 
that the clarinet tones are lower than the tones of the flute 
(the zero mode does not count).
In a musical language, the clarinet is in B-flat major 
whereas the flute is in C major.

\begin{figure}[h!]
\begin{center}
\includegraphics[width=0.7\textwidth]{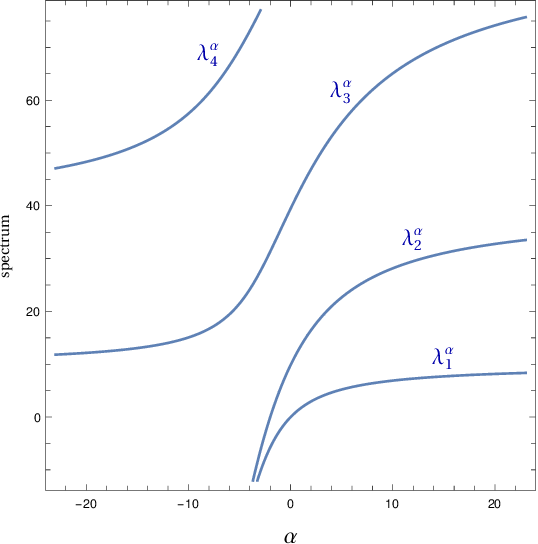}
\end{center}
\caption{Eigenvalues of $-\Delta_\alpha^{I_a}$ as functions of~$\alpha$
for $a:=1$.}
\label{Fig.Robin}
\end{figure}

\subsection{Rectangular boxes}\label{Sec.pipeds}
Given positive numbers $a_1, \dots ,a_d$, 
a \emph{rectangular box} of edges having the lengths $a_1, \dots, a_d$
is the Cartesian product
\begin{equation}\label{parallelepiped}
  I_{a_1,\dots,a_d} 
  := I_{a_1} \times \dots \times I_{a_d}
  \,,
\end{equation}
see Figure~\ref{Fig.rectangle}.
If all the edges have the same length~$a$,
we denote by $Q_a := I_{a,\dots,a}$ 
a \emph{cube} of edges having the length~$a$.

\begin{figure}[h!]
\begin{center}
\begin{tabular}{cc}
\includegraphics[width=0.4\textwidth]{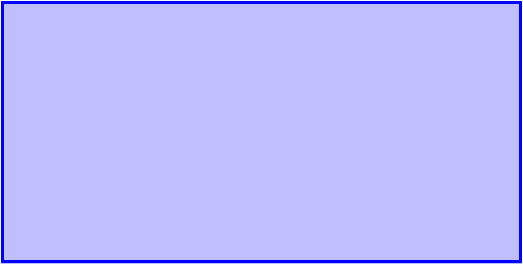}
& \includegraphics[width=0.4\textwidth]{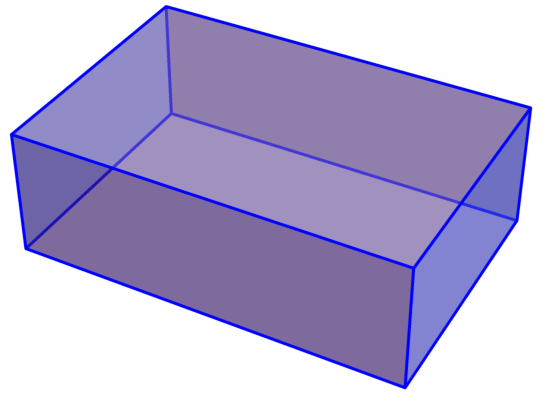}
\end{tabular}
\end{center}
\caption{Rectangular boxes in two and three dimensions.}
\label{Fig.rectangle}
\end{figure}

The spectral problem for the Robin Laplacian in $I_{a_1,\dots,a_d}$
can be solved by a separation of variables.
Consequently,
\begin{equation}\label{evs.piped}
  \sigma_\mathrm{p}\big(-\Delta_\alpha^{I_{a_1,\dots,a_d}}\big)
  = \left\{
   \lambda_{n_1}^\alpha(I_{a_1}) + \dots + \lambda_{n_d}^\alpha(I_{a_d}) 
  \right\}_{n_1,\dots,n_d = 1}^\infty
  \,.
\end{equation}
The corresponding normalised eigenfunctions are given by
$$
  \psi_{n_1,\dots,n_d}^\alpha(x) 
  := \psi_{n_1}^\alpha(x_1) \dots \psi_{n_d}^\alpha(x_d)
  \,.
$$
Since they form a complete orthonormal set in $\sii(I_{a_1,\dots,a_d})$,
the spectrum of $-\Delta_\alpha^{I_{a_1,\dots,a_d}}$ is purely discrete
and it is exhausted by the eigenvalues~\eqref{evs.piped}.

Note that the lowest eigenvalue is simple 
and the corresponding eigenfunction is nowhere zero 
(in fact, it is positive for our normalisation).
As usual in spectral theory, we arrange the eigenvalues
into a non-decreasing sequence 
$$
\begin{aligned}
  \sigma_\mathrm{p}\big(-\Delta_\alpha^{I_{a_1,\dots,a_d}}\big)
  &= \{\lambda_n^\alpha(I_{a_1,\dots,a_d})\}_{n=1}^\infty 
  \\
  &= \{\lambda_1^\alpha(I_{a_1,\dots,a_d}) 
  < \lambda_2^\alpha(I_{a_1,\dots,a_d}) 
  \leq \lambda_3^\alpha(I_{a_1,\dots,a_d}) \leq \dots\}
  \,,
\end{aligned}  
$$
where each eigenvalue is repeated according to its multiplicity
(so the sequence is \emph{not strictly} increasing 
if there are degeneracies).
The corresponding set of eigenfunctions will be denoted by
$\{\psi_n^\alpha\}_{n\in\Nat^*}$.
It is not completely trivial to obtain this non-decreasing sequence 
of eigenvalues for higher-dimensional rectangular boxes
and analyse the degeneracies 
(try for a square and a three-dimensional cube).

\subsection{Monotonicity of Dirichlet eigenvalues}
Given any open set $\Omega \subset \Real^d$, 
we denote by $\lambda_n^\alpha(\Omega)$ the numbers 
as defined by the minimax principle~\eqref{infsup.ess} 
with $H := -\Delta_\alpha^\Omega$.
They represent either discrete eigenvalues
(repeated according to multiplicities)
or the lowest point of the essential spectrum of $-\Delta_\alpha^\Omega$.
That is why we use the quotation marks in the following theorem.

\begin{theo}[Monotonicity of Dirichlet ``eigenvalues''] 
\label{Prop.monotonicity}
Given arbitrary open sets $\Omega_1,\Omega_2$, 
one has
$$
  \Omega_1 \subset \Omega_2
  \qquad \Longrightarrow \qquad
  \forall n \in \Nat^* \,, \quad
  \lambda_n^D(\Omega_1) \geq \lambda_n^D(\Omega_2)
  \,.
$$
\end{theo}
\begin{proof}
If $\psi \in W_0^{1,2}(\Omega_1)$, then the \emph{trivial extension}
\begin{equation}\label{extension}
  \tilde{\psi}(x) := 
  \begin{cases}
    \psi(x) & \mbox{if} \quad x \in \Omega_1 \,,
    \\
    0 & \mbox{if} \quad x \not\in \Omega_1 \,,
  \end{cases}
\end{equation}
belongs to $W_0^{1,2}(\Omega_2)$.
Consequently, 
$W_0^{1,2}(\Omega_1) \subset W_0^{1,2}(\Omega_2)$.
Using~\eqref{infsup.ess}, 
we therefore get 
$$
\begin{aligned}
  \lambda_n^D(\Omega_2)
  &= \inf_{\stackrel[\dim\mathscr{L}_n=n]{}{\mathscr{L}_n \subset W_0^{1,2}(\Omega_2)}} \
  \sup_{\stackrel[\psi \not= 0]{}{\psi \in \mathscr{L}_n}}
  \frac{\displaystyle\int_{\Omega_2} |\nabla\psi|^2}
  {\displaystyle\int_{\Omega_2} |\psi|^2}
  \\
  &\leq
  \inf_{\stackrel[\dim\mathscr{L}_n=n]{}{\mathscr{L}_n \subset W_0^{1,2}(\Omega_1)}} \
  \sup_{\stackrel[\psi \not= 0]{}{\psi \in \mathscr{L}_n}}
  \frac{\displaystyle\int_{\Omega_2} |\nabla\psi|^2}
  {\displaystyle\int_{\Omega_2} |\psi|^2} 
  \\
  &=
  \inf_{\stackrel[\dim\mathscr{L}_n=n]{}{\mathscr{L}_n \subset W_0^{1,2}(\Omega_1)}} \
  \sup_{\stackrel[\psi \not= 0]{}{\psi \in \mathscr{L}_n}}
  \frac{\displaystyle\int_{\Omega_1} |\nabla\psi|^2}
  {\displaystyle\int_{\Omega_1} |\psi|^2} 
  = \lambda_n^D(\Omega_1)
\end{aligned}  
$$  
for every $n \in \Nat^*$.
\end{proof}

It follows from Theorem~\ref{Prop.monotonicity}
that the larger membrane produces a lower fundamental tone
(or a quantum particle in a larger cavity has a lower ground-state energy),
which is in agreement with a physical intuition.

\begin{center}
\fbox{The monotonicity does not hold for Neumann eigenvalues !}
\end{center}
\begin{exam}[Non-monotonicity of Neumann eigenvalues]
A classical counterexample is given in 
the left realisation in Figure~\ref{Fig.counter.Neumann}:
an inscribed thin rectangle~$\Omega_1$ 
along the diagonal of a circumscribed rectangle~$\Omega_2$.
Of course, there is no contradiction for the lowest eigenvalue,
because $\lambda_1^N(\Omega_1) = 0 = \lambda_1^N(\Omega_2)$
due to the availability of the constant eigenfunction. 
However, we get a contradiction already for the second eigenvalue.
Let the lengths of the sides of the circumscribed (respectively inscribed)
rectangle be $a_2 \geq b_2$ (respectively $a_1 \geq b_1$).
Then
$$
\begin{aligned}
  \lambda_2^N(\Omega_1)
  &= \min\Big\{
  \mbox{$\big(\frac{\pi}{a_1}\big)^2$}+0,
  0+ \mbox{$\big(\frac{\pi}{b_1}\big)^2$}\Big\}
  =  \mbox{$\big(\frac{\pi}{a_1}\big)^2$} \,,
  \\
  \lambda_2^N(\Omega_2)
  &= \min\Big\{
 \mbox{$\big(\frac{\pi}{a_2}\big)^2$}+0,
  0+\mbox{$\big(\frac{\pi}{b_2}\big)^2$}
  \Big\}
  =  \mbox{$\big(\frac{\pi}{a_2}\big)^2$} \,.
\end{aligned}  
$$
The geometry dictates
$$
  a_1=\sqrt{a_2^2+b_2^2}- a_2 \, \frac{b_1}{b_2}
  \qquad \mbox{with} \qquad 
  b_1 < \frac{b_2}{a_2} \sqrt{a_2^2+b_2^2}
  \,.
$$ 
Consequently, $a_1 > a_2$ for all sufficiently small~$b_1$
(the result is intuitively clear from the picture), namely for
$$
  b_1 < \frac{b_2}{a_2} \sqrt{a_2^2+b_2^2} - b_2 
  \,.
$$
Under this condition, we get
$$
  \lambda_2^N(\Omega_2) > \lambda_2^N(\Omega_1)
  \,,
$$
which is a reversed inequality with respect to that of 
Theorem~\ref{Prop.monotonicity}.
\end{exam}
\begin{figure}[h!]
\begin{center}
\setlength{\unitlength}{2.8cm}
\begin{picture}(2,1.2)(0,-0.2) 
\thicklines
\put(0,0){\line(1,0){2}}
\put(0,1){\line(1,0){2}}
\put(0,0){\line(0,1){1}}
\put(2,0){\line(0,1){1}}
\put(0.2,0){\line(-1,2){0.08}}
\put(0.2,0){\line(2,1){1.68}}
\put(0.12,0.16){\line(2,1){1.68}}
\put(1.88,0.84){\line(-1,2){0.08}}
\thinlines
\put(1.8,0.1){\small $\Omega_2$}
\put(1.7,0.85){\small $\Omega_1$}
\end{picture}
\quad
\psset{unit=1cm}
\begin{pspicture}(2,2.2)(-2,-2.2)
\psset{linewidth=0.8pt}
\pscircle(0,0){2}
\thicklines
\put(1.9,0){\line(0,1){0.55}}
\put(1.9,0){\line(0,-1){0.55}}
\put(-1.9,0){\line(0,1){0.55}}
\put(-1.9,0){\line(0,-1){0.55}}
\put(-1.9,-0.55){\line(1,0){3.8}}
\put(-1.9,0.55){\line(1,0){3.8}}
\put(0,-1.7){$\Omega_2$}
\put(0,-0.1){$\Omega_1$}
\end{pspicture}
\caption{Counterexamples to the monotonicity of Neumann eigenvalues.}%
\label{Fig.counter.Neumann}
\end{center}
\end{figure}
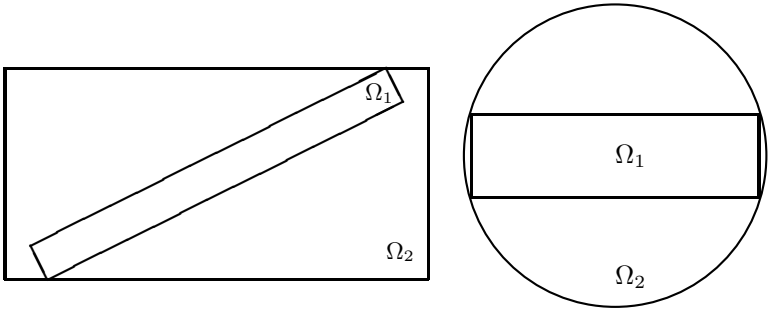
\begin{exam}[Non-monotonicity of Neumann eigenvalues]
Yet another counterexample to the monotonicity of Neumann 
eigenvalues is presented in 
right part of Figure~\ref{Fig.counter.Neumann}.
If the radius of the circumscribed disk~$\Omega_2$ is~$R$
and the lengths of the sides of the inscribed rectangle are $a \geq b$,
then $R^2=(a/2)^2+(b/2)^2$, 
so that the second Neumann eigenvalues satisfy
$$
  \lambda_2^N(\Omega_2)=(j_{1,1}'/R)^2 > (\pi/a)^2 = \lambda_2^N(\Omega_1)
$$
for all sufficiently small~$b$,
where~$j_{1,1}'$ denotes the first root of~$J_1'(x)=0$.
Indeed, $j_{1,1}' \approx 1.84$, while $\pi/2 \approx 1.57$.
(On the other hand, 
if the disk~$\Omega_2$ is inscribed in the rectangle~$\Omega_1$,
the inequality $\lambda_2^N(\Omega_2) < \lambda_2^N(\Omega_1)$ 
remains valid for all $a,b \geq 2R$.)
\end{exam}

As a replacement of domain monotonicity for the Neumann eigenvalues,
in 2023 Funano~\cite{Funano_2023} proved that 
that there exists a universal constant that controls by how much 
the quotient between two eigenvalues of 
convex domains~$\Omega_1$ and~$\Omega_2$ can vary, 
when~$\Omega_1$ is contained in~$\Omega_2$.
More recently, Freitas and Kennedy~\cite{Freitas-Kennedy}
found the optimal constant in the case of 
the lowest non-trivial eigenvalue.

\subsection{Absence of the essential spectrum}
Now we are in a position to establish 
the desired result that the spectrum of the Dirichlet Laplacian
in bounded sets is purely discrete.

\begin{theo}\label{Thm.bounded}
Let $\Omega$ be any bounded open set. 
Then
$$
  \sigma_\mathrm{ess}(-\Delta_D^{\Omega})  
  = \varnothing \,.
$$
\end{theo}
\begin{proof}
Since~$\Omega$ is bounded, there exists a cube~$Q$ 
such that $\Omega \subset Q$. 
By Theorem~\ref{Prop.monotonicity},
one has 
$
  \lambda_n^D(\Omega) \geq \lambda_n^D(Q)
$
for every $n \in \Nat^*$.
Taking the limit $n\to\infty$, 
we get 
$$
  \lambda_\infty^D(\Omega) \geq \lambda_\infty^D(Q) = \infty
  \,.
$$
Here the equality holds because we know that 
the spectrum of $-\Delta_D^{Q}$ is purely discrete.
\end{proof}
\begin{center}
\fbox{The absence does not hold for the Neumann Laplacian !}
\end{center}
\begin{exam}[Rooms and passages]\label{Ex.rooms}
A classical counterexample for the Neumann Laplacian consists 
of a planar domain~$\Omega$ depicted in Figure~\ref{Fig.rooms},
which is originally due to Fraenkel in 1979~\cite{Fraenkel_1979}
(see also \cite[Sec.~V.4.9]{Edmunds-Evans}).  
It is made up of an infinite sequence
of squares~$\mathcal{R}_j$ (``rooms'') 
of decreasing sizes $h_j$  
joined together by thin pipes~$\mathcal{P}_j$ (``passages'')
of diminishing cross-sections~$\delta_j$. 

More specifically, defining $h_0 := 0$, $h_j := j^{-3/2}$
and $\delta_j := j^{-6}$ for every $j \in \Nat^*$,
we introduce an $n^\mathrm{th}$ room
$$
  \mathcal{R}_n := \big(h_0+h_1+\dots+h_{n-1},h_0+h_1+\dots+h_{n-1}+h_n\big) 
  \times \left(-\frac{h_n}{2},\frac{h_n}{2}\right)
$$
and an $n^\mathrm{th}$ passage
$$
  \mathcal{P}_n := \left[h_0+h_1+\dots+h_{n-1},h_0+h_1+\dots+h_{n-1}+h_n\right]
  \times \left(-\frac{\delta_n}{2},\frac{\delta_n}{2}\right)
$$
and join together the odd rooms with even passages by setting 
$$
  \Omega := 
  \bigcup_{\stackrel[n \ \textrm{odd}]{}{n \in \Nat^*}}
  \mathcal{R}_n
  \cup
  \bigcup_{\stackrel[n \ \textrm{even}]{}{n \in \Nat^*}}
  \mathcal{P}_n
  \,.
$$

The domain~$\Omega$ is bounded because 
$$
  \displaystyle \sum_{j=1}^\infty h_j = \zeta(3/2) =: l < \infty
  \,,
$$ 
where~$\zeta$ denotes the Riemann zeta function;
in fact, $\Omega \subset [0,l] \times [-1/2,1/2]$.
However, the boundary~$\partial\Omega$ is not of class~$C^0$
because of the troublesome point $(l,0) \in \partial\Omega$.
\end{exam}
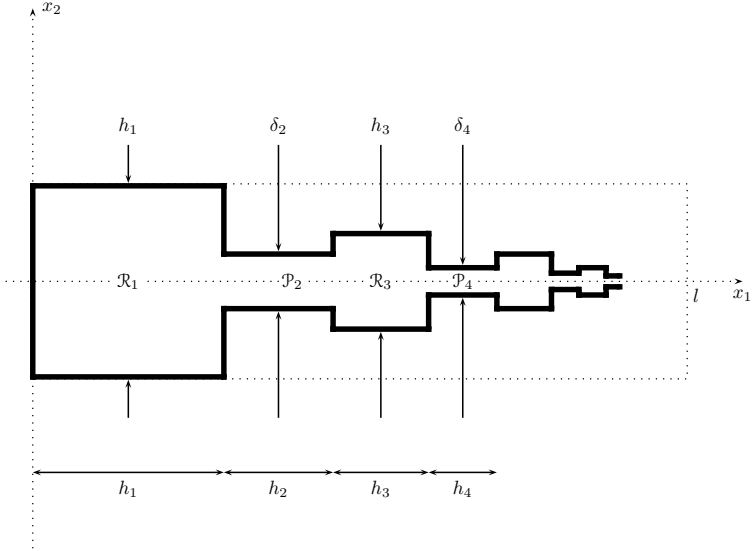
\begin{figure}[h!]
\begin{center}
\resizebox{1.75\textwidth}{!}{\begin{pspicture}(-1,-5)(27,5)
\psset{unit=0.5cm,origin={0,0}}
\thicklines
\psline[linestyle=dotted](-1,0)(2.8,0)
\psline[linestyle=dotted](4.2,0)(9,0)
\psline[linestyle=dotted](10.1,0)(12.1,0)
\psline[linestyle=dotted](13.3,0)(15.2,0)
\psline[arrows=->,linestyle=dotted](16.1,0)(26,0)
\rput(26,-0.5){$x_1$}
\psline[arrows=->,linestyle=dotted](0,-10)(0,10)
\rput(0.7,10){$x_2$}
\psframe[linestyle=dotted](-0.1,-3.6)(24,3.6)
\rput(24.3,-0.5){$l$}
\psline[linewidth=.1cm]{C-C}(0,-3.5)(0,3.5)
\psline[linewidth=.1cm]{C-C}(0,3.5)(7,3.5)
\psline[linewidth=.1cm]{C-C}(0,-3.5)(7,-3.5)
\psline[linewidth=.1cm]{C-C}(7,3.5)(7,1)
\psline[linewidth=.1cm]{C-C}(7,-3.5)(7,-1)
\psline[arrows=<->](0,-7)(7,-7)
\rput(3.5,-7.6){$h_1$}
\psline[arrows=->](3.5,-5)(3.5,-3.6)
\psline[arrows=->](3.5,5)(3.5,3.6)
\rput(3.5,5.7){$h_1$}
\rput(3.5,0){$\mathcal{R}_1$}
\psline[linewidth=.1cm]{C-C}(7,1)(11,1)
\psline[linewidth=.1cm]{C-C}(7,-1)(11,-1)
\psline[arrows=<->](7,-7)(11,-7)
\rput(9,-7.6){$h_2$}
\psline[arrows=->](9,-5)(9,-1.1)
\psline[arrows=->](9,5)(9,1.1)
\rput(9,5.7){$\delta_2$}
\rput(9.5,0){$\mathcal{P}_2$}
\psline[linewidth=.1cm]{C-C}(11,1)(11,1.75)
\psline[linewidth=.1cm]{C-C}(11,-1)(11,-1.75)
\psline[linewidth=.1cm]{C-C}(11,1.75)(14.5,1.75)
\psline[linewidth=.1cm]{C-C}(11,-1.75)(14.5,-1.75)
\psline[linewidth=.1cm]{C-C}(14.5,1.75)(14.5,0.5)
\psline[linewidth=.1cm]{C-C}(14.5,-1.75)(14.5,-0.5)
\psline[arrows=<->](11,-7)(14.5,-7)
\rput(12.75,-7.6){$h_3$}
\psline[arrows=->](12.75,-5)(12.75,-1.85)
\psline[arrows=->](12.75,5)(12.75,1.85)
\rput(12.75,5.7){$h_3$}
\rput(12.75,0){$\mathcal{R}_3$}
\psline[linewidth=.1cm]{C-C}(14.5,0.5)(17,0.5)
\psline[linewidth=.1cm]{C-C}(14.5,-0.5)(17,-0.5)
\psline[arrows=<->](14.5,-7)(17,-7)
\rput(15.75,-7.6){$h_4$}
\psline[arrows=->](15.75,-5)(15.75,-0.6)
\psline[arrows=->](15.75,5)(15.75,0.6)
\rput(15.75,5.7){$\delta_4$}
\rput(15.75,0){$\mathcal{P}_4$}
\psline[linewidth=.1cm]{C-C}(17,0.5)(17,1)
\psline[linewidth=.1cm]{C-C}(17,-0.5)(17,-1)
\psline[linewidth=.1cm]{C-C}(17,1)(19,1)
\psline[linewidth=.1cm]{C-C}(17,-1)(19,-1)
\psline[linewidth=.1cm]{C-C}(19,1)(19,0.3)
\psline[linewidth=.1cm]{C-C}(19,-1)(19,-0.3)
\psline[linewidth=.1cm]{C-C}(19,0.3)(20,0.3)
\psline[linewidth=.1cm]{C-C}(19,-0.3)(20,-0.3)
\psline[linewidth=.1cm]{C-C}(20,0.3)(20,0.5)
\psline[linewidth=.1cm]{C-C}(20,-0.3)(20,-0.5)
\psline[linewidth=.1cm]{C-C}(20,0.5)(21,0.5)
\psline[linewidth=.1cm]{C-C}(20,-0.5)(21,-0.5)
\psline[linewidth=.1cm]{C-C}(21,0.5)(21,0.2)
\psline[linewidth=.1cm]{C-C}(21,-0.5)(21,-0.2)
\psline[linewidth=.1cm]{C-C}(21,0.2)(21.5,0.2)
\psline[linewidth=.1cm]{C-C}(21,-0.2)(21.5,-0.2)

\end{pspicture}}
\end{center}
\caption{\emph{Rooms and passages} as an example of 
a bounded domain for which the Neumann Laplacian has an essential spectrum.}
\label{Fig.rooms}
\end{figure}
\begin{prop}[Rooms and passages]
Let~$\Omega$ be the (bounded) domain of Example~\ref{Ex.rooms}.
Then 
\begin{equation}\label{spec.rooms}
  0 \in \sigma_\mathrm{ess}(-\Delta_N^\Omega)
  \,.
\end{equation}
\end{prop}
\begin{proof}
For each odd $n \in \Nat^*$, 
define a function $u_n  \in W^{1,2}(\Omega)$
by requiring
$$
\begin{aligned}
  u_n(x) &:=
  \begin{cases}
    h_n^{-1} & \mbox{if} \quad x \in \mathcal{R}_n \,, 
    \\
    0 & \mbox{if} \quad x \in \Omega \setminus 
    (\mathcal{P}_{n-1}\cup\mathcal{R}_n\cup\mathcal{P}_{n+1})
    \,,
  \end{cases}
  \\
  \nabla u_n(x) &:= \pm \big((h_{n} h_{n \mp 1})^{-1},0\big) 
  \quad \mbox{if} \quad x \in \mathcal{P}_{n \mp 1}
  \,.
\end{aligned}  
$$
We have
$$
\begin{aligned}
  \|u_n\|^2 &= 1 + 
  \mbox{$\frac{1}{3}$} h_n^{-2} (h_{n-1}\delta_{n-1}+h_{n+1}\delta_{n+1})
  \geq 1 
  \,,
  \\
  \|\nabla u_n\|^2 &= (h_n h_{n-1})^{-2} h_{n-1} \delta_{n-1}
  + (h_n h_{n+1})^{-2} h_{n+1} \delta_{n+1}
  \xrightarrow[n\to\infty]{} 0
  \,.
\end{aligned}
$$
Consequently,
$$
  \frac{\|\nabla u_n\|^2}{\|u_n\|^2} \xrightarrow[n\to\infty]{} 0
  \,,
$$
which implies that $0 \in \sigma_\mathrm{ess}(-\Delta_N^\Omega)$
by the minimax principle,
because the functions~$u_n$ span an infinite-dimensional
subspace of $W^{1,2}(\Omega)$
(as the elements of the subsequence $\{u_{4n+1}\}_{n \in \Nat}$
have mutually disjoint supports).
\end{proof}

The feature of Example~\ref{Ex.rooms}
is the irregular point $(l,0) \in \partial\Omega$.
In general, a certain regularity of the boundary~$\partial\Omega$ 
is needed in order to ensure that the Neumann Laplacian
has a purely discrete spectrum in bounded sets.
The required regularity can be characterised in terms of 
the following extension property.

\begin{defi}\label{Def.extension}
\emph{An open set $\Omega \subset \Real^d$ 
is said to satisfy the \emph{extension property}
if there exists a bounded (linear) operator 
$E:W^{1,2}(\Omega) \to W^{1,2}(\Real^d)$ satisfying
$
  (E\psi)(x) = \psi(x)
$
for all $\psi \in W^{1,2}(\Omega)$ and all $x \in \Omega$.}
\end{defi}

Note that an analogous definition for Dirichlet boundary conditions is trivial, 
just because of the availability of
the trivial extension~\eqref{extension} for arbitrary sets.
This is the main reason behind the robust result 
of Theorem~\ref{Thm.bounded} about the emptiness
of the essential spectrum of the Dirichlet Laplacian
for \emph{any} bounded set.
Now we are in a position to establish the same result
under the extra hypothesis of Definition~\ref{Def.extension}.

\begin{theo}\label{Thm.bounded.Neumann}
Let $\Omega$ be any bounded open set
with the extension property. 
Then
$$
  \sigma_\mathrm{ess}(-\Delta_N^{\Omega})  
  = \varnothing \,.
$$
\end{theo}
\begin{proof} 
Since~$\Omega$ is bounded, there exists a cube~$Q$ 
such that $\Omega \subset Q$. 
We have the following bounded maps:
$$
  W^{1,2}(\Omega)
  \xrightarrow[]{\ E \ }
  W^{1,2}(\Real^d)
  \xrightarrow[]{\ R_1 \ }
  W^{1,2}(Q)
  \xrightarrow[]{\ \iota \ }
  \sii(Q)
  \xrightarrow[]{\ R_2 \ }
  \sii(\Omega)
  \,,
$$
where~$E$ is the extension of Definition~\ref{Def.extension},
$\iota$ is the embedding 
and $R_1,R_2$ are elementary restrictions.   
By the absence of the essential spectrum 
in Neumann rectangular boxes (established in Section~\ref{Sec.pipeds}), 
the embedding~$\iota$ is compact. 
More specifically, 
by writing 
$
  \iota = (-\Delta_N^Q+I)^{-1/2}(-\Delta_N^Q+I)^{1/2}
$, 
the compactness of~$\iota$ 
follows from the fact that 
the absence of the essential spectrum of $-\Delta_N^Q$
is equivalent (see \cite[Corol.~4,2,3]{Davies})
to the compactness of the resolvent $(-\Delta_N^Q+I)^{-1}$,
and therefore to the compactness of $(-\Delta_N^Q+I)^{-1/2}$.
In summary, the composed map 
$
  R_2 \circ \iota \circ R_1 \circ E 
$ 
is compact too. 
It follows that the embedding 
$W^{1,2}(\Omega) \hookrightarrow \sii(\Omega)$ is compact,
which means that the form domain of the Neumann Laplacian
is compactly embedded in the Hilbert space.
\end{proof}

The family of all open sets with the extension property is a wide one
(see \cite[Sec.~V.4.4]{Edmunds-Evans}).
It includes the so-called open sets with \emph{minimally smooth boundary},
which cover all bounded open sets with boundary of class~$C^{0,1}$
(\ie~Lipschitz regularity).
However, to have the conclusion of Theorem~\ref{Thm.bounded.Neumann},
it is enough to assume that the boundary of~$\Omega$
is merely continuous (see \cite[Sec.~V.4.17]{Edmunds-Evans}).

\subsection{Unbounded domains}\label{Sec.unbounded}
The situation becomes more difficult, even for the Dirichlet Laplacian,
if~$\Omega$ is an \emph{unbounded} quasi-bounded set.
It is still true that the spectrum of the Dirichlet Laplacian
is purely discrete if~$\Omega$ has finite volume.
Another example of sufficient condition is the following
result due to Berger and Schechter \cite{Berger-Schechter_1972}
(see~\cite[Thm.~V.5.17]{Edmunds-Evans} for a more general statement):
\begin{theo}[Berger--Schechter's criterion]
\label{Thm.Berger}
Let~$\Omega$ be an arbitrary open set. 
One has
$$
  \limsup_{\stackrel[x\in\Omega]{}{|x|\to\infty}}
  \big|\Omega \cap B_1(x)\big| = 0
  \qquad\Longrightarrow\qquad
  \sigma_\mathrm{ess}(-\Delta_D^\Omega) = \varnothing \,.
$$
\end{theo}

It is interesting to compare the sufficient condition
of Theorem~\ref{Thm.Berger} with
the characterisation~\eqref{q-bounded.equivalent}:
While quasi-bounded domains are just ``narrow at infinity'',
the sufficient condition of Theorem~\ref{Thm.Berger} requires
that the narrowness must be ``inessential in an integral sense''
to have a purely discrete spectrum.

Example~\ref{Ex.urchin} (spiny urchin) 
shows that Theorem~\ref{Thm.Berger} 
represents just a sufficient condition.
Indeed, since~$\Omega$ is built by removing from~$\Real^2$
just sets of measure zero (semi-infinite lines),
it follows that $|\Omega \cap B_1(x)| = |B_1(x)|$ for every $x\in\Omega$.
On the other hand, the following proposition shows
that the spectrum is still purely discrete.

\begin{prop}[Spiny urchin]\label{Prop.urchin}
Let $\Omega$ be the (unbounded) domain of Example~\ref{Ex.urchin}.
Then 
$$
  \sigma_\mathrm{ess}(-\Delta_D^\Omega) = \varnothing
  \,.
$$
\end{prop}
\begin{proof}
To prove that the spectrum of the Dirichlet Laplacian in~$\Omega$
is purely discrete, 
let us impose an extra Neumann condition on 
the circle $\Sigma_m := \partial B_m$.
More specifically, 
for every $m \in \Nat^*$,
we employ the decomposition
$$
  \Omega = \underbrace{(\Omega \cap B_m)}_{\Omega_m^\mathrm{int}} 
  \cup \underbrace{(\Omega \cap \partial B_m)}_{\Sigma_m}
  \cup \underbrace{(\Omega \setminus \overline{B}_m)}_{\Omega_m^\mathrm{ext}} 
  \,,
$$
which leads to the direct-sum decomposition of the Hilbert space
\begin{equation}\label{orthogonal}
  \sii(\Omega) 
  = \sii(\Omega_m^\mathrm{int}) \oplus \sii(\Omega_m^\mathrm{ext})
  \,.
\end{equation}
Recall that the Dirichlet Laplacian $H := -\Delta_D^\Omega$
is the operator in $\sii(\Omega)$
associated with the form 
$$
  h[\psi] := \int_{\Omega} |\nabla\psi|^2
  \,, \qquad
  \dom h := W_0^{1,2}(\Omega)
  \,. 
$$
The same operator with the extra Neumann condition on~$\Sigma_m$
is introduced as the operator~$H_m^N$ in $\sii(\Omega)$
associated with the form 
$$
  h_m^N[\psi] := \int_{\Omega} |\nabla\psi|^2
  \,, \qquad
  \dom h_m^N := 
  \big[ W_0^{1,2}(\Omega) \upharpoonright \Omega_m^\mathrm{int} \big]
  \oplus
  \big[ W_0^{1,2}(\Omega) \upharpoonright \Omega_m^\mathrm{ext} \big]
  \,,  
$$
where $W_0^{1,2}(\Omega) \upharpoonright \Omega_m^\mathrm{int}$
means the set of restrictions $\psi\upharpoonright\Omega_m^\mathrm{int}$
with $\psi \in W_0^{1,2}(\Omega)$,
and similarly for $W_0^{1,2}(\Omega) \upharpoonright \Omega_m^\mathrm{ext}$.
Note that $h_m^N$ acts in the same way as~$h$,
while the form domain of the former is larger
(\cf~Figure~\ref{Fig.bracketing}). 
Consequently,
\begin{equation}\label{urchin.bracketing}
  H \geq H_m^N
  = H_m^{\mathrm{int},N} \oplus H_m^{\mathrm{ext},N}
  \,,
\end{equation}
where $H_m^{\mathrm{int},N}$ is the operator 
in $\sii(\Omega_m^\mathrm{int})$ associated with the form 
$$
  h_m^{\mathrm{int},N}[\psi] 
  := \int_{\Omega_m^\mathrm{int}} |\nabla\psi|^2
  \,, \qquad
  \dom h_m^{\mathrm{int},N} := 
  W_0^{1,2}(\Omega) \upharpoonright \Omega_m^\mathrm{int}
  \,,  
$$
and $H_m^{\mathrm{ext},N}$ in $\sii(\Omega_m^\mathrm{ext})$
is defined analogously.

As a consequence of~\eqref{urchin.bracketing}, 
we get (\cf~Corollary~\ref{Corol.minimax})
$$
  \forall k \in \Nat^* \,, \qquad
  \lambda_k(H) \geq \lambda_k(H_m^N) 
  \,,
$$
and therefore (taking the limit $k \to \infty$)
\begin{equation}\label{urchin.lower}
\begin{aligned}
  \inf\sigma_\mathrm{ess}(H) 
  &\geq \inf\sigma_\mathrm{ess}(H_m^N) 
  \\
  &= \min\left\{ 
  \inf\sigma_\mathrm{ess}(H_m^{\mathrm{int},N}),
  \inf\sigma_\mathrm{ess}(H_m^{\mathrm{ext},N})
  \right\}
  \\
  &= \inf\sigma_\mathrm{ess}(H_m^{\mathrm{ext},N})
  \\
  &\geq \inf\sigma(H_m^{\mathrm{ext},N})
  \\
  &=
  \inf_{\stackrel[\psi \not= 0]{}{\psi \in \dom h_m^{\mathrm{ext},N}}}
  \frac{h_m^{\mathrm{ext},N}[\psi]}{\ \|\psi\|_{\sii(\Omega_m^\mathrm{ext})}^2}
  \,.
\end{aligned}  
\end{equation}
Here the second equality follows from the fact
that the spectrum of $H_m^{\mathrm{int},N}$ is purely discrete
(the extension of $\dom h_m^{\mathrm{int},N}$ 
to $W^{1,2}(B_m)$ is trivial
and the spectrum of the Neumann Laplacian in any disk is purely discrete,
see Theorem~\ref{Thm.bounded.Neumann}).
In summary, we have obtained a lower estimate to
the threshold of the essential spectrum of~$H$ 
through the spectrum of the ``exterior'' operator $H_m^{\mathrm{ext},N}$,
for any~$m$. 
It remains to analyse the spectral threshold of~$H_m^{\mathrm{ext},N}$.

\begin{figure}[h!] 
\begin{center}
\includegraphics[width=0.7\textwidth]{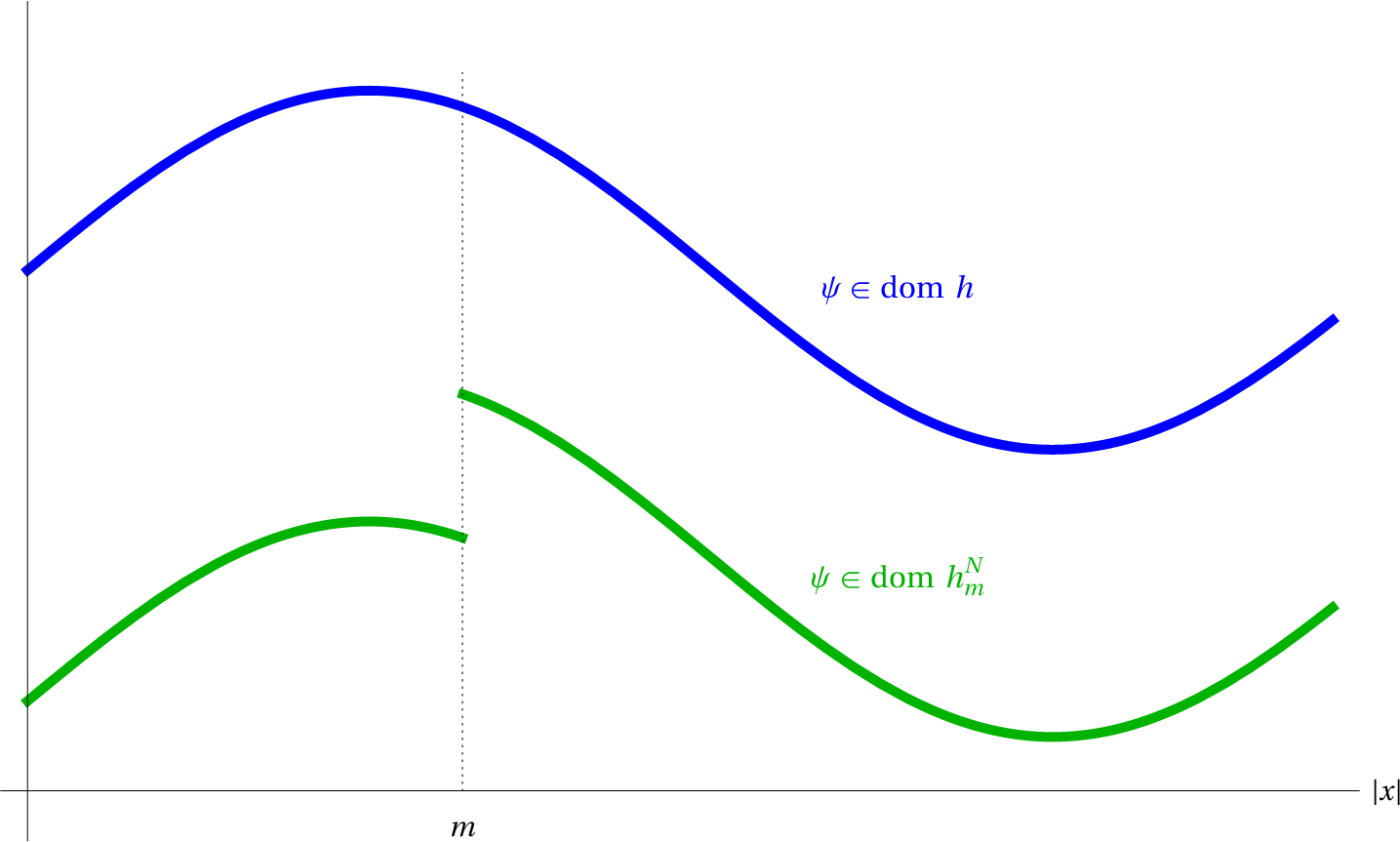}
\caption{Schematical visualisation of the effect
of introducing the Neumann boundary condition.}\label{Fig.bracketing}
\end{center}
\end{figure}

Let 
$
  \psi \in 
  C_0^\infty(\Omega) \upharpoonright \Omega_m^\mathrm{ext}
$,
a core of $h_m^{\mathrm{ext},N}$.
Passing to polar coordinates and neglecting the radial component, 
we have
\begin{equation}\label{urchin.lower2}
\begin{aligned}
  h_m^{\mathrm{ext},N}[\psi] 
  &= \int_{(m,\infty)\times S^1} 
  \left[
  \left|\partial_r\psi\right|^2
  + \frac{|\partial_\theta\psi|^2}{r^2}
  \right] \, r \, \der r \, \der\theta
  \\
  &\geq
  \sum_{j=0}^\infty 
  \int_{(m+j,m+j+1)\times S^1} 
  \frac{|\partial_\theta\psi|^2}{r^2}
  \ r \, \der r \, \der\theta
  \\
  &\geq 
  \sum_{j=0}^\infty 
  \int_{(m+j,m+j+1)\times S^1} 
  \left(\frac{\pi}{\pi/2^{m+j}}\right)^2 \frac{|\psi|^2}{r^2} 
  \ r \, \der r \, \der\theta
  \\
  &\geq 
  \sum_{j=0}^\infty 
  \int_{(m+j,m+j+1)\times S^1} 
  \left(\frac{2^{m+j}}{m+j+1}\right)^2 |\psi|^2
  \ r \, \der r \, \der\theta
  \\ 
  &\geq \left(\frac{2^m}{m+1}\right)^2 
  \int_{(m,\infty)\times S^1} 
  |\psi|^2 
  \ r \, \der r \, \der\theta
  \\
  &= \left(\frac{2^m}{m+1}\right)^2 
  \|\psi\|_{\sii(\Omega_m^\mathrm{ext})}^2 
  \,.
\end{aligned}
\end{equation}
Here the second inequality follows from 
the one-dimensional spectral bound
$$
  -\Delta_D^{(a,b)} \geq \left(\frac{\pi}{b-a}\right)^2
$$
for any real numbers $a<b$
(\cf~\eqref{D.spec} and~\eqref{Rayleigh}),
with help of Fubini's theorem,
by noticing that the angular distance on the circle~$\Sigma_m$ 
between two closest rays of the urchin is $\pi/2^m$
(\cf~Figure~\ref{Fig.urchin1}). 

\begin{figure}[h!] 
\begin{center}
\includegraphics[width=0.45\textwidth]{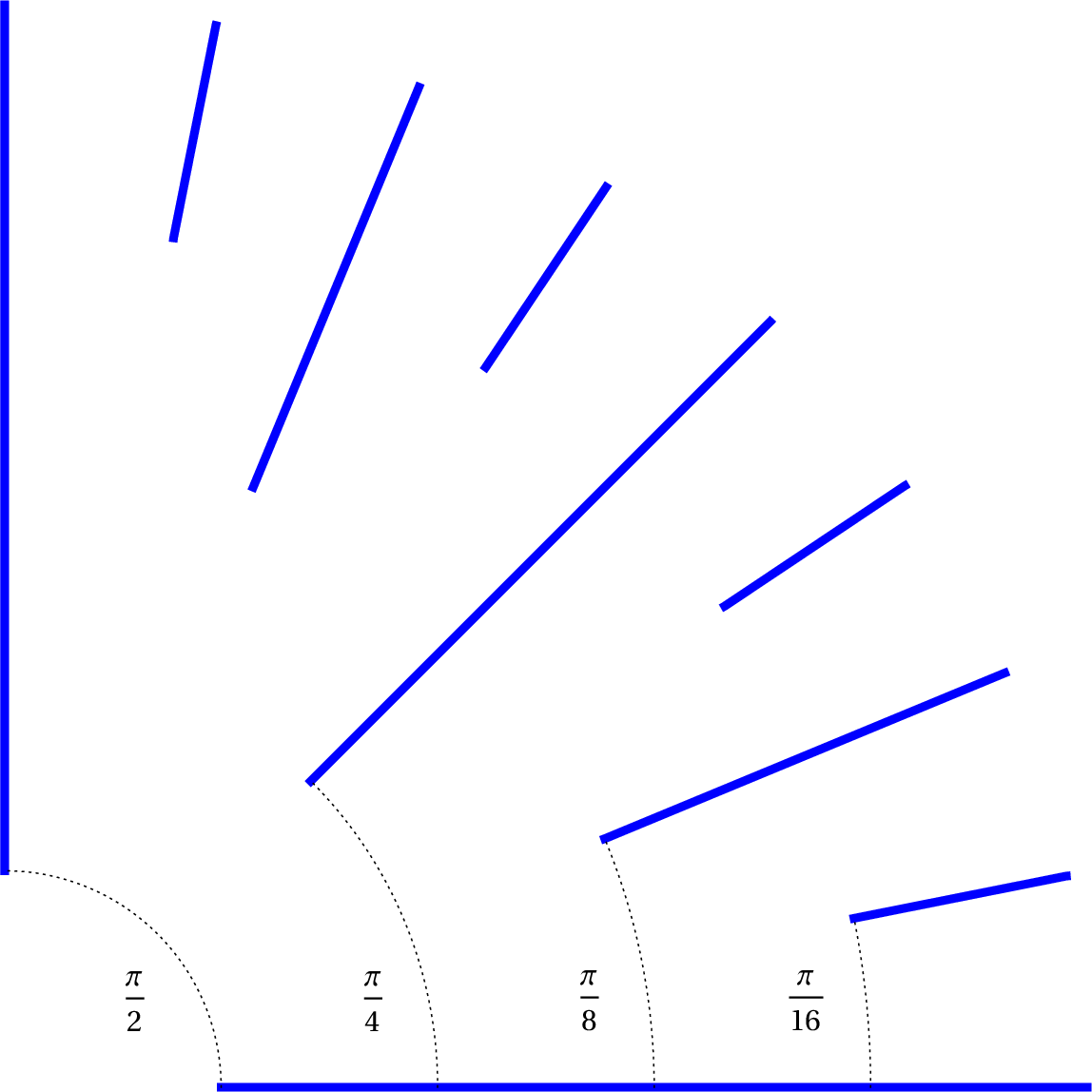}
\caption{The angular distance between two closest
rays of the spiny urchin.}\label{Fig.urchin1}
\end{center}
\end{figure}

In summary, from~\eqref{urchin.lower} and~\eqref{urchin.lower2}
we deduce 
$$
  \inf\sigma_\mathrm{ess}(H) 
  \geq \left(\frac{2^m}{m+1}\right)^2
  \,.
$$
Since the right-hand side tends to $+ \infty$ as $m \to \infty$,
while the left-hand side is independent of~$m$,
we obtain $\inf\sigma_\mathrm{ess}(H) = +\infty$.
That is, $\sigma_\mathrm{ess}(H) = \varnothing$ by
the minimax principle (Theorem~\ref{minimax}).
\end{proof}

In fact, it turns out that the property that the spectrum
of the Dirichlet Laplacian in~$\Omega$ is purely discrete
depends in an essential way on the dimension of~$\partial\Omega$.
Any quasi-bounded domain whose boundary consists
of reasonably regular $(d-1)$-dimensional hypersurfaces
has no essential spectrum.
For irregular boundaries, 
the reality can be drastically different,
as the following example demonstrates.

\begin{exam}[Spiny urchin with sparse spines]
If we replace the ``solid spines'' 
by ``dots accumulating at infinity'',
\ie, we define~$\dot\Omega$ as the domain in~$\Real^2$
obtained by deleting from the plane the union of the sets
\begin{multline*}
  \dot{S}_m :=
  \big\{
  (r\cos\theta,r\sin\theta) \, : \quad
  r = m + \sqrt{j}
  \quad\mbox{for}\quad j\in\Nat
  \\
  \quad \land \quad
  \theta = n\pi/2^m
  \quad\mbox{for}\quad n=1,2,\dots,2^{m+1}
  \big\}
  \,,
\end{multline*}
then exactly the same proof as that of Theorem~\ref{Thm.conical}
for quasi-conical domains implies 
$$
  \sigma(-\Delta_D^{\dot\Omega})
  = \sigma_\mathrm{ess}(-\Delta_D^{\dot\Omega})
  =[0,\infty)
  \,.
$$
This is obvious since a finite number of points
in an open planar set (\eg, an arbitrarily large disk)
form a set of capacity zero 
(\cf~Remark~\ref{Rem.Hardy}), 
so that  
$
  W_0^{1,2}( \dot{\Omega} \cap B_R) 
  = W_0^{1,2}(B_R)
$
for any $R>0$. 
\end{exam}

More generally, one has the following result.
\begin{theo}[{\cite[Thm.~1]{Adams_1970}}]
Let $d \geq 2$. If~$\partial\Omega$ consists only of isolated
points with no finite accumulation point,
then 
$$
  \sigma_\mathrm{ess}(-\Delta_D^\Omega) \not= \varnothing
  \,.
$$
\end{theo}

Finally, let us remark that in $d=1$
one knows that quasi-boundedness is necessary and sufficient
for an arbitrary (not necessary connected)
open subset $\Omega\subset\Real$ to have a purely discrete spectrum.
In higher dimensions, the necessary and sufficient conditions
can be obtained in terms of capacity \cite[Thm.~VIII.3.1]{Edmunds-Evans}. 

\subsection{Zero is always in the spectrum
of the Neumann Laplacian}
Finally, let us mention another peculiarity of Neumann boundary conditions.
If~$\Omega$ is bounded, then~$0$ is never in the spectrum 
of the Dirichlet Laplacian, for the spectrum is purely discrete
and the only solution of   
$
  0 = \|\nabla\psi\|^2
$
is $\psi=0$ (because~$\psi$ must be a constant function
and the Dirichlet boundary conditions force the constant to be zero).
On the other hand, non-zero constant functions are admissible
eigenfunctions of the Neumann Laplacian in bounded domains
(more generally in domains of finite volume),
so~$0$ is always in the spectrum of the Neumann Laplacian in such domains.
What is more, the property that~$0$ is in the spectrum 
of the Neumann Laplacian actually holds in the full generality
of arbitrary domains.
(Since the Neumann Laplacian is non-negative,
the result says that the bottom of the spectrum starts by zero.)

\begin{theo}\label{Thm.N.0} 
Let $\Omega$ be an arbitrary open set.
Then
$$
  0 \in \sigma(-\Delta_N^\Omega) 
  \,.
$$
\end{theo}

\begin{proof}
The statement of Theorem~\ref{Thm.N.0} in a greater 
generality of Riemannian manifolds can be found in \cite[Thm.~5.2.10]{Davies_1989}.
Our proof is partially inspired 
by the proof of \cite[Thm.~2.12]{CFKS}.
We are grateful to Markus Holzmann for letting us know about the idea.

The bound $\inf\sigma(-\Delta_N^\Omega) \geq 0$
follows by the non-negativity of the Neumann Laplacian.
To prove the opposite inequality, we recall
the minimax principle (\cf~\eqref{Rayleigh}),
\begin{equation}\label{variational}
  \inf\sigma(-\Delta_N^\Omega) =
  \inf_{\stackrel[\phi\not=0]{}{\phi \in W^{1,2}(\Omega)}}
  \frac{\|\nabla\phi\|^2}{\|\phi\|^2}
  \leq \frac{\|\nabla\psi\|^2}{\|\psi\|^2}
  \,,
\end{equation}
where the inequality holds for any non-zero $\psi \in W^{1,2}(\Omega)$.
For every $n \in \Nat^*$, we define $\psi_n(x) := \varphi_n(|x|)$,
where $\varphi_n \in C_0^\infty([0,\infty))$ is such that 
$0 \leq \varphi_n \leq 1$ and $|\varphi_n'| \leq C$ 
for all $n \in \Nat^*$ with some constant~$C$ independent of~$n$
and
$$
  \varphi_n(r) := 
  \begin{cases}
    1 & \mbox{if} \quad r < n \,,
    \\
    0 & \mbox{if} \quad r > n+1 \,.
  \end{cases}
$$
Clearly, for every $n \in \Nat^*$,
the restriction of~$\psi_n$ to~$\Omega$ 
(that we again denote by~$\psi_n$) 
belongs to $W^{1,2}(\Omega)$.
Let us take~$n$ sufficiently large so that $\Omega \cap B_n \not= \varnothing$. 
Since 
$
  |\nabla\psi_n(x)| = |\varphi_n'(|x|)| 
  = |\varphi_n'(|x|)| \chi_{B_{n+1} \setminus B_n}(x)
$ 
and $\psi_n(x) \geq \chi_{B_n}(x)$ for all $x \in \Real^d$, 
we have
$$
  \|\nabla\psi_n\|^2 
  \leq C^2 \, |\Omega \cap (B_{n+1}\setminus B_{n})|
  \,, \qquad
  \|\psi_n\|^2 \geq |\Omega \cap B_{n}| 
  \,.
$$
Replacing~$\psi$ by~$\psi_n$ in~\eqref{variational},
we therefore get the bound
\begin{equation}\label{Marcus}
  \inf\sigma(-\Delta_N^\Omega)
  \leq C^2 \,
  \frac{|\Omega \cap (B_{n+1}\setminus B_{n})|}{|\Omega \cap B_{n}|}
  =  C^2 \,
  \frac{|\Omega \cap B_{n+1}| - |\Omega \cap B_{n}|}{|\Omega \cap B_{n}|}
  \,.
\end{equation}
If the right-hand side equals to zero for some~$n$
(which is the case of bounded domains)
or in the limit as $n \to \infty$
(which is the case of the whole space~$\Real^d$),
we deduce $\inf\sigma(-\Delta_N^\Omega) \leq 0$.
In fact, it is enough to assume that one achieves 
the zero limit for a subsequence.
By contradiction, let us assume
$$
  \exists c>0, \ n_0>0, \quad \forall n \geq n_0, \qquad
  \frac{\omega_{n+1}-\omega_n}{\omega_n} \geq c
  \,,
$$ 
where $\omega_n := |\Omega \cap B_{n}|$.
That is, $\omega_{n+1} \geq (1+c) \, \omega_n$.
By recurrence, 
$$
  \forall k \in \Nat
  \,, \qquad
  \omega_{n+k} \geq (1+c)^k \, \omega_n
  \,.
$$
Estimating the left-hand side by the volume of the whole ball~$B_{n+k}$,
we obtain
$$
  \forall k \in \Nat
  \,, \qquad
  (n+k)^d \, |B_1|= |B_{n+k}| \geq (1+c)^k \, \omega_n
  \,.
$$
Since the right hand-side is exponentially growing with~$k$,
while the left-hand has a polynomial growth as a function of~$k$,
we arrive at a contradiction for all sufficiently large~$k$.
That is, the right-hand side of~\eqref{Marcus} 
tends to zero as $n \to \infty$ and the theorem is proved.
\end{proof}

\section{Spectral isoperimetric inequalities}
In this section, we look at extremal properties 
of the lowest eigenvalue $\lambda_1^\alpha(\Omega)$ 
as regards the shape of the underlying domain~$\Omega$.

\begin{center}
\fbox{Is the optimal geometry a ball ?}
\end{center}
\begin{figure}[h!t  ]
\begin{center}
\includegraphics[width=0.35\textwidth]{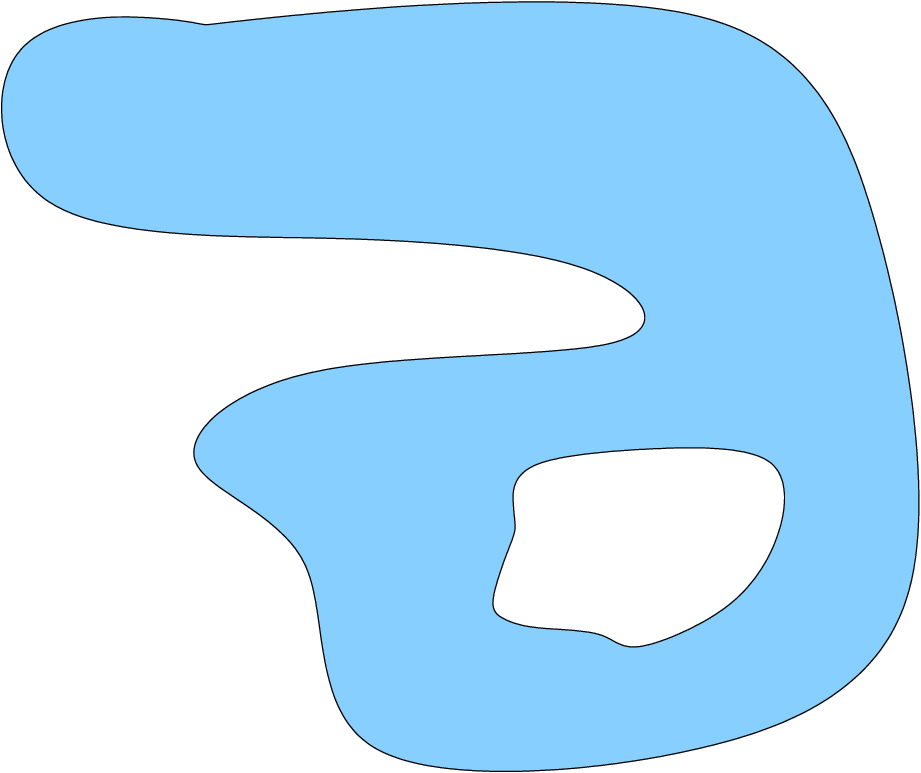}
\qquad\qquad
\includegraphics[width=0.3\textwidth]{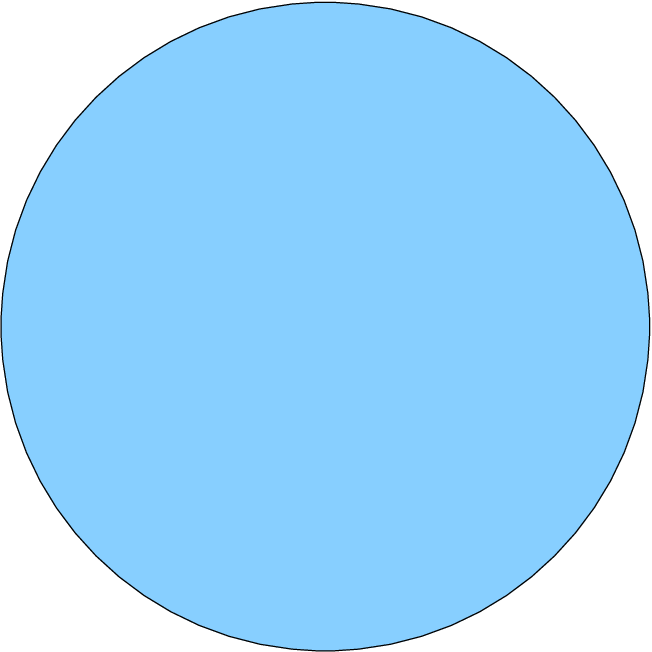}
\end{center}
\caption{An arbitrary bounded domain~$\Omega$ (left)
to be compared with the ball~$B$ (right).}
\label{Fig.domain}
\end{figure}

Let us begin by recalling some classical geometric facts.
For simplicity, 
you can assume that~$\Omega$ is a \emph{smooth} bounded domain,
in order to have classical definitions of its volume and boundary area,
but the domain can be multiply connected, see Figure~\ref{Fig.domain}.

\subsection{Geometric isoperimetric inequalities}
The (geometric) \emph{isoperimetric inequality}
states that among all sets of a given perimeter,
the ball has the largest volume.
That is, 
\begin{equation}\label{isoperimetric}
  \max_{|\partial\Omega|=\const} |\Omega| = |B|
  \,,
\end{equation}
where the maximum is taken over all bounded sets $\Omega \subset \Real^d$
of the fixed perimeter $|\partial\Omega|=\const$,
$B$~denotes the ball of the same perimeter as~$\Omega$
(\ie\ $|\partial B|=|\partial\Omega| = \const$)
and~$|\Omega|$ denotes the volume of~$\Omega$.
It is indeed an inequality because \eqref{isoperimetric}~is equivalent 
to the statement
\begin{equation}\label{isoperimetric.bis}
  \forall \Omega, \ |\partial\Omega| = \const, \qquad
  |\Omega| \leq |B|
  \,.
  \qquad\qquad
  (|\partial B| = |\partial\Omega| = \const)
\end{equation}
Moreover, the inequality becomes equality if, and only if, $\Omega=B$.

By scaling, \eqref{isoperimetric}~is equivalent 
to the  \emph{isochoric inequality}
stating that among all sets of a given volume,
the ball has the smallest perimeter.
That is, 
\begin{equation}\label{isochoric}
  \min_{|\Omega|=\const} |\partial\Omega| = |\partial B|
  \,,
\end{equation}
where the minimum is taken over all bounded sets $\Omega \subset \Real^d$
of a fixed volume $|\Omega|=\const$
and now~$B$ denotes the ball of the same volume as~$\Omega$ 
(\ie\ $|B|=|\Omega| = \const$).
Again, one is concerned with an inequality because~\eqref{isochoric}
is equivalent to the statement
\begin{equation}\label{isochoric.bis}
  \forall \Omega, \ |\Omega| = \const, \qquad
  |\partial\Omega| \geq |\partial B|
  \,.
  \qquad\qquad
  (|B| = |\Omega| = \const)
\end{equation}
Moreover, the inequality becomes equality if, and only if, $\Omega=B$.

The two inequalities~\eqref{isoperimetric.bis} and~\eqref{isochoric.bis}
can be stated as a unique inequality
(without any further constraints on the set~$\Omega$)
\begin{equation}\label{isochoric.both}
  \forall \Omega, \qquad
  |\partial\Omega|^d - d^d \, |B_1| \, |\Omega|^{d-1} \geq 0 
  ,
\end{equation}
and the inequality becomes equality if, and only if, $\Omega=B$.
Indeed, if~$R$ denotes the radius of~$B$, 
then the isoperimetric constraint requires 
$|\partial\Omega|=|\partial B_R| = R^{d-1} |\partial B_1|$,
while~\eqref{isoperimetric.bis} states that 
$|\Omega| \leq |B_R| = R^d |B_1|$;
eliminating~$R$ and using that $|\partial B_1| = d |B_1|$, 
we arrive at~\eqref{isochoric.both}.

The history of the geometric optimisation 
problems~\eqref{isoperimetric} and~\eqref{isochoric}
is briefly as follows
(we refer to~\cite{BuZa,Blasjo_2005} 
for a more detailed overview and references).
The optimality of ball was certainly known to ancient Greeks
(the solution is traditionally attributed 
to the legendary queen of Carthage Dido in about 900 BC),
but a rigorous proof was beyond the reach of their mathematical apparatus.
The first rigorous proof in two dimensions
was given by Steiner in 1838 AD, after three millennia.
In fact, he gave five nice proofs, but all of them were incomplete,
because he \emph{a priori}
assumed that the maximisation/minimisation problem
actually admits a solution.
This was settled by Weierstrass in 1879.
In 1884, Schwartz established the first proof in three dimensions.
A proof for higher dimensions was given by Hurwitz in 1901.
Skipping many subsequent achievements,
let us only mention a quantitative version of the isoperimetric inequality 
established by
Fusco, Maggi and Pratelli in 2008 \cite{Fusco-Maggi-Pratelli_2008}.

\subsection{The Faber--Krahn inequality}
Going from geometric to spectral quantities,
one may ask the question whether the ball is the extremal set 
also when optimising eigenvalues instead of the geometric data. 
The most celebrated result is certainly the \emph{Faber--Krahn inequality}
stating that it is indeed the case 
for the lowest Dirichlet eigenvalue under the isochoric constraint.  

\begin{theo}[Spectral isochoric inequality, Dirichlet case]\label{Thm.FK}
One has
\begin{equation}\label{FK}
  \min_{|\Omega|=\const} \lambda_1^D(\Omega) = \lambda_1^D(B)
  \,,
\end{equation}
where the minimum is taken over all bounded domains 
$\Omega \subset \Real^d$
of a fixed volume $|\Omega|=\const$
and~$B$ denotes the ball of the same volume as~$\Omega$ 
(\ie\ $|B|=|\Omega| = \const$).
\end{theo}
\begin{proof}
Let $\psi_1 := \psi_1^D$ denote a real-valued eigenfunction
of~$-\Delta_D^\Omega$ corresponding to $\lambda_1^D(\Omega)$. 
By the minimax principle (Theorem~\ref{minimax}),
one has
\begin{equation}\label{Rayleigh.ground}
  \lambda_1^D(\Omega) 
  = \inf_{\stackrel[\psi\not=0]{}{\psi \in W_0^{1,2}(\Omega)}} 
  \frac{\displaystyle \int_\Omega |\nabla\psi|^2}
  {\displaystyle \int_\Omega |\psi|^2}
  = \frac{\displaystyle \int_\Omega |\nabla\psi_1|^2}
  {\displaystyle \int_\Omega |\psi_1|^2}
  = \frac{\displaystyle \int_\Omega 
  \big|\nabla|\psi_1|\big|^2}
  {\displaystyle \int_\Omega \big||\psi_1|\big|^2}
  \,,
\end{equation}
where the second equality holds because $\lambda_1^D(\Omega)$
is an eigenvalue, so the infimum is actually achieved. 
It follows (\cf~Proposition~\ref{Prop.achieved})
that the absolute value~$|\psi_1|$ is also
a minimiser of~\eqref{Rayleigh.ground} and thus eigenfunction 
of~$-\Delta_D^\Omega$ corresponding to $\lambda_1^D(\Omega)$. 
Hence, $\lambda_1^D(\Omega)$ admits a non-negative eigenfunction.
 
Define the \emph{symmetric-decreasing rearrangement} 
(or the \emph{Schwarz symmetrisation},
see Figure~\ref{Fig.Schwarz})
$$
  \psi_1^*(x) := \int_0^\infty \chi_{\{|\psi_1|>t\}^*}(x) \, \der t
  \,,
$$
where~$S^*$ denotes the \emph{symmetric rearrangement} 
of any bounded measurable set $S \subset \Real^d$,
\ie\ $S^* := B_r(0)$ with $|B_r(0)| = |S|$.
Then $\psi_1^* \in W_0^{1,2}(\Omega^*)$, 
where $\Omega^*:=B_R(0)$ is a ball such that $|\Omega^*| = |\Omega|$.
Notice that~$\psi_1^*$ is constructed from~$\psi_1$ 
by rearranging the level sets of~$\psi_1$ in balls of the same volume.
Clearly, $\psi_1^*$~is non-negative, 
radially symmetric (\ie, $\psi_1^*(x)=\psi_1^*(y)$ if $|x|=|y|$)
and non-increasing as a function of the distance from the origin
(\ie, $\psi_1^*(x) \geq \psi_1^*(y)$ if $|x| \leq |y|$). 
Since the functions~$\psi_1$ and~$\psi_1^*$ are obviously \emph{equimeasurable}
(\ie~their level sets have the same measure),
we immediately get  
$$
  \|\psi_1^*\|_{\sii(\Omega^*)} = \|\psi_1\|_{\sii(\Omega)}
  \,.
$$
It is more difficult to prove that the derivative diminishes 
after the symmetric rearrangement 
(see, \eg, \cite[Lem.~7.17]{LL} for a proof)
$$
  \|\nabla \psi_1^*\|_{\sii(\Omega^*)} 
  \leq \|\nabla \psi_1\|_{\sii(\Omega)}
  \,.
$$

\begin{figure}[h!]
\begin{center}
\includegraphics[width=0.95\textwidth]{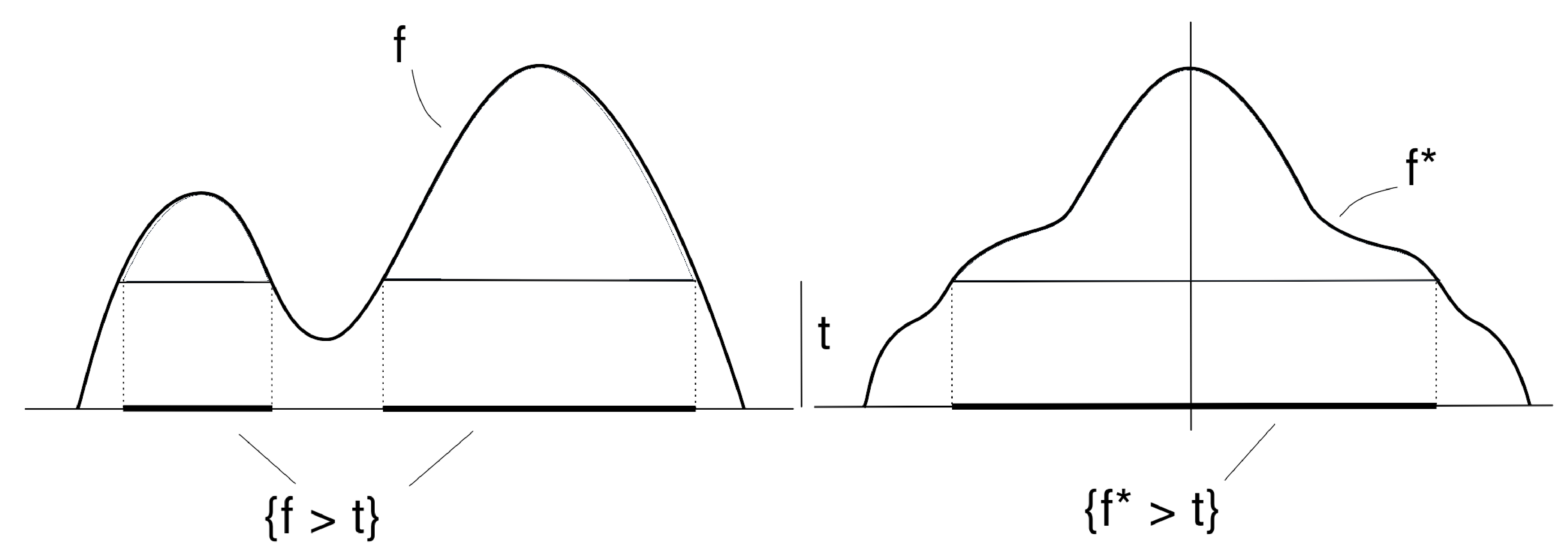}
\end{center}
\caption{The symmetric-decreasing rearrangement 
of a one-di\-men\-sio\-nal function~$f$.
(Source: \emph{Wikipedia} \cite{Wiki-Schwarz}.)}
\label{Fig.Schwarz}
\end{figure}

After these preliminaries,
we are now in a position to establish~\eqref{FK}. 
Using~$\psi_1^*$ as a trial function 
in the variational characterisation
of $\lambda_1^D(B)$ with $B:=\Omega^*$,
we get
$$
  \lambda_1^D(B) 
  = \inf_{\stackrel[\psi \not= 0]{}{\psi \in W_0^{1,2}(\Omega^*) }}
  \frac{\|\nabla\psi\|_{\sii(\Omega^*)}^2}{\ \|\psi\|_{\sii(\Omega^*)}^2}
  \leq \frac{\|\nabla\psi_1^*\|_{\sii(\Omega^*)}^2}{\ \|\psi_1^*\|_{\sii(\Omega^*)}^2}
  \leq \frac{\|\nabla\psi_1\|_{\sii(\Omega)}^2}{\ \|\psi_1\|_{\sii(\Omega)}^2}
  = \lambda_1^D(\Omega)
  \,. 
$$
This is equivalent to~\eqref{FK}. 
\end{proof}

Theorem~\ref{Thm.FK} implies the spectral isoperimetric 
inequality as a corollary.

\begin{coro}[Spectral isoperimetric inequality, Dirichlet case]
\label{Corol.FK}
One has
\begin{equation}\label{FK.perim}
  \min_{|\partial\Omega|=\const} \lambda_1^D(\Omega) = \lambda_1^D(B)
  \,,
\end{equation}
where the minimum is taken over all bounded domains 
$\Omega \subset \Real^d$
of a fixed perimeter $|\partial\Omega|=\const$
and~$B$ denotes the ball of the same perimeter as~$\Omega$ 
(\ie\ $|\partial B|=|\partial\Omega| = \const$).
\end{coro}
\begin{proof}
The spectral isochoric inequality~\eqref{FK} means 
$$
 \forall \Omega, \ |\Omega| = \const, \qquad
  \lambda_1^D(\Omega) \geq \lambda_1^D(B)
  \,.
  \qquad\qquad
  (|B| = |\Omega| = \const)
  \,.
$$
The (geometric) isochoric inequality~\eqref{isochoric.bis} implies 
$|\partial\Omega| \geq |\partial B|$.
It follows that there exists a ball $\tilde{B} \supset B$ 
such that $|\partial\tilde{B}| = |\partial \Omega|$. 
By the monotonicity of Dirichlet eigenvalues 
(Theorem~\ref{Prop.monotonicity}), 
$\lambda_1^D(B) \geq \lambda_1^D(\tilde{B})$.
\end{proof}

The brief history of~\eqref{FK} and~\eqref{FK.perim} goes as follows.
In 1877, 
based on explicitly solvable models,
Lord Rayleigh~\cite{Rayleigh_1877} conjectured
the validity of the spectral isochoric inequality.
In 1918, Courant~\cite{Courant_1918} 
established the spectral isoperimetric inequality.
In 1923--4,
Faber~\cite{Faber_1923} and Krahn~\cite{Krahn_1924} 
independently proved the spectral isochoric inequality.
A quantitative version of the Faber--Krahn inequality was established
by Fusco, Maggi and Pratelli in 2009~\cite{Fusco-Maggi-Pratelli_2009}.

\subsection{The Bossel--Daners inequality}
Next, let us investigate the spectral isochoric and isoperimetric 
inequalities for the Robin Laplacian $-\Delta_\alpha^\Omega$. 
We assume that~$\Omega$ is a bounded Lipschitz domain,
in order to guarantee the validity of~\eqref{Ass.Robin}
as well as that the spectrum is purely discrete
(recall Theorem~\ref{Thm.bounded.Neumann}). 

The Neumann case $\alpha=0$ is trivial, 
because $\lambda_1^N(\Omega)=0$
for any bounded domain~$\Omega$ 
(the corresponding eigenfunction is any non-zero constant). 
The problem is interesting for the first non-trivial eigenvalue 
$\lambda_2^N(\Omega)$ (so as it is for higher Dirichlet eigenvalues),
but we shall not consider these optimisation problems here.
Let us start with the Robin Laplacian $-\Delta_\alpha^\Omega$ 
with a \emph{positive} boundary parameter~$\alpha$

\begin{theo}[Spectral isochoric inequality, positive Robin case]\label{Thm.Bossel}
For every $\alpha > 0$,
one has
\begin{equation}\label{Bossel}
  \min_{|\Omega|=\const} \lambda_1^\alpha(\Omega) = \lambda_1^\alpha(B)
  \,,
\end{equation}
where the minimum is taken over all bounded Lipschitz domains 
$\Omega \subset \Real^d$
of a fixed volume $|\Omega|=\const$
and~$B$ denotes the ball of the same volume as~$\Omega$ 
(\ie\ $|B|=|\Omega| = \const$).
\end{theo}

Again, Theorem~\ref{Thm.Bossel} implies the corresponding 
spectral isoperimetric inequality.

\begin{coro}[Spectral isoperimetric inequality, positive Robin case]
For every $\alpha > 0$,
one has
\begin{equation*}
  \min_{|\partial\Omega|=\const} \lambda_1^\alpha(\Omega) = \lambda_1^\alpha(B)
  \,,
\end{equation*}
where the minimum is taken over all bounded Lipschitz domains 
$\Omega \subset \Real^d$
of a fixed perimeter $|\partial\Omega|=\const$
and~$B$ denotes the ball of the same perimeter as~$\Omega$ 
(\ie\ $|\partial B|=|\partial\Omega| = \const$).
\end{coro}
\begin{proof}
The spectral isochoric inequality~\eqref{Bossel} means 
$$
 \forall \Omega, \ |\Omega| = \const, \qquad
  \lambda_1^\alpha(\Omega) \geq \lambda_1^\alpha(B)
  \,.
  \qquad\qquad
  (|B| = |\Omega| = \const)
  \,.
$$
The (geometric) isochoric inequality~\eqref{isochoric.bis} implies 
$|\partial\Omega| \geq |\partial B|$.
It follows that there exists another ball $\tilde{B} \supset B$ 
with the same centre as~$B$
and such that $|\partial\tilde{B}| = |\partial \Omega|$.
Contrary to the Dirichlet case, however,
now we cannot use the monotonicity of eigenvalues to conclude that 
$\lambda_1^\alpha(B) \geq \lambda_1^\alpha(\tilde{B})$.
Anyway, the same inequality holds by a simple scaling.
Indeed, if~$R$ and~$\tilde{R}$ denote the radii of~$B$ and~$\tilde{B}$,
respectively, then
$$
  \lambda_1^\alpha(B) =  \left(\frac{\tilde{R}}{R}\right)^2 \, \lambda_1^{\tilde\alpha}(\tilde{B}) 
  \qquad \mbox{with} \qquad
  \tilde\alpha := \frac{R}{\tilde{R}} \, \alpha
  \,,
$$
which can easily be deduced by using the variational characterisation
of $\lambda_1^\alpha(B)$ 
and performing a scaling change variables in the integrals.
Since $\tilde{R} \geq R$ and~$\alpha$ is non-negative,
we get $\lambda_1^\alpha(B) \geq \lambda_1^\alpha(\tilde{B})$.   
\end{proof}

The brief history of~\eqref{Bossel} is as follows.
In 1986, 
Bossel~\cite{Bossel_1986} established the spectral isochoric inequality 
in dimension two.
The extension to higher dimensions was achieved 
by Daners in 2006~\cite{Daners_2006},
following the original idea of Bossel's. 
In 2010,
Bucur and Giacomini~\cite{Bucur-Giacomini_2010}
proposed an alternative proof based on the theory of special functions 
of bounded variation. 
In 2023,
Alvino, Nitsch, Trombetti \cite{Alvino-Nitsch-Trombetti}
developed an alternative proof in two dimensions 
based on a Talenti-type comparison principle.
 
\subsection{Bareket's conjecture}
Let us now look at the Robin Laplacian $-\Delta_\alpha^\Omega$
with a negative boundary parameter~$\alpha$.
It seems to be natural to expect that the ball is again 
the optimal set for the lowest eigenvalue.
However, since $\lambda_1^\alpha(\Omega)$ is negative whenever $\alpha<0$
(indeed, choose a constant trial function 
in the variational characterisation of $\lambda_1^\alpha(\Omega)$),
now it makes sense to maximise it.	 

\begin{conj}[Spectral isochoric inequality, negative Robin case]\label{Bareket}
For every $\alpha < 0$, one has
\begin{equation*}
  \max_{|\Omega|=\const} \lambda_1^\alpha(\Omega) = \lambda_1^\alpha(B)
  \,,
\end{equation*}
where the maximum is taken over all bounded Lipschitz domains 
$\Omega \subset \Real^d$
of a fixed volume $|\Omega|=\const$
and~$B$ denotes the ball of the same volume as~$\Omega$ 
(\ie\ $|B|=|\Omega| = \const$).
\end{conj}

This conjecture was raised by Bareket in 1977.
Moreover, she proved it for two-dimensional domains close to the ball
provided that~$|\alpha|$ is small. 
In 2007, 
Brock and Daners \cite{Brock-Daners_2007} rejuvenated the conjecture
by re-stating it as an open problem during an Oberwolfach conference.
In 2013,
by computing the derivative of $\lambda_1^\alpha(\Omega)$ 
with respect to $\alpha=0$, Daners~\cite{Daners_2013} gave another
support for the validity of the conjecture. 
In the same year, 
during a conference in Erlangen \cite{Erlangen},
Freitas formulated the conjecture as problem P1 
from ``Top~5 Open Problems'' in spectral geometry.
In 2015,
Ferone, Nitsch and Trombetti~\cite{Ferone-Nitsch-Trombetti_2015}
proved the conjecture for arbitrary domains closed to the ball,
with the closedness depending on~$\alpha$. 
 
Since we state this spectral isochoric inequality as a conjecture,
it might be expected that something goes wrong.
Indeed, in collaboration with Freitas~\cite{FK7},
we disproved the conjecture by showing that there exists another domain
which is better than the ball, at least if~$|\alpha|$ is large,
see Figure~\ref{Fig.annulus}.
\begin{theo}[Freitas \& Krej\-\v{c}i\v{r}\'ik \cite{FK7},
disproval of Conjecture~\ref{Bareket}]\label{Thm.counter.Bareket}
For every positive numbers $R_1 < R_2$, 
there exists $\alpha_0=\alpha_0(R_1,R_2) < 0$ such that, 
for all $\alpha \leq \alpha_0$, 
\begin{equation}\label{counter}
  \lambda_1^\alpha(B_R) < \lambda_1^\alpha(A_{R_1,R_2})
  \,,
\end{equation}
where $A_{R_1,R_2} := B_{R_2} \setminus \overline{B_{R_2}}$
is a \emph{spherical shell} 
and~$B_R$ with radius $R:= (R_2^d - R_1^d)^{1/d} < R_2$
is a ball of the same volume as $A_{R_1,R_2}$
(\ie\ $|B_R|=|A_{R_1,R_2}|$).
\end{theo}
\begin{proof}
Because of the rotational symmetry, the spectral problem
for the balls and spherical shells can be solved by separation of variables
in terms of special (namely, Bessel) functions.
Employing the known asymptotics of the Bessel functions,
it is tedious but straightforward to establish the following asymptotics:
$$
\begin{aligned}
  \lambda_1^\alpha(B_R) 
  &= -\alpha^2 + \frac{d-1}{R} \, \alpha + o(\alpha) \,,
  \\
  \lambda_1^\alpha(A_{R_1,R_2}) 
  &= -\alpha^2 + \frac{d-1}{R_2} \, \alpha + o(\alpha) \,,
\end{aligned}
\qquad \mbox{as} \qquad
\alpha \to -\infty \,.
$$
(For such type of asymptotics in general smooth domains,    
see \cite{Pankrashkin_2013,Exner-Minakov-Parnovski_2014,
Pankrashkin-Popoff_2015a,Pankrashkin-Popoff_2016}.)
Since the condition $|B_R|=|A_{R_1,R_2}|$ implies $R < R_2$
and~$\alpha$ is negative, we get the desired inequality~\eqref{counter}
for all sufficiently large~$|\alpha|$. 
\end{proof}
\begin{figure}[h!]
\begin{center}
\includegraphics[width=0.5\textwidth]{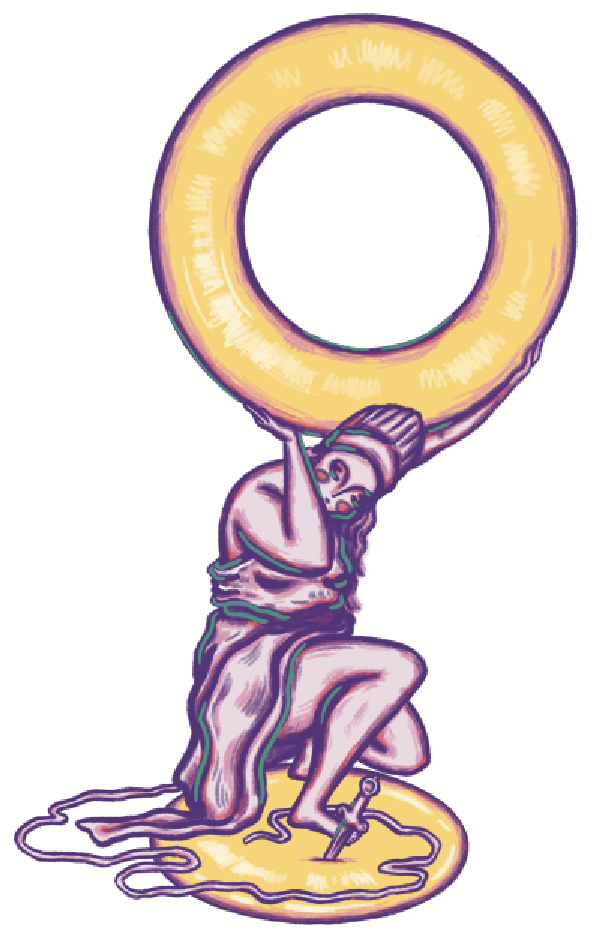}
\end{center}
\caption{As Dido must admit, 
spherical shells are more optimal than the ball in the spectral
isochoric maximisation problem of the negative Robin problem.
(Courtesy of Ane\v{z}ka K\r{u}lov\'a.)}
\label{Fig.annulus}
\end{figure}

Theorem~\ref{Thm.counter.Bareket} 
is remarkable for it provides the first known example 
where the extremal domain for the lowest eigenvalue of 
the Robin Laplacian is not a ball.
It remains open to show that spherical shells are the maximisers.
\begin{OProblem}
For every $\alpha < 0$, there exist positive numbers $R_1<R_2$
such that 
\begin{equation*}
  \max_{|\Omega|=\const} \lambda_1^\alpha(\Omega) 
  = \lambda_1^\alpha(A_{R_1,R_2})
  \,,
\end{equation*}
where the maximum is taken over all bounded Lipschitz domains 
$\Omega \subset \Real^d$
of a fixed volume $|\Omega|=\const$
and~$A_{R_1,R_2}$ denotes a spherical shell of the same volume as~$\Omega$ 
(\ie\ $|A_{R_1,R_2}|=|\Omega| = \const$).
\end{OProblem}

At the same time, it is still believed 
(and supported by numerical experiments~\cite{AFK})
that the ball is the maximiser within a more restricted class of domains
(in dimensions $d\geq 3$, it is not enough to assume that
the domain is simply connected~\cite{Ferone-Nitsch-Trombetti_2016}).
\begin{OProblem}
Does Conjecture~\ref{Bareket} hold within the class of 
\emph{simply connected} bounded Lipschitz domains if $d=2$
and \emph{convex} bounded domains if $d \geq 2$ ?
\end{OProblem}

In the same paper~\cite{FK7}, 
we also get a positive result about Conjecture~\ref{Bareket}
for small values of~$|\alpha|$ and planar domains.
Here it is important that the smallness depends solely
on the fixed area.  

\begin{theo}[Freitas \& Krej\v{c}i\v{r}\'ik \cite{FK7},
proof of Conjecture~\ref{Bareket}, 
$d=2$ and small $|\alpha|$]\label{Thm.Bareket}
There exists a negative number $\alpha_0=\alpha_0(|\Omega|)$ 
such that, for all $\alpha \in [\alpha_0,0]$
and any planar bounded domain~$\Omega$ of class~$C^2$ of fixed area, 
one has 
\begin{equation}\label{small}
  \lambda_1^\alpha(\Omega) \leq \lambda_1^\alpha(B)
  \,,
\end{equation}
where~$B$ is a ball of the same area as $\Omega$
(\ie\ $|B|=|\Omega|$).
\end{theo}
\begin{proof}
By the analytic perturbation theory
and the (geometric) iso\-perimetric inequality, 
it is easy to see that
$$
  \left. 
  \frac{\der \lambda_1^\alpha(\Omega)}{\der\alpha} 
  \right|_{\alpha=0}
  = \alpha \, \frac{|\partial\Omega|}{|\Omega|}
  \geq \alpha \, \frac{|\partial B|}{|B|}
  = \left. 
  \frac{\der \lambda_1^\alpha(B)}{\der\alpha} 
  \right|_{\alpha=0}
  \,.
$$
This argument (overtaken from~\cite{Daners_2013})
gives~\eqref{small}, but with  $\alpha_0=\alpha_0(\Omega)$. 
In~\cite{FK7}, we instead develop the method of parallel coordinates 
due to~\cite{Payne-Weinberger_1961}, which is based on using  
trial functions whose level lines are parallel to~$\partial\Omega$. 
In this way, we get~\eqref{small} 
with  $\alpha_0=\alpha_0(|\Omega|,|\partial\Omega|)$. 
To get rid of the dependence on the perimeter,
we make a careful comparison of Robin-Neumann annuli 
with Robin disks using explicit solutions
obtained in terms of Bessel functions. 
\end{proof}

An extension of this uniform result to higher dimensions 
has recently been achieved for convex domains in~\cite{GKP2}.

The isoperimetric constraint seems to be simpler.

\begin{theo}[Spectral isoperimetric inequality, negative Robin case]
\label{Thm.perimeter}
For every $\alpha < 0$,
one has
\begin{equation*}
  \max_{|\partial\Omega|=\const} \lambda_1^\alpha(\Omega) = \lambda_1^\alpha(B)
  \,,
\end{equation*}
where the maximum is taken over all bounded domains 
$\Omega \subset \Real^d$
of a fixed perimeter $|\partial\Omega|=\const$
which are either of class $C^2$ if $d=2$ or convex if $d \geq 2$
and~$B$ denotes the ball of the same perimeter as~$\Omega$ 
(\ie\ $|\partial B|=|\partial\Omega| = \const$).
\end{theo}

In 2017, 
Antunes, Freitas and  Krej\v{c}i\v{r}\'ik~\cite{AFK}
gave a first proof of Theorem~\ref{Thm.perimeter}
for planar domains of class $C^2$.
In 2019,
Bucur, Ferone, Nitsch and Trombetti
\cite{Bucur-Ferone-Nitsch-Trombetti_2019}
established the theorem for convex domains 
in arbitrary dimensions.
Independently,
Vikulova \cite{Vikulova_2022} extended the proof of~\cite{AFK} 
to convex domains in~$\Real^3$.

It remains open to show that the extra geometric hypotheses 
are superfluous.
\begin{OProblem}
Extend Theorem~\ref{Thm.perimeter} to all 
bounded Lipschitz
domains $\Omega \subset \Real^3$ 
without assuming the convexity.
\end{OProblem}
%
 
\section{Properties of eigenfunctions}
%
Up to now, we were exclusively interested in qualitative
and quantitative properties of the spectrum. 
Let us now look at properties of the eigenfunctions.
 
\subsection{Regularity}
By~\eqref{Laplace}, $\psi \in \dom(-\Delta_\alpha^\Omega)$
particularly satisfies $\psi \in W^{1,2}(\Omega)$ 
and $\Delta\psi \in \sii(\Omega)$. 
\emph{Elliptic regularity theory} ensures that 
\begin{equation}\label{elliptic}
  \psi \in W^{1,2}(\Omega) \quad \& \quad
  \Delta\psi \in \sii(\Omega)
  \qquad \Longrightarrow \qquad
  \nabla^2\psi \in \sii_\mathrm{loc}(\Omega)
  \,,
\end{equation}
where~$\nabla^2$ denotes the Hessian.
What is more, if~$\Omega$ is ``nice''
(\eg, bounded and $C^2$-smooth),
the hypotheses together with the boundary condition~\eqref{bc}
imply the global extra regularity
$\nabla^2\psi \in \sii(\Omega)$.
That is why, one has the characterisation~\eqref{nice}
of the operator domain. 

The gain of regularity~\eqref{elliptic}
is indeed remarkable:
Not only that the sum of second derivatives 
is square integrable, but each mixed second derivatives separately,
at least locally.
Heuristically, \eqref{elliptic}~can be understood 
by a formal(!) integration by parts:
$$
\begin{aligned}
  \|\Delta\psi\|^2
  = \int_{\Omega} |\Delta\psi|^2
  &= \sum_{j,k=1}^d \int_{\Omega}
  \partial_j\partial_j\bar{\psi}
  \ \partial_k\partial_k\psi
  \\
  &= -\sum_{j,k=1}^d \int_{\Omega}
  \partial_j\bar{\psi}
  \ \partial_j\partial_k\partial_k\psi
  \\
  &= -\sum_{j,k=1}^d \int_{\Omega}
  \partial_k\partial_j\bar{\psi}
  \ \partial_j\partial_k\psi
  = \int_{\Omega} |\nabla^2\psi|^2
  = \|\nabla^2\psi\|^2
  \,.
\end{aligned}  
$$
Of course, this formal computation suffers 
from at least the following defects:
\begin{enumerate}
\item 
the boundary terms are disregarded,
\item
the third derivatives $\partial_j\partial_k\partial_k\psi$
are not defined,
\item
the mixed derivatives $\partial_j\partial_k\psi$ 
for $j \not= k$ are not defined.
\end{enumerate}
The tricks of elliptic regularity theory are precisely
about overcoming these problems.
First, the contact with the boundary~$\partial\Omega$ 
is solved by multiplying~$\psi$ by a cut-off function;
that is why only the local regularity is obtained 
in the full generality of arbitrary domains~$\Omega$. 
The other problems are solved by replacing
the customary partial derivative $\partial_k\psi$
by the \emph{difference quotient}
$$
  \partial_k^\delta\psi(x) 
  := \frac{\psi(x_1,\dots,x_k+\delta,\dots,x_d) - \psi(x)}{\delta}
$$
and taking the limit $\delta \to 0$ only after 
the manipulations in the spirit of the integration by parts above.  
We refer to \cite[Sec.~6.3]{Evans} for more details. 
 
It follows from~\eqref{elliptic} that any function~$\psi$
from $\dom(-\Delta_\alpha^\Omega)$ 
belongs to $W_\mathrm{loc}^{2,2}(\Omega)$.
If $d=1,2$, the \emph{Sobolev embedding theorem}
\cite[Thm.~4.12]{Adams2} then guarantees that~$\psi$ is continuous.
This is true for 
$
  \psi \in W_\mathrm{loc}^{1,2}(\Omega)
  \subset \dom(\delta_\alpha^\Omega)
$ if $d=1$,
but not for $d=2$.
The former can be seen by writing,
for all $x,y$ in any connected component of $\Omega \subset \Real$,
$$
  |\psi(x) - \psi(y)|
  = \left| \int_y^x \psi' \right|
  \leq \|\psi'\| \, |x-y|^{1/2} 
  \,,
$$
so~$\psi$ is actually H\"older continuous with exponent~$1/2$.
By developing this idea 
(why integrability of derivatives ensure smoothness),
one has the multidimensional Sobolev embedding
\begin{equation}\label{Sobolev}
  k > d/2
  \qquad \Longrightarrow \qquad
  W_\mathrm{loc}^{k,2}(\Omega) 
  \subset 
  C^{k-[\frac{d}{2}]-1}(\Omega)
  \,.
\end{equation}

Now, let~$\psi$ be an eigenfunction of $-\Delta_\alpha^\Omega$
corresponding to an eigenvalue~$\lambda$.
We know that 
$\psi \in W^{1,2}(\Omega) \cap W^{2,2}_\mathrm{loc}(\Omega)$.
The validity of the equation (in a weak sense, as usual)
\begin{equation}\label{ev.eq}
   -\Delta\psi = \lambda \psi
   \qquad \mbox{in} \qquad \Omega
\end{equation}
imply that $\Delta\psi \in W^{2,2}_\mathrm{loc}(\Omega)$.
We can thus differentiate~\eqref{ev.eq},
apply~\eqref{elliptic} and conclude with 
$\psi \in W^{3,2}_\mathrm{loc}(\Omega)$. 
Repeating this boot-strap procedure, we get that 
$\psi \in W^{k,2}_\mathrm{loc}(\Omega)$
for every $k \in \Nat$.
By applying~\eqref{Sobolev}, we therefore conclude 
with the following regularity result.

\begin{theo}\label{Thm.regularity}
Let $\psi$ be any eigenfunction of $-\Delta_\alpha^\Omega$,
where $\Omega \subset \Real^d$ is an arbitrary open set.
Then
$$
  \psi \in C^\infty(\Omega)
  \,.
$$
\end{theo}

Note carefully that we are not stating any result about
the boundary regularity of the eigenfunctions.
However, it is true that $\psi \in C^\infty(\overline{\Omega})$
provided that~$\Omega$ is sufficiently smooth
(\eg, $\Omega$~bounded with~$\partial\Omega$ of class~$C^\infty$). 
 
\subsection{Positivity of the ground state}
The Robin Laplacian $-\Delta_\alpha^\Omega$ 
is a \emph{real} operator in the sense that 
$
  \psi \in \dom(-\Delta_\alpha^\Omega)
$
implies  
$
  \bar\psi \in \dom(-\Delta_\alpha^\Omega)
$
and 
$
  \overline{-\Delta_\alpha^\Omega\psi} 
  = -\Delta_\alpha^\Omega\bar{\psi}
$
for every $\psi \in \dom(-\Delta_\alpha^\Omega)$. 
Moreover, $-\Delta_\alpha^\Omega$ is self-adjoint,
so its spectrum is real. 
It follows that any eigenfunction of $-\Delta_\alpha^\Omega$ 
can be chosen to be real-valued
(we have already used it in the proof of Theorem~\ref{Thm.FK}).

It is a highly non-trivial fact that if the lowest point
in the spectrum of $-\Delta_\alpha^\Omega$ is a discrete eigenvalue
(so-called \emph{ground-state eigenvalue}),
then it possesses an eigenfunction 
(so-called \emph{ground state})
which does not change sign.
Moreover, the eigenspace is one-dimensional. 

\begin{theo}[Positivity of the ground state]\label{Thm.positive}
Let $\Omega \subset \Real^d$ be an arbitrary domain. 
Let 
$$
  \lambda_1 :=
  \inf\sigma(-\Delta_\alpha^\Omega) 
  \in \sigma_\mathrm{disc}(-\Delta_\alpha^\Omega)
  \,.
$$
Then the eigenvalue~$\lambda_1$ 
is simple and the corresponding eigenfunction
can be chosen to be positive.
\end{theo}
\begin{proof}
Let~$\psi_1$ be a real eigenfunction
of $-\Delta_\alpha^\Omega$ corresponding to~$\lambda_1$.
By the minimax principle (Theorem~\ref{minimax}),
one has
\begin{equation}\label{Rayleigh.ground.bis}
  \lambda_1 = \inf_{\stackrel[\psi\not=0]{}{\psi \in W^{1,2}(\Omega)}} 
  \frac{\displaystyle \int_\Omega |\nabla\psi|^2
  + \alpha \int_{\partial\Omega} |\psi|^2}
  {\displaystyle \int_\Omega |\psi|^2}
  \,.
\end{equation}
(If $\alpha = \infty$, the infimum should be taken 
over $W_0^{1,2}(\Omega)$
and the boundary term should be disregarded.)
Since~$\lambda_1$ is an eigenvalue, 
the infimum is achieved and one also has
\begin{equation}\label{ground.achieved}
  \lambda_1 =  
  \frac{\displaystyle \int_\Omega |\nabla\psi_1|^2
  + \alpha \int_{\partial\Omega} |\psi|^2}
  {\displaystyle \int_\Omega |\psi_1|^2}
  \,.
\end{equation}
If a real-valued function~$\psi$ belongs to $W^{1,2}(\Omega)$, 
then also $|\psi| \in W^{1,2}(\Omega)$
and $|\nabla|\psi|| = |\nabla\psi|$
(the analogous claim holds for $W_0^{1,2}(\Omega)$).
Consequently, the absolute value~$|\psi_1|$
is an admissible trial function in~\eqref{Rayleigh.ground.bis}
and one has
$$
  \lambda_1 \leq   
  \frac{\displaystyle \int_\Omega \big|\nabla|\psi_1|\big|^2
  \alpha \int_{\partial\Omega} \big||\psi|\big|^2}
  {\displaystyle \int_\Omega \big||\psi_1|\big|^2}
  = \frac{\displaystyle \int_\Omega |\nabla\psi_1|^2
  \alpha \int_{\partial\Omega} |\psi|^2}
  {\displaystyle \int_\Omega |\psi_1|^2}
  = \lambda_1
  \,.
$$
It follows (\cf~Proposition~\ref{Prop.achieved})
that~$|\psi_1|$ is also a minimiser of~\eqref{Rayleigh.ground}.
Hence $|\psi_1|$ is a non-negative eigenfunction 
of $-\Delta_\alpha^\Omega$ corresponding to~$\lambda_1$. 

We claim that $|\psi_1|$ is positive in~$\Omega$.
By contradiction, let us assume that there exists 
a point $x_0 \in \Omega$ such that $|\psi_1|(x_0)=0$.
By Theorem~\ref{Thm.regularity},
we know that~$|\psi_1|$ 
is necessarily infinitely smooth inside~$\Omega$.
Since~$|\psi_1|$ solves the differential equation
(which can be considered in the classical sense)
$
  -\Delta |\psi_1| = \lambda_1 |\psi_1| 
  ,
$
we know that it satisfies the Harnack inequality 
(see, \eg, \cite[Thm.~8.20]{Gilbarg-Trudinger})
$$
  \sup_{U}|\psi_1| \leq C \, \inf_{U}|\psi_1|
  \,,
$$
where~$U$ is any bounded domain containing~$x_0$
and~$C$ is a positive constant depending on~$U$.
Consequently, $|\psi_1|=0$ identically in~$U$. 
Using the arbitrariness of~$U$,
it follows that $\psi_1=0$ identically in~$\Omega$,
a contradiction.
So~$\lambda_1$ admits the positive eigenfunction~$|\psi_1|$. 

By the argument above, we have actually shown more, 
namely that~$\psi_1$ is either positive or negative 
(because $|\psi_1| > 0$ in~$\Omega$). 
So it is impossible that there is another eigenfunction
from the same eigenspace which would be orthogonal 
to~$\psi_1$ in $\sii(\Omega)$.
Hence, the eigenvalue~$\lambda_1$ is simple. 
\end{proof}

\subsection{The nodal-line conjecture}
In this section, let us restrict to Dirichlet boundary conditions.
Furthermore, to ensure that the spectrum of $-\Delta_D^\Omega$
is purely discrete, let us assume that~$\Omega$ is a bounded domain.

Let us arrange the eigenvalues of $-\Delta_D^\Omega$
in a non-decreasing sequence 
$$
  \sigma(-\Delta_D^\Omega)
  = \sigma_\mathrm{disc}(-\Delta_D^\Omega)
  = \{\lambda_1^D(\Omega) < \lambda_2^D(\Omega) 
  \leq \lambda_3^D(\Omega) \leq \dots \to +\infty \}
  \,,
$$
where each eigenvalue is repeated according to its multiplicity.
Notice that $\lambda_1^D(\Omega)$ is simple  
due to Theorem~\ref{Thm.positive},
so the inequality between~$\lambda_1^D(\Omega)$ 
and~$\lambda_2^D(\Omega)$ is strict.
Let $\{\psi_n^D\}_{n=1}^\infty$ denote a corresponding set
of real eigenfunctions, normalised to~$1$ in $\sii(\Omega)$.
Since the eigenfunctions are mutually orthogonal
and~$\psi_1^D$ can be chosen to be positive 
due to Theorem~\ref{Thm.positive},
the other eigenfunctions are forced to change sign.
For $n \geq 2$,
it makes thus sense to introduce the \emph{nodal set} of~$\psi_n^D$
by setting
$$
  \mathcal{N}(\psi_n^D) := (\psi_n^D)^{-1}(0)
  = \{x \in \Omega: \ \psi_n^D(x) = 0 \}
  \,.
$$
The connected components of $\Omega \setminus \mathcal{N}(\psi_n^D)$
are called \emph{nodal domains} of~$\psi_n^D$. 
 
In dimension $d=1$, the situation is particularly simple.
Without loss of generality, we can consider $I_a := (0,a)$ with $a>0$.  
Recalling Section~\ref{Sec.strings}, 
the eigenfunctions are given 
by~\eqref{D.spec} and~\eqref{D.efs}, respectively.
Consequently,
$$
  \mathcal{N}(\psi_n^D) 
  = \left\{a \frac{k}{n} \right\}_{k=1}^{n-1}
$$
and each~$\psi_n^D$ has exactly~$n$ nodal domains,
see Figure~\ref{Fig.nodal.segment}.

\begin{figure}[h!]
\begin{center}
\begin{tabular}{cccc}
{\small $\lambda_1^D$} 
& {\small $\lambda_2^D$}   
& {\small $\lambda_3^D$}  
& {\small $\lambda_4^D$}  
\\
\includegraphics[width=0.22\textwidth]{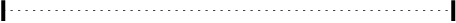}
& \includegraphics[width=0.22\textwidth]{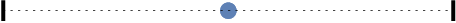}
& \includegraphics[width=0.22\textwidth]{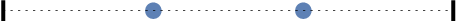}
& \includegraphics[width=0.22\textwidth]{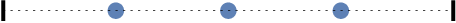}
\end{tabular}
\end{center}
\caption{Nodal sets (points) of the eigenfunctions in the segment
corresponding to the lowest four eigenvalues.
The $n^\text{th}$ eigenfunction has exactly $n$ nodal domains.}
\label{Fig.nodal.segment}
\end{figure}

In dimension $d=2$, the sets
$\mathcal{N}(\psi_n^D)$ are called \emph{nodal lines} 
and they form spectacular crossing curves or closed loops
(\emph{Chladni's patterns}),
see Figures~\ref{Fig.nodal.square} and~\ref{Fig.nodal.disk}
for the cases of squares 
(studied in Section~\ref{Sec.pipeds}) 
and disks, respectively.

\begin{figure}[h!t]
\begin{center}
\begin{tabular}{cccccc}
{\small $\lambda_1^D$} 
& {\small $\lambda_2^D$ ($= \lambda_3^D$)}  
& {\small $\lambda_3^D$ ($= \lambda_2^D$)}
& {\small $\lambda_4^D$}  
& {\small $\lambda_5^D$ ($= \lambda_6^D$)}
& {\small $\lambda_6^D$ ($= \lambda_5^D$)} 
\\
\includegraphics[width=0.135\textwidth]{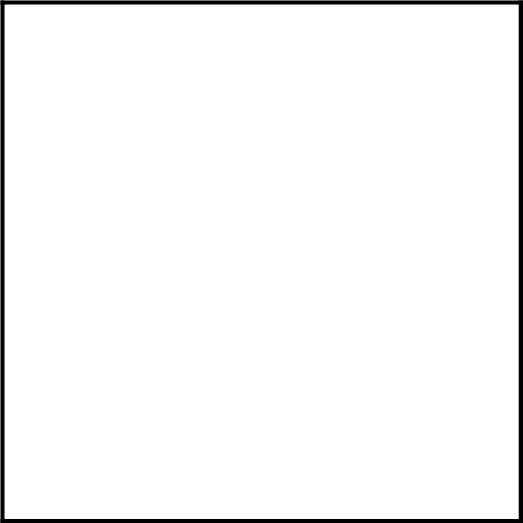}
&\includegraphics[width=0.135\textwidth]{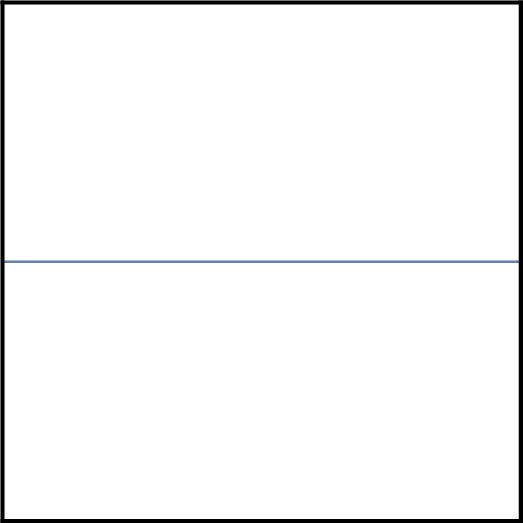}
&\includegraphics[width=0.135\textwidth]{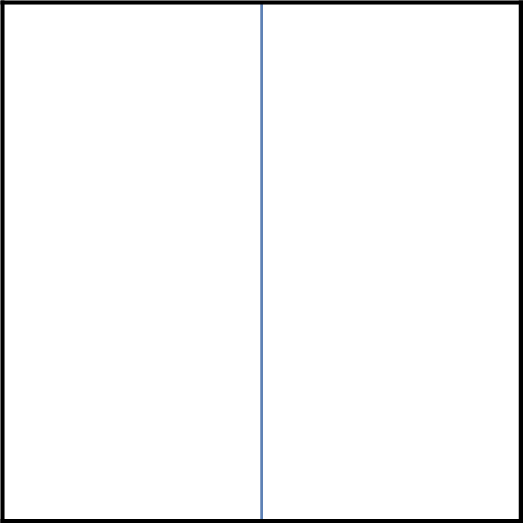}
&\includegraphics[width=0.135\textwidth]{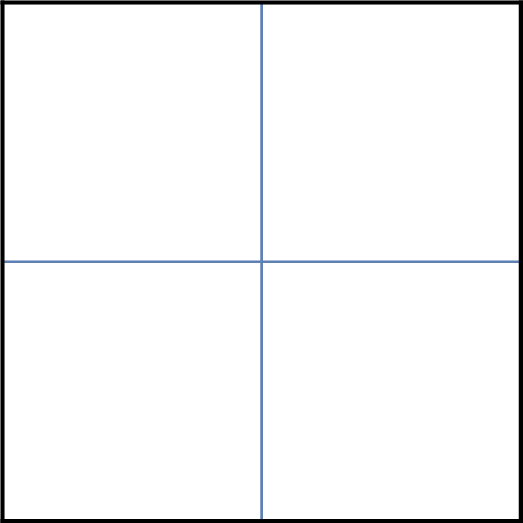}
&\includegraphics[width=0.135\textwidth]{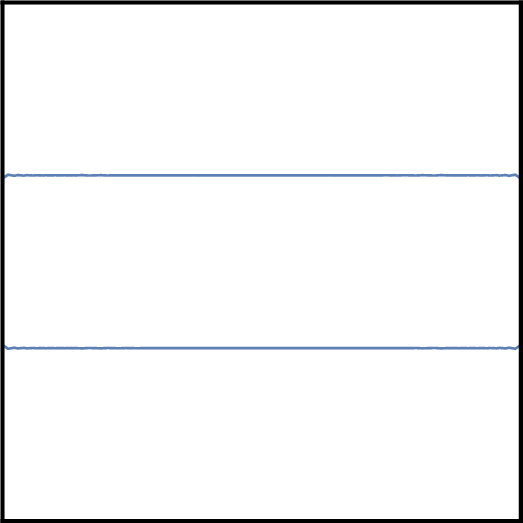}
&\includegraphics[width=0.135\textwidth]{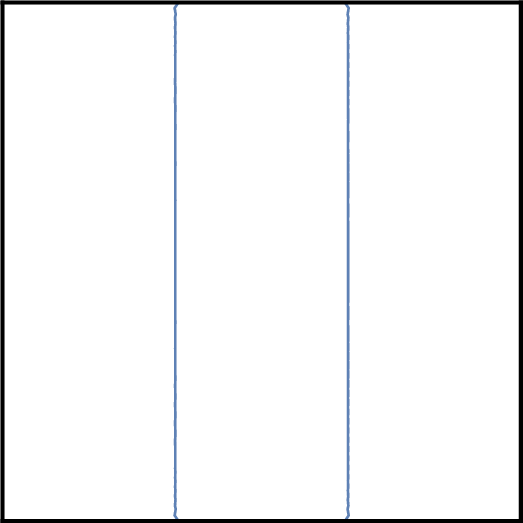}
\end{tabular}
\medskip \\
\begin{tabular}{ccc}
{\small $\lambda_2^D$ ($= \lambda_3^D$)}  
& {\small $\lambda_5^D$ ($= \lambda_6^D$)}
& {\small $\lambda_5^D$ ($= \lambda_6^D$)}
\\
\includegraphics[width=0.135\textwidth]{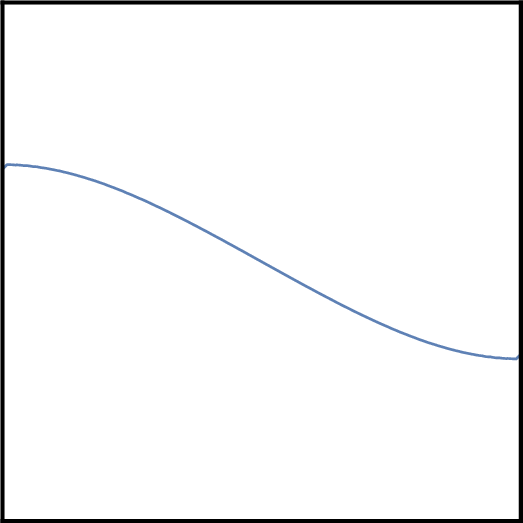}
&\includegraphics[width=0.135\textwidth]{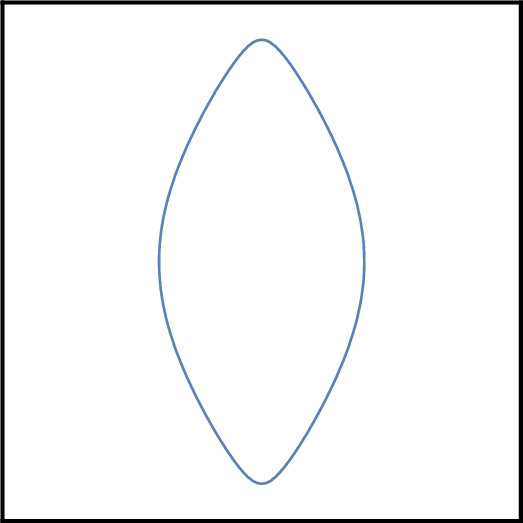}
&\includegraphics[width=0.135\textwidth]{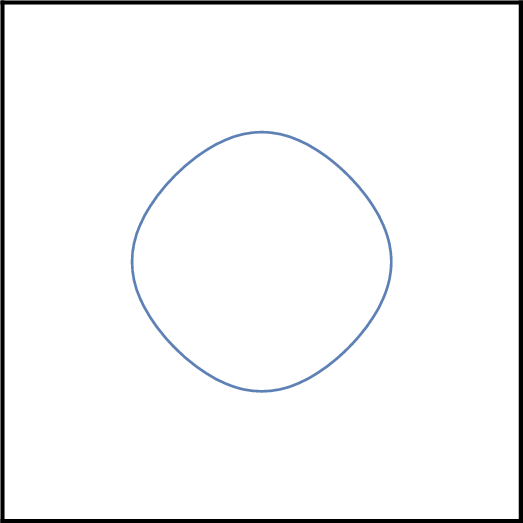}
\end{tabular}
\end{center}
\caption{Nodal lines of the eigenfunctions in the square
corresponding to the lowest six eigenvalues.
The first row depicts the nodal lines for the eigenfunctions
as obtained by a separation of variables.
The second row combines the eigenfunctions 
corresponding to multiple eigenvalues in a non-trivial way.
The fifth eigenfunction is the lowest eigenfunction
of the square which admits a closed nodal line.}
\label{Fig.nodal.square}
\end{figure}
\begin{figure}[h!t]
\begin{center}
\begin{tabular}{cccccc}
{\small $\lambda_1^D$} 
& {\small $\lambda_2^D$ ($= \lambda_3^D$)}  
& {\small $\lambda_3^D$ ($= \lambda_2^D$)}
& {\small $\lambda_4^D$ ($= \lambda_5^D$)}
& {\small $\lambda_5^D$ ($= \lambda_4^D$)}
& {\small $\lambda_6^D$}  
\\
\includegraphics[width=0.135\textwidth]{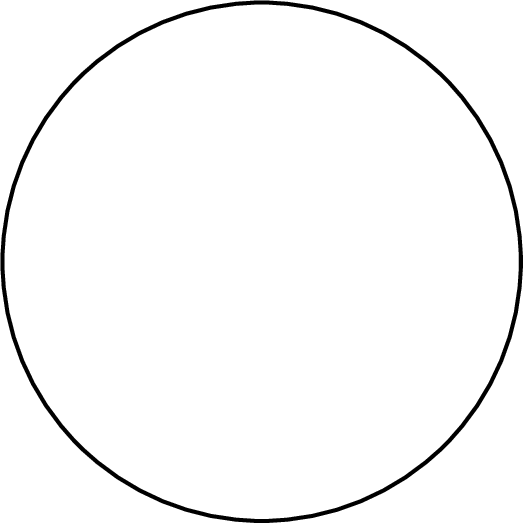}
&\includegraphics[width=0.135\textwidth]{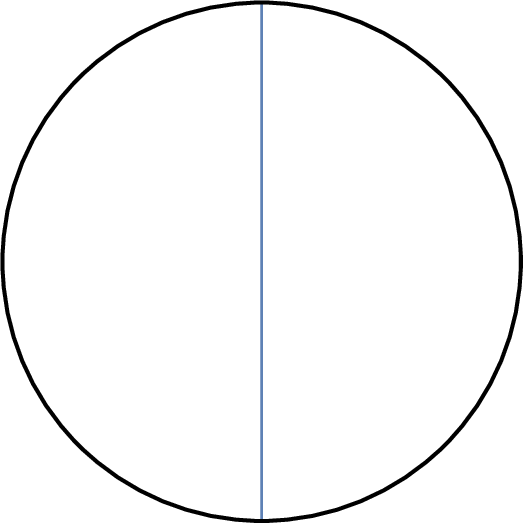}
&\includegraphics[width=0.135\textwidth]{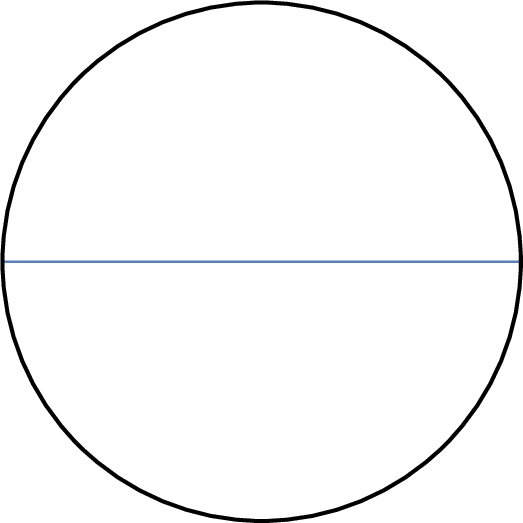}
&\includegraphics[width=0.135\textwidth]{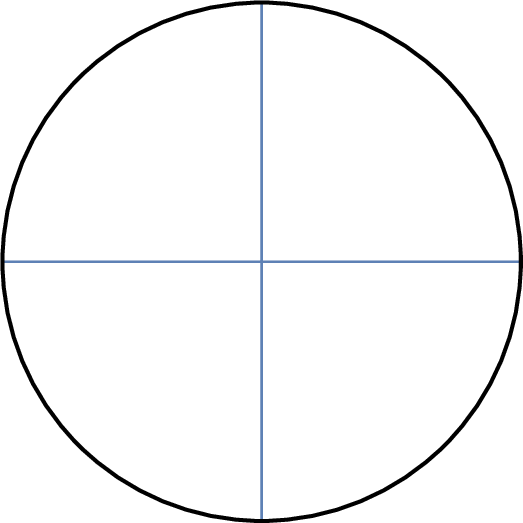}
&\includegraphics[width=0.135\textwidth]{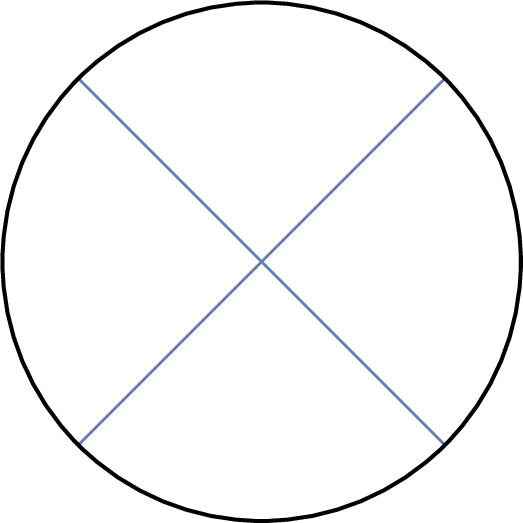}
&\includegraphics[width=0.135\textwidth]{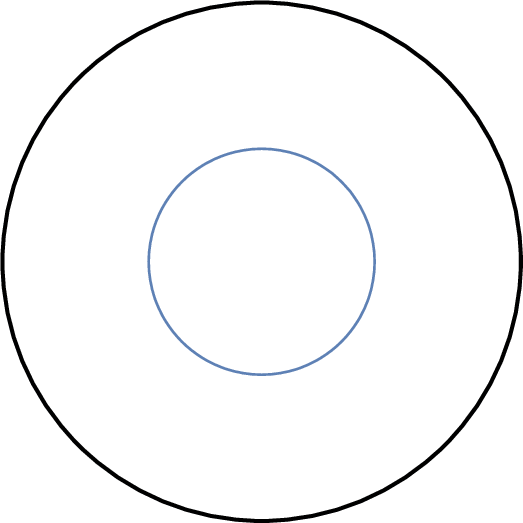}
\end{tabular}
\end{center}
\caption{Nodal lines of the eigenfunctions in the disk
corresponding to the lowest six eigenvalues.
The sixth eigenfunction is the lowest eigenfunction
of the disk which admits a closed nodal line.}
\label{Fig.nodal.disk}
\end{figure}

In particular, and this fact extends to dimensions $d \geq 2$,
it is no longer true that~$\psi_n^D$ has \emph{exactly}~$n$ nodal domains. 
Anyway, the examples suggest that~$\psi_n^D$ has \emph{at most}~$n$
nodal domains. 
This observation holds in the full generality.

\begin{theo}[Nodal domain theorem, 
Courant 1923 \cite{Courant_1923}]\label{Thm.Courant} 
Let $\Omega \subset \Real^d$ be an arbitrary bounded domain.
For every $n \geq 1$, 
$$
\mbox{
$\psi_n^D$ has at most~$n$ nodal domains. 
}
$$
\end{theo}
\begin{proof}
This theorem is originally announced 
and proved for $d=2$ in~\cite{Courant_1923}
(see also \cite[Sec.~VI.6]{CH1}).
In higher dimensions, not all the existing proofs 
meet the necessary mathematical rigour
(\cf~\cite[Rem.~6 in App.~D]{Berard-Meyer_1982}).
On the other hand, the general rigorous approach of
\cite{Alessandrini_1998} requires apparently 
unnecessary hypotheses about the regularity of~$\Omega$.
We rely on some arguments of 
\cite[App.~D]{Berard-Meyer_1982}
used to prove the theorem on manifolds.
 
The case $n=1$ is trivial due to the positivity of~$\psi_1^D$.
Henceforth we therefore assume $n \geq 2$.
Let $\Omega_1,\dots,\Omega_k$ be the nodal domains of~$\psi_n^D$.
Suppose, by contradiction, that $k > n$. 
We abbreviate $u_j := \psi_n^D \chi_{\Omega_j}$,
the restriction of~$\psi_n^D$ to its $j^\text{th}$ nodal domain.
Let us consider the function
$$
  \psi := \sum_{j=1}^{k-1} 
  c_j \, u_j 
  \,,
$$
where $c_1,\dots,c_{k-1} \in \Real$
are constants to be chosen later.

First of all, we claim that $\psi \in W_0^{1,2}(\Omega)$,
so it is an admissible test function in the variational 
characterisation of the eigenvalues of $-\Delta_D^\Omega$ 
given by Theorem~\ref{minimax}.
To verify it, we have to be careful 
when the boundary~$\partial\Omega$ is not sufficiently regular. 
Clearly, it is enough to check
that each $u_j \in W_0^{1,2}(\Omega)$.
Fix $j \in \{1,\dots,k\}$ 
and let us assume (without loss of generality)
that $u_j > 0$ in~$\Omega_j$.
For every positive~$\eps$, 
let us introduce the superlevel set 
$
  \Omega_j^\eps := \{x \in \Omega_j : u_j > \eps \}
$.
We set $u^\eps := u-\eps$ in~$\Omega_j^\eps$
and extend it by zero elsewhere.
Let $\{\eps_i\}_{i=1}^\infty$ be the set 
of regular values of~$u$ 
(\ie, $\forall x \in \Omega$,
$u(x)=\eps_i \Rightarrow \nabla u(x) \not= 0$)
which tend to zero as $i \to \infty$
(by Sard's theorem, critical values form a set of measure zero). 
Then $\Omega_j^{\eps_i}$ is smooth
and since~$u$ is smooth and 
$u^{\eps_i}=0$ on $\partial\Omega_j^{\eps_i}$,
it follows that
$
  u^{\eps_i} \in W_0^{1,2}(\Omega_j^{\eps_i})
  \subset W_0^{1,2}(\Omega_j)
$.
Since the volume $|\Omega_j \setminus \Omega_j^{\eps_i}|$
vanishes as $i \to \infty$,
it is straightforward to check that 
$u^{\eps_i} \to u_j$ in $W^{1,2}(\Omega_j)$ as $i \to \infty$.
Since $W_0^{1,2}(\Omega_j)$ is a closed subspace of $W^{1,2}(\Omega_j)$, 
we have just established that 
$u_j \in W_0^{1,2}(\Omega_j) \subset W_0^{1,2}(\Omega)$.

Second, we claim that it is possible to choose 
the constants $c_1,\dots,c_{k-1}$ in such a way that 
$c_1^2 + \dots + c_{k-1}^2 \not= 0$ 
(so that~$\psi$ is non-trivial)
and
$$
  \forall i = 1, \dots, n-1 
  \,, \qquad
  0 = (\psi_i^D,\psi) 
  = \sum_{j=1}^{k-1} c_j \int_{\Omega_j} \psi_i^D \psi_n^D
  \,.
$$
Indeed, it is enough to notice that we deal with 
a homogeneous system of linear equations,
where the number of unknowns is larger than 
the number of equations 
(because $k-1 > n-1$).

Since~$\psi$ is orthogonal to the first~$n-1$ eigenfunctions,
it follows that the linear span
$
  \mathcal{M}_{n}
  := \obal\{\psi_1^D,\dots,\psi_{n-1}^D,\psi\}
$ 
is an $n$-dimensional subspace of $W_0^{1,2}(\Omega)$.
Choosing 
$\mathscr{L}_{n} := \mathcal{M}_{n}$ in~\eqref{infsup},
it follows that 
$$
  \lambda_{n}^D(\Omega) 
  \leq \frac{\displaystyle \int_\Omega |\nabla\psi|^2}
  {\displaystyle \int_\Omega |\psi|^2}
  = \frac{\displaystyle \sum_{j=1}^{k-1} c_j 
  \int_{\Omega_j} |\nabla u_j|^2}
  {\displaystyle \sum_{j=1}^{k-1} c_j \int_{\Omega_j} |u_j|^2}
  = \frac{\displaystyle \sum_{j=1}^{k-1} c_j \,
  \lambda_n^D(\Omega) \int_{\Omega_j} |u_j|^2}
  {\displaystyle \sum_{j=1}^{k-1} c_j \int_{\Omega_j} |u_j|^2}
  = \lambda_n^D(\Omega) 
  .
$$  
Consequently (\cf~Proposition~\ref{Prop.achieved}), 
$\psi$~is an admissible eigenfunction
of $-\Delta_D^\Omega$ corresponding to the eigenvalue $\lambda_n^D(\Omega) $.
However, $\psi = 0$ on the non-empty open set~$\Omega_k$,
which implies that actually $\psi = 0$ on~$\Omega$
due to the strong maximum principle (or the Harnack inequality).
This is a contradiction with the fact that~$\psi$
has been constructed as a non-trivial function.  
\end{proof} 

We continue to assume that~$\Omega$ is a bounded domain,
but the hypotheses that there are two eigenvalues
below the essential spectrum would be enough 
for the intriguing conjecture below.
Since the ground state~$\psi_1^D$ is positive, 
the second eigenfunction~$\psi_2^D$ 
has at least two nodal domains.
It follows by Theorem~\ref{Thm.Courant}
that~$\psi_2^D$ has \emph{exactly two} nodal domains.

In 1967 Payne~\cite{Payne1} 
conjectured that the nodal line $\mathcal{N}(\psi_2^D)$
in dimension $d=2$ cannot be a closed curve
(in fact, he writes: \emph{This almost certainly must be the case\dots}).
There are some musical arguments to support this expectation.
A reformulation suitable also for higher dimensions
is that the nodal set touches the boundary.
\begin{conj}[Nodal-line conjecture]\label{Conj.nodal}
For any domain~$\Omega$, \
$$  
  \overline{\mathcal{N}(\psi_2^D)} \cap \partial\Omega
  \not= \varnothing
  \,.
$$
\end{conj}

As already mentioned, 
a variant of the conjecture was raised by Payne \cite{Payne1} 
for planar domains in 1967.
In 1973,
Payne~\cite{Payne2} proved it under some symmetry assumptions
about~$\Omega$.
(For other results under symmetry, 
see \cite{Lin_1987,Putter,Damascelli}.)
In 1991,
Jerison~\cite{Jerison0} announced a proof of the conjecture
for long thin convex planar domains 
(see also \cite{Jerison1,GJerison,
Grieser-Jerison_1998,Grieser-Jerison_2009,
Beck_2018,Beck-Canzani-Marzuola_2021}).
The method extends to higher dimensions~\cite{Jerison}.
In 1992,
Melas~\cite{Melas} established the conjecture for all convex planar domains
(see also \cite{Alessandrini_1994}).
In 1993,
Yau~\cite{Yau_1993} extended the conjecture to higher dimensions
and manifolds.
In 2008,
Freitas and Krej\v{c}i\v{r}\'ik \cite{FK4} proved the conjecture
for thin (possibly non-convex)
tubular neighbourhoods of curves in arbitrary dimensions
and thin strips on surfaces. 
Moreover, an estimate on the location of the nodal set was given
(in fact, for all eigenfunctions), see Figure~\ref{Fig.nodal}.
See also~\cite{KT2} for an extension to tubular neighbourhoods
of hypersurfaces. 

\begin{figure}[h!]
\begin{center}
\includegraphics[width=0.7\textwidth]{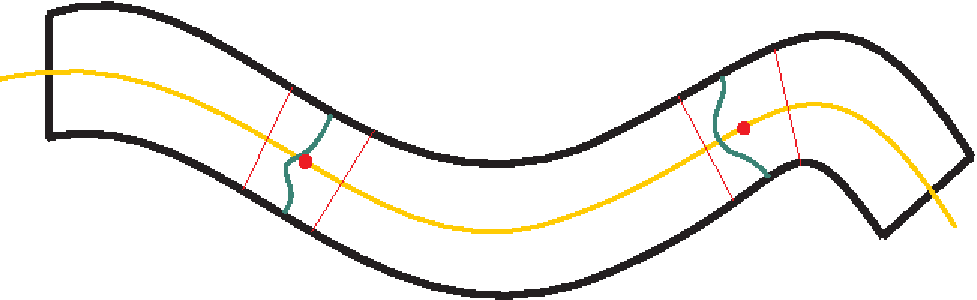}
\end{center}
\caption{The geometric setting of~\cite{FK4} for $d=2$.
The nodal lines (green) are located close to the nodal points (red)
of a one-dimensional Schr\"odinger operator on the base curve (yellow)
with a geometric potential.}
\label{Fig.nodal}
\end{figure}

On the negative side, in 1997,
Hoffmann-Ostenhofs and Nadirashvili~\cite{H2ON}
gave a counterexample for multiply connected planar domains,
see Figure~\ref{Fig.nodal.H2ON}. 
(See~\cite{Fournais} for an extension to higher dimensions.)
In 2002,
Freitas~\cite{Freitas_2002} showed that the conjecture 
does not necessarily hold for domains on surfaces.
In 2007,
Freitas and Krej\v{c}i\v{r}\'ik \cite{FK2} 
gave a counterexample for unbounded domains,
see Figure~\ref{Fig.nodal.unbounded}.
In 2013,
Kennedy~\cite{Kennedy_2013} gave a simply connected counterexample 
in dimensions $d \geq 3$.
As the last development, in 2025, Freitas and Leylekian~\cite{Freitas-Leylekian} 
constructed a planar merely doubly connected example with a closed nodal line,
see Figure~\ref{Fig.FL}.

\begin{figure}[h!]
\begin{center}
\includegraphics[width=0.25\textwidth]{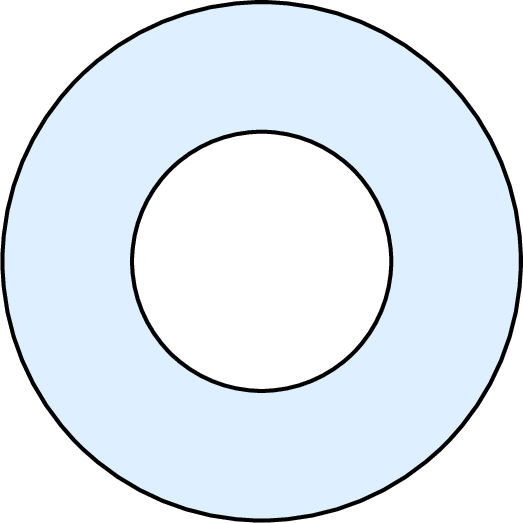}
\quad
\includegraphics[width=0.25\textwidth]{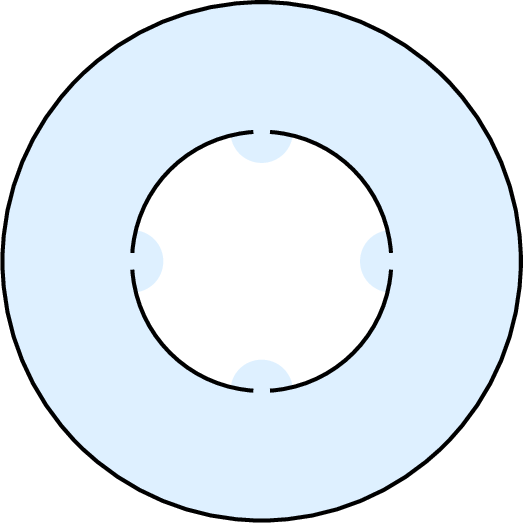}
\quad
\includegraphics[width=0.25\textwidth]{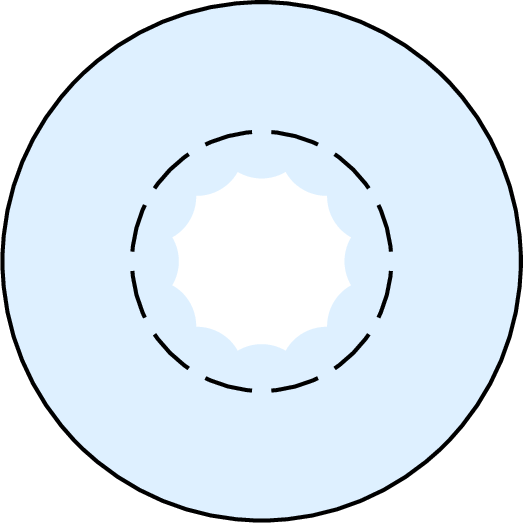}
\end{center}
\caption{The multiply connected domain of \cite{H2ON} 
for which the nodal line of the second eigenfunction is a closed curve.
The idea is to start with two concentric circles;
the union of the interior of the inner circle (disk)
and the exterior of the inner circle lying inside
the outer circle (annulus) forms a disconnected open set.
Assuming that the annulus is such that its 
first eigenvalue is greater than the first eigenvalue
of the inner disk and simultaneously smaller than
the second eigenvalue of the inner disk,
it follows that the second eigenfunction of
the disconnected domain is the first eigenfunction
in the annulus extended by zero to the inner disk.
Digging a sufficient number of small holes in the inner circle
will make the set connected and the eigenfunction
of the annulus will penetrate in the interior of the inner circle 
in such a way that it produces a closed nodal line eventually.
(Numerically, 6 holes are sufficient
in a similar geometric setting~\cite{Dahne-Gomez-Serrano-Hou_2021}.)}
\label{Fig.nodal.H2ON}
\end{figure}
\begin{OProblem}
Prove Conjecture~\ref{Conj.nodal} for all bounded simply connected 
domains in~$\Real^2$ 
and all bounded convex domains in~$\Real^d$ with $d \geq 3$. 
\end{OProblem}
\begin{figure}[h!t]
\begin{center}
\includegraphics[width=0.45\textwidth]{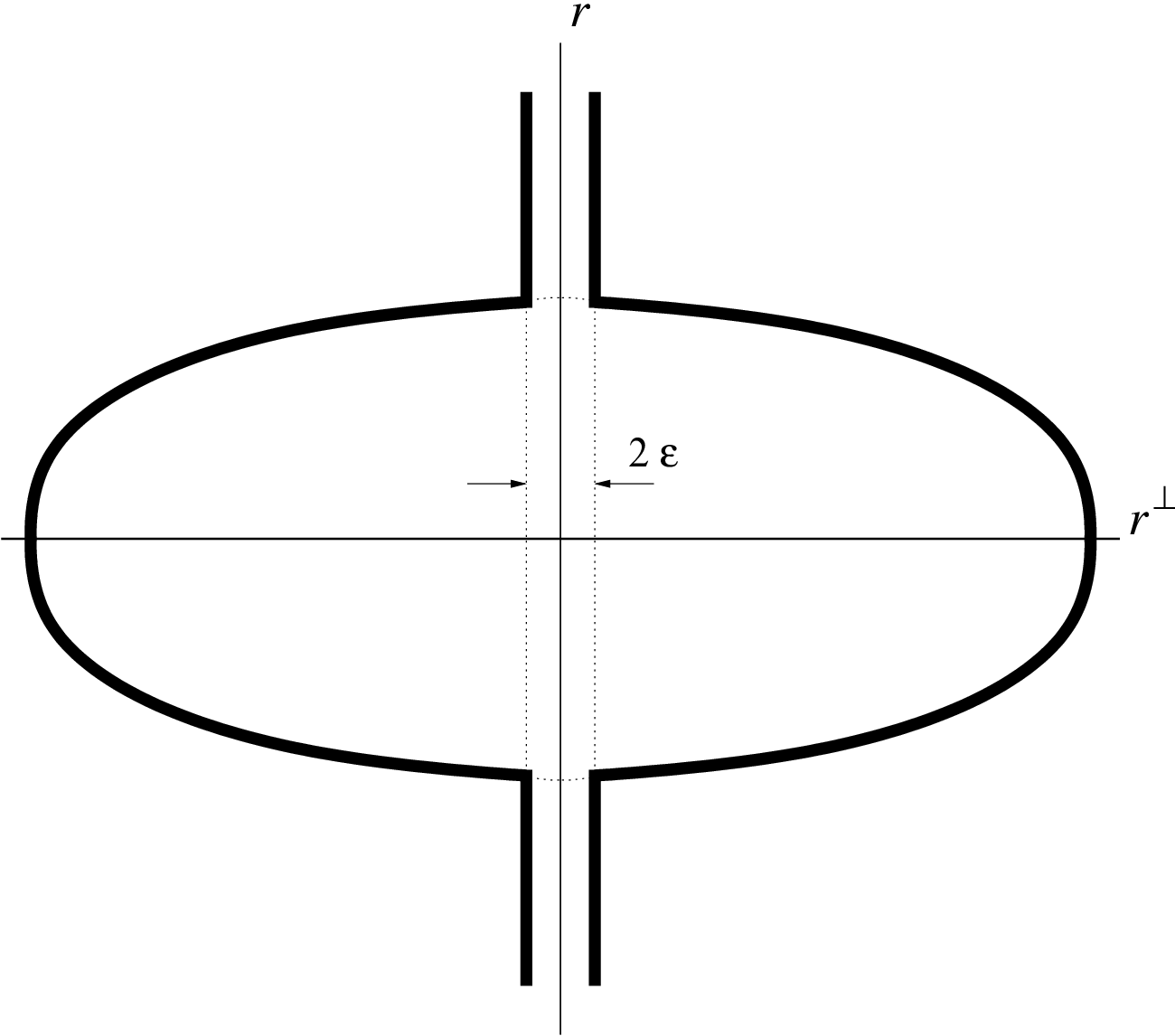}
\qquad
\includegraphics[width=0.45\textwidth]{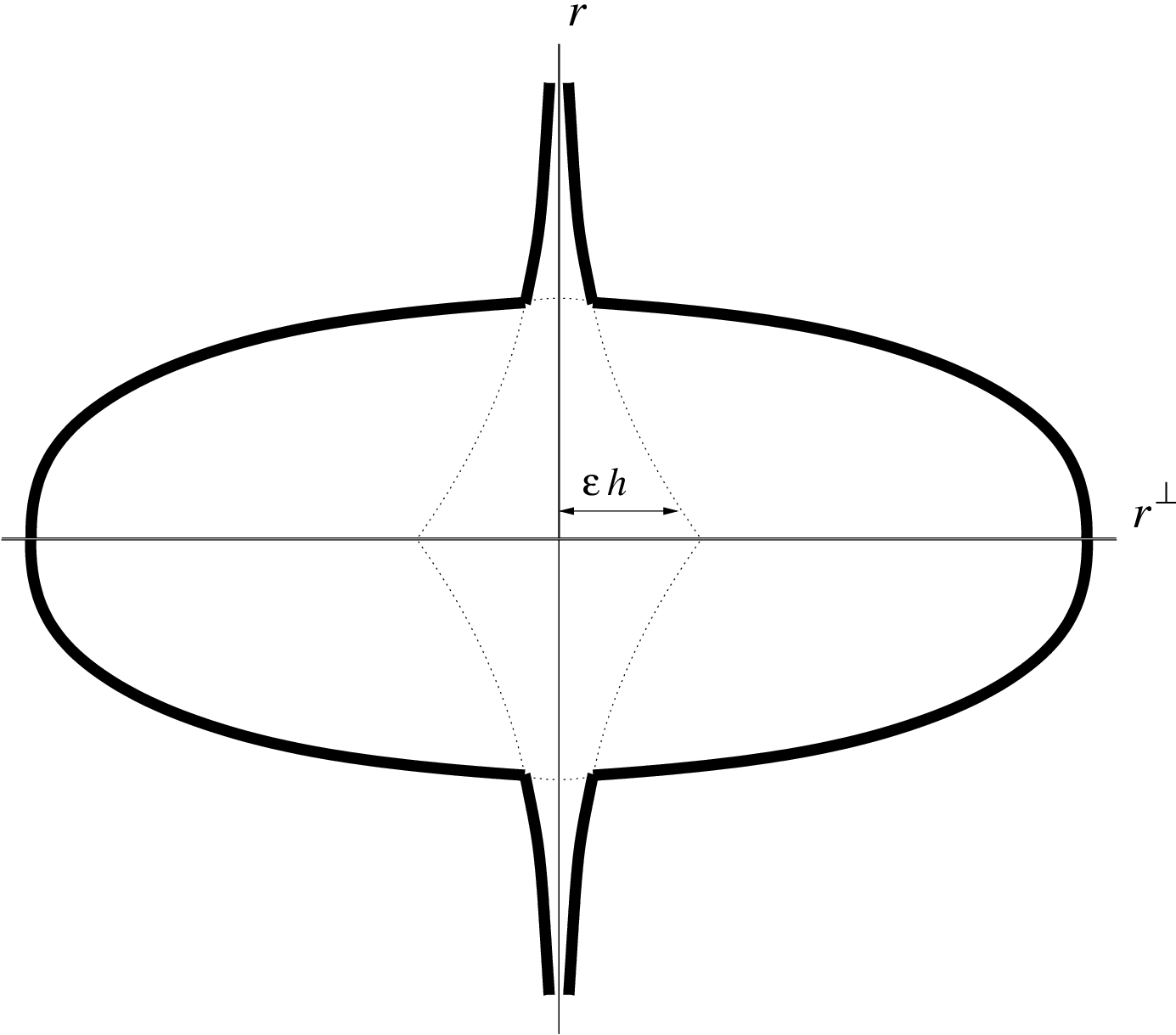}
\end{center}
\caption{Unbounded domains of~\cite{FK2}
for which the nodal line 
of the second eigenfunction does not touch the boundary.
The construction starts with a bounded convex domain 
which is invariant under reflections 
through two orthogonal lines~$r$ and~$r^\bot$.
Assuming that the bounded domain is sufficiently long
in the direction~$r^\bot$, its second eigenvalue is simple 
and has a nodal line coinciding with the axis~$r$
inside the domain.
Appending two sufficiently thin semi-infinite strips 
invariant under a reflection through~$r$
will make the whole vertical line~$r$ to be the nodal line
of the obtained unbounded domain.
The left (respectively, right) figure represents
a quasi-cylindrical (respectively, quasi-bounded) realisation.}
\label{Fig.nodal.unbounded}
\end{figure}
\begin{figure}[h!t]
\begin{center}
\includegraphics[height=2.3cm]{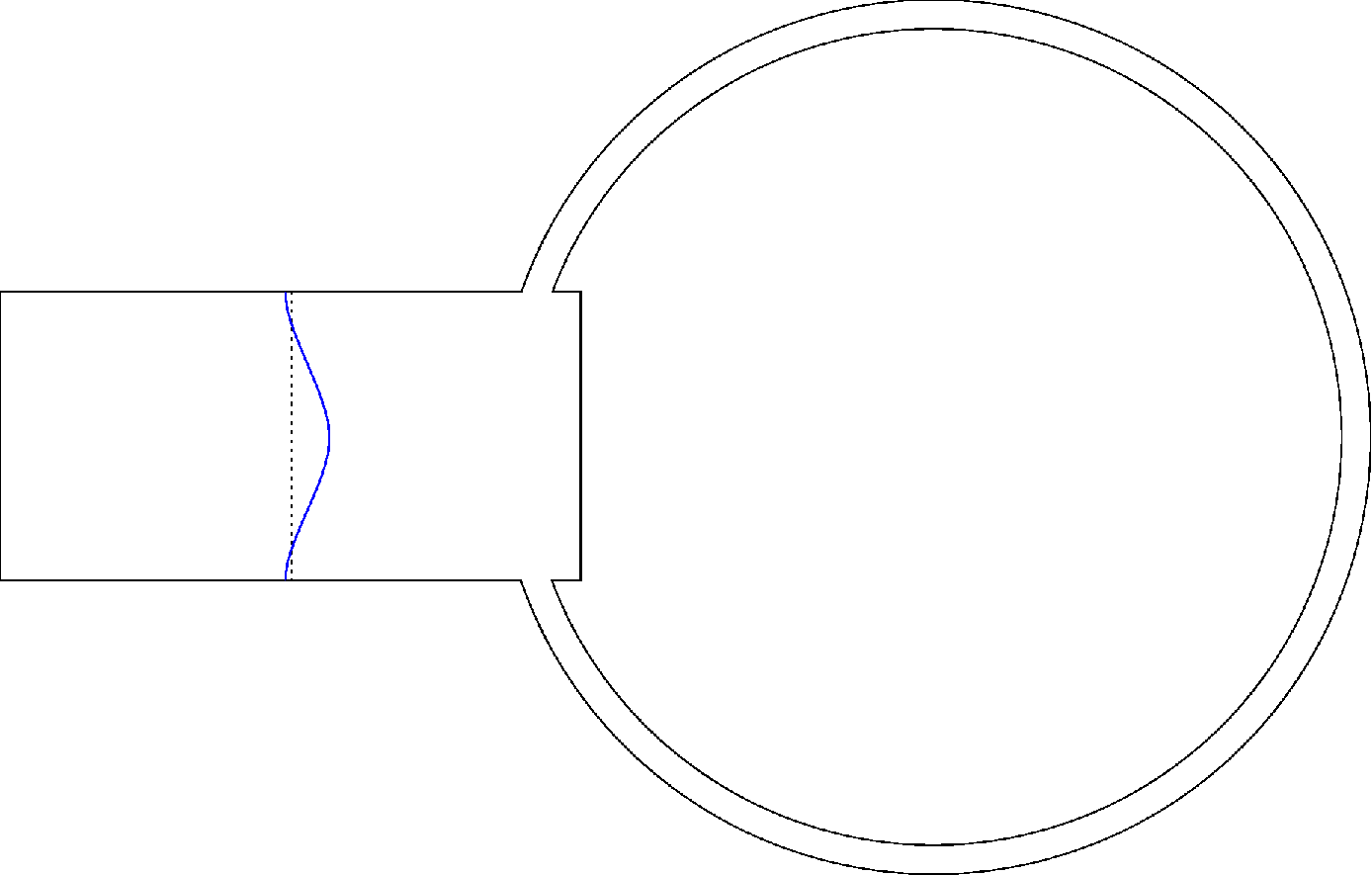} \quad
\includegraphics[height=2.3cm]{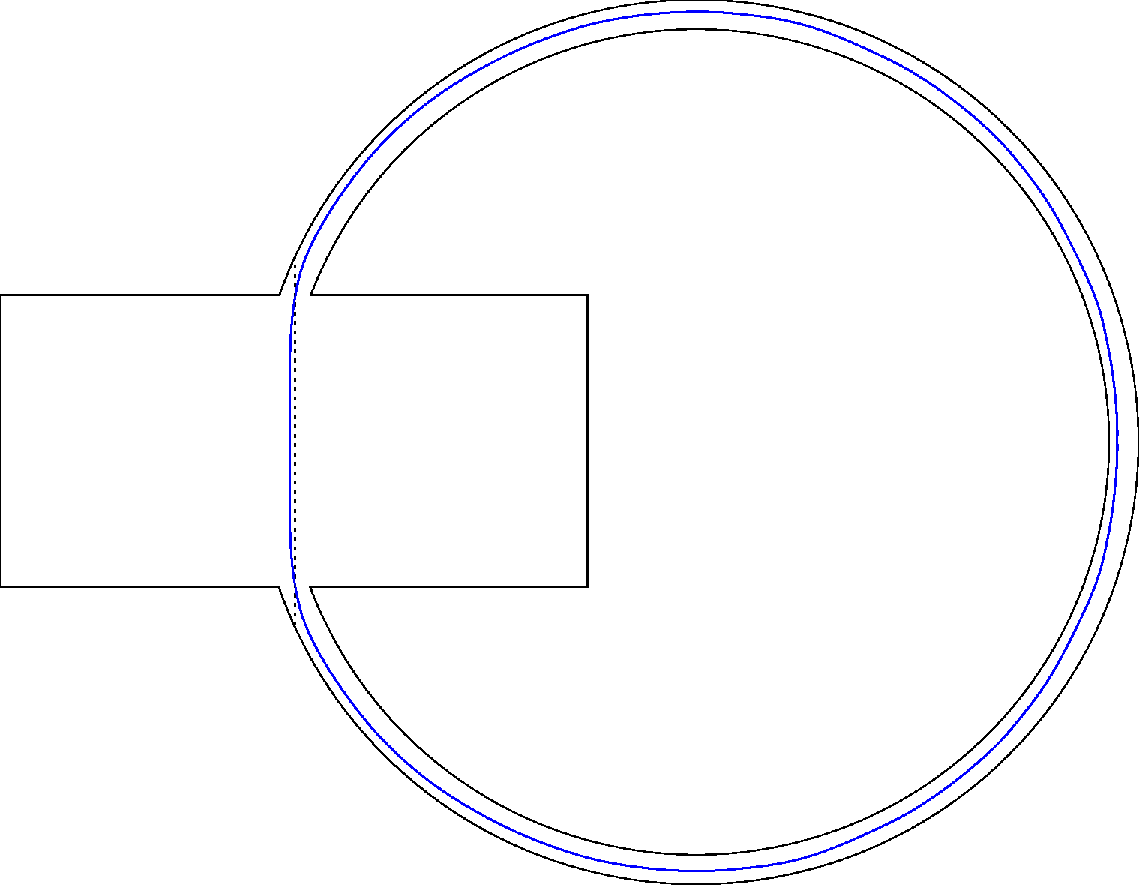} \quad
\includegraphics[height=2.3cm]{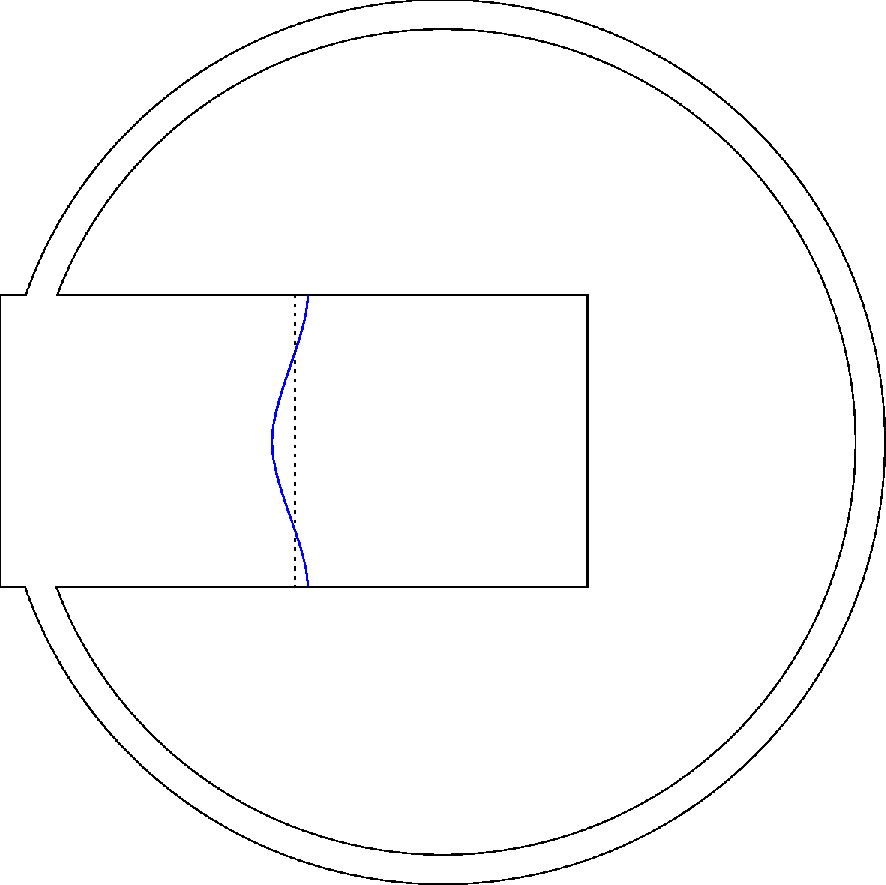}
\end{center}
\caption{The doubly connected domain of~\cite{Freitas-Leylekian}
admitting a closed nodal line. 
The idea is reminiscent of~\cite{FK2}, 
but now a thin annulus-like appendix is used 
instead of semi-infinite strips 
as an attachment to the symmetric bounded convex domain.
A continuity argument ensures that there is a location 
of the bounded domain for which the nodal line remains trapped 
inside the constructed domain (the middle realisation in the figure).}
\label{Fig.FL}
\end{figure}

\newpage
 
\subsection{The hot-spots conjecture}
In this section, let us restrict to Neumann boundary conditions.
Furthermore, to ensure that the spectrum of $-\Delta_N^\Omega$
is purely discrete, let us assume that~$\Omega$ is 
a sufficiently regular (\eg, Lipschitz) bounded domain. 
 
Let us arrange the eigenvalues of $-\Delta_N^\Omega$
in a non-decreasing sequence 
$$
  \sigma(-\Delta_N^\Omega)
  = \sigma_\mathrm{disc}(-\Delta_N^\Omega)
  = \{ 0 = \lambda_1^N(\Omega) < \lambda_2^N(\Omega) 
  \leq \lambda_3^N(\Omega) \leq \dots \to +\infty \}
  \,,
$$
where each eigenvalue is repeated according to its multiplicity.
Notice that $\lambda_1^N(\Omega)$ is simple due to Theorem~\ref{Thm.positive},
so the inequality between~$\lambda_1^N(\Omega)$ 
and~$\lambda_2^N(\Omega)$ is strict again.
Let $\{\psi_n^N\}_{n=1}^\infty$ denote a corresponding set
of real eigenfunctions, normalised to one in $\sii(\Omega)$.
Since the eigenfunctions are mutually orthogonal
and~$\psi_1^N$ can be chosen to be positive 
due to Theorem~\ref{Thm.positive} (in fact, it is a constant),
the other eigenfunctions are forced to change sign. 

In 1974 Rauch conjectured
that any eigenfunction corresponding to
the second eigenvalue~$\lambda_2^N(\Omega)$  
attains its maximum and minimum at boundary points only. 

\begin{conj}[Hot-spots conjecture]
\label{Conj.Rauch}
For any eigenfunction~$\psi_2^N$ 
corresponding to the second eigenvalue of $-\Delta_N^\Omega$, one has
$$
  \forall x \in \Omega, \quad
  \min_{\partial\Omega} \psi_2^N 
  < \psi_2^N(x) < 
  \max_{\partial\Omega} \psi_2^N
  \,.
$$
\end{conj}

According to Ba\~nuelos and Burdzy~\cite{Banuelos-Burdzy_1999}, 
Conjecture~\ref{Conj.Rauch} was indeed raised by Rauch 
during a conference in 1974 but 
``it has never appeared in print under his name''. 
It is certainly true that Conjecture~\ref{Conj.Rauch}
does not appear in Rauch's conference report~\cite{Rauch_1975},
where however a strong heuristic support 
for the validity of it can be deduced from: 
\begin{quote}
\emph{Since the first eigenvalue is zero, the large-time behaviour
of the heat semigroup generated by the Neumann Laplacian
is determined by the second eigenfunction 
and Neumann boundary conditions model an insulating interface,
it is expected that ``hot spots'' and ``cold spots'' 
of a medium living inside~$\Omega$ will move towards 
the boundary~$\partial\Omega$ for large times.}
\end{quote}

In dimension $d=1$,  
the situation is particularly simple.
Without loss of generality, we can consider $I_a := (0,a)$ with $a>0$.  
Recalling Section~\ref{Sec.strings}, 
the eigenfunctions are given 
by~\eqref{N.spec.point} and~\eqref{N.efs}, respectively.
In particular, the (unique) eigenfunction~$\psi_2^N$
corresponding to $\lambda_2^N((0,a))$ indeed 
attains its maximum and minimum at the boundary,
see Figure~\ref{Fig.hot}. 

\begin{figure}[h!t]
\begin{center}%
\begin{tabular}{ccc}%
{\small $\psi_2^N$} & {\small $\psi_3^N$} & {\small $\psi_4^N$}%
\\
\includegraphics[width=0.3\textwidth]{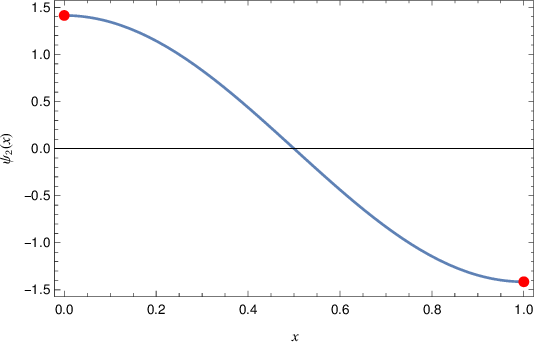}
&\includegraphics[width=0.3\textwidth]{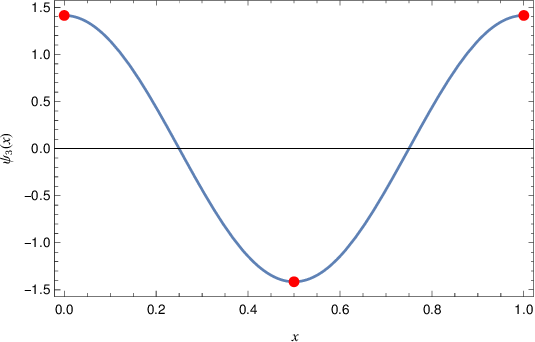}
&\includegraphics[width=0.3\textwidth]{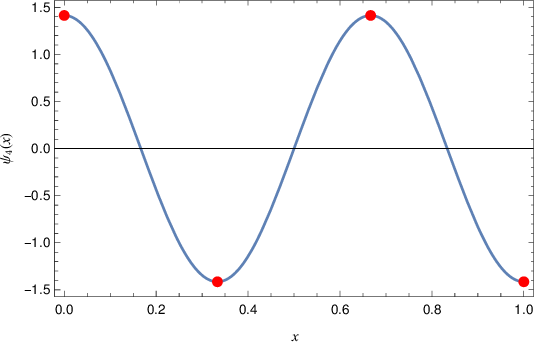}%
\end{tabular}%
\end{center}%
\caption{Eigenfunctions of $-\Delta_N^{(0,1)}$.
The hot and cold spots (red dots)
of~$\psi_2^N$ are clearly located 
on the boundary of the string $(0,1)$.}\label{Fig.hot}
\end{figure}

A pre-millennium history of the conjecture is as follows. 
As already mentioned, the conjecture is attributed 
to Rauch and the year 1974. 
In 1985, 
Kawohl \cite[Sec.~II.5]{Kawohl_1985} established the conjecture
for Cartesian products $(0,a) \times \omega$,
where $\omega \subset \Real^{d-1}$ is another bounded Lipschitz domain.
In 1999,
Ba\~nuelos and Burdzy~\cite{Banuelos-Burdzy_1999}
proved the conjecture for various special domains,
including obtuse triangles and long symmetric convex domains. 
(For other results under symmetry, 
see \cite{Jerison-Nadirashvili_2000,Pascu_2002}.) 

On the negative side, in 1999,
Burdzy and Werner \cite{Burdzy-Werner_1999}
gave a counterexample for multiply connected domains 
(see also \cite{Bass-Burdzy_2000,Burdzy_2005}).
In 2002,
Freitas~\cite{Freitas_2002} showed that the conjecture 
does not necessarily hold for domains on surfaces.
In 2024,
De Dios Pont \cite{deDios}
constructed counterexamples of convex domains
in all Euclidean spaces of sufficiently large dimensions.

New positive developments have been achieved after the turn of the millennium.   
In 2004,
Atar and Burdzy \cite{Atar-Burdzy_2004}
showed that the conjecture holds for \emph{lip domains},
see Figure~\ref{Fig.lips}. 
(For other results under small parameter, 
see \cite{Miyamoto_2009,Miyamoto_2013,Siudeja_2015}.) 
In 2019,
Krej\v{c}i\v{r}\'ik and Tu\v{s}ek \cite{KT3} proved the conjecture
for thin strips on surfaces. 
Moreover, we located the maxima and minima of 
all the eigenfunctions in the thin strips. 
In 2020,
Judge and Mondal \cite{Judge-Mondal_2020,Judge-Mondal_2022} 
established the conjecture for general triangles.
In 2023,
Rohleder~\cite{Rohleder1} (see also~\cite{Rohleder2})
gave an analytic proof 
of the probabilistic result~\cite{Atar-Burdzy_2004}.
In 2024,
Kennedy and Rohleder~\cite{Kennedy-Rohleder_2024}
established the conjecture for higher-dimensional 
analogues of the planar setting considered in~\cite{Atar-Burdzy_2004}.
In 2025, Hatcher~\cite{Hatcher_2025} 
proved the conjecture for some non-convex polygons, 
including L- and Swiss-cross-shaped domains 
(see also \cite{Hatcher_2025a,Hatcher_2025b}
for his recent works on manifolds).

\begin{figure}[h!]
\begin{center}
\includegraphics[width=0.97\textwidth]{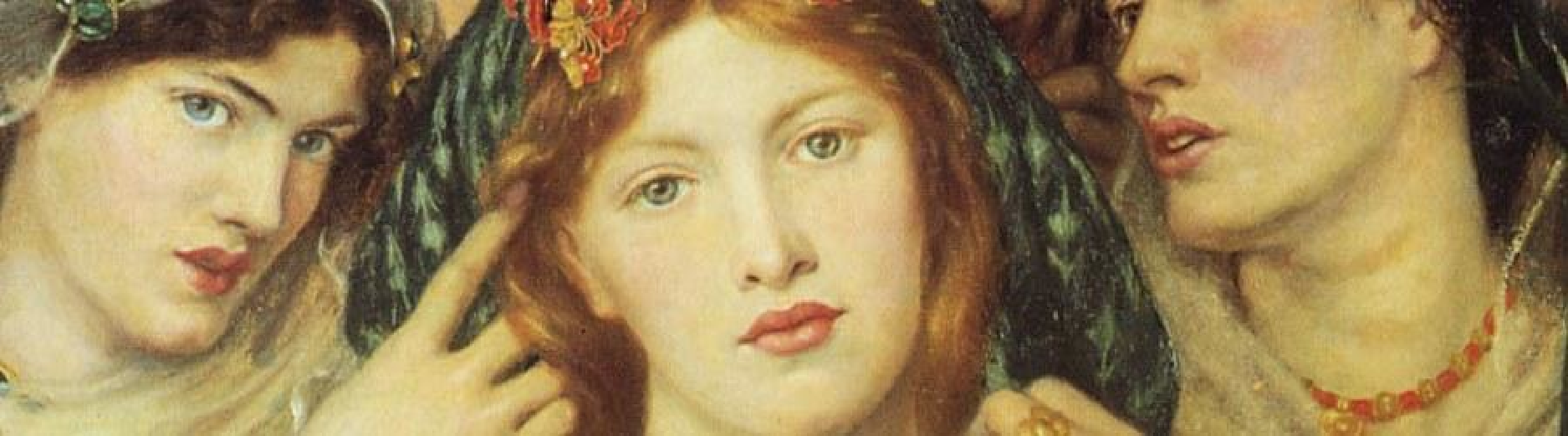}
\end{center}
\caption{In~\cite{Atar-Burdzy_2004}, Conjecture~\ref{Conj.Rauch}
is proved for \emph{lip domains}, 
\ie, planar domains  delimited by graphs of two Lipschitz functions 
with Lipschitz constant equal to one.
(Dante Gabriel Rossetti: \emph{The Beloved})}
\label{Fig.lips}
\end{figure}

\section{Inverse spectral geometry}
To conclude the discussion of bounded domains,
let us look at an inverse problem 
stated by this famous question 
of Kac's from 1966 \cite{Kac_1966}:
\begin{center}
\fbox{Can one hear the shape of a drum ?}
\end{center}
More specifically, 
given a sequence of Dirichlet eigenvalues 
$\{\lambda_n^D(\Omega)\}_{n=1}^\infty$,
can one reconstruct the geometry of the domain~$\Omega$
on which the Dirichlet Laplacian $-\Delta_D^\Omega$ acts?
In particular, do there exist isospectral 
(having the same spectrum)
incongruent (having different shapes) domains?  
 
Let us investigate this problem inductively
by starting with specific geometries.

\subsection{Can one hear the shape of a string?}
The simplest non-trivial class of domains is given by 
one-dimensional intervals $I_a := (0,a)$ with $a>0$
of Section~\ref{Sec.strings}. 
Here the geometry of~$I_a$ is determined by the length~$a$ of the string.
Since the spectrum of $-\Delta_D^{I_a}$ 
is given by~\eqref{D.spec}, 
the knowledge of the lowest eigenvalue 
$\lambda_1^D(I_a) = \left(\frac{\pi}{a}\right)^2$ is sufficient 
to determine the geometry of~$I_a$.
In summary, the answer is YES. 
 
\subsection{Can one hear the shape of a rectangular drum?}
Now, let us consider the Cartesian product of two strings,
\ie, a rectangle $I_{a,b} := (0,a) \times (0,b)$
with positive~$a$ and~$b$ of Section~\ref{Sec.pipeds}.
Without loss of generality, let us assume $a \leq b$.
Here the geometry is determined by the side lengths~$a$ and~$b$. 
From the analysis in Section~\ref{Sec.pipeds}, we know that 
$$
  \sigma(-\Delta_D^{I_{a,b}}) = \left\{
  \left(\frac{n\pi}{a}\right)^2 + \left(\frac{m\pi}{b}\right)^2
  \right\}_{n,m=1}^\infty
  \,.
$$
Consequently, the knowledge of the two lowest eigenvalues 
$
  \lambda_1^D(I_{a,b}) := 
  \left(\frac{\pi}{a}\right)^2 + \left(\frac{\pi}{b}\right)^2
$ 
and 
$
  \lambda_2^D(I_{a,b}) := 
  \left(\frac{\pi}{a}\right)^2 + \left(\frac{2\pi}{b}\right)^2
$ 
is sufficient to determine the geometry of~$I_{a,b}$
(for squares, the knowledge of the lowest eigenvalue is sufficient).
In summary, the answer is YES.

In general, we see that the knowledge of the lowest eigenvalue
is not sufficient to determine the geometry.
In fact, it is not sufficient to determine 
even the dimension of the domain
(there exist $a,b,c > 0$ such that 
$\lambda_1^D(I_{a,b}) = \lambda_1^D(I_{c})$). 
However, in the case of rectangles, 
the lowest eigenvalue is sufficient to determine 
the geometry of~$I_{a,b}$ 
if we restrict to membranes of fixed area.

In higher dimensions $d \geq 3$, 
it is clear that the first~$d$ eigenvalues 
are sufficient to determine the geometry of the rectangular box
$I_{a_1,\dots,a_d}$. 
 
\subsection{Can one hear the area of a drum?}
Before answering the general question,
let us consider the more modest problem 
whether one can hear the area of a drum. 
Here the answer is YES
because of the following robust result.
Indeed, the leading-order term in the asymptotics 
of the eigenvalues at infinity contains the area~$|\Omega|$. 

\begin{theo}[Weyl's law, 1911 \cite{Weyl_1911}]\label{Thm.Weyl.law}
Let $\Omega \subset \Real^d$ be an arbitrary bounded open set. 
Then
\begin{equation}\label{Weyl.law}
  \lambda_n^D(\Omega)^{d/2} 
  = \frac{(2\pi)^d}{|B_1|} \frac{n}{|\Omega|}
  + o(n)
\end{equation}
as $n \to \infty$.
\end{theo}
\begin{proof} 
We are inspired by \cite[Sec.~XIII.15]{RS4}.

\fbox{Heuristic proof}
First of all, 
let us give a ``physical proof''
based on a quantum-mechanical intuition. 
Let us consider the operator
$$
  H := -\frac{\hbar^2}{2m} \Delta_D^\Omega
$$
representing the quantum Hamiltonian of 
a free particle of mass~$m$
constrained to the open set $\Omega \subset \Real^d$. 
Here $\hbar := h / (2\pi)$ denotes the reduced Planck constant,
where~$h$ is the Planck constant.
The eigenvalues~$\lambda$ of~$H$ represent energies of the particle.  
From the eigenvalue equation $H\psi = \lambda\psi$,
the following relationship is obvious:
$$
  \mbox{large-energy limit $\lambda \to \infty$}
  \qquad \Longleftrightarrow \qquad
  \mbox{semiclassical limit $h \to 0$}.
$$
In this regime, the Bohr--Sommerfeld quantisation condition states
that each quantum bound state is associated with a volume $h^d$ 
in the phase space:
$$
  \mbox{1 bound state}
  \quad \sim \quad
  \mbox{volume $h^d$ in the phase space}.
$$
Consequently,
$$
\begin{aligned}
  N_D^\Omega(\lambda) 
  :=&\  \# \{\mbox{eigenvalues of }- \Delta_D^\Omega \leq \lambda\} 
  \\
  \sim&\ \{(x,p) \in \Omega\times\Real^d : \ |p|^2 \leq \lambda \} / h^d
  \qquad \mbox{as} \qquad \lambda \to \infty
  \\
  = &\ |\Omega| \, |B_{\lambda^{1/2}}| / (2\pi)^d
  = |\Omega| \, |B_1| \, \lambda^{d/2} / (2\pi)^d
  \,,
\end{aligned}  
$$
because in our units, $m = \frac{1}{2}$ and $\hbar=1$ 
(then $H = |p|^2$, where $p := -i\hbar\nabla$ is the quantum momentum). 
The obtained asymptotic formula for~$N_D^\Omega(\lambda)$ is just 
an equivalent way of writing~\eqref{Weyl.law}.   

\fbox{Cubes}
From now on, let us continue with mathematically rigorous arguments.
First, let us argue that the result~\eqref{Weyl.law} holds 
for cubes $Q_a := (0,a)^d$ with $a > 0$.
In fact, the asymptotics~\eqref{Weyl.law} becomes exact 
for the one-dimensional cube (or string) $I_a = (0,a)$. 
In higher dimensions,
we know that any eigenvalue $\Lambda \in \sigma(-\Delta_D^{Q_a})$
satisfies 
$$
\begin{aligned}
  \Lambda
  &=
  \left(\frac{k_1\pi}{a}\right)^2 
  + \dots + \left(\frac{k_d\pi}{a}\right)^2
  \qquad \mbox{with some} \qquad
  (k_1,\dots,k_d) \in {\Nat^*}^d
  \\
  &= \frac{\pi^2}{|Q_a|^{2/d}} \, (k_1^2 + \dots + k_d^2)
  \,.
\end{aligned}
$$
To determine $N_D^{Q_a}(\lambda)$, we have to count all these eigenvalues
which are less than or equal to~$\lambda$.
In other words, we have to count all the points of ${\Nat^*}^d$ 
which are contained in the ball of radius 
$$
  \kappa := \sqrt{\frac{|Q_a|^{2/d} \, \lambda}{\pi^2}}
  = \frac{|Q_a|^{1/d} \, \lambda^{1/2}}{\pi}
  \,,
$$
see Figure~\ref{Fig.Weyl} for $d=2$.
This is a purely number-theoretic problem. 
The intuition that the number should be asymptotically as $\lambda \to \infty$ 
equal to the volume of a hyperoctant of this ball is correct, \ie,
\begin{equation}\label{Weyl.Dirichlet}
  N_D^{Q_a}(\lambda) \sim \frac{|B_\kappa|}{2^d}
  = \frac{|B_1| \, \kappa^d }{2^d}
  = \frac{|B_1| \, |Q_a| \, \lambda^{d/2} }{(2\pi)^d}
  \qquad \mbox{as} \qquad \lambda \to \infty
  \,.
\end{equation} 
This is just  an equivalent way of writing~\eqref{Weyl.law} for cubes.

Defining $N_N^{Q_a}$ as the counting function
of the Neumann eigenvalues in the cube~$Q_a$,
we note that that the procedure above leads 
to the same asymptotics
\begin{equation}\label{Weyl.Neumann}
  N_N^{Q_a}(\lambda) \sim N_D^{Q_a}(\lambda) 
  \qquad \mbox{as} \qquad \lambda \to \infty
  \,.
\end{equation} 
Indeed, now we have to count all the points of ${\Nat}^d$ 
which are contained in the ball of the same radius~$\kappa$ 

\begin{figure}[h!t]
\begin{center}
\includegraphics[width=0.5\textwidth]{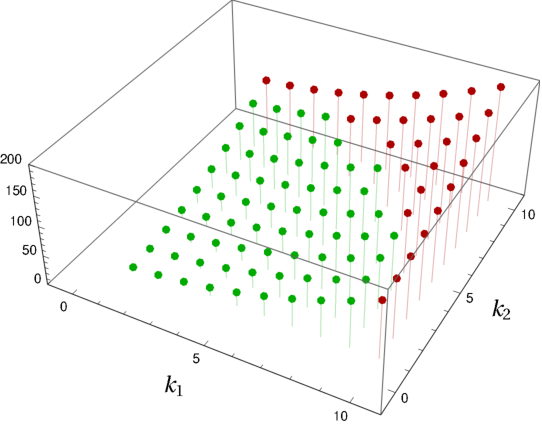}
\includegraphics[width=0.45\textwidth]{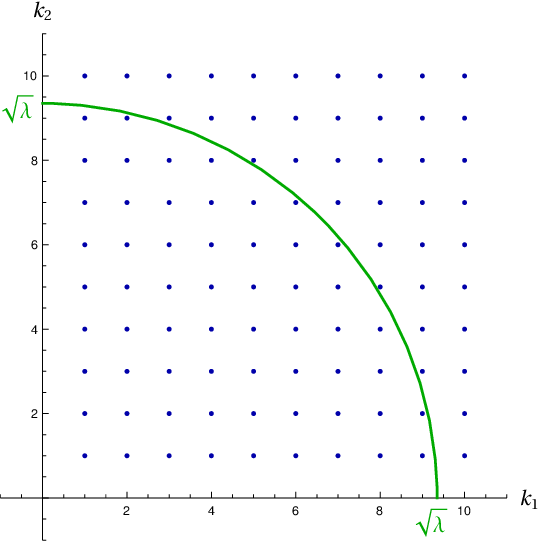}
\end{center}
\caption{Counting the number of Dirichlet eigenvalues 
$k_1^2+k_2^2$ (green and red dots)
which are less than or equal to~$\lambda$ (green dots) 
is equivalent to counting the points $(k_1,k_2)$ (blue dots)
inside the circle of radius~$\sqrt{\lambda}$ (green curve).}
\label{Fig.Weyl}
\end{figure}

\fbox{General contented sets}
Let $\Omega \subset \Real^d$ be a general bounded open set.
Given any $j \in \Nat$, 
we consider the subsets and supersets
$$
  \interior\left( \bigcup_{a \in A_j^-} \overline{Q_j(a)} \right)
  =: \Omega_j^- \subset \Omega \subset \Omega_j^+ :=
  \interior\left( \bigcup_{a \in A_j^+} \overline{Q_j(a)} \right)
$$ 
where
$$
  Q_j(a) := \left(\frac{a_1}{2^j},\frac{a_1+1}{2^j}\right)
  \times \dots \times \left(\frac{a_d}{2^j},\frac{a_d+1}{2^j}\right)
  \,, \qquad
  a := (a_1,\dots,a_d) \in \Int^d 
  \,,
$$
and the unions are taken over all 
$a \in A_j^-$ (respectively, $a \in A_j^+$)
such that $Q_j(a) \subset \Omega$
(respectively, $Q_j(a) \cap \Omega \not= \varnothing$).
That is, $\Omega_j^-$ (respectively, $\Omega_j^+$)
is formed by all the shrinking cubes $Q_j(a)$ contained in~$\Omega$  
(respectively, that intersect~$\Omega$),
see Figure~\ref{Fig.Jordan}.

\begin{figure}[h!]
\begin{center}
\includegraphics[width=0.4\textwidth]{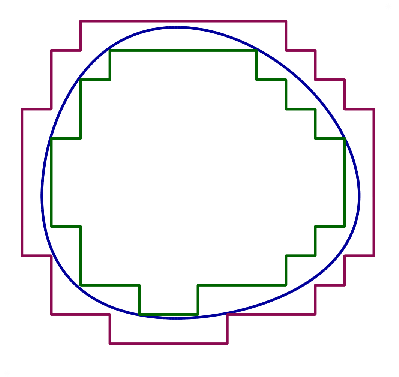}
\end{center}
\caption{Approximating~$\Omega$ (blue) by sets formed by cubes, 
from inside by~$\Omega_j^-$ (green) 
and from outside by~$\Omega_j^+$ (red).}
\label{Fig.Jordan}
\end{figure}

By the minimax principle (Dirichlet--Neumann bracketing), 
we have 
\begin{multline*}
  \lefteqn{\bigoplus_{a \in A_j^+} -\Delta_N^{Q_j(a)} =
  -\Delta_N^{\bigcup_{a \in A_j^+} Q_j(a)} \leq 
  -\Delta_N^{\Omega_j^+} \leq 
  -\Delta_D^{\Omega_j^+} \leq}
  \\
  -\Delta_D^\Omega
  \\
  \leq -\Delta_D^{\Omega_j^-}
  \leq  -\Delta_D^{\bigcup_{a \in A_j^-} Q_j(a)}
  = \bigoplus_{a \in A_j^-} -\Delta_D^{Q_j(a)}
  \,.
\end{multline*}  
Consequently, 
\begin{multline*}
  (\# A_j^+) \, N_N^{Q_j(a)}(\lambda) \geq
  \sum_{a \in A_j^+} N_N^{Q_j(a)}(\lambda) \geq
  \\
  N_D^\Omega(\lambda)
  \\
  \geq \sum_{a \in A_j^-} N_D^{Q_j(a)}(\lambda)
  = (\# A_j^-) \, N_D^{Q_j(a)}(\lambda)
  \,.
\end{multline*}  
Using the asymptotics~\eqref{Weyl.Dirichlet} and~\eqref{Weyl.Neumann} for cubes, 
it follows that 
$$
  (\# A_j^\pm) \, N_\iota^{Q_j(a)}(\lambda)
  \sim
  \frac{|B_1| \,  (\# A_j^\pm) \, |Q_j(a)| \, \lambda^{d/2} }{(2\pi)^d}
  = \frac{|B_1| \, |\Omega_j^\pm| \, \lambda^{d/2} }{(2\pi)^d}
$$
as $\lambda \to \infty$ for both $\iota \in \{D,N\}$.
Consequently,
$$
  \frac{|B_1| \, |\Omega_j^+| }{(2\pi)^d} \geq 
  \limsup_{\lambda\to\infty} \frac{N_D^\Omega(\lambda)}{\lambda^{d/2}} 
  \geq
  \liminf_{\lambda\to\infty} \frac{N_D^\Omega(\lambda)}{\lambda^{d/2}} 
  \geq \frac{|B_1| \, |\Omega_j^-| }{(2\pi)^d}
  \,.
$$
By definition, $\Omega$~is \emph{contented} (or \emph{Jordan-measurable}) 
if the boundary~$\partial\Omega$ has measure zero
(see~\cite[Sec.~3]{Spivak-calculus}). 
Then 
$$
  \lim_{j\to\infty} |\Omega_j^-| 
  = |\Omega|
  = \lim_{j\to\infty} |\Omega_j^+|
  \,,
$$ 
which concludes the proof of~\eqref{Weyl.law} for contented sets. 

\fbox{General sets}
For general sets, more refined methods are used
\cite{Birman-Solomyak_1970,Rozenblum_1972}.
\end{proof}

It is interesting to note that Weyl's law (Theorem~\ref{Thm.Weyl.law}) 
together with the Faber--Krahn inequality 
(Theorem~\ref{Thm.FK}, where it is additionally known that
the minimum is achieved only for the ball)
enables one to hear whether the drum is circular. 
 
\subsection{Can one hear the perimeter of a drum?}
Motivated by Theorem~\ref{Thm.Weyl.law}, 
maybe it is a good idea to look at the next terms 
in the asymptotics of the eigenvalues 
to reconstruct the geometry of the underlying domain?
Maybe the next term contains information about 
the perimeter of the domain?
This was exactly what Weyl thought in 1913~\cite{Weyl_1913}. 

\begin{conj}[Two-term Weyl's law, 1913~\cite{Weyl_1913}]\label{Conj.Weyl} 
Let $\Omega \subset \Real^d$ be a smooth bounded open set. Then
\begin{equation}\label{Weyl.conj}
  N_D^\Omega(\lambda)
  = \frac{|B_1|}{2\pi} \, |\Omega| \, \lambda^{d/2} 
  - \frac{1}{2\,(d-1)!\,|\Sphere^{d-1}|} 
  \, |\partial\Omega| \, \lambda^{(d-1)/2} 
  + o(\lambda^{(d-1)/2} )
\end{equation}
as $\lambda \to \infty$.
\end{conj}

The validity of this two-term Weyl's law remains a mystery. 
As already mentioned, this conjecture 
was raised by Weyl in 1913~\cite{Weyl_1913}.
In 1980, 
Ivrii~\cite{Ivrii_1980} and Melrose~\cite{Melrose_1980}
established the conjecture
under the additional assumption that 
the set of periodic billiard trajectories in~$\Omega$ has measure zero.
While this extra condition is conjectured to be satisfied for 
all Euclidean domains, it has been verified only for a few classes only,
namely convex analytic domains and polygons.
We refer to the 1997 book by Safarov and Vassiliev~\cite{Safarov-Vassiliev}
for a self-contained general approach to the problem
and further references. 
 
In summary, the answer to the question of the title 
of this subsection is NOT KNOWN.

\begin{OProblem}
Prove Conjecture~\ref{Conj.Weyl}
for all bounded smooth domains.
\end{OProblem}

The feature of Conjecture~\ref{Conj.Weyl} is to consider 
the asymptotics for individual eigenvalues. 
By ``averaging with respect to spectral parameter'',
Frank and Larson~\cite{Frank-Larson_2020,Frank-Larson_2025} 
have recently managed to 
establish an analogue of the two-term asymptotics
for Riesz means of the eigenvalues. 
 
\subsection{P\'olya's conjecture}
Assuming the validity of Conjecture~\ref{Conj.Weyl},
it immediately follows that  
$
  N_D^\Omega(\lambda)
  \leq \frac{|B_1|}{2\pi} \, |\Omega| \, \lambda^{d/2} 
$
for all sufficiently large~$\lambda$.
In 1954, P\'olya \cite{Polya_1954} conjectured that this inequality 
actually holds uniformly.   

\begin{conj}[P\'olya 1954 \cite{Polya_1954}]\label{Conj.Polya} 
Let $\Omega \subset \Real^d$ be a bounded open set. 
Then, for all $\lambda \geq 0$,
\begin{equation} 
  N_D^\Omega(\lambda)
  \leq \frac{|B_1|}{2\pi} \, |\Omega| \, \lambda^{d/2} 
  \,.
\end{equation}
\end{conj}

In 1961,  
P\'olya \cite{Polya_1961} resolved his own conjecture
under the additional assumption that~$\Omega$ is a tiling domain 
(\ie~a domain $\Omega \subset \Real^d$ such that~$\Real^d$ 
can be covered, up to a set of measure zero, 
by a disjoint union of copies of~$\Omega$, 
see Figure~\ref{Fig.honey}).
In 1966,
under the P\'olya's tiling-domain assumption,
Kellner~\cite{Kellner_1966} proved
an analogue of the conjecture for Neumann boundary conditions
(now a reverse inequality).
In 1997,
Laptev \cite{Laptev_1997} established the conjecture
for Cartesian products of domains
if it holds for one of the domains.

\begin{figure}[h!]
\begin{center}
\includegraphics[width=0.55\textwidth]{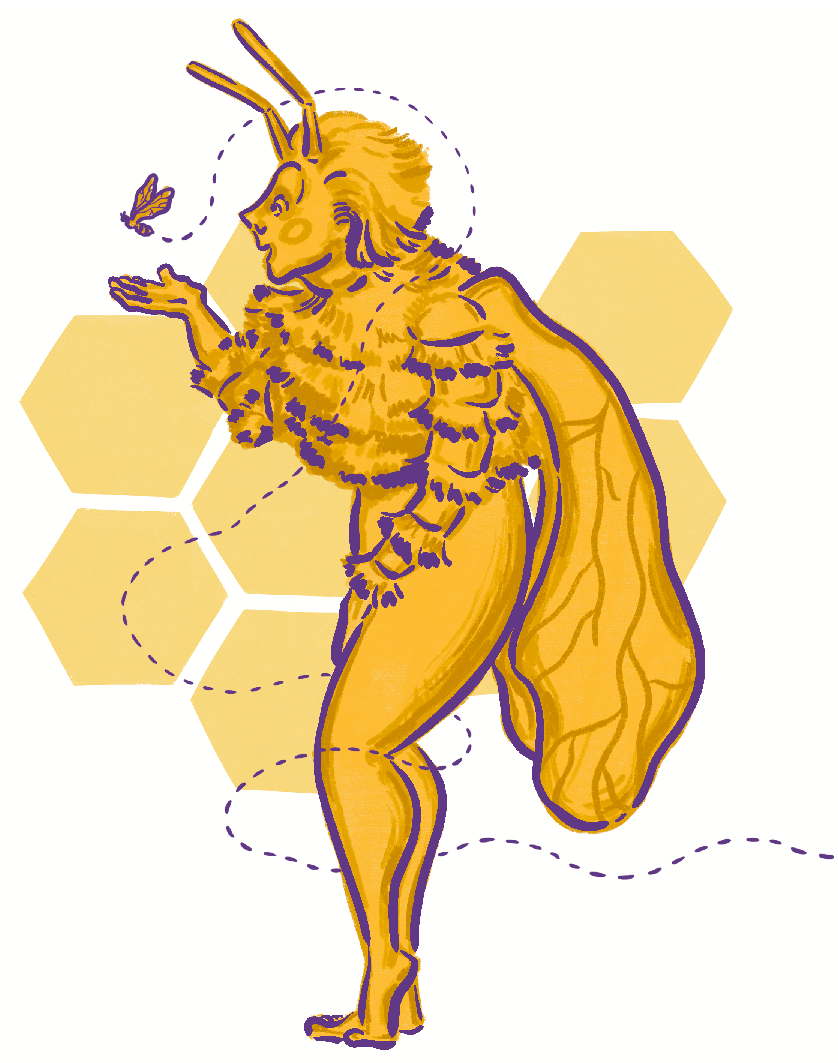}
\end{center}
\caption{As honeybees know very well, an example of a tiling domain
(for which P\'olya's result~\cite{Polya_1961} applies)
is a regular hexagon. 
(Courtesy of Ane\v{z}ka K\r{u}lov\'a.)}
\label{Fig.honey}
\end{figure}

As the most recent development,
in 2023,
Filonov, Levitin, Polterovich 
and Sher~\cite{Filonov-Levitin-Polterovich-Sher_2023}
proved the conjecture for balls
(including the Neumann case, which is computer-assisted).
The computer-assisted proof is also used for 
their more recent proof for Dirichlet (two-dimensional) 
annuli~\cite{Filonov-Levitin-Polterovich-Sher_2025}.
 
The inclusion of balls and annuli is interesting 
because they represent non-tiling domains.
In parallel, there are proofs for thin domains by P.~Freitas and I.~Salavessa
\cite{Freitas-Salavessa_2023} (see also~\cite{Freitas-Mao-Salavessa}),
which includes results for thin domains also valid on manifolds.
See also~\cite{He-Wang} for other thin domains.

\begin{OProblem}
Prove Conjecture~\ref{Conj.Polya} for all bounded  domains.
\end{OProblem}

\subsection{Can one hear the shape of a drum?}
Now let us come back to the initial question of this section
whether the geometry of the domain can be reconstructed   
from the Dirichlet eigenvalues. 
It turns out that the answer is NO

\begin{figure}[h!]
\begin{center}
\includegraphics[width=0.7\textwidth]{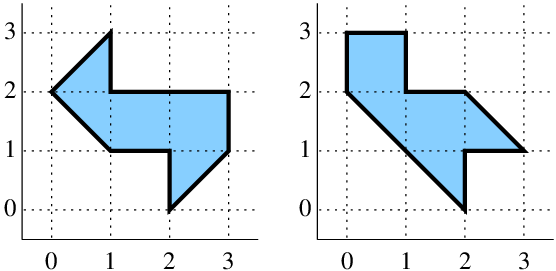}
\end{center}
\caption{Isospectral incongruent domains 
due to~\cite{Gordon-Webb-Wolpert_1992}. 
Note that the domains have the same area and perimeter,
so they are indistinguishable on the level of 
the two-term Weyl's law~\eqref{Weyl.conj}.}
\label{Fig.iso}
\end{figure}

A brief history of this discovery is as follows. 
In 1985,
Sunada~\cite{Sunada_1985} introduced a fundamental 
method for constructing isospectral Riemannian manifolds.
In 1992,
Gordon Webb and Wolpert~\cite{Gordon-Webb-Wolpert_1992}
constructed isospectral incongruent Euclidean domains, 
see Figure~\ref{Fig.iso}.
To do so they employed a non-trivial application of Sunada's method:
Their drums consist of seven congruent triangles. The symmetry behind the
construction method dictates how the triangles ought to be glued together to
provide the isospectral drums. It also yields a map (alias transplantation)
between eigenfunctions of one drum to eigenfunctions of the other sharing
the same eigenvalue.
In 2009,
Band, Parzanchevski and Ben-Shach~\cite{Band-Parzanchevski-Ben-Shach_2009}
presented new isospectral examples
(including graphs, see also~\cite{Band-Berkolaiko-Joyner-Liu}).

\section{Quasi-cylindrical domains or Waveguides}
%
Finally, let us consider quasi-cylindrical domains.
Spectral analysis of this type of domains is typically the most complicated.
The only general result is that there is always 
some essential spectrum
(see Theorem~\ref{Thm.cylindrical} below),
but there might also be some discrete eigenvalues. 

Because of the geometric complexity of quasi-cylindrical domains,
we restrict ourselves to 
Dirichlet boundary conditions
and to a special geometric subclass: 
\textbf{tubes}, 
see Figure~\ref{f-tube}.
Our motivation is twofold.
First, the tubular geometry is rich enough to demonstrate
the complexity of quasi-cylindrical domains.
Second, the Dirichlet Laplacian in tubes is
a reasonable model for the Hamiltonian
in quantum-waveguide nanostructures.

\begin{figure}[h]
\begin{center}
\includegraphics[width=0.9\textwidth]{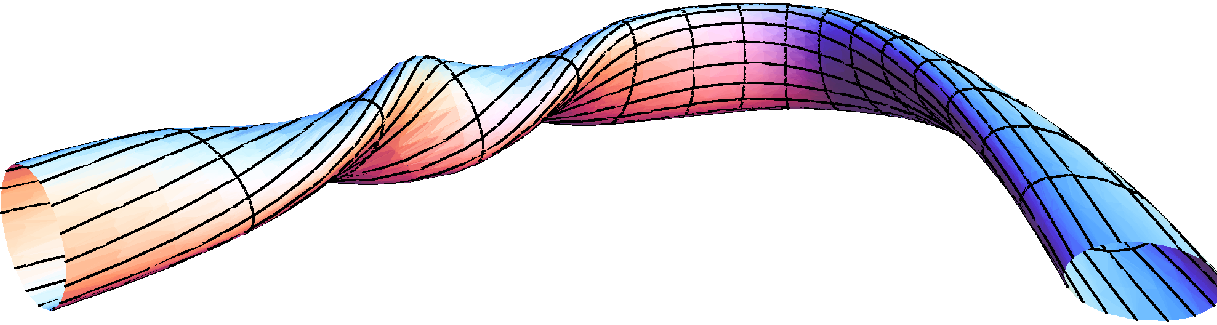}
\caption{An example of a tube of elliptical cross-section.
The geometric deformations of twisting and bending 
are demonstrated on the left and right
part of the picture, respectively.}\label{f-tube}
\end{center}
\end{figure}

\subsection{There is always some essential spectrum}
Before considering the special geometric setting of tubes,
let us establish a very general result, 
which is not even restricted to quasi-cylindrical domains.

\begin{theo}[General location of the essential spectrum]
\label{Thm.cylindrical}
Let $\Omega \subset \Real^d$ be an arbitrary non-empty open set.
Set
\begin{equation*}
  R_\mathrm{max} :=
  \sup\big\{
  R \, : \ \mbox{$\Omega \supset$ infinite sequence 
  of disjoint balls of radius~$R$}
  \big\}
\end{equation*}
(by convention, we set $R_\mathrm{max} := 0$ if there is no such a sequence.)
There exists a dimensional constant~$c_d$ such that
\begin{equation}\label{ess.location}
  \inf\sigma_\mathrm{ess}(-\Delta_D^\Omega)
  \leq \frac{c_d}{R_\mathrm{max}^2}
\end{equation}
(by convention, we interpret the right-hand side as~$+\infty$ or~$0$
if $R_\mathrm{max} := 0$ or $R_\mathrm{max} := +\infty$, respectively).
\end{theo}
\begin{proof}
If $R_\mathrm{max} = 0$, the right-hand side of~\eqref{ess.location}
can be interpreted as~$+\infty$ and there is nothing to be proved.
Let us therefore assume $R_\mathrm{max} > 0$.
Let $\{x_n\}_{n\in\Nat^*} \subset \Omega$
be a set of points such that
$\{B_R(x_n)\}_{n\in\Nat^*} \subset \Omega$
is the set of mutually disjoint balls
for all $R \in (0,R_\mathrm{max})$.
Then there also exists a sequence of cubes $\{Q_a(x_n)\}_{n\in\Nat^*}$
such that $Q_a(x_n) \subset B_R(x_n)$;
in fact, choosing the inscribed cubes, we have the relation $R^2 = d a^2$.
The idea is to construct a non-compact sequence
supported on the disjoint cubes.
Let~$\psi$ be the first eigenfunction of $-\Delta_D^{Q_a(0)}$,
normalised to~$1$ in $\sii(Q_a(0))$,
and recall (\cf~\eqref{evs.piped})
that the corresponding eigenvalue is given by
$$
  \lambda_1^D(Q_a(0)) = d \left(\frac{\pi}{2a}\right)^2
  = \left(\frac{\pi d}{2 R}\right)^2 =: \frac{c_d}{R^2}
  \,.
$$
For all $n \in \Nat^*$, we set
$$
  \psi_n(x) := \psi(x-x_n)
$$
(the first eigenfunction of $-\Delta_D^{Q_a(x_n)}$)
and extend it by zero to the whole~$\Omega$.
Then the functions $\psi_n$'s are mutually orthonormal in $\sii(\Omega)$
and satisfy $\|\nabla\psi_n\|_{\sii(\Omega)}^2 = c_d/R^2$.
Hence, choosing the $n$-dimensional subspace
$\mathscr{L}_n := \obal\{\psi_1,\dots,\psi_n\}$
in the minimax principle (Theorem~\ref{minimax}), we get
\begin{equation}\label{cylindrical.upper}
  \lambda_n^D(\Omega) \leq \frac{c_d}{R^2}
\end{equation}
for \emph{all} $n\in\Nat^*$.
Consequently, by Theorem~\ref{minimax},
$$
  \inf\sigma_\mathrm{ess}(-\Delta_D^\Omega)
  = \lim_{n\to\infty} \lambda_n^D(\Omega)
  \leq \frac{c_d}{R^2}
  \,.
$$
Since the argument holds for all $R \in (0,R_\mathrm{max})$,
we conclude with the stated inequality.
\end{proof}

As a consequence of Theorem~\ref{Thm.cylindrical}, 
we get the following implications:

\begin{center}
\begin{tabular}{clcr}
  & $\Omega$ is quasi-conical
  & $\Longrightarrow$ 
  & $\inf\sigma_\mathrm{ess}(-\Delta_D^\Omega) = 0$\,,
  \\
  & $\Omega$ is quasi-cylindrical
  & $\Longrightarrow$
  & $\sigma_\mathrm{ess}(-\Delta_D^\Omega) \not=\varnothing$\,,
  \\
  & $\Omega$ is quasi-bounded
  & $\Longleftarrow$ 
  & $\sigma_\mathrm{ess}(-\Delta_D^\Omega)=\varnothing$\,.
\end{tabular}
\end{center}

The first implication (in fact, much more)
has been established previously, 
see Corollary~\ref{Corol.positivity}. 
The last implication says that the quasi-boundedness 
is a \emph{necessary} condition for the discreteness of the spectrum
of the Dirichlet Laplacian
(by Theorem~\ref{Thm.bounded}, the boundedness is a sufficient condition).
It is the middle implication which is of interest for us
as regards quasi-cylindrical domains.
Let us highlight it as a corollary.

\begin{coro}
Let $\Omega \subset \Real^d$ be any quasi-cylindrical open set. 
Then 
$$
  \sigma_\mathrm{ess}(-\Delta_D^\Omega) \not=\varnothing \,.
$$
\end{coro}

\subsection{Straight tubes}
The special class of quasi-cylindrical domains we shall consider
are obtained as a ``local'' perturbation of the \emph{straight tube}
\begin{equation}\label{straight}
  \Omega_0 := \Real \times \omega
  \,,
\end{equation}
where $\omega \subset \Real^{d-1}$ is an arbitrary bounded domain
(the \emph{cross-section} of a waveguide modelled by~$\Omega_0$),
see Figure~\ref{Fig.straight}.

\begin{figure}[h]
\begin{center}
\includegraphics[width=0.4\textwidth]{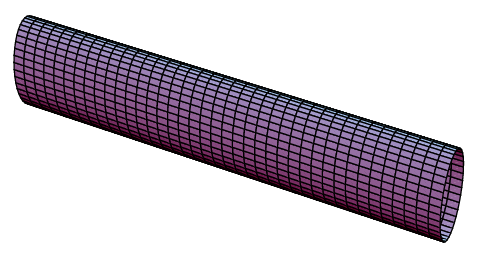}
\caption{A straight tube.}\label{Fig.straight}
\end{center}
\end{figure}

Since~$\Omega_0$ is a Cartesian product of two domains,
it is easily seen that 
\begin{equation}\label{separation}
  -\Delta_D^{\Omega_0} 
  \cong -\Delta_D^{\Real} \otimes I_\omega
  + I_{\Real} \otimes -\Delta_D^{\omega}
  \qquad \mbox{in} \qquad
  \sii(\Omega_0) \cong \sii(\Real) \times \sii(\omega)
  \,,
\end{equation}
where $I_{\Real}$ and $I_{\omega}$ denote the identity operators
on $\sii(\Real)$ and $\sii(\omega)$, respectively,
and~$\otimes$ denotes the closed tensor product~\cite[Sec.~8.2]{ABG}. 
This is the precise statement of the ``separation of variables'' in $\Omega_0$. 

Since the real axis~$\Real$ is a quasi-conical domain,
its Dirichlet spectrum is purely essential 
(see Corollary~\ref{Corol.positivity})
$$
  \sigma(-\Delta_D^{\Real})
  = \sigma_\mathrm{ess}(-\Delta_D^{\Real})
  = [0,\infty)
  \,.
$$

On the other hand, since~$\omega$ is bounded, 
its Dirichlet spectrum is purely discrete (see Theorem~\ref{Thm.bounded})
$$
  \sigma(-\Delta_D^{\omega})
  = \sigma_\mathrm{disc}(-\Delta_D^{\omega})
  =: \{E_1 < E_2 \leq E_3 \leq \dots \to +\infty \} 
  \,.
$$
Let~$\mathcal{J}_n$ denote the eigenfunction of $-\Delta_D^\omega$
corresponding to~$E_n$. We choose the eigenfunctions normalised to one 
in $\sii(\omega)$ and~$\mathcal{J}_1$ positive.

Since the spectrum of the operator
on the right-hand side of~\eqref{separation} is obtained 
as the sum of the individual spectra, it follows that 
the spectrum of $-\Delta_D^{\Omega_0}$ coincides with
the semi-axis $[E_1,\infty)$.
In particular, it is purely essential 
(see Figure~\ref{Fig.straight.spec}): 
\begin{equation}\label{spec.straight}
  \sigma(-\Delta_D^{\Omega_0}) 
  = \sigma_\mathrm{ess}(-\Delta_D^{\Omega_0}) 
  = [E_1,\infty) .
\end{equation}

Notice that $E_1>0$ (otherwise $\int_\omega |\nabla\mathcal{J}_1|^2 = 0$,
which would imply that $\mathcal{J}_1=\const$ almost everywhere in~$\omega$,
and the constant would have to be equal to zero due to the Dirichlet
boundary conditions).
Hence, the structure of the spectrum~\eqref{spec.straight} suggests
that we deal with a reasonable model 
for a semiconductor waveguide nanostructure
(the ionisation energy~$E_1$ is strictly positive).
We shall frequently use the following Poincar\'e inequality
\begin{equation}\label{Poincare}
  \forall \mathcal{J} \in W_0^{1,2}(\omega)
  \,, \qquad
  \|\nabla\mathcal{J}\|_{\sii(\omega)}^2 
  \geq E_1 \, \|\mathcal{J}\|_{\sii(\omega)}^2
  \,,
\end{equation}
which follows from the variational definition of~$E_1$
(\cf~Theorem~\ref{minimax}).

\begin{figure}[h!t]
\begin{center}
\includegraphics[width=0.4\textwidth]{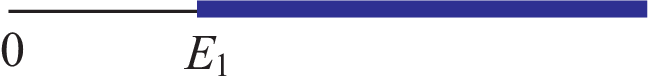}
\caption{Spectrum of a straight tube.}\label{Fig.straight.spec}
\end{center}
\end{figure}

Despite the fact that the tubes we consider are three-dimensional,
they are \emph{quasi-one-dimensional} in the sense that
there is only one ``infinite direction''.   
For straight tubes, 
using the criticality of~$\Real$ 
(\cf~Theorem~\ref{Thm.critical}),
we get the following criticality of $\Real\times\omega$. 

\begin{prop}[Criticality of straight tubes]\label{Prop.critical}
Let $V \in L^\infty(\Omega_0)$ be such that 
$V \stackrel[\not=]{}{\leq} 0$.
Then 
$$
  \inf\sigma(-\Delta_D^{\Omega_0} + V) < E_1
  \,. 
$$

\end{prop}
\begin{proof}
The proof is based on the variational characterisation
$$
  \inf\sigma(-\Delta_D^{\Omega_0} + V - E_1) = 
  \inf_{\stackrel[\psi\not=0]{}{\psi \in W_0^{1,2}(\Omega_0)}} 
  \frac{Q[\psi]}{\|\psi\|^2}
  \,,
$$
where 
$
  \|\psi\|^2 = \int_{\Omega_0} 
  |\psi(s,t)|^2 \, \der s \, \der t
$ 
and
$$
  Q[\psi] := \int_{\Omega_0} |\nabla\psi|^2 
  + \int_{\Omega_0} V(s,t) \, |\psi(s,t)|^2 \, \der s \, \der t
  - E_1 \int_{\Omega_0} |\psi(s,t)|^2 \, \der s \, \der t
  \,.
$$
Hence it is enough to find a function $\psi \in W_0^{1,2}(\Omega_0)$
such that $Q[\psi] < 0$.
The main idea is that this can be achieved 
for a trial function built from $(s,t)\mapsto\mathcal{J}_1(t)$,
a generalised eigenfunction of $-\Delta_D^{\Omega_0}$ 
corresponding to~$E_1$.
More specifically, let
$\psi_n(s,t):=\varphi_n(s) \, \mathcal{J}_1(t)$,
where (recall Figure~\ref{Fig.1D})
\begin{equation}\label{trial}
  \varphi_n(s) :=
  \begin{cases}
    1 & \mbox{if} \quad |s| \leq n \,,
    \\
    0 & \mbox{if} \quad |s| \geq 2n \,,
    \\
    \displaystyle
    \frac{2n-|s|}{n} & \mbox{otherwise} \,. 
  \end{cases}
\end{equation}
Note that 
$\varphi_n \to 1$ pointwise as $n \to \infty$
and $\|\varphi_n'\|_{\sii(\Real)}^2 = 2/n$.
The latter makes the ``longitudinal energy'' 
disappear as $n \to \infty$:
$$
  \int_{\Omega_0} |\partial_s\psi_n(s,t)|^2 
  \, \der s \, \der t 
  = \|\varphi_n'\|_{\sii(\Real)}^2
  \xrightarrow[n \to \infty]{} 0
  \,.
$$
On the other hand, the ``transversal energy'' 
is compensated by~$E_1$:
$$
  \int_{\Omega_0} |\nabla_{\!t}\psi_n(s,t)|^2 \, \der s \, \der t
  - E_1 \int_{\Omega_0} |\psi_n(s,t)|^2 \, \der s \, \der t
  =  0
  \,.
$$
It remains to notice that 
$$
  \int_{\Omega_0} V(s,t) \, |\psi_n(s,t)|^2 \, \der s \, \der t
  \xrightarrow[n \to \infty]{}  
  \int_{\Omega_0} V(s,t) \, |\mathcal{J}_1(t)|^2
  \, \der s \, \der t  < 0
$$
by the monotone convergence theorem 
(the limit can be $-\infty$).
\end{proof}

Of course, the most interesting situation is when~$V$
\emph{vanishes at infinity} in the sense that 
$$
  \lim_{s_0 \to \infty} 
  \esssup_{|s| > s_0,\, t\in\omega} |V(s,t)| = 0
  \,.
$$
In this case,
$$
 \sigma_\mathrm{ess}(-\Delta_D^{\Omega_0} + V) 
 = [E_1,\infty)
 \,,
$$
so the lowest point in the spectrum of 
$-\Delta_D^{\Omega_0} + V$
corresponds to a discrete eigenvalue.

\subsection{Curved tubes}
The straight tube~$\Omega_0$ can be considered as built
by translating the cross-section $\omega \subset \Real^{d-1}$ 
along the straight line $\{(s,0,\dots,0) :s\in\Real\}$ in~$\Real^d$.
A curved tube~$\Omega$ is obtained
by translating~$\omega$ along 
a general curve~$\Gamma$ in~$\Real^d$
with respect to an arbitrary moving frame. 
For simplicity, 
and also because of the physical motivation we have in mind,
let us restrict to $d=3$.

To rigorously implement this idea of definition,
let $\Gamma:\Real\to\Real^3$ be a 
$C^2$-smooth curve
which is (without loss of generality) 
parameterised by its arc-length 
(\ie~$|\Gamma'(s)|=1$ for all $s \in \Real$).
By the regularity hypothesis, 
the tangent vector field $T := \Gamma'$ is $C^1$-smooth
and the \emph{curvature} $\kappa := |\Gamma''|$ is continuous.
Let $(T,N_1,N_2)$ be a \emph{relatively parallel adapted frame} 
of~$\Gamma$ obeying the equations
(\cf~\cite{Bishop_1975,KSed,KZ3})
\begin{equation}\label{frame.parallel} 
\begin{pmatrix}
    T \\
    N_1 \\
    N_2
\end{pmatrix}'
=
\begin{pmatrix}
    0 & k_1 & k_2 \\
   -k_1 &  0 & 0 \\
    -k_2  & 0 & 0 
\end{pmatrix}
\begin{pmatrix}
T \\
    N_1 \\
    N_2
\end{pmatrix}
,
\end{equation}
where $k_1,k_2$ are continuous functions satisfying 
\begin{equation}\label{curvature} 
  k_1^2+k_2^2 = \kappa^2 \,.
\end{equation}

The general moving frame $(T,N_1^\theta,N_2^\theta)$
is then introduced by rotating the relatively parallel adapted frame
$(T,N_1,N_2)$ with respect to a given 
$C^1$-smooth rotation function 
$\theta:\Real\to\Real$:
$$
\begin{pmatrix}
    N_1^\theta \\
    N_2^\theta
\end{pmatrix}
=
\begin{pmatrix}
    \cos\theta & -\sin\theta  \\
    \sin\theta & \cos\theta
\end{pmatrix}
\begin{pmatrix}
    N_1 \\
    N_2
\end{pmatrix}
\,.
$$   
Using~\eqref{frame.parallel}, it is straightforward to check
that the new frame evolves along the curve via 
\begin{equation}\label{frame.general} 
\begin{pmatrix}
    T \\
    N_1^\theta \\
    N_2^\theta
\end{pmatrix}'
=
\begin{pmatrix}
    0 & k_1^\theta & k_2^\theta \\
   -k_1^\theta &  0 & -\theta' \\
    -k_2^\theta  & \theta' & 0 
\end{pmatrix}
\begin{pmatrix}
    T \\
    N_1^\theta \\
    N_2^\theta
\end{pmatrix}
,
\end{equation}
where 
$$
\begin{pmatrix}
    k_1^\theta \\
    k_2^\theta
\end{pmatrix}
=
\begin{pmatrix}
    \cos\theta & -\sin\theta  \\
    \sin\theta & \cos\theta
\end{pmatrix}
\begin{pmatrix}
    k_1 \\
    k_2
\end{pmatrix}
\,.
$$
The familiar Frenet frame corresponds to the special choice
$\theta'=-\tau$ (torsion)
together with $k_1=\kappa\cos\theta$ and $k_2=-\kappa\sin\theta$. 

With these preliminaries,
the general \emph{tube}~$\Omega$ is defined by 
\begin{equation}\label{tube}
  \Omega := 
  \Big\{
  \underbrace{
  \Gamma(s) + t_1 \, N_1^\theta(s) + t_2 \, N_2^\theta(s)
  }_{\mathscr{L}(s,t)}
  : \, (s,t) \in \Omega_0
  \Big\}
  \,.
\end{equation}
In this way, $\Omega$ can obviously be understood 
as a deformation of the straight tube~$\Omega_0$
(corresponding to $\Gamma(s) := (s,0,0)$).

It is clear from the equations of motion of the general 
moving frame~\eqref{frame.general} that there are 
two independent geometric effects in curved tubes:
\begin{itemize}
\item
$\Omega$ is \emph{bent}
\ $:\Longleftrightarrow$ \
$\kappa \not= 0$ (\ie~$\Gamma$ is not a straight line).
\item
$\Omega$ is 
\emph{twisted}
\ $:\Longleftrightarrow$ \
$\theta' \not= 0$ 
(\ie~$\omega$ is not translated parallelly along~$\Gamma$).
\end{itemize}
The definition of bending is visually clear. 
On the other hand, the definition of twisting may be confusing.
Indeed, even if $\theta'\not=0$ (and $\kappa=0$), 
it is possible that~$\Omega$ is congruent 
to the straight tube $\Omega_0$. 
For this reason, to have a non-trivially twisted tube,
it is also important to assume that the cross-section~$\omega$ 
is not circular.
By~$\omega$ being \emph{circular} we mean that
it is a disk or annulus centred at the origin of 
$\Real^2 \supset \omega$. 
(As usual, we identify open sets which differ 
by a set of capacity zero, 
in particular a disk with 
a countable number of points removed is also circular,
\cf~Remark~\ref{Rem.Hardy}.) 
Examples of purely bent and purely twisted tubes 
can be seen in Figures~\ref{Fig.bend} and~\ref{Fig.twist}, respectively.
Of course, $\Omega$~can be simultaneously bent and twisted,
as in Figure~\ref{f-tube}. 

\subsection{The tube as a Riemannian manifold}
Our strategy to deal with the curved geometry of the tube~$\Omega$
is to use the identification
\begin{equation}\label{identification}
  \Omega \cong (\Omega_0,G)
  \,,
  \qquad \mbox{where} \qquad
  G := (\nabla\mathscr{L}) \cdot (\nabla\mathscr{L})^\top
\end{equation}
is the metric tensor 
induced by the mapping 
$
  \mathscr{L}: \Omega_0 \to \Real^3
$
from~\eqref{tube}.  
Here the dot ``$\cdot$'' denotes the matrix multiplication.

In other words, we parameterise~$\Omega$ globally
by means of the ``coordinates'' $(s,t)$ of~\eqref{tube}.
To this aim, we need to impose natural restrictions
in order to ensure that
$\mathscr{L}:\Omega_0 \to \Omega$ is a diffeomorphism
(and thus~$\Omega$ an embedded submanifold of~$\Real^3$). 

Using~\eqref{frame.general}, we find
\begin{equation}\label{metric}
  G =
  \begin{pmatrix}
    f^2+f_1^2+f_2^2  & f_1 & f_2 \\
    f_1 & 1 & 0 \\
    f_2 & 0 & 1 \\
  \end{pmatrix}
  , \qquad
  \begin{aligned}
  f(s,t)
  &:= 1 - t_1 \, k_1^\theta(s) - t_2 \, k_2^\theta(s) \,,
  \\
  f_1(s,t)
  &:= t_2 \, \theta'(s) \,,
  \\
  f_2(s,t)
  &:= -t_1 \, \theta'(s) \,.
\end{aligned}
\end{equation}
Consequently,
$$
  |G| := \det(G) = f^2 \,.
$$
By virtue of the inverse function theorem,
the mapping~$\mathscr{L}$ induces a \emph{local} $C^1$-diffeomorphism
provided that the Jacobian~$f$ does not vanish on $\Omega_0$.
One has the uniform bounds
\begin{equation}\label{1<G<1}
  0 <
  1 - a \, \|\kappa\|_\infty
  \ \leq \ f(s,t) \ \leq \
  1 + a \, \|\kappa\|_\infty
  < \infty
\end{equation}
valid for every $(s,t) \in \Omega_0$,
where $\|\cdot\|_\infty := \|\cdot\|_{L^\infty(\Real)}$.
Consequently,
the positivity of~$f$ is guaranteed by the hypothesis
\begin{equation}\label{Ass.basic1}
  \kappa \in L^\infty(\Real)
  \qquad\mbox{and}\qquad
  a \, \|\kappa\|_\infty < 1
  \,,
\end{equation}
where 
\begin{equation}\label{quantity.a}
  a := \sup_{t\in\omega} |t|
  \,.
\end{equation}
Condition~\eqref{Ass.basic1} is a natural one 
to avoid ``local self-intersections'', 
see Figure~\ref{Fig.self} for a two-dimensional analogue.
 
To guarantee that $\Omega$~is an embedded submanifold of~$\Real^3$,
we have to ensure that $\mathscr{L}:\Omega_0 \to \Omega$
is a \emph{global} diffeomorphism.
This is the case if,
in addition to~\eqref{Ass.basic1}, 
we \emph{ad hoc} assume that
\begin{equation}\label{Ass.basic2}
  \mathscr{L} \quad \mbox{is injective} 
  \,.
\end{equation}

In other words, hypothesis~\eqref{Ass.basic1} ensures that 
$(\Omega_0,G)$ is a Riemannian manifold
and~$\mathscr{L}$ represents its \emph{immersion} in~$\Real^3$.
Both hypotheses~\eqref{Ass.basic1} and~\eqref{Ass.basic2}
then ensure that $(\Omega_0,G)$ is an \emph{embedded}
submanifold of~$\Real^3$. 
Giving up the geometrical interpretation of~$\Omega$
being a non-self-inter\-secting tube in~$\Real^3$,
it is possible to work under the hypothesis~\eqref{Ass.basic1} only. 

\begin{figure}[h!]
\begin{center}
\includegraphics[width=0.5\textwidth]{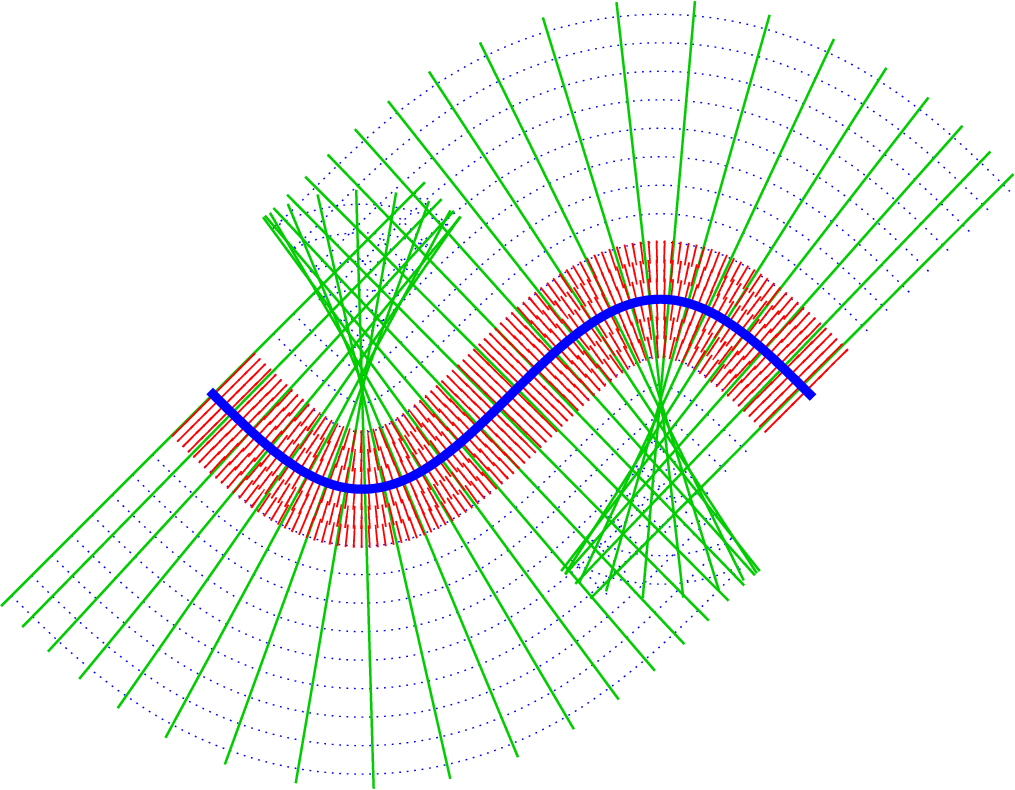}
\end{center}
\caption{A Euclidean two-dimensional tube (red) constructed 
by means of straight lines (green)
emanating from the base planar curve (blue).
A smallness restriction on the tube radius~$a$ must be imposed,
in order to avoid ``local self-intersections''.}
\label{Fig.self}
\end{figure}

\subsection{The tube as a quantum Hamiltonian}
As a particular consequence of the fundamental 
hypotheses~\eqref{Ass.basic1} and \eqref{Ass.basic2},  
$\Omega$~is an open set. 
Therefore the Dirichlet Laplacian $-\Delta_D^\Omega$ is well defined.
Moreover, recalling the identification~\eqref{identification},
we can identify $-\Delta_D^\Omega$
with an operator $H := U(-\Delta_D^\Omega)U^{-1}$
acting in the Hilbert space
\begin{equation}\label{Hilbert}
  \Hilbert := 
  \sii\big(\Omega_0,f(s,t) \, \der s \, \der t\big)
  \,,
\end{equation}
where 
$
  U : \sii(\Omega) \to \Hilbert:
  \{u \mapsto  u \circ \mathscr{L}\}
$.
By definition, 
$H$~is the operator associated with the quadratic form~$h$
in~$\Hilbert$ defined by 
$$
  h[\psi] := \delta_D^\Omega[U^{-1}\psi] 
  \,, \qquad
  \psi \in \dom h := U \dom(\delta_D^\Omega)
  \,.
$$  
In other words, one has to pass to the curvilinear coordinates~$(s,t)$
in the integral~\eqref{depending}. 
Using that $C_0^\infty(\Omega)$ is a core of~$\delta_D^\Omega$
and that $\mathscr{L}:\Omega_0 \to \Omega$ 
is a $C^1$-diffeomorphism, one has
\begin{equation}\label{form}
  h[\psi] = \big(\partial_i\psi,G^{ij}\partial_j\psi\big)
  \,, \qquad
  \psi \in \dom h = 
  \overline{C_0^1(\Omega_0)}^{\vertiii{\cdot}}
  \,,
\end{equation}
where $\vertiii{\psi} := \sqrt{h[\psi] + \|\psi\|^2}$,
$G^{ij}$ stands for the coefficients of the inverse metric 
\begin{equation}\label{inverse}
  G^{-1} = \frac{1}{f^2}
  \begin{pmatrix}
    1 & -f_1 & -f_2 \\
    -f_1 & f^2+f_1^2 & f_1 f_2 \\
    -f_2 & f_2 f_1 & f^2+f_2^2 \\
  \end{pmatrix}
\end{equation}
and the Einstein summation convention is adopted,
with the range of indices being $1,2,3$.
In the distributional sense, $H$~acts as the Laplace--Beltrami operator
\begin{equation}\label{LB}
 H =
  -|G|^{-1/2} \partial_i |G|^{1/2} G^{ij} \partial_j \,.
\end{equation}

In fact, \eqref{form}~is a general formula for any curvilinear coordinates.
Using the particular form of the inverse metric~\eqref{inverse},
we find
\begin{equation}\label{form.bis}
  h[\psi] = \| f^{-1} (\partial_s\psi - \theta'\partial_\tau \psi)\|^2
  + \|\nabla_{\!t}\psi\|^2
  \,,
\end{equation}
where $\nabla_{\!t} := (\partial_{t_1},\partial_{t_2})$
is the transverse gradient
and 
$\partial_\tau := \tau \cdot \nabla_{\!t}$ 
with $\tau := (t_2,-t_1)$ is the transverse angular derivative
(do not confuse the vector~$\tau$ with the torsion
that we shall never mention any more).
Here we implicitly assume that $\psi = \psi(s,t)$. 
Moreover, with an abuse of notation,
we denote by the same symbol~$\theta'$ the function
$(s,t) \mapsto \theta'(s)$. 

Similarly, the general formula~\eqref{LB} in our case reads
\begin{equation}\label{Hamiltonian}
  H = - f^{-1} (\partial_s-\theta'\partial_\tau) 
  f^{-1} (\partial_s-\theta'\partial_\tau)
  - f^{-1} \nabla_{\!t} \cdot f \nabla_{\!t}
  \,.
\end{equation}
\begin{prop}\label{Prop.equivalence}
In addition to~\eqref{Ass.basic1}, let us assume
\begin{equation}\label{Ass.twist}
  \theta' \in L^\infty(\Real)
  \,.
\end{equation}
Then there exist positive constants~$C_\pm$ 
depending on $\|\kappa\|_\infty$, $\|\theta'\|_\infty$ and~$a$
such that 
\begin{equation}\label{equivalence}
  \forall \psi \in C_0^1(\Omega_0)
  \,, \qquad
  C_- \, \|\psi\|_{W^{1,2}(\Omega_0)}
  \leq \vertiii{\psi} \leq 
  C_+ \, \|\psi\|_{W^{1,2}(\Omega_0)}
  \,.
\end{equation}
In particular, $\dom h = W_0^{1,2}(\Omega_0)$.
\end{prop}
\begin{proof}
Let $\psi \in C_0^1(\Omega_0)$.
Recalling~\eqref{1<G<1},
\begin{multline*}
  \frac{\| \partial_s\psi - \theta'\partial_\tau \psi\|_{\Hilbert_0}^2}
  {1+a\,\|\kappa\|_\infty} 
  + (1-a\,\|\kappa\|_\infty) \, \|\nabla_{\!t}\psi\|_{\Hilbert_0}^2
  \\
  \leq h[\psi] \leq 
  \\
  \frac{\| \partial_s\psi - \theta'\partial_\tau \psi\|_{\Hilbert_0}^2}
  {1-a\,\|\kappa\|_\infty} 
  + (1+a\,\|\kappa\|_\infty) \, \|\nabla_{\!t}\psi\|_{\Hilbert_0}^2
  \,,
\end{multline*}
where $\Hilbert_0:=\sii(\Omega_0)$.
The upper bound in~\eqref{equivalence} is concluded by
$$
\begin{aligned}
  \| \partial_s\psi - \theta'\partial_\tau \psi\|_{\Hilbert_0}^2
  &\leq
  2 \, \|\partial_s\psi\|_{\Hilbert_0}^2 + 
  2 \, \|\theta'\partial_\tau \psi\|_{\Hilbert_0}^2
  \\
  &\leq
  2 \, \|\partial_s\psi\|_{\Hilbert_0}^2 + 
  2 \, \|\theta'\|_\infty^2 \, a^2 \, 
  \, \|\nabla_{\!t} \psi\|_{\Hilbert_0}^2
  \,,
\end{aligned} 
$$
where we have used the pointwise bound
$|\partial_\tau\psi| \leq |t| |\nabla_{\!t}\psi|$ 
and~\eqref{quantity.a}.
For the lower bound, we write
$$
\begin{aligned}
  \| \partial_s\psi - \theta'\partial_\tau \psi\|_{\Hilbert_0}^2
  &\geq
  \frac{\delta}{1+\delta} \, \|\partial_s\psi\|_{\Hilbert_0}^2
  -\delta \, \|\theta'\partial_\tau \psi\|_{\Hilbert_0}^2
  \\
  &\geq
  \frac{\delta}{1+\delta} \, \|\partial_s\psi\|_{\Hilbert_0}^2
  -\delta \, \|\theta'\|_\infty^2 \, a^2 \, 
  \, \|\nabla_{\!t} \psi\|_{\Hilbert_0}^2
\end{aligned} 
$$
with any positive~$\delta$.
It remains to choose~$\delta$ sufficiently small,
for instance 
$$
  \delta := \frac{(1-a\,\|\kappa\|_\infty)(1+a\,\|\kappa\|_\infty)}
  {1+\|\theta'\|_\infty^2 \, a^2}
  \,.
$$
This establishes the lower bound in~\eqref{equivalence}.
From the equivalence of norms,
the conclusion $\dom h = W_0^{1,2}(\Omega_0)$ readily follows.
\end{proof}

From now on, we always assume~\eqref{Ass.basic1} 
and~\eqref{Ass.basic2} as standing hypotheses.
(The latter can be omitted when~$\Omega$ is understood
as an immersed submanifold of~$\Real^3$
and the results below are restated for~$H$.) 
 
\subsection{Asymptotically straight tubes}
The essential spectrum of the Laplacian in a manifold is
determined by the behaviour of the metric at infinity
(and possibly at the boundary) only.
Inspecting the dependence of the coefficients of~\eqref{metric}
on large ``longitudinal distances''~$s$,
we see that the metric~$G$ converges to the Euclidean metric provided that
\begin{equation}\label{Ass.decay}
  \lim_{|s|\to\infty} \kappa(s) = 0
  \qquad\mbox{and}\qquad
  \lim_{|s|\to\infty} \theta'(s) = 0
  \,.
\end{equation}
Therefore the following stability result is not surprising.
\begin{theo}[Stability of the essential spectrum]\label{Thm.ess}
Under the hypotheses~\eqref{Ass.decay},
\begin{equation}\label{stability.ess}
  \sigma_\mathrm{ess}(-\Delta_D^\Omega) 
  = [E_1,\infty)
  \,.
\end{equation}
\end{theo}
\begin{proof}
The inclusion $\sigma_\mathrm{ess}(H) \subset [E_1,\infty)$
is obtained by a Neumann bracketing argument.
The opposite inclusion 
$\sigma_\mathrm{ess}(H) \supset [E_1,\infty)$ 
follows by the approximate eigenfunctions
\begin{equation}\label{singular.Weyl}
  \psi_n(s,t) := \varphi_n(s) \, e^{i k s} \, \mathcal{J}_1(t)
  \,,
\end{equation}
where $k \in \Real$,
$
\displaystyle
  \varphi_n(s) := n^{-1/2} \, \varphi(n^{-1}s-n)
$ 
with $n \in \Nat^*$ and $\varphi \in C_0^\infty(\Real)$ 
satisfying $\|\varphi\|_{\sii(\Real)}=1$. 
To avoid additional assumptions,
the Weyl criterion adapted to quadratic forms 
\cite[App.]{KL} must be used. In detail:

\fbox{$\sigma_\mathrm{ess}(H) \subset [E_1,\infty)$}
(Neumann bracketing)
 
Given any arbitrary positive number~$s_0$, 
we divide $\Omega_0$ into an interior and
an exterior part by considering:
$$
  I_\mathrm{int} := (-s_0,s_0)  
  \qquad\mbox{and}\qquad
  I_\mathrm{ext} := 
  (-\infty,-s_0) \cup (s_0,\infty)
  \,.
$$
We impose the Neumann condition on 
the interface $\Sigma_\pm := \{\pm s_0\} \times \omega$. 
On the level of forms, it leads to considering 
the quadratic form  
which acts as~$h$ in $\Omega_0$ 
but satisfies no continuity (in the Sobolev setting) on~$\Sigma_\pm$.  
More specifically, let us consider the quadratic form
$h_\mathrm{int}^N$ in 
$
  \Hilbert_\mathrm{int} :=
  \sii(I_\mathrm{int}\times\omega,f(s,t)\,\der s \, \der t)
$ 
defined by
$$
\begin{aligned}
  h_\mathrm{int}^N[\psi] 
  &:= (\partial_i\psi,G^{ij}\partial_j\psi)_{\Hilbert_\mathrm{int}}
  \,,
  \\
  \dom h_\mathrm{int}^N 
  &:= \{\psi \upharpoonright (I_\mathrm{int}\times\omega) : \ 
  \psi \in \dom h \}
  \,.
\end{aligned}  
$$
We denote by~$H_\mathrm{int}^N$ the operator 
associated with $h_\mathrm{int}^N$ in $\Hilbert_\mathrm{int}$.
Similarly, we introduce a form $h_\mathrm{ext}^N$ 
and the associated operator $H_\mathrm{ext}^N$ 
in the Hilbert space
$
  \Hilbert_\mathrm{ext} :=
  \sii(I_\mathrm{ext}\times\omega,f(s,t)\,\der s \, \der t)
$. 
We set $H^N := H_\mathrm{int}^N \oplus H_\mathrm{ext}^N$.
Since the form domain~$H^N$ is larger than the form domain of~$H$,
while the forms act in the same way, one has  
$
  H \geq H^N 
$.
By the minimax principle,
\begin{equation}\label{bracketing}
\begin{aligned}
  \inf\sigma_\mathrm{ess}(H) 
  &\geq \inf\sigma_\mathrm{ess}(H^N)
  \\
  &= \min\left\{\inf\sigma_\mathrm{ess}(H_\mathrm{int}^N),
  \inf\sigma_\mathrm{ess}(H_\mathrm{ext}^N)\right\}
  \\
  &= \inf\sigma_\mathrm{ess}(H_\mathrm{ext}^N)
  \\
  &= \inf\sigma(H_\mathrm{ext}^N)
  \,.
\end{aligned} 
\end{equation}
Here the first equality holds because $H_\mathrm{int}^N$
is an operator with compact resolvent.
To estimate the lowest point in 
the spectrum of $H_\mathrm{ext}^N$,
we use the crude bound $G^{-1} \geq \diag(0,1,1)$.
Then, for every $\psi \in \dom h_\mathrm{ext}^N$,
$$
\begin{aligned}
  h_\mathrm{ext}^N[\psi] 
  &\geq \|\nabla_{\!t}\psi\|_{\Hilbert_\mathrm{ext}}^2
  \\
  &\geq \big(\inf_{I_\mathrm{ext}\times\omega} f \big) \
  \|\nabla_{\!t}\psi\|_{\sii(I_\mathrm{ext}\times\omega)}^2
  \\
  &\geq E_1 \, \big(\inf_{I_\mathrm{ext}\times\omega} f \big) \ 
  \|\psi\|_{\sii(I_\mathrm{ext}\times\omega)}^2
  \\
  &\geq E_1 \, 
  \frac{\displaystyle\inf_{I_\mathrm{ext}\times\omega} f}
  {\displaystyle\sup_{I_\mathrm{ext}\times\omega} f} \, 
  \|\psi\|_{\Hilbert_\mathrm{ext}}^2
  \,.
\end{aligned} 
$$
Here the third inequality follows from~\eqref{Poincare}
with the help of Fubini's theorem. 
Using the first of the hypotheses~\eqref{Ass.decay}, 
we see that the spectrum~$H_\mathrm{ext}^N$
is estimated from below by~$E_1$ times a function of~$s_0$ 
tending to~$1$ as $s_0 \to \infty$.
Since~$s_0$ can be chosen arbitrarily large,
it follows from~\eqref{bracketing} that 
$\inf\sigma_\mathrm{ess}(H) \geq E_1$.

\fbox{$\sigma_\mathrm{ess}(H) \supset [E_1,\infty)$}
(Weyl criterion adapted to quadratic forms)

By the classical Weyl criterion
\cite[Thm.~7.22]{Weidmann},
it suffices to construct, 
for each $k \in \Real$,
a sequence 
$$
  \{\psi_n\}_{n=1}^\infty \subset \dom H
  \quad \mbox{such that} \quad
  \begin{aligned}
    \mbox{(i)} \quad & \liminf_{n\to\infty} \|\psi_n\|_\mathcal{H} > 0
    \,,
    \\
    \mbox{(ii)} \quad & \|[H-(k^2+E_1)]\psi_n\|_\mathcal{H} 
    \xrightarrow[n\to\infty]{}0 
    \,.
  \end{aligned}
$$ 
One is tempted to use a tensor product
of plane waves ``localised at infinity''
and the first Dirichlet eigenfunction~$\mathcal{J}_1$
in the cross-section.
Namely, let us consider~\eqref{singular.Weyl}.
Note that 
\begin{equation}\label{varphi}
\begin{aligned}
  \|\varphi_n\|_{\sii(\Real)} 
  &= \|\varphi\|_{\sii(\Real)} = 1 \,,
  \\
  \|\varphi_n'\|_{\sii(\Real)} 
  &= n^{-1} \, \|\varphi'\|_{\sii(\Real)} \to 0  \,,
  \\
  \|\varphi_n''\|_{\sii(\Real)} 
  &= n^{-2} \, \|\varphi''\|_{\sii(\Real)} \to 0 \,,
\end{aligned} 
\qquad\mbox{and}\qquad
  \inf\supp\varphi_n 
  \to \infty 
\end{equation}
as $n \to \infty$.
Then~(i) follows at once by using~\eqref{1<G<1}
and the fact that 
$\|\psi_n\|_{\Hilbert_0}=1$,
where $\Hilbert_0 := \sii(\Omega_0)$.
Moreover, under additional assumptions
about the decay of~$\kappa$ and $\theta'$ at infinity
(involving derivatives),
it is indeed possible to show that~$\psi_n$
is the desired sequence.

To avoid the additional assumptions,
we use the Weyl criterion adapted to quadratic forms 
\cite[App.]{KL}
requiring to construct a sequence
$$
  \{\psi_n\}_{n=1}^\infty \subset \dom h
  \quad \mbox{such that} \quad
  \begin{aligned}
    \mbox{(i)} \quad & 
    \liminf_{n\to\infty}\|\psi_n\|_\mathcal{H} > 0
    \,,
    \\
    \mbox{(ii')} \quad & \|[H-(k^2+E_1)]\psi_n\|_{\mathcal{H}_{-1}} 
    \xrightarrow[n\to\infty]{}0 
    \,.
  \end{aligned}
$$ 
The advantage is that the sequence is required to belong 
to the form domain instead of the operator domain.
What is more, the weaker convergence in the dual space
$\Hilbert_{-1} := \Hilbert_1^*$ is required,
where $\Hilbert_1 := \dom h$ is equipped with 
the norm~$\vertiii{\cdot}$ introduced below~\eqref{form}.
Note that $\Hilbert_1 \subset \Hilbert=\Hilbert^* \subset \Hilbert_{-1}$
and 
$$
   \|[H-(k^2+E_1)]\psi_n\|_{\mathcal{H}_{-1}} 
   = \sup_{\stackrel[\phi\not=0]{}{\phi\in\Hilbert_1}} 
   \frac{|h(\phi,\psi_n) - (k^2+E_1) (\phi,\psi_n)_{\Hilbert}|}
   {\|\phi\|_{\Hilbert_1}}
   \,.
$$
Therefore checking~(ii') reduces to an analysis 
on the level of the form~$h$. 

Let $\phi \in C_0^1(\Omega_0)$, a core of~$\mathcal{H}_1$.
We write  
\begin{multline*}
  |h(\phi,\psi_n) - (k^2+E_1) (\phi,\psi_n)_{\Hilbert}|
  \\
  \leq |h_1(\phi,\psi_n) - k^2 \, (\phi,\psi_n)_{\Hilbert}|
  + |h_2(\phi,\psi_n) - E_1 \, (\phi,\psi_n)_{\Hilbert}|
  \,,
\end{multline*}
where (with any $\psi \in \Hilbert_1$)
$$
  h_1[\psi] := 
  \| f^{-1} (\partial_s\psi - \theta'\partial_\tau \psi)\|_{\Hilbert}^2
  \qquad \mbox{and} \qquad
  h_2[\psi] := \|\nabla_{\!t}\psi\|_{\Hilbert}^2
  \,.
$$

For the ``transverse'' form~$h_2$, 
a repeated integration by parts 
using that $f(s,t)$ is linear in~$t$
yields
$$
\begin{aligned}
  h_2(\phi,\psi_n) - E_1 \, (\phi,\psi_n)_{\Hilbert}
  &= -(\phi,\nabla_{\!t}\psi_n \cdot \nabla_{\!t}f)_{\Hilbert_0}
  \\
  &= (\nabla_{\!t}\phi,\psi_n\nabla_{\!t}f)_{\Hilbert_0}
  = \Big(\nabla_{\!t}\phi,\psi_n\frac{\nabla_{\!t}f}{f}\Big)_{\Hilbert}
  \,.
\end{aligned}  
$$
Consequently, 
$$
\begin{aligned}
  \frac{|h_2(\phi,\psi_n) - E_1 \, (\phi,\psi_n)_{\Hilbert}|}
  {\|\phi\|_{\Hilbert_1}}
  &\leq \frac{\|\nabla_{\!t}\phi\|_\Hilbert}{\|\phi\|_{\Hilbert_1}} 
  \, \|\psi_n\|_\Hilbert \,
  \left\| \frac{\nabla_{\!t}f}{f} \right\|_{\infty,n}
  \\
  &\leq \|f\|_{\infty,n}^{1/2}
  \left\| \frac{\nabla_{\!t}f}{f} \right\|_{\infty,n}
  \xrightarrow[n\to\infty]{}0
  \,, 
\end{aligned}  
$$
where 
$
  \| \cdot \|_{\infty,n} := \| \cdot \|_{L^\infty(\supp\varphi_n\times\omega)}
$.
The convergence holds because 
the first assumption of~\eqref{Ass.decay} 
implies that
$\left\| \nabla_{\!t}f\right\|_{\infty,n} \to 0$
as $n \to \infty$.

For the ``longitudinal'' form~$h_1$, 
we further decompose
$$
  h_1(\phi,\psi_n) - k^2 \, (\phi,\psi_n)_{\Hilbert}
  = I_1 + I_2 + I_3 + I_4
  \,,
$$
where
$$
\begin{aligned}
  I_1 &:=  (f^{-1}\partial_s \phi,
  f^{-1}\partial_s \psi_n)_\Hilbert
  - k^2 \, (\phi,\psi_n)_{\Hilbert}
  \,,
  \\
  I_2 &:= - (f^{-1}\partial_s \phi,
  f^{-1}\theta'\partial_\tau \psi_n)_\Hilbert
  \,,
  \\
  I_3 &:= - (f^{-1}\theta'\partial_\tau \phi,
  f^{-1}\partial_s \psi_n)_\Hilbert
  \,,
  \\
  I_4 &:= (f^{-1}\theta'\partial_\tau \phi,
  f^{-1}\theta'\partial_\tau \psi_n)_\Hilbert
  \,.
\end{aligned}  
$$
We start to estimate the last integral as follows:
$$
\begin{aligned}
  \frac{|I_4|}{\|\phi\|_{\Hilbert_1}}
  &\leq a^2 \, \left\| \frac{\theta'}{f} \right\|_{\infty,n}^2
  \frac{\|\nabla_{\!t} \phi\|_\Hilbert}{\|\phi\|_{\Hilbert_1}} 
  \, \|\nabla_{\!t} \psi_n\|_\Hilbert 
  \\
  &\leq a^2 \, \left\| \frac{\theta'}{f} \right\|_{\infty,n}^2 
  \|f\|_{\infty,n}^{1/2} \, \sqrt{E_1}
  \xrightarrow[n\to\infty]{}0
  \,, 
\end{aligned} 
$$
where the convergence holds due to the second assumption of~\eqref{Ass.decay}.
Similarly,
$$
\begin{aligned}
  \frac{|I_3|}{\|\phi\|_{\Hilbert_1}}
  &\leq a \, \left\| \frac{\theta'}{f^2} \right\|_{\infty,n}
  \frac{\|\nabla_{\!t} \phi\|_\Hilbert}{\|\phi\|_{\Hilbert_1}} 
  \, \|\partial_s \psi_n\|_\Hilbert 
  \\
  &\leq a \, \left\| \frac{\theta'}{f} \right\|_{\infty,n}^2 
  \|f\|_{\infty,n}^{1/2} \, 
  \big(\|\varphi_n'\|_{\sii(\Real)} + |k| \, \|\varphi_n\|_{\sii(\Real)} \big) 
  \xrightarrow[n\to\infty]{}0
\end{aligned} 
$$
and
$$
\begin{aligned}
  \frac{|I_2|}{\|\phi\|_{\Hilbert_1}}
  &\leq a \, \left\| \frac{\theta'}{f^2} \right\|_{\infty,n}
  \frac{\|\partial_s \phi\|_\Hilbert}{\|\phi\|_{\Hilbert_1}} 
  \, \|\nabla_{\!t} \psi_n\|_\Hilbert 
  \\
  &\leq a \, \left\| \frac{\theta'}{f} \right\|_{\infty,n}^2 
  \|f\|_{\infty,n} \, C_-^{-1}
  \, \sqrt{E_1}
  \xrightarrow[n\to\infty]{}0
  \,,
\end{aligned} 
$$
where~$C_-$ is the constant from Proposition~\ref{Prop.equivalence}.
Finally, we write $I_1 = I_1^{(a)} + I_1^{(b)}$, where
$$
\begin{aligned}
  I_1^{(a)}  
  := \ & (\partial_s \phi,\partial_s \psi_n)_{\Hilbert_0}
  - k^2 \, (\phi,\psi_n)_{\Hilbert_0}
  \\
  = \ & (\phi,-\partial_s^2 \psi_n)_{\Hilbert_0}
  - k^2 \, (\phi,\psi_n)_{\Hilbert_0}
  \\
  = \ & -\big(\phi,(\varphi_n''+2ik\varphi_n') \, e^{iks} \, \mathcal{J}_1
  \big)_{\Hilbert_0}
\end{aligned}  
$$
and
$$
  I_1^{(b)}  
  :=  \big(\partial_s \phi,(f^{-1}-1)\partial_s \psi_n\big)_{\Hilbert_0}
  - k^2 \, \big(\phi,(f-1)\psi_n\big)_{\Hilbert_0}
  \,.
$$
Consequently,
$$
\begin{aligned}
  \frac{|I_1^{(a)}|}{\|\phi\|_{\Hilbert_1}}
  &\leq \frac{\|\phi\|_{\Hilbert_0}}{\|\phi\|_{\Hilbert_1}} \,
  \big(\|\varphi_n''\|_{\sii(\Real)} 
  + 2 \, |k| \, \|\varphi_n'\|_{\sii(\Real)} \big) 
  \\
  &\leq 
  \left\| \frac{1}{f} \right\|_{\infty,n}^{1/2}
  \big(\|\varphi_n''\|_{\sii(\Real)} 
  + 2 \, |k| \, \|\varphi_n'\|_{\sii(\Real)} \big) 
  \xrightarrow[n\to\infty]{}0
  \,,
\end{aligned}
$$
where the convergence holds due to~\eqref{varphi},
and 
$$
\begin{aligned}
  \frac{|I_1^{(b)}|}{\|\phi\|_{\Hilbert_1}}
  &\leq \frac{\|\partial_s \phi\|_{\Hilbert_0}}{\|\phi\|_{\Hilbert_1}} 
  \left\| \frac{1}{f} -1 \right\|_{\infty,n}
  \|\partial_s \psi_n\|_{\Hilbert_0}
  + k^2 \,
  \frac{\|\phi\|_{\Hilbert_0}}{\|\phi\|_{\Hilbert_1}} \,
  \|\psi_n\|_{\Hilbert_0}
  \\
  &\leq
  C_-^{-1} \, 
  \left\| \frac{1}{f} -1 \right\|_{\infty,n}
  \big(\|\varphi_n'\|_{\sii(\Real)} + |k| \, \|\varphi_n\|_{\sii(\Real)} \big)
  +  k^2 \, C_-^{-1} \xrightarrow[n\to\infty]{}0
  \,,
\end{aligned}  
$$
where the convergence holds
due to the first assumption of~\eqref{Ass.decay}.
\end{proof}

A brief history of Theorem~\ref{Thm.ess} is as follows.
In 1995,
Duclos and Exner~\cite{DE} established~\eqref{stability.ess}
for tubes of the cross-section being a disk,
assuming no twisting $\theta' = 0$ 
and under extra hypotheses about the decay 
of the curvature~$\kappa$ at infinity.
In 2005,
Chenaud, Duclos, Freitas and Krej\v{c}i\v{r}\'ik \cite{ChDFK}
considered tubes of arbitrary cross-sections,
removed the extra decay hypotheses,
but still assumed that there is no twisting.
In 2008,
Krej\v{c}i\v{r}\'ik \cite{K6-with-erratum}
established Theorem~\ref{Thm.ess}
under the present minimal hypotheses
(see \cite{KL} for generalisations to tubes 
of arbitrary dimension and codimension).

\begin{figure}[h!t]
\begin{center}
\begin{tabular}{cc}
& \bigskip \\
\includegraphics[width=0.47\textwidth]{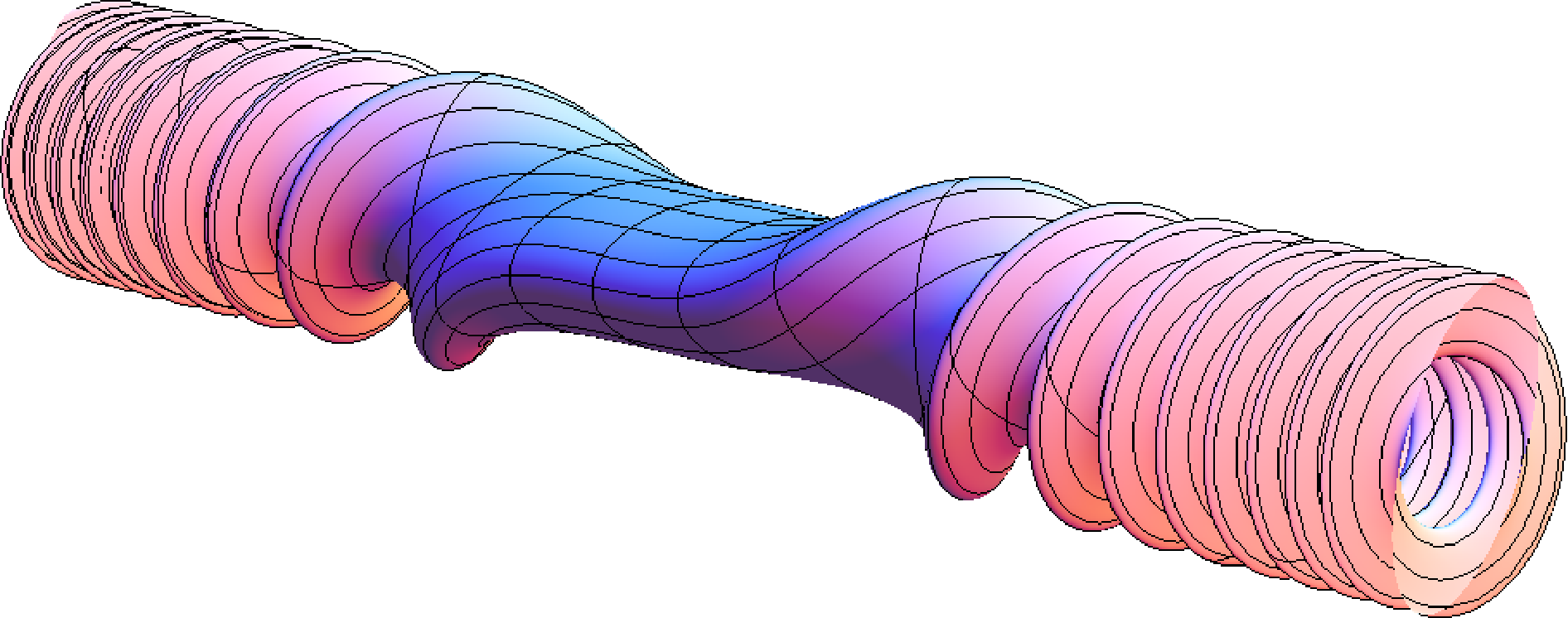}
& \includegraphics[width=0.47\textwidth]{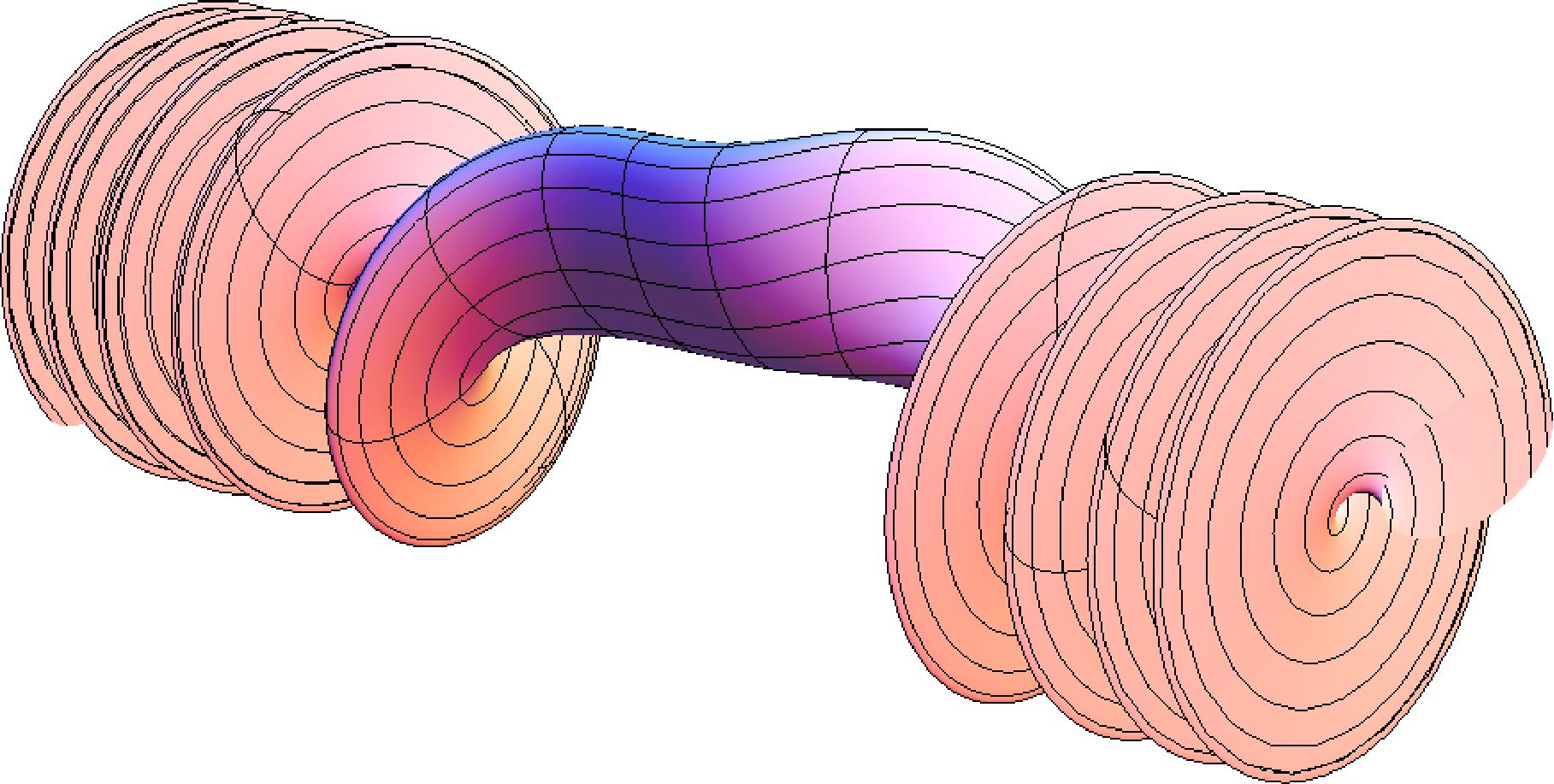}
\end{tabular}
\end{center}
\caption{Tubes with diverging twisting~\cite{K11}. 
A non-empty essential spectrum (left)
and purely discrete spectrum (right).}
\label{Fig.diverge}
\end{figure}

Theorem~\ref{Thm.ess} is optimal in the sense that 
there are examples of Euclidean tubes for which 
$
  \inf\sigma_\mathrm{ess}(-\Delta_D^{\Omega}) < E_1
$
or 
$
  \inf\sigma_\mathrm{ess}(-\Delta_D^{\Omega}) > E_1
$ 
if~$\Omega$ is not twisted but periodically bent 
\cite{KKriz,K6-with-erratum}
or not bent but periodically twisted
\cite{EKov_2005,K6-with-erratum,
Briet-Kovarik-Raikov-Soccorsi_2009}, 
respectively.
What is more, the spectrum can be purely discrete 
for tubes with asymptotically diverging twisting~\cite{K11},
see Figure~\ref{Fig.diverge}.

\begin{OProblem}
In the quasi-bounded realisation of Figure~\ref{Fig.diverge},
what are the \emph{Weyl-type asymptotics} 
for the accumulation of the discrete eigenvalues at infinity?
\end{OProblem}

Here the Weyl law must necessarily be non-standard
because $|\Omega|=\infty$.
In this direction, a Berezin-type upper bound for the eigenvalue moments
has been established in \cite{Barseghyan-Khrabustovskyi_2019}. 
 
\subsection{The effect of bending}\label{Sec.bend}
Now we restrict ourselves to tubes which are bent ($\kappa\not=0$)
but untwisted ($\theta'=0$),
see Figure~\ref{Fig.bend}.

\begin{figure}[h!]
\begin{center}
\includegraphics[width=0.6\textwidth]{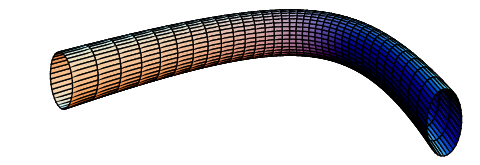}
\end{center}
\caption{A bent untwisted tube.}\label{Fig.bend}
\end{figure}

It turns out that bending is \emph{supercritical} 
in the sense that it gives rise to a spectrum below the energy~$E_1$.
\begin{theo}[Supercriticality]\label{Thm.bending}
Let $\kappa \not= 0$ and $\theta' = 0$.
Then
$$
  \inf\sigma(-\Delta_D^\Omega) < E_1
  \,.
$$
\end{theo}
\begin{proof}
The main idea is that $(s,t)\mapsto\mathcal{J}_1(t)$
is a ``zero point'' of the functional 
$Q[\psi] := h[\psi] - E_1 \, \|\psi\|^2$,
which is not a ``minimiser''.

\fbox{Zero point}
In detail, for every $n \in \Nat^*$, set
$\psi_n(s,t):=\varphi_n(s)\mathcal{J}_1(t)$,
where the sequence $\{\varphi_n\}_{n \in \Nat^*}$ 
is defined in~\eqref{trial}.
We claim that
\begin{equation}\label{first}
  Q[\psi_n] := h[\psi_n] - E_1 \, \|\psi_n\|^2
  \xrightarrow[n\to\infty]{} 0
  \,,
\end{equation}
as in the proof of Proposition~\ref{Prop.critical}.
Indeed, the ``longitudinal energy'' of~\eqref{form.bis}
disappears as $n \to \infty$:
$$
  \|\partial_s\psi_n\|^2 
  = \int_{\Omega_0} 
  \frac{|\varphi_n'(s)|^2 \, |\mathcal{J}_1(t)|^2}{f(s,t)} 
  \, \der s \, \der t
  \leq \frac{\|\varphi_n'\|_{\sii(\Real)}^2 
  \, \|\mathcal{J}_1\|_{\sii(\omega)}^2   }
  {1-a \, \|\kappa\|_\infty} 
  \xrightarrow[n\to\infty]{} 0 \,.
$$
The striking fact behind~\eqref{first} is that 
the ``transverse energy'' is compensated by~$E_1$,
despite the presence of the Jacobian~$f$ 
in the Hilbert space~\eqref{Hilbert}.
This follows by an explicit integration by parts 
and the observation that $f(s,t)$ is \emph{linear} in~$t$:
$$
\begin{aligned}
  \|\nabla_{\!t}\psi_n\|^2 - E_1 \, \|\psi_n\|^2
  &= \int_{\Omega_0} 
  |\varphi_n(s)|^2 \, |\nabla_{\!t}\mathcal{J}_1(t)|^2 
  \, f(s,t) \, \der s \, \der t
  \\
  & \qquad - E_1 \int_{\Omega_0} 
  |\varphi_n(s)|^2 \, |\mathcal{J}_1(t)|^2 
  \, f(s,t) \, \der s \, \der t
  \\
  &= - \int_{\Omega_0} 
  |\varphi_n(s)|^2 \, \mathcal{J}_1(t) \, \nabla_{\!t}\mathcal{J}_1(t) 
  \cdot
  \nabla_{\!t} f(s,t) \, \der s \, \der t
  \\
  &= - \frac{1}{2} \int_{\Omega_0} 
  |\varphi_n(s)|^2 \, \nabla_{\!t}|\mathcal{J}_1(t)|^2 
  \cdot
  \nabla_{\!t} f(s,t) \, \der s \, \der t
  \\
  &= \frac{1}{2} \int_{\Omega_0} 
  |\varphi_n(s)|^2 \, |\mathcal{J}_1(t)|^2 
  \cdot
  \Delta_{t} f(s,t) \, \der s \, \der t = 0 \,.
\end{aligned}  
$$

\fbox{Not minimiser}
In the second step, one considers a small perturbation
$$
  \psi_{n,\eps}(s,t) := \psi_{n}(s,t) + \eps \, \phi(s,t)
  \,, 
$$
where $\eps \in \Real$ and $\phi \in \dom h$ is real-valued.
Then
\begin{equation}\label{quadratic}
  Q[\psi_{n,\eps}] = Q[\psi_{n}] + 2 \, \eps \, Q(\psi_n,\phi)
  + \eps^2 \, Q[\phi] \,.
\end{equation}
Here the first term on the right-hand side is independent of~$\eps$
and vanishes as $n \to \infty$.
The last term is independent of~$n$
and can be made negligible with respect to the middle term
by taking~$\eps$ small.
It remains to show that~$\phi$ can be chosen 
in such a way that  
$$
  \left.
  \frac{\partial Q[\psi_{n,\eps}]}{\partial \eps} 
  \right|_{\eps = 0}
  = 2 \, Q(\psi_n,\phi)
  < 0 
$$
for all sufficiently large~$n$.
In fact, by taking 
$$
  \phi(s,t) := j(s) \, \xi(t) \, \mathcal{J}_1(t)
  \,,
$$
where $j \in C_0^\infty(\Real)$ and~$\xi$ will be determined 
in a moment, $Q(\psi_n,\phi)$ is independent of~$n$ 
for all sufficiently large~$n$. 
This follows from the fact $\varphi_n=1$ 
on the support of~$j$ for all sufficiently large~$n$. 
More specifically, we write $Q=Q_1+Q_2$, where
$$
  Q_1(\psi_n,\phi) 
  := \big(f^{-1}\partial_s \psi_n,f^{-1}\partial_s\phi\big)
  = \int_{\Omega_0} 
  \frac{\varphi_n'(s) \, \mathcal{J}_1(t) \, \partial_s\phi(s,t)}
  {f(s,t)} \, \der s \, \der t 
  = 0
$$
for all sufficiently large~$n$
(because $\varphi_n'=0$ on the support of~$j$)
and
$$
\begin{aligned}
  Q_2(\psi_n,\phi) 
  :=\ & (\nabla_{\!t} \psi_n,\nabla_{\!t} \phi)
  - E_1 \, (\psi_n,\phi)
  \\
  =\ &  \int_{\Omega_0} 
  \varphi_n(s) \, \nabla_{\!t}\mathcal{J}_1(t) 
  \cdot \nabla_{\!t} \phi(s,t)
  \, f(s,t) \, \der s \, \der t
  \\
  & \qquad - E_1 \int_{\Omega_0} 
  \varphi_n(s) \, \mathcal{J}_1(t)
  \, \phi(s,t) \, f(s,t) \, \der s \, \der t
  \\
  =\ &
  -\int_{\Omega_0} 
  \varphi_n(s) \, \nabla_{\!t}\mathcal{J}_1(t) 
  \, \phi(s,t) \cdot
  \nabla_{\!t} f(s,t) \, \der s \, \der t
  \\
  =\ &
  -\frac{1}{2} \int_{\Omega_0} 
  \nabla_{\!t}|\mathcal{J}_1(t)|^2 
  \, j(s) \, \xi(t) \cdot
  \nabla_{\!t} f(s,t) \, \der s \, \der t
  \\
  =\ &
  \frac{1}{2} \int_{\Omega_0} 
  |\mathcal{J}_1(t)|^2 
  \, j(s) \, \nabla_{\!t}\xi(t) \cdot
  \nabla_{\!t} f(s,t) \, \der s \, \der t
  \,,
\end{aligned} 
$$  
where the last but one equality holds for all sufficiently large~$n$
(because $\varphi_n=1$ on the support of~$j$). 
Choosing (note that~$\theta$ is necessarily constant 
by the hypothesis $\theta'=0$)
$$
  \xi(t) := 
  \begin{pmatrix}
  \alpha_1 & \alpha_2
  \end{pmatrix}
  \begin{pmatrix}
    \cos\theta & \sin\theta \\
    -\sin\theta & \cos\theta 
  \end{pmatrix}
  \begin{pmatrix}
    t_1 \\ t_2
  \end{pmatrix}
$$
with any $\alpha_1,\alpha_2 \in \Real$
and recalling that~$\mathcal{J}_1$ is normalised to~$1$ in $\sii(\omega)$, 
we arrive at
$$
  Q_2(\psi_n,\phi) = \frac{1}{2} \int_\Real 
  j(s) \, [\alpha_1 k_1(s) + \alpha_2 k_2(s)] \, \der s
$$
for all sufficiently large~$n$.
We claim that there exists a bounded interval $I \subset \Real$
and numbers $\alpha_1,\alpha_2 \in \Real$ such that 
$$
  \alpha_1 k_1 + \alpha_2 k_2 < 0 
  \qquad \mbox{on} \qquad I
  \,.
$$
Indeed, this follows by the hypothesis $\kappa \not=0$,
the relationship~\eqref{curvature} and the continuity of~$\kappa$.  
Choosing~$j$ non-negative and $\supp j := I$,
we eventually get the desired result that 
$$
  \mbox{$Q_2(\psi_n,\phi)$ is negative and independent of~$n$
  for all sufficiently large~$n$.}
$$

\fbox{Conclusion} 
In summary, the sum of the last two terms 
on the right-hand side of~\eqref{quadratic} 
are independent of~$n$ for all sufficiently large~$n$ 
and can be made negative by choosing~$\eps$ positive 
and sufficiently small. 
Then we choose~$n$ so large that the sum with the first term 
on the right-hand side of~\eqref{quadratic} remains negative.
\end{proof}

For asymptotically straight tubes, 
Theorem~\ref{Thm.bending} 
together with the stability of 
the essential spectrum (Theorem~\ref{Thm.ess})
implies the existence of discrete eigenvalues,
see Figure~\ref{Fig.bend.spec}.

\begin{coro}\label{Corol.bending} 
Let $\kappa \not= 0$ and $\theta' = 0$.
Assume in addition that
$\displaystyle \lim_{|s|\to\infty} \kappa(s) = 0$.
Then
$$
  \sigma_\mathrm{disc}(-\Delta_D^\Omega) \cap (0,E_1) 
  \not= \varnothing
  \,.
$$
\end{coro}
\begin{figure}[h]
\begin{center}
\includegraphics[width=0.4\textwidth]{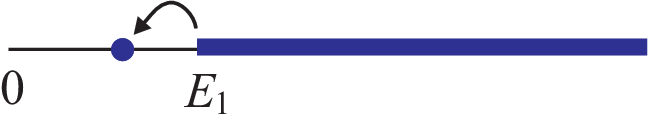}
\caption{Spectrum of a bent tube.}\label{Fig.bend.spec}
\end{center}
\end{figure}

The moral of this section is that 
$$
  \fbox{bending acts as an attractive interaction}
$$
in the sense that it diminishes the spectrum
with respect to the spectrum of 
straight tubes (\cf~Theorem~\ref{Thm.bending}). 
If the bent tube is asymptotically straight,
then the spectrum below~$E_1$ is discrete
(\cf~Corollary~\ref{Corol.bending}).
For tubes which are not asymptotically straight, however,
the effect of bending can be so strong that it pushes down
the essential spectrum too. For instance, 
this always happens for periodically bent tubes
(\ie~$\theta'=0$ and $\kappa\not=0$ is periodic).
This is clear from Theorem~\ref{Thm.bending} 
and the fact that periodic systems admit 
no discrete eigenvalues. 
 
A brief history of the effect of bending is as follows. 
In 1992,
physicists Goldstone and Jaffe~\cite{GJ} 
essentially devised the present variational proof.
In 1995,
the proof was mathematically rectified 
by Duclos and Exner~\cite{DE}.
However, their trial function requires an extra smoothness
of the base curve.
In 2005,
Chenaud, Duclos, Freitas and Krej\v{c}i\v{r}\'ik \cite{ChDFK}
removed this extra requirement.
In 2025,
Baldelli and Krej\v{c}i\v{r}\'ik \cite{BK5}
generalised Theorem~\ref{Thm.bending} 
to the nonlinear setting of the $p$-Laplacian.
 
The supercritical effect of bending is unwanted for quantum waveguides,
where the discrete eigenvalues correspond to \emph{bound states}.
Contrary to what happens in classical physics, 
a quantum particle gets trapped in a bent tube.
To make the transport in the waveguide more stable,
the following question arises:
\begin{center}
\fbox{How to get rid of the discrete spectrum ?}
\end{center}

\subsection{The effect of twisting}\label{Sec.twist}
Now we restrict ourselves to tubes 
which are twisted ($\theta'\not=0$)
but unbent ($\kappa=0$),
see Figure~\ref{Fig.twist}.

\begin{figure}[h!]
\begin{center}
\includegraphics[width=0.6\textwidth]{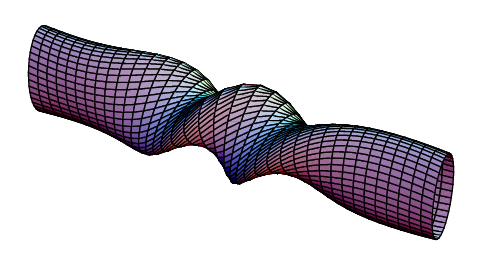}
\end{center}
\caption{A twisted unbent tube.}\label{Fig.twist}
\end{figure}

The effect of twisting is more subtle because of the following result.
\begin{prop}\label{Prop.twist}
Let $\kappa = 0$.
Assume in addition 
$\displaystyle \lim_{|s|\to\infty} \theta'(s) = 0$.
Then
$$
  \sigma(-\Delta_D^\Omega) =
  \sigma_\mathrm{ess}(-\Delta_D^\Omega) 
  = [E_1,\infty)
  \,.
$$
\end{prop}

\begin{proof}
By Theorem~\ref{Thm.ess}, $[E_1,\infty) \subset \sigma(H)$.
The opposite inclusion follows from the fact that $f=1$ if $\kappa=0$: 
\begin{equation}\label{Poincare.pre} 
  h[\psi] \geq \|\nabla_{\!t}\psi\|^2 
  = \int_{\Omega_0} |\nabla_{\!t}\psi(s,t)|^2 \, \der s \, \der t 
  \geq  E_1 \int_{\Omega_0} |\psi(s,t)|^2 \, \der s \, \der t 
  = E_1 \, \|\psi\|^2 
\end{equation}
for every $\psi \in \dom h$.
Here the second inequality follows from
the Poincar\'e inequality~\eqref{Poincare}
with the help of Fubini's theorem.
\end{proof}

Proposition~\ref{Prop.twist} implies that, 
for asymptotically straight unbent tubes, 
the spectrum is purely essential and coincides with that 
of straight tubes, see Figure~\ref{Fig.straight.spec}. 
So, there is no effect twisting on the first glance.
Contrary to the case of straight tubes
(\cf~Proposition~\ref{Prop.critical}), 
however, twisted tubes are \emph{subcritical}
in the sense that
there is always a Hardy-type inequality 
whenever the tube is non-trivially twisted.  

\begin{theo}[Local Hardy inequality]\label{Thm.Hardy.local}
Let $\kappa = 0$. 
Assume that $\theta'\not=0$ and~$\omega$ is not circular.
Let $I \subset \Real$ be any bounded interval on which 
$\theta'$~is not identically equal to zero.
Then there is a positive constant~$\lambda_1^I$ 
(depending on $\theta'\upharpoonright I$ and~$\omega$)
such that 
\begin{equation}\label{Hardy.local}
  H - E_1 \geq \lambda_1^I \ \chi_{I \times \omega} \,.
\end{equation}
\end{theo}
\begin{proof}
Given a bounded interval $I \subset \Real$,
let us consider the quadratic form in $\sii(I\times\omega)$
defined by 
$$
\begin{aligned}
  h^I[\psi] &:= 
  \|\partial_s\psi-\theta'\partial_\tau\psi\|_{\sii(I\times\omega)}^2
  + \|\nabla_{\!t}\psi\|_{\sii(I\times\omega)}^2 
  \,,
  \\
  \dom h^I &:= \big\{
  \psi\!\upharpoonright\! (I\times\omega)
  : \ \psi\in W_0^{1,2}(\Real\times\omega)
  \big\}
  \,.
\end{aligned}  
$$
The corresponding operator~$H^I$ acts as~\eqref{LB} 
and satisfies Neumann boundary conditions on $(\partial I) \times \omega$. 
Since~$I$ is bounded, $H^I$~is an operator with compact resolvent.
Consequently, the infimum
$$
  \lambda_1^I := 
  \inf_{\stackrel[\psi\not=0]{}{\psi\in \dom h^I}}
  \frac{h^I[\psi] - E_1 \, \|\psi\|_{\sii(I\times\omega)}^2}
  {\,\|\psi\|_{\sii(I\times\omega)}^2} 
$$
is achieved by a positive function $\psi_1 \in \dom h^I$. 
One immediately has 
$$
  h[\psi] - E_1 \, \|\psi\|^2
  \geq h^I[\psi] - E_1 \, \|\psi\|_{\sii(I\times\omega)}^2
  \geq \lambda_1^I \, \|\psi\|_{\sii(I\times\omega)}^2
  = \lambda_1^I \ \|\chi_I \, \psi\|^2
$$
for every $\psi \in \dom h$,
similarly as in~\eqref{Poincare.pre} 
crucially employing that $f=1$ if the tube is unbent.
We claim that $\lambda_1^I$ is positive
if the interval~$I$ is chosen in such a way
that $\theta'\not=0$ on~$I$.  
By contradiction, assume that $\lambda_1^I=0$.
Then 
\begin{equation}\label{2eqs}  
\begin{aligned}
  \|\partial_s\psi_1-\theta'\partial_\tau\psi_1\|_{\sii(I\times\omega)}^2
  &=0 
  \,, \\
  \|\nabla_{\!t}\psi_1\|_{\sii(I\times\omega)}^2
  - E_1 \, \|\psi_1\|_{\sii(I\times\omega)}^2 &= 0
  \,.
\end{aligned}  
\end{equation}
Writing $\psi_1(s,t) = \varphi(s)\,\mathcal{J}_1(t) + \phi(s,t)$,
where 
$
  (\mathcal{J}_1,\phi(s,\cdot))_{\sii(\omega)} = 0
$
for almost every $s \in I$,
we deduce from the second equality in~\eqref{2eqs} that $\phi=0$.
The first identity in~\eqref{2eqs} is then equivalent to
\begin{multline*}
  \|\varphi'\|_{\sii(I)}^2 \, 
  \|\mathcal{J}_1\|_{\sii(\omega)}^2
  + \|\theta' \varphi\|_{\sii(I)}^2 \, 
  \|\partial_\tau\mathcal{J}_1\|_{\sii(\omega)}^2
  \\
  - 2 \, (\mathcal{J}_1,\partial_\tau\mathcal{J}_1)_{\sii(\omega)}
  \, (\varphi',\theta'\varphi)_{\sii(I)}
  = 0
  \,.
\end{multline*}
Since 
%
  $(\mathcal{J}_1,\partial_\tau\mathcal{J}_1)_{\sii(\omega)}=0$
%
by an integration by parts, it follows that~$\varphi$ must be constant
and that
$$
  \|\theta'\|_{\sii(I)} = 0
  \qquad\mbox{or}\qquad
  \|\partial_\tau\mathcal{J}_1\|_{\sii(\omega)} = 0
  \,.
$$
However, this is impossible under the stated assumptions because
$\|\theta'\|_{\sii(I)}$ vanishes if, and only if,
$\theta'=0$ almost everywhere in~$I$,
and $\partial_\tau\mathcal{J}_1 = 0$ identically in~$\omega$
if, and only if, $\omega$~is circular.
\end{proof}

The inequality of Theorem~\ref{Thm.Hardy.local} is particularly 
interesting in the setting of Proposition~\ref{Prop.twist}.
Indeed, if the twist vanishes at infinity, the spectrum starts with~$E_1$,
so there can be no positive \emph{constant}~$c$ such that 
$H-E_1 \geq c$ (Poincar\'e inequality) holds.  
However, there can be a non-trivial non-negative 
\emph{function} (Hardy weight)
$\rho:\Real\times\omega \to \Real$ such that 
$H-E_1 \geq \rho$ (Hardy inequality) holds.

We call the Hardy inequality of Theorem~\ref{Thm.Hardy.local} \emph{local}
because the Hardy weight is compactly supported there.
However, there is always a way how to deduce a \emph{global} 
Hardy inequality (\ie~with an everywhere positive Hardy weight)
from the local one.
One approach is based on a standard argument of partition of unity 
subordinated to a finitely local covering 
\cite[Lem.~3.1]{Pinchover-Tintarev_2007},
which was first applied to twisted tubes in \cite{BK5}.
Here we present an alternative approach based on 
rather tedious estimates,  
which was first applied to twisted tubes in~\cite{EKK}
and under the present minimal hypotheses in \cite{K-Padova}. 
The advantage of the latter approach is that it
yields an explicit form for the Hardy weight. 

\begin{theo}[Global Hardy inequality or Subcriticality]\label{Thm.Hardy}
Let $\kappa = 0$. 
Assume that $\theta'\not=0$ and~$\omega$ is not circular.
Then there exists a positive constant~$c$ such that 
\begin{equation}\label{Hardy}
  H - E_1 \geq c \, \rho \,.
\end{equation}
where $\rho(s,t) := (1+s^2)^{-1}$.
\end{theo}
\begin{proof}
Let $\psi \in C_0^1(\Real\times\omega)$, a core of $\dom h$.

$\circ$
By Theorem~\ref{Thm.Hardy.local}, 
\begin{equation}\label{Hardy.local.bis}
  Q[\psi] := h[\psi] - E_1 \, \|\psi\|^2 
  \geq \lambda_1^I \, \|\chi_I \psi\|^2
  \,,
\end{equation}
where $\lambda_1^I$ is positive and $I \subset \Real$
is any bounded interval on which~$\theta'$ 
is not identically equal to zero.
Here we abbreviate $\chi_I := \chi_{I \times \omega}$.
Let us also recall the orthogonality relation
\begin{equation}\label{orthogonality}
  \Re\, (\psi,\partial_\tau\psi)_{\sii(\omega)} = 0
  \,,
\end{equation}
which was essentially used in the proof of Theorem~\ref{Thm.Hardy.local}
and follows by an integration by parts.

$\circ$
Let us write $I = (s_0-R,s_0+R)$,
so that $s_0 \in \Real$ is the centre of the interval~$I$
and $R > 0$ its half-width.
Let us define an auxiliary cut-off function 
$\eta \in C_0^\infty(\Real \setminus \{s_0\})$
such that $0 \leq \eta \leq 1$
and $\eta(s) := 1$ if $|s-s_0| > R$.
We denote by the same symbol~$\eta$ the function 
$(s,t) \mapsto \eta(s)$ on $\Real\times\omega$,
and similarly for its derivatives.
We write
\begin{equation*}
  \psi = \eta\psi + (1-\eta)\psi \,.
\end{equation*}

$\circ$
Applying this decomposition, 
one has the identity
\begin{multline*} 
  \|\partial_s\psi-\theta'\partial_\tau\psi\|^2
  = \|\partial_s(\eta\psi)-\theta'\partial_\tau(\eta\psi)\|^2
  + \|\partial_s((1-\eta)\psi)-\theta'\partial_\tau((1-\eta)\psi)\|^2
  \\
  + I_1 + I_2 + I_3 + I_4
  \,,
\end{multline*}
where 
$$
\begin{aligned}
  I_1 :=&\ 2 \, \Re \big(
  \partial_s(\eta\psi),
  \partial_s((1-\eta)\psi)
  \big)
  \\
  =&\ 2 \, \big(\partial_s\psi,\eta(1-\eta)\partial_s\psi\big)
  + 2 \, \Re\big(\psi,[\eta(1-\eta)]'\partial_s\psi\big) 
  + 2 \, \big(\psi,\eta'(1-\eta)'\psi\big)
  \\
  =&\ 2 \, \big(\partial_s\psi,\eta(1-\eta)\partial_s\psi\big)
  - \big(\psi,[\eta(1-\eta)]''\psi\big) 
  - 2 \, \big(\psi,\eta'^2\psi\big)
  \\
  \geq&\ 
  2 \, \underbrace{\big(\partial_s\psi,\eta(1-\eta)\partial_s\psi\big)}_{A \geq 0}
  - \underbrace{\|[\eta(1-\eta)]''+2\eta'^2\|_\infty}_{C>0} 
  \, \|\chi_I\psi\|^2 \,,
\end{aligned}  
$$
$$
  I_2 := 2 \, \Re \big(
  \theta'\partial_\tau(\eta\psi),
  \theta'\partial_\tau((1-\eta)\psi)
  \big)  
  = 2 \, \underbrace{\big(
  \theta'\partial_\tau\psi,
  \eta(1-\eta)\theta'\partial_\tau\psi)
  \big)}_{B \geq 0}  
  \,,
$$
(using~\eqref{orthogonality})
$$
\begin{aligned}
  I_3 := - 2 \, \Re \big(
  \theta'\partial_\tau(\eta\psi),
  \partial_s((1-\eta)\psi)
  \big)  
  &= - 2 \, \Re \big(
  \theta'\eta\partial_\tau\psi,
  (1-\eta)\partial_s\psi
  \big)  
  \\
  &\geq 
  - 2 \, \sqrt{A} \, \sqrt{B}
\end{aligned}  
$$
and, similarly,
$$
\begin{aligned}
  I_4 := - 2 \, \Re \big(
  \partial_s(\eta\psi),
  \theta'\partial_\tau((1-\eta)\psi)
  \big)  
  &= - 2 \, \Re \big(
  \eta\partial_s\psi,
  \theta'(1-\eta)\partial_\tau\psi 
  \big)  
  = I_3
  \\
  &\geq  
  - 2 \, \sqrt{A} \, \sqrt{B}
  \,.
\end{aligned}  
$$
Using that 
$
  2A+2B-4\sqrt{A}\sqrt{B} 
  = 2 \, \big(\sqrt{A}-\sqrt{B}\big)^2 
  \geq 0
$, 
we therefore arrive at 
\begin{multline*}
  \|\partial_s\psi-\theta'\partial_\tau\psi\|^2
  \\
  \geq \|\partial_s(\eta\psi)-\theta'\partial_\tau(\eta\psi)\|^2
  + \|\partial_s((1-\eta)\psi)-\theta'\partial_\tau((1-\eta)\psi)\|^2
  - C \, \|\chi_I\psi\|^2
  \,.  
\end{multline*}
Neglecting the second term on the right-hand side
and using~\eqref{Poincare}, we eventually get
\begin{equation}\label{Hardy.delta} 
  Q[\psi]
  \geq \|\partial_s(\eta\psi)-\theta'\partial_\tau(\eta\psi)\|^2
  - C \, \|\chi_I\psi\|^2
  \,.  
\end{equation}

$\circ$
Since $\eta\psi$ vanishes at $s=s_0$, 
one has a classical version of the Hardy inequality~\eqref{Hardy.1D}
for the first term on the right-hand side of~\eqref{Hardy.delta}: 
\begin{equation}\label{Hardy.1D.QW}
  \|\partial_s(\eta\psi)-\theta'\partial_\tau(\eta\psi)\|^2
  \geq \frac{1}{4} \, \left\|\frac{\eta\psi}{s-s_0}\right\|^2 
  \,,
\end{equation}
where~$s$ denotes the function $(s,t) \mapsto s$.
Indeed, for every $\alpha \in \Real$, one has
\begin{multline*}
  \left\|\partial_s(\eta\psi)-\theta'\partial_\tau(\eta\psi)
  - \alpha \frac{\eta\psi}{s-s_0}\right\|^2
  = \left\|\partial_s(\eta\psi)-\theta'\partial_\tau(\eta\psi)\right\|^2
  + \alpha^2 \left\|\frac{\eta\psi}{s-s_0}\right\|^2 
  \\
  + 2 \, \alpha \, \Re\left(\partial_s(\eta\psi)-\theta'\partial_\tau(\eta\psi),
  \frac{\eta\psi}{s-s_0}\right)
  \,.
\end{multline*}
Recalling~\eqref{orthogonality},
one has
$$
\begin{aligned}
  2 \, \Re\left(\partial_s(\eta\psi)-\theta'\partial_\tau(\eta\psi),
  \frac{\eta\psi}{s-s_0}\right)
  &= 2 \, \Re\left(\partial_s(\eta\psi),
  \frac{\eta\psi}{s-s_0}\right)
  \\
  &= \int_{\Omega_0} \frac{\partial_s|\eta\psi|^2}{s-s_0}
  = \int_{\Omega_0} \frac{|\eta\psi|^2}{(s-s_0)^2}
  \,,
\end{aligned}
$$
where the last equality follows by an integration by parts.
Choosing $\alpha := -1/2$, we get the desired claim.
In summary, plugging~\eqref{Hardy.1D.QW} into~\eqref{Hardy.delta}, 
we arrive at
\begin{equation}\label{Hardy.1D.bis}
  Q[\psi]
  \geq \frac{1}{4} \, \left\|\frac{\eta\psi}{s-s_0}\right\|^2 
  - C \, \|\chi_I\psi\|^2
  \geq \frac{1}{4} \, \left\|\tilde\rho \, \eta\psi\right\|^2 
  - C \, \|\chi_I\psi\|^2
  \,,
\end{equation}
where $\tilde{\rho}(s,t) := [1+(s-s_0)^2]^{-1/2}$.

$\circ$
The negative term on the right-hand side of~\eqref{Hardy.1D.bis}
can be controlled by~\eqref{Hardy.local.bis}.
In more detail, given any $\beta \in \Real$, 
let us interpolate between the estimates~\eqref{Hardy.local.bis}
and~\eqref{Hardy.1D.bis} as follows: 
$$
  Q[\psi] = (1-\beta) \, Q[\psi] + \beta \, Q[\psi] 
  \geq [(1-\beta) \, \lambda_1^I - C \, \beta] \, \|\chi_I\psi\|^2
  + \frac{\beta}{4} \, \left\|\tilde\rho \, \eta\psi\right\|^2 
  \,.
$$
Choosing~$\beta$ sufficiently small, for instance
$$
  \beta := 
  \frac{4\lambda_1^I}{1+4(\lambda_1^I+C)}
$$
which makes the constants standing in front 
of $\|\chi_I\psi\|^2$ and $\left\|\tilde\rho \, \eta\psi\right\|^2$ 
equal, we get
$$
  Q[\psi] \geq 
  \frac{\beta}{4}
  \left(
  \|\chi_I \psi\|^2 + \left\|\tilde\rho \, \eta\psi\right\|^2
  \right)
  \geq 
  \frac{\beta}{4} 
  \left(
  \|\chi_I \tilde\rho \, \psi\|^2 +\left\|\tilde\rho \, \eta\psi\right\|^2 
  \right)
  \geq 
  \frac{\beta}{4} \, \left\|\tilde\rho\psi\right\|^2
  \,.
$$

$\circ$
The last obtained inequality means 
$H-E_1 \geq \frac{\beta}{4}\tilde{\rho}$.
The ultimate inequality~\eqref{Hardy} follows by the choice
$$
  c := \frac{\beta}{4}
  \, \sup_{s \in \Real} \frac{1+s^2}{1+(s-s_0)^2}
  \,.
$$
This concludes the proof of Theorem~\ref{Thm.Hardy}.
\end{proof}

The moral of this section is that 
$$
  \fbox{twisting acts as a repulsive interaction}
$$
in the sense that it has the tendency to raise the spectrum
with respect to the spectrum of straight tubes. 
This is subtle for asymptotically straight twisted unbent tubes,
for which the spectrum as a set coincides with the spectrum
of straight tubes. 
In this case, the repulsiveness is quantified by the existence
of Hardy-type inequalities at the threshold~$E_1$ of the spectrum
(\cf~Theorem~\ref{Thm.Hardy}),
where the function~$\rho$ necessarily vanishes at infinity.

For tubes which are not asymptotically straight, however,
the effect of twisting can be so strong that it pushes up
the essential spectrum too. For instance, 
this always happens for periodically twisted tubes
(\ie~$\kappa=0$, 
$\omega$ is not circular
and $\theta'\not=0$ is periodic).
This is clear from the proof of Theorem~\ref{Thm.Hardy} 
yielding the Poincar\'e inequality
$H-E_1 \geq \lambda_1^I > 0$, 
where~$I$ is the period of twisting. 
What is more, the spectrum can be purely discrete 
for tubes with asymptotically diverging twisting~\cite{K11},
see Figure~\ref{Fig.diverge}.

The existence of the Hardy inequality in twisted waveguides
can be used to eliminate the unwanted bound states. 
Indeed, consider an unbent and non-trivially twisted tubes
satisfying the second hypothesis of~\eqref{Ass.decay}.
Starting to bend it under the first hypothesis of~\eqref{Ass.decay},
there will be no discrete spectrum for all sufficiently weak bendings
\cite{EKK,K6-with-erratum}. 
 
A brief history of the effect of twisting is as follows.
In 2008,
Ekholm, Kova\v{r}\'ik and Krej\v{c}i\v{r}\'ik \cite{EKK}
established Theorem~\ref{Thm.Hardy} under extra hypotheses.
Subsequently,
Krej\v{c}i\v{r}\'ik  \cite{K6-with-erratum}
devised an alternative, more robust approach, 
which yields the subcriticality under the present minimal hypotheses 
(see also~\cite{KZ1}).
In 2025,
Baldelli and Krej\v{c}i\v{r}\'ik \cite{BK5}
generalised Theorem~\ref{Thm.Hardy} 
to the nonlinear setting of the $p$-Laplacian
(and higher dimensions for the first time).

%
\bibliography{bib}
\end{document}